\documentclass[11pt, a4paper]{article}
\usepackage{amsmath}
\usepackage{amsthm}
\usepackage{bbm}
\usepackage{mathtools}
\usepackage{mathrsfs} 
\usepackage{tikz-cd}  
\usepackage[numbers]{natbib}
\usepackage{mathpazo}
\usepackage{amssymb}
\usepackage[T1]{fontenc}
\numberwithin{equation}{section}
\usepackage[margin=1in]{geometry} 
\usepackage{indentfirst}          
\usepackage{enumitem} 
\usepackage{hyperref} 
\hypersetup{
    colorlinks=true,
    linkcolor=blue,
    citecolor=red,
    urlcolor=blue
}
\usepackage{cleveref} 

\newtheorem{thm}{Theorem}[section]

\newtheorem{definition}[thm]{Definition}
\newtheorem{exm}[thm]{Example}

\newtheorem{lemma}[thm]{Lemma}
\newtheorem{cor}[thm]{Corollary}
\newtheorem{prop}[thm]{Proposition}

\newtheorem{rmk}[thm]{Remark}

\newcommand{\GG}{{^{G}_{G}\mathcal{YD}^{\Phi}}}

\newcommand{\BGGA}{{^{\mathbb{G}}_{\mathbb{G}}\mathcal{YD}^{\pi^*(\Phi_{\underline{a}})}}}

\newcommand{\BGGo}{{^{\mathbb{G}}_{\mathbb{G}}\mathcal{YD}}}

\newcommand{\AAC}{{^{A}_{A}\mathcal{YD}(\mathcal{C})}}
\newcommand{\RRC}{{^{R}_{R}\mathcal{YD}(\mathcal{C})}}
\newcommand{\RRH}{{^{R\#H}_{R\#H}\mathcal{YD}}}
\newcommand{\AACR}{{^{A}_{A}\mathcal{YD}(\mathcal{C})}_{\operatorname{rat}}}
\newcommand{\AACOPCR}{{^{A^{\operatorname{cop}}}_{A^{\operatorname{cop}}}\mathcal{YD}(\overline{\mathcal{C}})}_{\operatorname{rat}}}

\newcommand{\BBCR}{{^{B}_{B}\mathcal{YD}(\mathcal{C})}_{\operatorname{rat}}}

\newcommand{\AARC}{{\mathcal{YD}(\mathcal{C})^A_A}}

\newcommand{\BBC}{{^{B}_{B}\mathcal{YD}(\mathcal{C})}}
\newcommand{\GGAA}{{^{G}_{G}\mathcal{YD}^{\Phi_{\underline{a}}}}}

\newcommand{\HH}{{^{H}_{H}\mathcal{YD}}}
\newcommand{\HHJ}{{^{H^J}_{H^J}\mathcal{YD}}}
\newcommand{\BNG}{{^{\AN}_{\AN}\mathcal{YD}}}
\newcommand{\AVG}{{^{A(V)}_{A(V)}\mathcal{YD}}}
\newcommand{\AMN}{\mathcal{A}(M\oplus N)}
\newcommand{\AN}{\mathcal{A}(N)}
\newcommand{\AMI}{{^{A(M_i)}_{A(M_i)}\mathcal{YD}}}

\newcommand{\AMID}{{^{A(M_i^*)}_{A(M_i^*)}\mathcal{YD}}}
\newcommand{\HHY}{{^{H'}_{H'}\mathcal{YD}}}

\newcommand{\NN}{\mathbb{N}_{0}}
\newcommand{\ZZ}{\mathbb{Z}}
\newcommand{\Gr}{\text{-Gr}}
\newcommand{\Mod}{\mathcal{M}}

\newcommand{\Zgr}{\mathbb{Z}\operatorname{-Gr}\mathcal{M}_{\mathbbm{k}}}
\newcommand{\BMN}{\mathcal{B}(M\oplus N)}

\newcommand{\sg}{\mathbbm{g}}

\newcommand{\bbk}{\mathbbm{k}}

\newcommand{\bfo}{\textbf{1}}
\newcommand{\BV}{\mathcal{B}(V)}
\newcommand{\BVCC}{{^{\BV}_{\BV}\mathcal{YD}(\mathcal{C})}}
\newcommand{\BMICC}{{^{\bB(M_i)}_{\bB(M_i)}\mathcal{YD}(\mathcal{C})}}
\newcommand{\BVdCC}{{^{\bB(V^*)}_{\bB(V^*)}\mathcal{YD}(\mathcal{C})}}
\newcommand{\BMIdCC}{{^{\bB(M_i^*)}_{\bB(M_i^*)}\mathcal{YD}(\mathcal{C})}}
\newcommand{\BN}{\mathcal{B}(N)}

\usetikzlibrary{knots}

\newcommand{\bB}{\mathcal{B}}

\newcommand{\ot}{\otimes}

\newcommand{\BVd}{\mathcal{B}(V^*)}

\makeatletter
\newcommand\blfootnote[1]{%
  \begingroup
  \renewcommand\thefootnote{}%
  \footnotetext{%
    \begin{minipage}[t]{\linewidth}
      \setlength{\parindent}{0pt}%
      #1
    \end{minipage}%
  }%
  \addtocounter{footnote}{-1}%
  \endgroup
}
\makeatother
\title{Reflection Theory and Cartan Graphs of Nichols Algebras over
Coquasi-Hopf Algebras}
\author{Bowen Li and Gongxiang Liu}
\date{}

\begin{document}

\maketitle

\blfootnote{
\textit{Affiliation:} Bowen Li:  School of Mathematical Sciences, Shanghai Jiao Tong University. \\ \textit{Email:} \href{mailto:lbwwysusu@gmail.com}{lbwwysusu@gmail.com} \\  Gongxiang Liu: School of Mathematics, Nanjing University. \textit{Email:} \href{mailto:gxliu@nju.edu.cn}{gxliu@nju.edu.cn}

}

\begin{abstract}
We develop a reflection theory for Nichols algebras in the
Yetter-Drinfeld category over a coquasi-Hopf algebra with bijective
antipode.  We prove that a  tuple of
finite-dimensional simple Yetter--Drinfeld modules admitting all
reflections gives rise to a Cartan graph. We characterize the
finiteness of the associated Cartan graph by tensor decomposability of the Nichols
algebra, obtain a factorization indexed by positive real roots, and
derive a criterion for finite-dimensionality. We also show that
braided monoidal equivalences preserve reflections and the associated
Cartan graphs. As applications, we prove that every Cartan graph associated with a
diagonal type tuple over a finite abelian group equipped with an abelian
\(3\)-cocycle is covered by one arising from a diagonal type tuple over 
 some finite abelian group $G$ with trivial associator, and that
their real-root sets agree at corresponding objects. We also
construct a Cartan graph of the former kind that cannot be obtained
from any diagonal tuple in \({}_G^G\mathcal{YD}\). Finally, we
exhibit a non-diagonal affine Nichols algebra over a coquasi-Hopf
algebra whose associated Cartan graph is the standard affine Cartan
graph of type \(A_2^{(1)}\).
\end{abstract}

\section{Introduction}

\subsection{Notation}
Throughout this paper, unless otherwise stated, we adopt the following notation.
\begin{itemize}
    \item $\mathbb I=\{1,2,\ldots,\theta\}$, where $\theta\geq2$.
    \item $H$ denotes a coquasi-Hopf algebra with bijective antipode,
    and $\HH$ denotes its category of Yetter--Drinfeld modules. We
    sometimes write $\mathcal C=\HH$.
    \item $M=(M_1,\ldots,M_\theta)$ denotes a tuple of
    finite-dimensional simple objects in $\HH$. By abuse of notation,
    \[
    \bB(M):=\bB(M_1\oplus\cdots\oplus M_\theta).
    \]
    \item $\mathcal G$ denotes a semi-Cartan graph, and
    $\mathcal W(\mathcal G)$ denotes its Weyl groupoid.
    \item $\mathbbm{k}$ is an algebraically closed field of
    characteristic zero.
\end{itemize}
\subsection{From Hopf algebras to coquasi-Hopf algebras}

The classification of finite-dimensional pointed Hopf algebras has
been a central problem in Hopf algebra theory for several decades.
A fundamental framework for this problem is the lifting method of
Andruskiewitsch and Schneider
\cite{AS98,AS00,AS02,AS10}.  In this approach, the associated graded
Hopf algebra of a pointed Hopf algebra is described by the
bosonization of a Nichols algebra in a Yetter--Drinfeld category.
The classification problem is therefore closely related to the
structure of Nichols algebras and to their subsequent liftings,
which have been studied extensively by means of cocycle
deformations \cite{AA+14,AG19}.

One of the main structural tools for Nichols algebras is the
reflection theory of tuples of simple Yetter--Drinfeld modules.
The iterated braided adjoint actions between the components of a tuple determine the entries of its generalized Cartan matrix.  When the iterated braided adjoint actions eventually vanish, the
corresponding reflections can be defined and iterated, giving rise
to a family of generalized Cartan matrices and an associated Weyl
groupoid. This theory was developed in
\cite{AHS10,HS10}, reformulated through categorical partial
dualization in \cite{HSdualpair}, and treated systematically in
\cite{rootsys}.

In the diagonal case, Heckenberger introduced the Weyl groupoid of a
Nichols algebra in \cite{Hec06} and classified the finite arithmetic
root systems in \cite{Hec09}.  An important role of the associated
Cartan graph is to provide convex orders on the positive real roots.
More precisely, for a finite Cartan graph, a reduced expression of a
longest morphism determines a sequence containing all positive real
roots, and this sequence defines a convex order; see
\cite[Section~9.3]{rootsys}.  Such an order supplies the
combinatorial framework in which root vectors can be arranged to
form PBW bases.  Using root vectors, convex orders, and
Lusztig isomorphisms, Angiono obtained explicit PBW
bases and presentations by generators and relations for
Nichols algebras of diagonal type \cite{A13,A15}; see also
\cite{AA17} for a broader account of their structure and
classification.

This explains why the passage from a semi-Cartan graph to a Cartan
graph is essential.  A semi-Cartan graph records the individual
reflections and the corresponding generalized Cartan matrices, but
its axioms alone do not provide the rank-two relations required to
compare root sequences arising from different reduced expressions.
The additional Cartan-graph axioms provide these relations and
thereby make it possible to pass from reflections to convex orders
on the positive real roots and root factorizations.  In the
diagonal case, these factorizations lead to PBW bases.  In the factorization obtained in this paper, the factor associated
with a positive real root \(\beta\) is the Nichols algebra
\(\mathcal B(M_\beta)\), where \(M_\beta\) is a finite-dimensional
simple object of \({}_H^H\mathcal{YD}\).

Pointed coquasi-Hopf algebras and their comodule categories have been
studied in connection with quasi-quantum groups, pointed tensor
categories, and deformations governed by group \(3\)-cocycles; see,
for example,
\cite{rank1,coquasitriangular,tame-rep,FQQR2,QQG,
huang2024classification}. Their infinitesimal parts are Nichols
algebras in braided monoidal categories whose associativity
constraints may be nontrivial.

Let \(G\) be a finite abelian group. When \(\Phi\) is an abelian
\(3\)-cocycle, Nichols algebras of diagonal type in
\({}_G^G\mathcal{YD}^{\Phi}\) can be related to Nichols algebras over
group algebras \cite{QQG}. When \(\Phi\) is non-abelian,
\({}_G^G\mathcal{YD}^{\Phi}\) may also contain simple objects that are
not of diagonal type. The rank-three Nichols algebra studied in
\cite{huang2024classification} provides such an example.

The passage to coquasi-Hopf algebras is not merely a change of
formulas.  The multiplication is associative only up to the
associator, and the corresponding Yetter--Drinfeld category is
generally non-strict.  Thus the duality functors, rational module
categories, coinvariant constructions, and braided adjoint actions
used in reflection theory must be formulated internally in the
braided monoidal category.  The purpose of this paper is to develop
these constructions for Nichols algebras over arbitrary
coquasi-Hopf algebras with bijective antipode and to obtain their
natural structural consequences: Cartan graphs, tensor
decomposability, root factorization, and a finite-dimensionality
criterion.

\subsection{Reflection theory and Cartan graphs}

Let $H$ be a coquasi-Hopf algebra with bijective antipode.  Our
categorical starting point is the theory of rational
Yetter--Drinfeld modules over locally finite graded Hopf algebras in
$\HH$.  Given a dual pair of such Hopf algebras, we construct a
braided monoidal equivalence between the corresponding categories of
rational Yetter--Drinfeld modules.  In particular, the canonical
Hopf pairing
\[
 \mathcal B(V^*)\ot\mathcal B(V)\longrightarrow\mathbbm{k}
\]
gives a braided monoidal equivalence
\[
 \Omega_V:
 {}_{\mathcal B(V)}^{\mathcal B(V)}
      \mathcal{YD}(\mathcal{C})_{\operatorname{rat}}
 \longrightarrow
 {}_{\mathcal B(V^*)}^{\mathcal B(V^*)}
      \mathcal{YD}(\mathcal{C})_{\operatorname{rat}}.
\]
This equivalence is the categorical partial dualization underlying
the reflection construction.

Suppose that $M$ admits the $i$-th reflection.  Let
$
 K_i^M
 =
 \mathcal B(M)^{\operatorname{co}\mathcal B(M_i)}
$
be the corresponding right coinvariant subalgebra.  We obtain an
isomorphism of Hopf algebras in $\HH$,
\[
 \Theta_i:
 \mathcal B(R_i(M))
 \xrightarrow{\ \sim\ }
 \Omega_i(K_i^M)\#\mathcal B(M_i^*).
\]
Thus the reflected Nichols algebra is described explicitly in terms
of the coinvariant part of $\mathcal B(M)$ and the Nichols algebra of
the dual simple object.

If $M$ admits all reflections, the family of tuples obtained from $M$
by sequences of reflections determines an associated semi-Cartan graph.
\[
 \mathcal G(M)
 =
 \mathcal G
 \bigl(
 \mathbb I,\mathcal X,r,(A^{[P]})_{[P]\in\mathcal X}
 \bigr).
\]
The first principal result of the paper is that this semi-Cartan
graph already satisfies all the axioms of a Cartan graph.

\medskip
\noindent\textbf{Theorem A.} \emph{(Theorem~\ref{thm5.6})}
Let $M$ be a tuple of finite-dimensional simple Yetter--Drinfeld
modules over a coquasi-Hopf algebra with bijective antipode.  If $M$
admits all reflections, then $\mathcal G(M)$ is a Cartan graph.
\medskip

The proof is based on one-sided coideal subalgebras and their
compatibility with reflections.  For every $i\in\mathbb I$, we
construct an algebra isomorphism in $\HH$
\[
 T_i^M:
 L_i^{R_i(M)}
 \xrightarrow{\ \sim\ }
 K_i^M.
\]
These maps play a role analogous to that of Lusztig maps for Nichols
algebras over Hopf algebras.  They allow us to transport one-sided coideal
subalgebras through successive reflections.

More precisely, let
\(\kappa=(i_1,\ldots,i_l)\) be a reduced reflection sequence and put
\[
 \beta_k
 =
 \operatorname{id}_{[M]}
 s_{i_1}\cdots s_{i_{k-1}}(\alpha_{i_k}).
\]
We construct finite-dimensional simple Yetter--Drinfeld modules
\(M_{\beta_k}\), where \(M_{\beta_k}\) is homogeneous of degree
\(\beta_k\), and a graded right coideal subalgebra
\(E^M(\kappa)\subseteq\mathcal B(M)\) such that multiplication induces
an isomorphism
\[
 \mathcal B(M_{\beta_l})\ot(\cdots (\mathcal{B}(M_{\beta_2})\ot
 \mathcal B(M_{\beta_1})))
 \xrightarrow{\ \sim\ }
 E^M(\kappa).
\]
The modules \(M_{\beta_k}\) and the associated coideal
factorizations are used to prove (CG3), while the comparison of the
factorizations associated with the alternating sequences
\((i,j,i,j,\ldots)\) and \((j,i,j,i,\ldots)\) in each rank-two
subsystem gives (CG4).  This proves Theorem A.

\subsection{Tensor decomposability and finite-dimensionality}

After proving that $\mathcal G(M)$ is a Cartan graph, we study how
its finiteness is reflected in the structure of $\mathcal B(M)$.  We
call $\mathcal B(M)$ tensor decomposable if, as a multigraded object
in $\HH$, it is isomorphic to an iterated tensor product of Nichols
algebras $\mathcal B(Q_l)$, where the $Q_l$ are finite-dimensional
simple objects of pairwise distinct nonzero multidegrees.

The main point is that tensor decomposability is not merely a
consequence of the existence of a finite root system.  It already
forces all reflections to exist.  Indeed, an appropriate homogeneous
component of the coinvariant algebra is the direct sum of the
subspaces $(\operatorname{ad}M_i)^n(M_j)$, where $j\neq i$ and
$n\geq0$.  Its finite-dimensionality implies that, for every
$j\neq i$, $(\operatorname{ad}M_i)^n(M_j)=0$ for all sufficiently
large $n$.  We then prove that tensor decomposability is preserved
by reflections and that the degrees of the tensor factors transform
under the corresponding simple reflections.

\medskip
\noindent\textbf{Theorem B.} \emph{(Theorem~\ref{thm:tensor-decomposability-criterion})}
Let $M$ be a tuple of finite-dimensional simple Yetter--Drinfeld
modules.  The following conditions are equivalent:

    (1) $M$ admits all reflections and $\mathcal G(M)$ is finite.
    
    (2)$M$ admits all reflections and $\mathcal B(P)$ is tensor
decomposable for some tuple $P$ obtained from $M$ by a sequence of
reflections.
    
    (3) $\mathcal B(M)$ is tensor decomposable.

\medskip

As a consequence, when $\mathcal G(M)$ is finite, every reduced
expression of a longest morphism gives a factorization of the
Nichols algebra into Nichols algebras associated with the positive
real roots.  In particular, the positive real roots index the tensor
factors of $\mathcal B(P)$ for every tuple $P$ in the reflection
orbit.  This is the root factorization stated in
Corollary~\ref{cor:root-factorization}.

The Cartan graph controls the number of factors occurring in the
root factorization, but it does not by itself determine whether the
Nichols algebras $\mathcal B(P_i)$ of the individual components are
finite-dimensional.  Combining the root factorization with the
behavior of Nichols algebras under reflections gives the following
criterion.

\medskip
\noindent\textbf{Theorem C.} \emph{(Theorem~\ref{thm6.7})}
Let $M$ be a tuple of finite-dimensional simple Yetter--Drinfeld
modules.  Then $\mathcal B(M)$ is finite-dimensional if and only if
$M$ admits all reflections, $\mathcal G(M)$ is finite, and
$\mathcal B(P_i)$ is finite-dimensional for every tuple $P$ obtained
from $M$ by a sequence of reflections and every $i\in\mathbb I$.
\medskip

Thus the finiteness problem separates naturally into the finiteness
of the associated Cartan graph and the finite-dimensionality of
$\mathcal B(P_i)$ for the individual components in the reflection
orbit.  The finite Cartan graph provides finitely many positive real
roots, while the last condition ensures that every corresponding
factor in the root factorization is finite-dimensional.

\subsection{Braided equivalences and applications}

The reflection construction is compatible with braided monoidal
equivalences.  Let $H$ and ${H'}$ be coquasi-Hopf algebras with
bijective antipodes and let
\[
 (F,\xi):
 {}_H^H\mathcal{YD}
 \longrightarrow
 {}_{H'}^{H'}\mathcal{YD}
\]
be a braided monoidal equivalence.  We prove that $F$ sends Nichols
algebras to Nichols algebras and preserves the vanishing and
nonvanishing of the iterated braided adjoint actions used to define
reflections.  Consequently, the maximal integers $n$ for which
$(\operatorname{ad}M_i)^n(M_j)\neq0$ are preserved, as are the
Cartan entries and the reflected tuples.

\medskip
\noindent\textbf{Theorem D.} \emph{(Theorem~\ref{thm7.2})}
The tuple $M$ admits all reflections if and only if $F(M)$ does.  If
these conditions hold, then $F$ induces an isomorphism between
$\mathcal G(M)$ and $\mathcal G(F(M))$.  Consequently, their Weyl
groupoids and real-root systems are naturally identified.
\medskip

For a diagonal tuple \(M\) in
\({}_G^G\mathcal{YD}^{\Phi}\), where \(G\) is a finite abelian group
and \(\Phi\) is an abelian \(3\)-cocycle, we obtain two complementary
results. First, we construct an associated Cartan graph that cannot
be obtained from any diagonal tuple in
\({}_K^K\mathcal{YD}\), for any finite abelian group \(K\). Second,
after pulling \(M\) back to a suitable finite abelian group and
twisting by an explicitly determined \(2\)-cochain, we obtain a
diagonal tuple \(M'\) such that \(\mathcal G(M')\) covers
\(\mathcal G(M)\). Their real-root sets agree at objects related by
the covering map. These comparisons allow us to transfer finiteness
results for arithmetic root systems and GK-dimension.

Finally, we consider a Nichols algebra in ${}_G^G\mathcal{YD}^{\Phi}$, where $G$ is an
elementary abelian group and $\Phi$ is a normalized non-abelian $3$-cocycle on $G$.  Its associated
Cartan graph is the standard affine Cartan graph of type
$A_2^{(1)}$.  We determine its real roots and show that the associated Tits cone
 is an open half-space.
Thus the example gives an affine Nichols algebra of non-diagonal type in the coquasi-Hopf setting.  

\subsection{Organization of the paper}

The paper is organized as follows.  Section~2 recalls the necessary
background on Cartan graphs, Weyl groupoids, coquasi-Hopf algebras,
Yetter--Drinfeld modules, and bosonization.  In Section~3, we develop
the theory of rational Yetter--Drinfeld modules and construct the
braided monoidal equivalences induced by dual pairs.  Section~4
studies coinvariant Nichols algebras, constructs reflections, and
recalls the associated semi-Cartan graph.
In Section~5, we investigate one-sided coideal subalgebras and the
Lusztig-type maps compatible with reflections.  Section~6 constructs
the finite-dimensional simple Yetter--Drinfeld modules associated
with the roots in reduced reflection sequences and proves that
$\mathcal G(M)$ is a Cartan graph.  In Section~7, we introduce tensor
decomposable Nichols algebras, establish the root factorization, and
prove the finite-dimensionality criterion.  Section~8 proves the
invariance of reflections and Cartan graphs under braided monoidal
equivalences. Finally, Section~9 compares Cartan graphs associated with diagonal
tuples over finite abelian groups equipped with abelian
\(3\)-cocycles with those arising over finite abelian groups with trivial $3$-cocycle,
and studies a non-diagonal affine Nichols algebra.
 \section{Preliminaries}

In this section, we recall the background needed throughout the paper.
We first review semi-Cartan graphs, Cartan graphs, Weyl groupoids, root
systems, and reduced sequences in Weyl groupoids, following
\cite{rootsys}. We then recall coquasi-Hopf algebras, Yetter--Drinfeld
modules, bosonization, and Yetter--Drinfeld modules over Hopf algebras
in braided monoidal categories. We also fix the notation used in the
subsequent sections.

\subsection{Cartan graphs and root systems}
We begin by recalling semi-Cartan graphs, Cartan graphs, Weyl
groupoids, and root systems. We follow the notation and conventions
of \cite{rootsys}, to which we refer for proofs and examples.
\begin{definition} \label{def5.14}
   Let $\mathbb{I}$ be a non-empty finite set,   \(\mathcal{X}\) a non-empty set, $r: \mathbb{I} \times \mathcal{X}\rightarrow 
   \mathcal{X}$, $A: \mathbb{I} \times \mathbb{I} \times \mathcal{X}\rightarrow \mathbb{Z}$ maps. For all $i,j \in \mathbb{I}$ and ${X} \in \mathcal{X}$ we write $r_i(X)=r(i,X)$, $a_{ij}^X=A(i,j,X)$ and $A^X=(a_{ij}^X)_{i,j \in \mathbb{I}} \in \mathbb{Z}^{\mathbb{I} \times \mathbb{I}}$. The quadruple \(\mathcal{G} = \mathcal{G}(\mathbb{I}, \mathcal{X}, (r_i)_{i \in I}, (A^X)_{X\in \mathcal{X}})\) is called a semi-Cartan graph if for all $X \in \mathcal{X}$, the matrix $A^X$ is a generalized Cartan matrix, and the following axioms hold.
\\   {(CG1)} For all \(i \in \mathbb{I}\), the map $r_i$ satisfies \(r_i^2 = \mathrm{id}_{\mathcal{X}}\).
    \\  {(CG2)} For all \(i \in \mathbb{I}\), \(X \in \mathcal{X}\), \(A^X\) and \(A^{r_i(X)}\) have the same \(i\)-th row.

\end{definition}

Kac-Moody Lie algebras, Kac-Moody Lie superalgebras and Nichols algebras over Hopf algebras provide an abundant class of examples of semi-Cartan graphs \cite{rootsys}. 

Given  semi-Cartan graphs, we can construct  Weyl groupoids, and define  the sets of real roots.
Given a semi-Cartan graph, its Weyl groupoid is defined as follows.
\begin{definition}
    We denote by $\mathcal{D}(\mathcal{X},\operatorname{End}(\mathbb{Z}^\mathbb{I}))$  to be the category with objects $\operatorname{Ob}\mathcal{D}(\mathcal{X},\operatorname{End}(\mathbb{Z}^\mathbb{I}))=\mathcal{X}$, and morphisms 
    $$ \operatorname{Hom}(X,Y)=\{ (Y,f,X)\mid f\in \operatorname{End}(\mathbb{Z}^\mathbb{I})\}, $$
    where the composition of morphisms is defined by 
    $$(Z,g,Y)\circ (Y,f,X)=(Z,gf,X),  \ \text{for all} \  X,Y,Z \in \mathcal{X}, \ f,g \in \operatorname{End}(\mathbb{Z}^\mathbb{I}).$$
   Let $\alpha_i$, $1\leq i\leq \theta$ be the standard basis of  $\mathbb{Z}^\mathbb{I}$, and 
    $$ s_i^X \in \operatorname{Aut}(\mathbb{Z}^{\mathbb{I}}), \ s_i^X(\alpha_j)=\alpha_j-a_{ij}^X\alpha_i, \text{for all} \ j.$$
    We call the smallest subcategory of $\mathcal{D}(\mathcal{X},\operatorname{End}(\mathbb{Z}^\mathbb{I}))$ which contains all morphisms $(r_i(X),s_i^X,X)$ with $i \in \mathbb{I}$, $X \in \mathcal{X}$ the Weyl groupoid of $\mathcal{G}$, denoted by $\mathcal{W}(\mathcal{G})$.
\end{definition}
Now let $X,Y \in \mathcal{X}$, $\omega \in \operatorname{Hom}(X,Y)$. The following definition is parallel to that in Lie algebra. We call 
$$ l(\omega)= \operatorname{min}\{ k \mid \omega=\operatorname{id}_Xs_{i_1}\cdots s_{i_k}, \ i_1,...,i_k \in \mathbb{I}\}$$
the length of $\omega$ and the decomposition $(i_1,...,i_k)$ is called the reduced decomposition of $\omega$.
Here we omit the subscript for simplicity

 For each $X \in \mathcal{X}$, the set of real roots of $\mathcal{G}$ at $X$ is defined as follows:
   $$ \Delta^{X,\operatorname{re}}=\{\omega(\alpha_i)\in \mathbb{Z}^{\mathbb{I}}\mid \omega \in \operatorname{Hom}(\mathcal{W}(\mathcal{G}),X), i \in \mathbb{I} \}.$$
We say $\mathcal{G}$ is a finite semi-Cartan graph if $\Delta^{X, \operatorname{re}}$ is finite for each $X \in \mathcal{X}$.
   
   The elements 
    $$\Delta^{X,\operatorname{re}}_+=\Delta^{X,\operatorname{re}} \cap \mathbb{N}_0^{\mathbb{I}},\ \Delta^{X,\operatorname{re}}_-=\Delta^{X,\operatorname{re}} \cap -\mathbb{N}_0^{\mathbb{I}}$$
   are called positive and negative real roots, respectively.
   For any \( X \in \mathcal{X} \) and \( i, j \in \mathbb{I} \), let  
\[
m_{ij}^X = |\Delta^{X,\operatorname{re}} \cap (\mathbb{N}_0 \alpha_i + \mathbb{N}_0 \alpha_j)|.
\]
With the above notations, we  can give the definition of a Cartan graph.
\begin{definition}\label{D-def6.3}
We say that a semi-Cartan graph \( \mathcal{G} \) is a {Cartan graph} if the following hold.
\begin{enumerate}
    \setcounter{enumi}{2} 
    \item[ {(CG3)}] For any \( X \in \mathcal{X} \), the set \( \Delta^{X,\operatorname{re}}=\Delta^{X,\operatorname{re}}_+ \cup \Delta^{X,\operatorname{re}}_{-}. \) 
    \item[ {(CG4)}] For any $i \ne j \in \mathbb{I}$ and $X \in \mathcal{X}$, if the cardinality $m_{ij}^{X} := | \Delta^{X,\operatorname{re}} \cap (\mathbb{N}_{0}\alpha_{i} + \mathbb{N}_{0}\alpha_{j}) |$ is finite, then $(r_{i}r_{j})^{m_{ij}^{X}}(X) = X$.
\end{enumerate}
\end{definition}
There are many examples showing that a semi-Cartan graph may not satisfy $(CG3)$ or $(CG4)$, see [\citealp{rootsys}, Example 9.1.26, Example 9.2.3].\par 
We are going to introduce root systems over semi-Cartan graphs.
\begin{definition}
 Let \( \mathcal{G} = \mathcal{G}(\mathbb{I}, \mathcal{X}, r, A) \) be a semi-Cartan graph.
For all \( X \in \mathcal{X} \), let \( (R^X)_{X\in\mathcal{X}} \) be a subset of \( \mathbb{Z}^{\mathbb{I}} \) with the following properties.
\begin{enumerate}[label=(\arabic*)]
    \item[ {(R1)}] \( 0 \notin R^X \) and \( \alpha_i \in R^X \) for all \( X \in \mathcal{X} \) and \( i \in \mathbb{I} \).
    \item[ {(R2)}] \( R^X \subseteq \mathbb{N}_0^{\mathbb{I}}\cup -\mathbb{N}_0^{\mathbb{I}} \) for all \( X \in \mathcal{X} \).
    \item[ {(R3)}] For any \( X \in \mathcal{X} \) and \( i \in \mathbb{I} \), \( s_i^X (R^X) = R^{r_i(X)} \).
    \item[ {(R4)}] If $i, j \in \mathbb{I}$ and $X \in \mathcal{X}$ such that $i \neq j$ and $m_{ij}^X$ in Definition \ref{D-def6.3} is finite, then $(r_ir_j)^{m_{ij}^X}(X)=X.$
\end{enumerate}
Then we say that the pair \( (\mathcal{G}, (R^X)_{X \in \mathcal{X}}) \) is a {root system}  over \( \mathcal{G} \). A root system over \( \mathcal{G} \) is said to be reduced if for all \( X \in \mathcal{X} \) and \( \alpha \in R^X \) the roots \( \alpha \) and \( -\alpha \) are the only rational multiples of \( \alpha \) in \( R^X \).  A root system is said to be finite if for all \( X \in \mathcal{X} \),  $R^X$ is finite. 
\end{definition}
If $\mathcal{G}$ is a  Cartan graph,     the pair $(\mathcal{G}, (\Delta^{X,\operatorname{re}})_{X \in \mathcal{X}} )$ is automatically a reduced root system over $\mathcal{G}$. Furthermore, if $(R^X)_{X\in \mathcal{X}}$ is finite and reduced, then $R^X=\Delta^{X, \operatorname{re}}$ for each $X \in \mathcal{X}$.
\begin{rmk} \upshape
There is a correspondence between connected, simply connected Cartan
graphs permitting a reduced root system and crystallographic Tits
arrangements with reduced root system. A  crystallographic Tits arrangement consists of  a pair $(\mathcal{A},T)$, where $\mathcal{A}$ is a set of linear hyperplanes in $\mathbb{R}^{\mathbb{I}}$, and $T$ is a Tits cone in $\mathbb{R}^{\mathbb{I}}$, satisfying some additional conditions. One may refer to \cite{affinenichols2} for details.
 \end{rmk}

\subsection{Reduced sequences and an equivalent form of Cartan graphs}
Our goal is to prove that the reflections constructed below give
rise to a Cartan graph. Rather than verifying  {(CG3)} and
 {(CG4)} directly, we use the equivalent conditions
 {(CG3')} and  {(CG4')}.  \par
Let \( \mathcal{G} = \mathcal{G}(\mathbb{I}, \mathcal{X}, r, A) \) be a semi-Cartan graph. Let \( X \in \mathcal{X}, l \geq 0 \), and \( \kappa = (i_1, \ldots, i_l) \in \mathbb{I}^l \).
 For all \( 1 \leq k \leq l \) let
\[
\beta_k^{X, \kappa} = \text{id}_X s_{i_1} \cdots s_{i_{k-1}} (\alpha_{i_k}),
\]
and let
\[
\Lambda^X(\kappa) = \{ \beta_k^{X, \kappa} \mid 1 \leq k \leq l \}.
\]

\begin{definition}
 We say that \( \kappa \) is \( X \)-{reduced} if for any \( 1 \leq k < l \),
\[
\alpha_{i_k} \notin \Lambda^{r_{i_k} \cdots r_{i_1}(X)}(i_{k+1}, \ldots, i_l).
\]
The integer \( l \) is called the {length of \( \kappa \)}.
\end{definition}
\begin{rmk}\upshape\label{rmk2.7}
  (1)  Let \( X \in \mathcal{X}, l \geq 2 \) and \( \kappa = (i_1, \ldots, i_l) \in \mathbb{I}^l \). If \( i_j = i_{j+1} \) for some \( 1 \leq j < l \), then $$ \alpha_{i_j} \in \Lambda^{r_{i_j} \cdots r_{i_1}(X)}(i_{j+1}, \ldots, i_l),$$ and hence \( \kappa \) is not \( X \)-reduced. \par 
  (2) A sequence $\kappa = (i_1, \dots, i_l)$ is $X$-reduced if and only if $(i_2, \dots, i_l)$ is $r_{i_1}(X)$-reduced and $\alpha_{i_1} \notin \Lambda^{r_{i_1}(X)}(i_2, \dots, i_l)$.\par 
  (3) $\kappa$ is $X$-reduced if and only if $\beta_p^{X,\kappa}\neq -\beta_q^{X,\kappa} $ for any $1\leq p <q\leq l$.
\end{rmk}
Now we consider a special case.
\begin{definition}
    Let \(i, j \in \mathbb{I}\) with \(i \neq j\) and \(\kappa = (i_k)_{k \geq 1} = (i, j, i, \ldots)\) with  
\(i_k = i\) if \(k\) is odd and \(i_k = j\) if \(k\) is even. Let \(X \in \mathcal{X}\). We call any $X$-reduced sequence with entries in $\{i, j\}$ and starting with $i$ a beginning of $\kappa$
Let \(\kappa_{ij}^X\) be the longest \(X\)-reduced beginning of \(\kappa\), if it exists, and \(\kappa\) otherwise. We write \(\overline{m}_{ij}^X\) for the length of \(\kappa_{ij}^X\).
\end{definition}
With the help of the notation above, we are in the position to introduce axioms characterizing Cartan graphs.

(CG3') For any \( X \in \mathcal{X} \) and any \( X \)-reduced sequence \( \gamma \), \( \Lambda^X(\gamma) \subseteq \mathbb{N}_0^\mathbb{I} \).

(CG4') For any \( X \in \mathcal{X} \) and any \( i, j \in \mathbb{I} \) with \( i \neq j \) and \( \overline{m}_{ij}^X < \infty \), we have

\[(r_i r_j)^{\overline{m}_{ij}^X}(X) = X, \quad \text{id}_X(s_i s_j)^{\overline{m}_{ij}^X}(\alpha_k) = \alpha_k\]
for all \( k \in \mathbb{I}-\{i,j\}  \).

The following lemma is important for our purposes. 
\begin{lemma}\textup{[\citealp{rootsys}, Corollary 9.2.20]}
    For any semi-Cartan graph $\mathcal{G}$, the following are equivalent.\par
     {(1)} $\mathcal{G}$ is a Cartan graph.\par
     {(2)} $\mathcal{G}$ satisfies (CG3') and (CG4').
\end{lemma}
We collect some lemmas needed later, which can be found in [\citealp{rootsys}, Section 9.2].
\begin{lemma}\label{lem2.10}
     Assume that $|\mathbb{I}| \geq 2$. Let $\mathcal{G}=\mathcal{G}(\mathbb{I},\mathcal{X},r,A)$ be a semi-Cartan graph,  $X\in \mathcal{X}$, $i,j \in \mathbb{I}$ with $i\neq j$\par 
    (1)  Assume $\kappa=(i,j,i,j...)=(i_1,...,i_m)$ is a $X$-reduced sequence beginning with $\kappa_{ij}^X$ with length $m \geq 2$. If $\overline{m}_{ji}^{r_j(X)}=m$, then $\beta_{m}^{X,\kappa}=\alpha_j$.\par 
    (2) Assume $\mathcal{G}$ satisfies (CG3'). 
   Then $\overline{m}_{ij}^{r_{k}(X)}=\overline{m}_{ji}^{r_k(X)}=\overline{m}_{ij}^X$ for  $k \in \{i,j\}$.
    \par 
    (3) Assume $\kappa=\kappa_{ij}^X=(i,j,i,j,...)$, $\overline{m}^X_{ij} < \infty$ and $\mathcal{G}$ satisfies (CG3'), then $\beta_{1}^{X,\kappa}=\alpha_i$, $\beta_{\overline{m}_{ij}}^{X,\kappa}=\alpha_j$, $\operatorname{id}_X(s_i s_j)^{\overline{m}_{ij}^X}(\alpha_i) = \alpha_i$ and $\operatorname{id}_X(s_i s_j)^{\overline{m}_{ij}^X}(\alpha_j) = \alpha_j$.
\end{lemma}
 The following lemma can be found in [\citealp{rootsys}, Section 9.3], which is used to give a new criterion to determine the finite-dimensionality of a Nichols algebra.
\begin{lemma}\label{lem2.11}
   (1) Let $\mathcal{G}$ be a Cartan graph, and $\omega=\operatorname{id}_Xs_{i_1}\cdots s_{i_l} \in \operatorname{Hom}(\mathcal{W}(\mathcal{G)},X)$ be a reduced decomposition of $\omega$.  Then $\kappa=(i_1,...,i_l)$ is $X$-reduced.\par 
   (2) Let $\mathcal{G}$ be a finite Cartan graph and $\kappa=(i_1,...i_n)$ be a reduced decomposition of longest element $\omega_0 \in \operatorname{Hom}(\mathcal{W}(\mathcal{G}),X),$ then $$ \Delta^{X,\operatorname{re}}_+=\Lambda^X(\kappa).$$
\end{lemma}
 \subsection{ Coquasi-Hopf algebras}
We follow the standard definition of a coquasi-Hopf algebra as introduced in \cite{Drinfeld}.  
 \begin{definition}
 A coquasi-Hopf algebra $H$ is a coalgebra $(H, \Delta, \varepsilon)$ equipped with a compatible quasi-algebra structure and an antipode $(\mathcal{S},\alpha, \beta)$. Namely, there exist:
\begin{itemize}
    \item Two coalgebra homomorphisms:
    \[
    m : H \otimes H \rightarrow H, \quad a \otimes b \mapsto ab,
    \]
    \[
    \mu : \mathbbm{k} \rightarrow H, \quad \lambda \mapsto \lambda 1_{H},
    \]
    
    \item A convolution-invertible map $\Phi : H^{\otimes 3} \rightarrow \mathbbm{k}$ called an \textit{associator},
    
    \item A coalgebra antimorphism $\mathcal{S} : H \rightarrow H$,
    
    \item Two linear functions $\alpha, \beta : H \rightarrow \mathbbm{k}$
\end{itemize}
such that for all $a, b, c, d \in H$ the following equalities hold:
\begin{align}
    a_1(b_1c_1)\Phi(a_2, b_2, c_2) &= \Phi(a_1, b_1, c_1)(a_2b_2)c_2, \\
    1_{H}a &= a = a1_{H}, \\
    \Phi(a_1, b_1, c_1d_1)\Phi(a_2b_2, c_2, d_2) &= \Phi(b_1, c_1, d_1)\Phi(a_1, b_2c_2, d_2)\Phi(a_2, b_3, c_3), \\
    \Phi(a, 1_{H}, b) &= \varepsilon(a)\varepsilon(b), \\
    \mathcal{S}(a_1)\alpha(a_2)a_3 &= \alpha(a)1_{H}, \quad a_1\beta(a_2)\mathcal{S}(a_3) = \beta(a)1_{H}, \\
    \Phi(a_1, \mathcal{S}(a_3), a_5)\beta(a_2)\alpha(a_4) &= \Phi^{-1}(\mathcal{S}(a_1), a_3, \mathcal{S}(a_5))\alpha(a_2)\beta(a_4) = \varepsilon(a).
\end{align}
\end{definition}

Throughout this paper, we use the Sweedler sigma notation $\Delta(a) = a_1 \otimes a_2$ for the coproduct and $a_1 \otimes a_2 \otimes \cdots \otimes a_{n+1}$ for the result of the $n$-iterated application of $\Delta$ on $a$. We say $H$ has a bijective antipode if $\mathcal{S}$ is bijective.

When $H$ is equipped with the structure of a coquasi-Hopf algebra with antipode $(\mathcal{S}, \alpha, \beta)$, a preantipode of $H$ always exists.
\begin{lemma}\textup{[\citealp{preantipode}, Theorem 3.10]}\label{lem1.3}
Let $H$ be a coquasi-Hopf algebra with antipode $(\mathcal{S}, \alpha, \beta)$, then 
\begin{equation}
    \mathbb{S}:= \beta * \mathcal{S} *\alpha 
\end{equation}
is a preantipode for $H$, where $*$ denotes the convolution product. 
\end{lemma}

Following the definition in [\citealp{preantipode}], we recall the notion of a preantipode, which plays an important role in further calculations.
\begin{definition}\textup{[\citealp{preantipode}, Definition 3.6]}
    Let $(H,\Phi)$ be a coquasi-bialgebra, a preantipode for $H$ is a $\mathbbm{k}$-linear map $\mathbb{S}:H \rightarrow H$ such that, for all $h \in H$,
    \begin{align*}
        \mathbb{S}(h_1)_1h_2\otimes \mathbb{S}(h_1)_2&=1_H \otimes \mathbb{S}(h),\\
        \mathbb{S}(h_2)_1\otimes h_1\mathbb{S}(h_2)_2&=\mathbb{S}(h)\otimes 1_H,\\
\Phi(h_1,\mathbb{S}(h_2),h_3)&=\varepsilon(h).
    \end{align*}
\end{definition}

\subsection{ Yetter-Drinfeld module categories over coquasi-Hopf algebras}
 We now turn our attention to the Yetter-Drinfeld module category structure.
 The definition of the Yetter-Drinfeld module category $\HH$ over an arbitrary coquasi-Hopf algebra $H$ was already given in [\citealp{YD-module}]. Since it plays a crucial role in our paper, we recall this definition.
 \begin{definition}\textup{[\citealp{YD-module}, Definition 3.1]}\label{def2.5}
Let $H$ be a coquasi-Hopf algebra with associator $\Phi$. A left-left Yetter--Drinfeld module over $H$ is a triple $(V, \delta_V, \rhd)$ such that:

\begin{itemize}
    \item $(V, \delta_V)$ is a left comodule of $H$ and we denote $\delta_V(v)$ by $v_{-1} \otimes v_0$ as usual;
    
    \item $\rhd: H \otimes V \to V$ is a $\mathbbm{k}$-linear map satisfying for all $h, l \in H$ and $v \in V$:
    \begin{align}
   & (hl) \rhd v = \frac{\Phi(h_2, (l_2 \rhd v_0)_{-1}, l_3)}{\Phi(h_1, l_1, v_{-1})\Phi((h_3 \rhd (l_2 \rhd v_0)_0)_{-1}, h_4, l_4)} 
    (h_3 \rhd (l_2 \rhd v_0)_0)_0, \label{1.9} \\
  &  1_H \rhd v = v, \label{1.10} \\ 
  & (h_1 \rhd v)_{-1} h_2 \otimes (h_1 \rhd v)_0 = h_1 v_{-1} \otimes (h_2 \rhd v_0).   \label{1.11}
    \end{align}
\end{itemize}
A morphism $f:(V,\delta_V,\rhd)\rightarrow (V',\delta_{V'},\rhd')$ is a colinear map $f:(V,\delta_V)\rightarrow (V',\delta_{V'})$ such that $f(h \rhd v)=h \rhd' f(v)$ for all $h\in H$.
\end{definition}
 The category $\HH$ is a $\mathbbm{k}$-linear  braided monoidal  abelian category. The unit object of $\HH$ is $\mathbbm{k}$, which is regarded as an object in $\HH$ via trivial structures. For $V,W \in \HH$, the tensor product of Yetter-Drinfeld modules is defined by:
\[
(V, \delta_V, \rhd) \otimes (W, \delta_W, \rhd) = (V \otimes W, \delta_{V \otimes W}, \rhd),
\]
where  $\delta_{V \otimes W}$ is given by 
\[
\delta_{V \otimes W} (v \otimes w) = v_{-1} w_{-1} \otimes v_0 \otimes w_0,
\]
and
\begin{equation}\label{1.12}
h \rhd (v \otimes w) =
\frac{\Phi (h_1, v_{-1} , w_{-2})\Phi((h_2\rhd v_0)_{-1}, (h_4 \rhd w_0 )_{-1}, h_5)}{\Phi ((h_2 \rhd v_0)_{-2}, h_3, w_{-1}) }
  (h_2 \rhd v_0)_0 \otimes (h_4 \rhd w_0)_0.
\end{equation}
The braiding $
c_{V,W} : V \otimes W \rightarrow W \otimes V
$
is given by:
\begin{equation}\label{1.13}
c_{V,W} (v \otimes w) = (v_{-1} \rhd w) \otimes v_0.
\end{equation}
\par 
In \cite{rightyd}, the authors proved that the category $\mathcal{YD}^H_H$ of right-right  Yetter-Drinfeld module is braided equivalent to $\mathcal{Z}_{l}(\mathcal{M}^H)$, which is the left center of the category of right comodules of $H$. The proof for $\HH$ is completely dual. We need to recall the definition of right center here for completeness.\par 
\begin{definition}
For a coquasi-Hopf algebra \( H \) with bijective antipode, the right center of the category of left \( H \)-comodules \( \mathcal{Z}_r(^H\mathcal{M}) \) has objects \((V, c_{-,V})\), where \( V \in {^H\mathcal{M}} \) and \( c_{W,V} : W \otimes V \longrightarrow V\otimes W \) is a natural transformation in \(^H\mathcal{M} \) that satisfies the commutative diagram
  \begin{center}
	\begin{tikzpicture}
		\node (A) at (0,0) {$ (X\ot W) \ot V$};
		\node (B) at (4,0) {$V \ot (X\ot W) $};
		\node (C) at (-2,-2) {$X \ot (W \ot V)$};
  	\node (D) at (0,-4) {$X \ot (V \ot W)$};
   \node (E) at (4,-4) {$(X\ot V) \ot W$};
   \node (F) at (6,-2) {$(V\ot X) \ot W$};
		\draw[->] (A) --node [above] {$c_{X\ot W,V}$} (B);
		\draw[->] (A) --node [ above left] {$a_{X,W,V}$} (C);	
		\draw[->] (C) --node [below left] {$\operatorname{id}_X\ot c_{W,V}$} (D)	;
		\draw[->] (D) --node [above] {$a^{-1}_{X,V,W}$} (E)	;
  	\draw[->] (E) --node [ below right] {$c_{X,V}\ot \operatorname{id}_W$} (F);	
		\draw[->] (F) --node [above right] {$a_{V,X,W}$} (B)	;
		\end{tikzpicture}
\end{center}
for all \( W, X \in ^H\mathcal{M} \) and \( c_{\mathbbm{k},V} = \mathrm{id}_V \). A morphism \( f : (V, c_{-,V}) \longrightarrow (V', c_{-,V'}) \) in the  center is a left \( H \)-comodule map satisfying
$$ c_{W,V'}\circ (\operatorname{id}_W \ot f)=(f\ot \operatorname{id}_W)\circ c_{W,V}.$$
for all \( W \in {^H\mathcal{M}} \).
 \end{definition}
 The category \( \mathcal{Z}_r(^H\mathcal{M}) \) is braided monoidal  with  tensor product \((V, c_{-,V}) \otimes (W, c_{-,W}) = (V \otimes W, c_{-,V \otimes W})\). 
 \begin{lemma}\label{lem1.6}
Let $H$ be a coquasi-Hopf algebra with bijective antipode. The category $\HH$ is braided monoidal isomorphic to the right center of the category of left \( H \)-comodules \(\mathcal{Z}_r(^H\mathcal{M})\).
 \end{lemma}
\begin{proof}
We construct mutually inverse braided monoidal functors between the two categories.\par 
 Let \((V, c_{V,-})\) be an object in  \(\mathcal{Z}(^H\mathcal{M})\); we can define a left linear map on $V$ as follows:
\begin{equation}
\rhd  : H \ot V \longrightarrow V, \quad h \rhd v = (\mathrm{id}_V \ot \varepsilon )c_{H,V}(h \otimes v). \label{2.1}    
\end{equation}
Let \( W \in {^H \mathcal{M}}\) and \( w^* \in W^* \). Define the map \( f_{w^*} : W \longrightarrow H, f_{w^*}(w) = w_{-1}w^*(w_0) \). Then \( f_{w^*} \) is an \( H \)-comodule map. By naturality of \( c_{-,V} \), we have
\begin{align*}
    c_{H,V}\circ (f_{w^*}\ot \operatorname{id}_V)(w\ot v)&=(\operatorname{id}_V \ot f_{w^*})\circ c_{W,V}(w\ot v)\\
    &=c_{H,V}(w_{-1}w^*(w_0) \ot v).
\end{align*} 
That is 
$$c_{H,V}(w_{-1}w^*(w_0) \ot v)=(\operatorname{id}_V \ot f_{w^*})\circ c_{W,V}(w\ot v).$$
Applying \( \operatorname{id}_V \ot \varepsilon \) to both sides leads to 
$$ w^*(w_0)(w_{-1}\rhd v)=(\operatorname{id}_V \ot \varepsilon \circ f_{w^*})\circ c_{W,V}(w\ot v)=(\operatorname{id}_V \ot w^*)\circ c_{W,V}(w\ot v).$$
Thus
\[
c_{W,V}(w\ot v)=(w_{-1}\rhd v)\ot w_0
\]
for every $W\in{}^H\mathcal M$.\par 
Since $c_{\mathbbm{k},V}=\operatorname{id}_V$ and the unit map
$u:\mathbbm{k}\to H$ is a left $H$-comodule map, equation
(\ref{1.10}) gives
\[
1_H\rhd v=\varepsilon(1_H)v=v.
\]
Next, since $c_{H,V}$ is a left $H$-comodule map, for
$h\in H$ and $v\in V$ we have
 $$(h_1 \rhd v)_{-1}h_2 \ot (h_1 \rhd v)_0 \ot h_3=\delta((h_1 \rhd v) \ot h_2)=h_1v_{-1}\ot (h_2 \rhd v_0) \ot h_3.$$
 Applying $\varepsilon$ to the last tensor factor gives
(\ref{1.11}). Relation (\ref{1.9}) follows from the hexagon axiom
applied to $W=X=H$. Indeed,
$c_{H\ot H,V}((h\ot l)\ot v)
=(h_1l_1\rhd v)\ot(h_2\ot l_2)$, and the hexagon axiom gives
 $$  
 \frac{\Phi(h_2, (l_2 \rhd v_0)_{-1}, l_3)}{\Phi(h_1, l_1, v_{-1})\Phi((h_3 \rhd (l_2 \rhd v_0)_0)_{-1}, h_4, l_4)}  (h_3 \rhd (l_2 \rhd v_0)_0)_0  \ot (h_5 \ot l_5).
 $$
Applying $\varepsilon$ to the second and third tensor factors gives
(\ref{1.9}). Thus $(V,\delta,\rhd)$ is a left-left
$H$-Yetter--Drinfeld module. Moreover, every morphism
\[
f:(V,c_{-,V})\longrightarrow(V',c_{-,V'})
\]
is a morphism in $\HH$.

Conversely, let $V$ be a Yetter--Drinfeld module. For every left
$H$-comodule $W$, define $c_{W,V}$ by (\ref{1.13}). Equation
(\ref{1.9}) gives the hexagon axiom, while (\ref{1.10}) gives
$c_{\mathbbm{k},V}=\operatorname{id}_V$. Equation (\ref{1.11})
implies that $c_{W,V}$ is a left $H$-comodule map. Hence
$(V,c_{-,V})$ is an object of $\mathcal{Z}_r({}^H\mathcal{M})$.
Similarly, every morphism in $\HH$ is a morphism in
$\mathcal{Z}_r({}^H\mathcal{M})$.  \par
The two constructions are clearly inverse to each other. Furthermore, the braiding in $\HH$ given by (\ref{1.13}) coincides with the braiding of the right center under the above correspondence. Hence we obtain a braided monoidal isomorphism.
\end{proof}
In order to simplify notation, we recall the following linear maps 
\( p_R \), \( q_R \), \( p_L \), \( q_L \)\(\in \operatorname{Hom}(H\otimes H, \mathbbm{k}) \), which is introduced in \cite{rightyd},
\[
\begin{aligned}
p_R(h, g) &= \Phi^{-1}(h, g_1 \beta(g_2), \mathcal{S}(g_3)), \\
q_R(h, g) &= \Phi(h, g_3, \alpha (\mathcal{S}^{-1}(g_2)) \mathcal{S}^{-1}(g_1)), \\
p_L(h,g)&=\Phi(\mathcal{S}^{-1}(h_3)\beta(\mathcal{S}^{-1}(h_2)),h_1,g),\\
q_L(h,g)&=\Phi^{-1}(\mathcal{S}(h_1)\alpha(h_2),h_3,g),
\end{aligned}
\]
for any \( h, g \in H \).

These  maps satisfy various compatibility relations with the antipode and associator, see \cite{rightyd} for the complete list.
\begin{rmk} \label{lem2.6}\upshape
Let  $H$ be a coquasi-Hopf algebra with  bijective antipode.\par 
  {(1)} \textup{[\citealp{reflection1}, Lemma 2.6]} The equation (\ref{1.11}) is equivalent to  
\begin{equation}\label{1.17}
    (h\rhd v)_{-1} \otimes (h\rhd v)_0=p_R((h_3 \rhd v_0)_{-1},h_4)   q_R(h_1v_{-2},\mathcal{S}(h_6))    (h_2v_{-1})\mathcal{S}(h_5)      \otimes (h_3 \rhd v_0)_0.
\end{equation}\par 
 {(2)} \textup{[\citealp{reflection1}, Lemma 2.6]} The equation (\ref{1.11}) implies:
\begin{equation}
    v_{-1} \otimes (h\rhd v_0)=q_L(h_3,v_{-1})p_L(\mathcal{S}(h_1),(h_4\rhd v_0)_{-1}h_6)\mathcal{S}(h_2)((h_4\rhd v_0)_{-2}h_5)\otimes (h_4 \rhd v_0)_0.
\end{equation}\par
 {(3)}[\citealp{LZW15}, Theorem 2.8] The inverse  braiding is given by:
\begin{equation}
\begin{aligned}
    c_{V,W}^{-1}(w\ot v)&=\Phi^{-1}(\mathcal{S}(v_{-2}),w_{-1},v_{-5})q_L( \mathcal{S}^{-1}(v_{-1}),w_{-2}v_{-6})p_R((\mathcal{S}^{-1}((v_{-3}\rhd w_0)_{-1})),\mathcal{S}^{-1}(v_{-4}))\\
  &v_0\ot (\mathcal{S}^{-1}(v_{-3})\rhd w_0)_0.
\end{aligned}     
\end{equation}\par 
 {(4)} [\citealp{LZW15}, Proposition 2.5]
The coquasi-Hopf algebra $H$ itself can be viewed as an object in $\HH$ via the structures
\begin{equation}
    \begin{aligned}
        &\delta_H(h)=h_1\mathcal{S}(h_3)\ot h_2, \\
        &h \rhd h'=\Phi(h_1,h_1',\mathcal{S}(h_7'))\Phi(h_2h_2', \mathcal{S}(h_4h_4'),h_7)g(h_5,h_5')q_R(\mathcal{S}(h_6'),\mathcal{S}(h_6))h_3h_3'.
    \end{aligned}
\end{equation}
Here $$ g(h,h'):=\Phi^{-1}(\mathcal{S}(h_1h_1'),h_3h_3',\mathcal{S}(h_5')\mathcal{S}(h_5))\chi(h_4,h_4')\alpha(h_2h_2'),$$ and 
$$ \chi(h,h')=\Phi(h_1h'_1,\mathcal{S}(h_5'),\mathcal{S}(h_4))\Phi^{-1}(h_2,h_2',\mathcal{S}(h_4'))\beta(h_3)\beta(h_3').$$

\end{rmk}

\subsection{Bosonization and Radford projection for  coquasi-Hopf algebras}
Bosonization combines a Hopf algebra in a Yetter--Drinfeld category
with the underlying coquasi-Hopf algebra. We recall the
coquasi-Hopf version from \cite{YD-module}, which will be used in the
study of projections, coinvariants, and reflections.\par
Let $H$ be a coquasi-Hopf algebra with associator $\Phi$. Suppose $R$  is a Hopf algebra in $\HH$, we denote 
$$ r^1 \otimes r^2:= \Delta_R(r).$$
\begin{lemma}\textup{[\citealp{YD-module}, Theorem 5.2]}
    Let us consider on $M=R\otimes H$ the following structures:
\begin{align*}
(r \otimes h)  (s \otimes k) &= \frac{\Phi(h_2,s_{-1}, k_2)\Phi(r_{-1},(h_3 \rhd s_0)_{-1},h_5k_4)}{\Phi (r_{-2},h_1, s_{-2}k_1)\Phi((h_3 \rhd s_0)_{-2}, h_4, k_3)}   r_0  (h_3 \rhd s_0)_0 \otimes h_6k_5, \\
 u_M(k) &= k1_R \otimes 1_H,\\
\Delta_M(r \otimes h) &= \Phi^{-1}(r_{-1}^1, r_{-2}^2, h_1)r_0^1 \otimes r_{-1}^2 h_2 \otimes r_0^2 \otimes h_3, \\
\varepsilon_M(r \otimes h) &= \varepsilon_R(r)\varepsilon_H(h),\\
\Phi_M((r \otimes h), (s \otimes k), (t \otimes l))& = \varepsilon_R(r)\varepsilon_R(s)\varepsilon_R(t)\Phi(h, k, l),
\end{align*}
where $r,s,t \in R$, $h,k,l \in H$.
With above operations, $R\ot H$ is a coquasi-bialgebra, we denote by $R\#H$.
\end{lemma}
\begin{rmk}\upshape
  The bosonization  $R\#H$ can be further equipped with an antipode, making it a coquasi-Hopf algebra. The antipode $(\mathcal{S},\alpha,\beta)$ is given by:
    \begin{align*}
        &\alpha(r\#h)=\varepsilon_R(r)\alpha_H(h) \\
        &\beta(r\#h)=\varepsilon_R(r)\beta_H(h), \\
        &\mathcal{S}(r\#h)=m_{R\#H}\circ(\mathcal{S}_H\ot \mathcal{S}_R)\circ c_{R,H}(r \ot h)=\mathcal{S}_H(r_{-1} \rhd h)\mathcal{S}_R(r_0).
    \end{align*}
    Here, $1\neq r \in R$, $h \in H$.
    In particular, if $r$ is a primitive element in $R$, then 
 $$ \mathcal{S}(r\#1)= \mathcal{S}_H(r_{-2})\alpha(r_{-1})\mathcal{S}_R(r_0)=-\mathcal{S}_H(r_{-2})\alpha(r_{-1})r_0.$$
    The antipode $\mathcal{S}$ is bijective if and only if  $\mathcal{S}_H$ and $\mathcal{S}_R$ are bijective.
    \end{rmk}
Now suppose $N$ and $H$ are both coquasi-Hopf algebras.  Furthermore, assume there exist morphisms of coquasi-Hopf algebras
$$ \pi: N \rightarrow H \ \text{ and } \sigma: H \rightarrow N$$
such that $\pi\sigma=\operatorname{id}_H$.
\begin{lemma}\textup{[\citealp{YD-module}, Theorem 5.8]} \label{Lemma 1.6}
   Under above assumptions, $L:=N^{\operatorname{co}H}$ is a Hopf algebra in $\HH$.
     For all $a \in N$, we  define the following linear map
   \begin{equation}
\tau(a)=\Phi_N(a_1,(\sigma \circ \mathbb{S}_H\circ \pi(a_3))_1, a_4)a_2(\sigma \circ \mathbb{S}_H\circ \pi(a_3))_2.
   \end{equation} 
   Then $\tau: N \rightarrow L$ is well-defined.
   For all $r, s \in L$, $h\in H$, $k \in \mathbbm{k}$, the Yetter-Drinfeld module structure of $L$ is given by
   \begin{align*}
       \operatorname{ad}(h)(r)&:=\tau(\sigma(h)r)=\Phi_{H}(h_{1}r_{-1},\mathbb{S}(h_3)_1,h_4)(h_2r_0)\mathbb{S}(h_3)_2, \\  r_{-1}\otimes r_0 &:=\rho_R(r):=\pi(r_1)\otimes r_2.
       \end{align*} Moreover, $L$ is a Hopf algebra in $\HH$ via
       \begin{align*}
           &m_L(r\otimes s):=rs, \ \ \ \  \ \  \ u_L(k)=k1_N, \\ 
       &\Delta_L(r):=\tau(r_1) \otimes \tau(r_2), \ \ \ \ \varepsilon_L(r):=\varepsilon_N(r),\\
&\mathcal{S}_R=\operatorname{ad}\circ (\operatorname{id}_H\ot \tau \circ \mathcal{S}_L \circ i)\circ \rho_R,
   \end{align*}
   where $i:L \rightarrow N$ is the embedding map.
   Furthermore, there is an isomorphism of coquasi-Hopf algebra $\phi: L\# H \rightarrow N$ given by 
   $$ \phi(r \otimes h)=r \sigma(h), \ \  \phi^{-1}(a)=\tau(a_1)\otimes \pi(a_2).$$
\end{lemma}

 The following condition occurs frequently in this paper. We give an example here. 
\begin{exm} \upshape
Let $R,R'$ be two Hopf algebras in $\HH$ together with a surjective Hopf algebra map $\pi': R' \rightarrow R$ and an injective map $\sigma': R \rightarrow R'$ such that $ \pi' \circ \sigma'= \operatorname{id}_R$. Let $M'=R'\# H$, and $M=R \# H$. Obviously, these maps naturally extend to yield a surjective coquasi-Hopf algebra map 
$\pi:M' \rightarrow M$ and an injective coquasi-Hopf algebra map
$\sigma:M \rightarrow M'$, such that $\pi \circ \sigma=\operatorname{id}_M$. \par

    Let $K:=M'^{\operatorname{co}M}$, then $K$ is an object in  $\HHY$, with the following structures:
    \begin{align*}
    &   \operatorname{ad}: M \otimes K \rightarrow K,  h \otimes k \mapsto \Phi_{M}(h_{1}k_{-1},\mathbb{S}_M(h_3)_1,h_4)(h_2k_0)\mathbb{S}_M(h_3)_2, \\
    &\operatorname{\rho}: K\rightarrow M \otimes K, k \mapsto \pi(k_1) \otimes k_2.
    \end{align*}
    To analyze the linear map $\operatorname{ad}$ in concrete calculations, the following computation is useful.
  Let $y\in R$ be a primitive element. One can  compute the preantipode $\mathbb{S}_M(y)$  in $M$. 
    \begin{align*}
        \mathbb{S}_M(y)= \beta * \mathcal{S}_M* \alpha(y)=\beta(y_1)\mathcal{S}_M(y_2)\alpha(y_3)=\beta(y_{-1})\mathcal{S}_M(y_0)=-\beta(y_{-3})\mathcal{S}_H(y_{-2})\alpha(y_{-1})y_0.
    \end{align*}  Then we can proceed to calculate $\operatorname{ad}(y)(x)$, where $x \in K$,
    \begin{equation}
\begin{aligned}\label{1.27}
    \operatorname{ad}(y)(x)&=yx+(y_{-1}x)\mathbb{S}_M(y_0)\\&=yx-(y_{-4}x)\beta(y_{-3})(\mathcal{S}_H(y_{-2})\alpha(y_{-1})y_0)\\
    &=yx-\Phi(y_{-7}x_{-1},\mathcal{S}_H(y_{-3}),y_{-1})\beta(y_{-6})((y_{-5}x_0)\mathcal{S}_H(y_{-4}))\alpha(y_{-2})y_0.
\end{aligned}
\end{equation}
On the other hand, note that for all $h \in H$, $\Delta_H(\mathbb{S}(h))=\beta(h_1)\mathcal{S}_H(h_3)\ot \mathcal{S}_H(h_2)\alpha(h_4)$. Therefore, we have
\begin{align*}
   yx-\operatorname{ad}(y_{-1})(x)y_0&=yx-\Phi(y_{-4}x_{-1},\mathbb{S}_M(y_{-2})_1,y_{-1}))((y_{-3}x_0)\mathbb{S}_M(y_{-2})_2)y_0\\
    &=yx-\Phi(y_{-7}x_{-1},\mathcal{S}_H(y_{-3}),y_{-1})\beta(y_{-6})((y_{-5}x_0)\mathcal{S}_H(y_{-4}))\alpha(y_{-2})y_0.
\end{align*}
This implies \begin{equation}
    \operatorname{ad}(y)(x)=yx-\operatorname{ad}(y_{-1})(x)y_0.
\end{equation}
\end{exm}
\subsection{Yetter-Drinfeld modules in arbitrary braided monoidal categories}
To study Yetter--Drinfeld modules over Hopf algebras in $\HH$, we
use the general theory of Yetter--Drinfeld modules in braided
monoidal categories. We recall the definition from \cite{bes97};
see also \cite[Section~3.4]{rootsys}.
Now we let $\mathcal{C}$ be a braided monoidal category and $R$ a Hopf algebra in $\mathcal{C}$. The category of Yetter-Drinfeld modules  $\RRC$ is well-defined. 
\begin{definition}
    Let $R$ be a Hopf algebra with bijective antipode in a braided monoidal category $\mathcal{C}$. Suppose
that $X$ is a left  module and left comodule over $R$ with structure map $\rho_X$ and $\delta_X$. The triple $(X,\rho_X,\delta_X)$ is a Yetter-Drinfeld module if in addition, in $\operatorname{Hom}_{\mathcal{C}}(R\otimes X, R\otimes X)$,
\begin{align*}
&(\mu_R\otimes\operatorname{id})
\circ a_{R,R,X}^{-1}
\circ(\operatorname{id}\otimes c_{X,R})
\circ a_{R,X,R}
\circ(\delta_X\circ\rho_X\otimes\operatorname{id})\\
&\quad{}\circ a_{R,X,R}^{-1}
\circ(\operatorname{id}\otimes c_{R,X})
\circ a_{R,R,X}
\circ(\Delta_R\otimes\operatorname{id})\\
&=(\mu_R\otimes\rho_X)
\circ a_{R,R,R\otimes X}^{-1}
\circ\bigl(\operatorname{id}\otimes
a_{R,R,X}\circ(\operatorname{id}\otimes c_{R,R})
\otimes\operatorname{id}\bigr)\\
&\quad{}\circ(\operatorname{id}\otimes a_{R,R,X}^{-1})
\circ a_{R,R,R\otimes X}
\circ(\Delta_R\otimes\delta_X).
\end{align*}
\end{definition}
It is  known that the category $\RRC$ is a  braided monoidal category. For $X,Y \in \RRC$ the braiding isomorphism in $\RRC$ is given by 
\begin{align*}
    c_{X,Y}^{\mathcal{YD}}&:X \otimes Y \rightarrow Y \otimes X, \\
 c_{X,Y}^{\mathcal{YD}}&:=(\rho_Y \otimes \operatorname{id}_X)\circ a^{-1}_{R,Y,X}\circ  (\operatorname{id}_R\otimes c_{X,Y})\circ a_{R,X,Y} 
 \circ(\delta_X \otimes \operatorname{id}_Y).
\end{align*}
\begin{rmk}\upshape \label{rmk2.13}
\par (1) According to the  definition of $ c_{X,Y}^{\mathcal{YD}}$, it is not hard to see, if $X \in {^R\mathcal{C}}$ and $Y \in {_R\mathcal{C}}$, then $ c_{X,Y}^{\mathcal{YD}}$ is a morphism in $\mathcal{C}$.\par
(2) [\citealp{rootsys}, Proposition 3.4.5]
Let \( V \) be an object in \( \mathcal{C} \), \((V, \lambda) \in {_R\mathcal{C}}\), and \((V, \delta) \in {^R\mathcal{C}}\).  
Then the following are equivalent.
\begin{enumerate}
    \item[(i)] \((V, \lambda, \delta)\) is a left Yetter-Drinfeld module over \( R \).
    \item[(ii)] For all \( X \in {_R\mathcal{C}} \), \( c_{V,X}^{\mathcal{YD}} \) is a morphism in \( {_R\mathcal{C}} \).
    \item[(iii)] \( c_{V,R}^{\mathcal{YD}} \) is a morphism in {\( _R\mathcal{C} \)}, where \( R \) is a left \( R \)-module by the multiplication in \( R \).
    \item[(iv)] For all \( X \in {^R\mathcal{C}} \), \( c_{X,V}^{\mathcal{YD}} \) is a morphism in \( {^R\mathcal{C}} \).
\end{enumerate}

\par
    (3) In general, let $(\mathcal{C}, \ot, \bfo, c)$ be a braided monoidal category, we can define the reverse monoidal category $\overline{\mathcal{C}}:=(\mathcal{C},\ot ,\bfo, \overline{c})$, where for $X,Y \in \mathcal{C}$,
$$ \overline{c}_{X,Y}:=c^{-1}_{Y,X}: X\ot Y \rightarrow Y \ot X.$$
For $X \in {_R\mathcal{C}}$ and $Y \in {^R\mathcal{C}}$, let
\begin{equation}
\overline{c}^{\mathcal{YD}}_{X,Y}=\overline{c}_{X,Y}\circ (\lambda_X \ot \operatorname{id}_Y)\circ ((\mathcal{S}^{-1}_R\ot \operatorname{id}_X )\ot\operatorname{id}_Y)\circ   
   (\overline{c}_{X,R} \ot \operatorname{id}_Y)\circ a^{-1}_{X,R,Y} \circ (\operatorname{id}_X \ot \delta_Y).
\end{equation}
By [\citealp{rootsys}, Proposition 3.4.8], $c^{\mathcal{YD}}_{Y,X}$ is an isomorphism in $\mathcal{C}$ with inverse $\overline{c}_{X,Y}^{\mathcal{YD}}$. 

\end{rmk}

The definition of the bosonization  naturally extends to the case where $K$  is a Hopf algebra in the
braided monoidal category $\RRC$, where $\mathcal{C}$ is now an arbitrary braided monoidal  category.
 We collect several key results from [\citealp{bes97},\citealp{B95},\citealp{AF00}] that will be important for the construction.
 \begin{lemma}\textup{[\citealp{bes97}, Theorem 3.9.5]}\label{lemequ}
     There is a braided monoidal isomorphism from the Hopf bimodule category to the Yetter-Drinfeld module category:
     $$ {^R_R\mathcal{C}^R_R}\cong \RRC.$$
 \end{lemma}
 \begin{lemma}\textup{[\citealp{B95}, Proposition 4.2.3]}\label{thm1.12}
    Let \( A \) be a Hopf algebra in \( \mathcal{C} \) and \( K \) a Hopf algebra in \( \AAC \). There is an obvious isomorphism of braided monoidal categories
    \begin{equation}
        {^{K\#A}_{K\#A}\mathcal{YD}}(\mathcal{C})\cong  {^{K}_{K}\mathcal{YD}}(\AAC).
    \end{equation}
\end{lemma}
\begin{lemma}\label{lem2.12}
\textup{[\citealp{AF00}, Theorem 3.2]} Let \( H \) and \( A \) be Hopf algebras in a braided monoidal category \( \mathcal{C} \). Let \(\pi : H \to A\) and \(\iota : A \to H\) be Hopf algebra morphisms such that \(\pi \circ \iota = \mathrm{id}_{A}\). If \( \mathcal{C} \) has equalizers and \( A \otimes (-)\) preserves equalizers, there is a Hopf algebra \( R \) in the braided monoidal category $\AAC$, such that  
\[ H \cong R \# A. \]    
\end{lemma}
\begin{rmk}\upshape  Let $H$ be a coquasi-Hopf algebra with bijective antipode.\par 
  (1) By [\citealp{quasibon}, Proposition 3.9, Lemma 4.5] and Lemma  \ref{lemequ}, we have a braided monoidal equivalence
  $$ \RRH\cong \RRC,$$where $\mathcal{C}=\HH$.\par 
  (2) Note that $\HH$ is an abelian category. Moreover, tensoring with $A$ is exact in the present category. Therefore the conditions in Lemma \ref{lem2.12} hold automatically. That is, we can talk about bosonization and taking coinvariants in $\HH$ without any restrictions.
\end{rmk}

 \section{Rational Yetter-Drinfeld modules  over coquasi-Hopf algebras}
 In this section, $H$ always represents a coquasi-Hopf algebra with bijective antipode, and $\mathcal{C}=\HH$ denotes the Yetter-Drinfeld module category of $H$ . We investigate rational Yetter-Drinfeld modules in $\HH$, a notion that will be important for our later study of Nichols algebras and their reflections.

\subsection{Rational  modules in $\HH$}
Since $H$ is not associative as an ordinary algebra, the usual
notion of an $H$-module is not available. However, if $R$ is an
algebra  in $\mathcal C=\HH$, then the category
${}_R\mathcal C$ of left $R$-modules in $\mathcal C$ is
well-defined. Motivated by \cite[Definition~12.2.3]{rootsys}, we
introduce rational modules in this setting.
\begin{definition} \label{def3.6} Let $R=\bigoplus_{n \geq 0} R(n)$ be an $\mathbb{N}_0$-graded Hopf algebra in $\mathcal{C}$.\par 
 {(1)}
    A left  module $(X,\lambda_X)$ over $R$ in $\mathcal{C}$  is called rational if for any $x \in X$, there exists a natural number $n_0$ such that $\lambda(R(n)\ot x)=0$ for all $n \geq n_0$. We denote the category of left rational modules over $R$ in $\mathcal{C}$ by ${_R\mathcal{C}}_{\operatorname{rat}}$.\par 
     {(2)} A left Yetter-Drinfeld module $X$ in $\RRC$  is called rational if $X$ is rational as a module of $R$ in $\mathcal{C}$. We denote the corresponding category  by $\RRC_{\operatorname{rat}}$.
\end{definition}
\begin{rmk}\upshape
We say that a $\mathbb{N}_0$-graded Hopf algebra $R$ is locally finite if $R(n)$ is finite-dimensional for all $n \geq 0$.
It is worth noting that for a locally finite $\mathbb{N}_0$-graded Hopf algebra $R$, this graded definition of rationality coincides with the classical one. Intuitively, the condition $\lambda(R(n) \otimes x) = 0$ for large $n$ implies that the action on any element $x$ is nontrivial on a finite-dimensional truncation of $R$, which recovers the standard notion of rationality in the non-graded setting.\end{rmk}
A natural question is whether rational modules are closed under tensor products. The following lemma ensures this.

\begin{lemma}\label{lem5.2}
    Let $R=\bigoplus_{n \geq 0} R(n)$ be an $\mathbb{N}_0$-graded Hopf algebra in $\mathcal{C}$.\par 
     {(1)}
    The category ${_R\mathcal{C}}_{\operatorname{rat}}$ is a monoidal subcategory of ${_R\mathcal{C}}$, which is closed under  direct sums, subobjects and quotient objects.\par
    {(2)} The category $\RRC_{\operatorname{rat}}$ is a monoidal subcategory of $\RRC$, which is closed under  direct sums, subobjects and quotient objects.\par
\end{lemma}

\begin{proof}
(1) Let $V, W \in {_R\mathcal{C}}_{\operatorname{rat}}$ with $R$-action $\lambda_V$, $\lambda_W$ respectively.  
It is standard that the tensor product $V \ot W$ is an object in $_R\mathcal{C}$, with $R$-action 
\begin{align*}
    \lambda_{V\ot W}&=(\lambda_V \ot \lambda_W)\circ(a_{R,V,R}^{-1}\ot\operatorname{id}_W)\circ a_{R\ot V,R,W}\circ ((\operatorname{id}_R \ot c_{R,V})\ot \operatorname{id}_W)\\& \circ (a_{R,R,V}\ot\operatorname{id}_W)\circ a^{-1}_{R\ot R,V,W}\circ (\Delta_R \otimes \operatorname{id}_{V\ot W}).
\end{align*}
Now take $r \in R(m)$, we have 
\begin{align*}
\lambda_{V\ot W}(r\ot(v\ot w))
&=
\frac{
\Phi(r^2_{-4},v_{-1},w_{-3})
\Phi(r^1_{-1},(r^2_{-3}\rhd v_0)_{-1},r^2_{-1}w_{-1})
}{
\Phi(r^1_{-2},r^2_{-5},v_{-2}w_{-4})
\Phi((r^2_{-3}\rhd v_0)_{-2},r^2_{-2},w_{-2})
}\\
&\quad{}\cdot
\lambda_V(r^1_0\ot(r^2_{-3}\rhd v_0)_0)
\otimes\lambda_W(r^2_0\ot w_0).
\end{align*}
Since $r^2_{-3} \in H$, and $V$ is an object in $\HH$, we have $(r^2_{-3}\rhd v_0)_0 \in V$. \par The expression $\lambda_{V\ot W}(r \ot (v\ot w))$ is a finite sum. Thus,  by the definition of the rational module, there exists a natural number $n_1$ such that for $n \geq n_1$,
$$ \lambda_V (R(n)\ot (r^2_{-3}\rhd v_0)_0)=0.$$
Similarly, there is a natural number $n_2$ such that  $$ \lambda_W(R(n) \ot w_0)=0, 
$$  for all   $n \geq n_2$.\par
Now we assume $m \geq 2\operatorname{max}\lbrace n_1,n_2 \rbrace$, since $R$ is $\mathbb{N}_0$-graded, we have $\Delta_R(r) \in \bigoplus_{p+q=m}R(p)\ot R(q).$ Therefore  
$$\lambda_{V\ot W}(R(m) \ot (v\ot w))=0.$$
This implies $V\ot W \in {_R\mathcal{C}_{\operatorname{rat}}}$.
Therefore, ${_R\mathcal{C}}_{\operatorname{rat}}$ is a monoidal subcategory of ${_R\mathcal{C}}$. The category ${_R\mathcal{C}}_{\operatorname{rat}}$ is obviously closed under arbitrary direct sums, subobjects and quotient objects.
\par 
(2) The forgetful functor $ \RRC\rightarrow {_R\mathcal{C}}$ is monoidal. Thus, for $V,W \in \RRC_{\operatorname{rat}},$ we have $V\ot W \in {_R\mathcal{C}}_{\operatorname{rat}} $. By definition it implies $V\ot W \in \RRC_{\operatorname{rat}}$. \par 
\end{proof}

\subsection{Properties of  pairings }
We now turn to the study of pairings in the category 
$\mathcal{C}=\HH$, which will later facilitate the construction of dualities between comodules and rational modules.\par 
Let $A$ and $B$ be objects in $\mathcal{C}$, and $\omega: A\otimes B \rightarrow \mathbbm{k}$ is a morphism in $\mathcal{C}$, called a pairing.  For subsets $X \subseteq A$ and $Y \subseteq B$, we define 
\begin{align*}
     X^{\perp}&:=\lbrace b \in B\mid \omega(x, b)=0, \ \text{for all}\ x\in X \rbrace, \\
          Y^{\perp}&:=\lbrace a \in A\mid \omega(a, y)=0, \ \text{for all}\ y\in Y \rbrace. 
\end{align*}
We say $\omega$ is non-degenerate if $A^{\perp}=0$ and $B^{\perp}=0$.\par 

\begin{lemma}\label{lem5.3}
    If $E$ is a subobject of $A$, then $E^{\perp}$ is a subobject of $B$. Similarly, if $F$ is  a subobject of $B$, then $F^{\perp}$ is a subobject of $A$.
\end{lemma}

\begin{proof}
Let $E$ be a subobject of $A$. We must show that $E^{\perp}$ is
an object of $\mathcal C$. The proof that $E^{\perp}$ is an
$H$-comodule is the same as in the Hopf algebra case; see
[\citealp{HSdualpair}, Lemma~2.3]. It remains to prove that
$H\rhd E^{\perp}\subseteq E^{\perp}$.
   \par 
   Thus, for $b\in E^{\perp}$, $h\in H$, and $e\in E$, it suffices
to verify that
\[
\omega(e,h\rhd b)=0.
\]
   By hexagon axiom in $\mathcal{C}$, we have the following commutative diagram:
    \begin{center}
	\begin{tikzpicture}
		\node (A) at (0,0) {$ A \ot (H \ot B)$};
		\node (B) at (4,0) {$A\ot (B \ot H)$};
		\node (C) at (-2,-2) {$(A \ot H) \ot B$};
  	\node (D) at (0,-4) {$(H \ot A) \ot B$};
   \node (E) at (4,-4) {$H \ot (A \ot B)$};
   \node (F) at (6,-2) {$(A \ot B) \ot H$};
		\draw[->] (A) --node [above] {$\operatorname{id}\ot c_{H,B}$} (B);
		\draw[->] (A) --node [ above left] {$a^{-1}_{A,H,B}$} (C);	
		\draw[->] (C) --node [below left] {$c^{-1}_{H,A} \ot \operatorname{id} $} (D)	;
		\draw[->] (D) --node [above] {$a_{H,A,B}$} (E)	;
  	\draw[->] (E) --node [ below  right] {$c_{H,A\ot B}$} (F);	
		\draw[->] (B) --node [above right] {$a_{A,B,H}^{-1}$} (F)	;
		\end{tikzpicture}
\end{center}
Since $E$ is an object in $\mathcal{C}$, for simplicity, we write $c_{A,H}^{-1}(e\ot h)=\sum_{1\leq i \leq n} h^i \ot e^i$ for some natural number $n$. Here $h^i \in H$, $e^i \in E$.
Now \begin{equation}
\begin{aligned}
 c_{H,A\ot B}\circ a_{H,A,B} \circ (c_{A,H}^{-1}\ot \operatorname{id})  ((e \ot h) \ot b)&=\sum_{1\leq i \leq n} c_{H,A\ot B}\circ a_{H,A,B}((h^i \ot e^i) \ot b)\\
 &=\sum_{1\leq i \leq n} \Phi(h^i_1,e^i_{-1},b_{-1})c_{H,A\ot B}(h^i_2 \ot (e^i_0 \ot b_0))\\
 &=\sum_{1\leq i \leq n}\Phi(h^i_1,e^i_{-1},b_{-1})(h^i_2 \rhd (e_0^i \ot b_0)) \ot h^i_3.
\end{aligned}
\end{equation}
Here $b_0\in E^{\perp}$ because $E^{\perp}$ is an $H$-comodule,
and $e_0^i\in E$ for all $i$. Since $\omega$ is a morphism in
$\mathcal C$ and the braiding $c$ is a natural isomorphism, we have
\begin{equation}
(\omega \ot \operatorname{id}) \circ c_{H,A\ot B}\circ a_{H,A,B} \circ (c_{A,H}^{-1}\ot \operatorname{id}) ((e \ot h) \ot b)=\sum_{1\leq i \leq n}\Phi(h^i_1,e^i_{-1},b_{-1}) \omega(e_0^i,b_0) \ot h^i_2=0.
\end{equation}   
Note that $a_{A,H,B}^{-1}(e\ot (h\ot b)) \in (A\ot H)\ot B$, 
thus by the hexagon axiom, we obtain:
$$
(\omega \ot \varepsilon) \circ a_{A,B,H}^{-1} \circ (\operatorname{id} \ot c_{H,B}) \circ a_{A,H,B} (a_{A,H,B}^{-1}(e\ot (h\ot b)))=0.
$$
That is 
\begin{equation}
    \Phi(e_{-1},(h_1 \rhd b)_{-1},h_2)(\omega \ot \varepsilon)((e_0 \ot (h_1 \rhd b)_0)\ot h_3)=\Phi*(\omega \ot \varepsilon)(e \ot (h_1 \rhd b )\ot h_2)=0.
\end{equation}
Note that $ \Phi$ is convolution invertible, thus 
\begin{equation}
   \omega(e, h \rhd b)=0
\end{equation}
   for all $e \in E$, $b \in E^{\perp}$, $h \in H$.
 Consequently, $E^{\perp}$ is a subobject of $B$. The argument for
$F^{\perp}$ is analogous.
\end{proof}
The following construction of the Yetter-Drinfeld submodule generated by a subcomodule is crucial for our subsequent analysis of rational modules.

Let $V \in {}_H^H\mathcal{YD}$ and let $V' \subseteq V$ be an $H$-subcomodule. We define the Yetter-Drinfeld submodule generated by $V'$, denoted by $\langle V' \rangle$, as the smallest Yetter-Drinfeld submodule of $V$ containing $V'$.

We provide an explicit filtration to construct $\langle V' \rangle$.
\begin{itemize}
    \item Set $X_0 := V'$.
    \item Assume $X_n$ is defined. Let $Y_n := \text{span}\{h \rhd x \mid h \in H, x \in X_n\}$.
    \item Define $X_{n+1}$ as the smallest $H$-subcomodule of $V$ containing $Y_n$:
    $$X_{n+1} := \operatorname{span}\{(f\otimes \operatorname{id}) \delta_V(y)\mid f\in H^*, y \in Y_n    \}.$$
\end{itemize}

\begin{lemma}
\label{lem:YD_generation}
With the notation above, $\langle V' \rangle = \bigcup_{n \ge 0} X_n$. 
\end{lemma}

\begin{proof}
It is direct to verify that $\bigcup_{n \ge 0} X_n$ is closed under the $H$-action and $H$-coaction by construction. Since any Yetter-Drinfeld submodule containing $V'$ must contain $X_0$ and be closed under the action and coaction, it must contain each $X_n$. Thus, the union is minimal.
\end{proof}

Let \( V, W \in \mathcal{C} \), and $\omega$ a pairing of \( A, B \) in \( \mathcal{C} \). We denote by \(\operatorname{Hom}_{\mathcal{C},\text{rat}}(A \otimes V, W)\) the set of all \( g \) in \(\operatorname{Hom}_{\mathcal{C}}(A \otimes V, W)\) such that for all \( v \in V \) there is a finite-dimensional subobject \( F \subseteq B \) in \( \mathcal{C} \) with \( g(F^{\perp} \otimes v) = 0 \).

     \begin{prop}\label{prop5.5}
Let $A,B,W\in\mathcal C$, and let
$\omega:A\otimes B\to\mathbbm{k}$ be a non-degenerate pairing in
$\mathcal C$. Assume that every $b\in B$ is contained in a
finite-dimensional subobject of $B$. Then, for every
$V\in\mathcal C$,  the map

\begin{align*}
&G_V : \operatorname{Hom}_{\mathcal{C}}(V, B \otimes W) \to \operatorname{Hom}_{\mathcal{C},\operatorname{rat}}(A \otimes V, W),\\
&f \mapsto (A \otimes V \xrightarrow{\operatorname{id}_A \otimes f} A \otimes (B \otimes W) \xrightarrow{a^{-1}_{A,B,W}} (A\ot B) \ot  W \xrightarrow{\omega\ot \operatorname{id}}  W)
\end{align*}
is bijective.
\end{prop}
\begin{proof}
\textbf{Step1:}\par
We first show that $G_V$ is well-defined and injective. Let $f \in \operatorname{Hom}_{\mathcal{C}}(V, B \otimes W)$, and $g=G_V(f)$.  By assumption, for each $v \in V$, there is a finite-dimensional subobject $F \subseteq B$ with $f(v)=\sum_{i \in I} b^i \ot w^i \in F \ot W$. Now for any $a \in F^{\perp}$, by definition of $G_V$ and Lemma \ref{lem5.3}, we have 
 $$ g(a \ot v)=\Phi(a_{-1},b^i_{-1},w^i_{-1})\omega(a_0,b^i_0) \ot w^i_0=0.$$
This shows that
$G_V(f)\in
\operatorname{Hom}_{\mathcal C,\operatorname{rat}}(A\ot V,W)$. If $g=0$, since $\omega$ is non-degenerate, we deduce that $f=0$. Therefore $G_V$ is injective.

\textbf{Step2:}\par
Assume that $B$ is finite-dimensional.  Then $B$ has a left dual $B^*=\operatorname{Hom}(B, \mathbbm{k})$. Since $\omega$ is non-degenerate, there is an isomorphism $A \cong B^*$ in $\mathcal{C}$  via
the following isomorphism:
$$ \operatorname{Hom}_{\mathcal{C}}(A\ot B, \mathbbm{k})\cong \operatorname{Hom}_{\mathcal{C}}(A,B^*).$$
 Therefore $A$ is a left dual of $B$, which is finite-dimensional. We have such an isomorphism 
 $$ \operatorname{Hom}_{\mathcal{C}}(V, B\ot W) \cong \operatorname{Hom}_{\mathcal{C}}(A \ot V, W).$$
Since $\operatorname{Hom}_{\mathcal{C},\operatorname{rat}}(A\ot V,W) \subseteq  \operatorname{Hom}_{\mathcal{C}}(A \ot V, W),$ and $G_V$ is injective,  we deduce that $G_V$ is bijective when $B$ is finite-dimensional.

\textbf{Step3:}\par
We now extend the bijectivity to the general case when $B$ may be infinite-dimensional. Let $V'\subseteq V$ be a finite-dimensional $H$-subcomodule, and
denote by $\langle V'\rangle$ the Yetter--Drinfeld submodule of
$V$ generated by $V'$. Assume that $\langle V'\rangle=V$. We
prove that $G_V$ is surjective in this case.
Let $g \in \operatorname{Hom}_{\mathcal{C},\operatorname{rat}}(A\ot V, W) $. Since $V'$ is finite-dimensional, the rationality of $g$ implies the existence of  a finite-dimensional subobject $ F\subseteq B$ in $\mathcal{C}$ such that $g(F^{\perp}\ot V')=0$. By Lemma \ref{lem5.3}, $F^{\perp}$ is an object in $\mathcal{C}$. Since $g $ is a morphism in $\mathcal{C}$, the same  argument as in the proof of
Lemma~\ref{lem5.3} gives
\[
g(F^{\perp}\ot(h\rhd V'))=0
\]
for every $h\in H$.
In the construction of $\langle V'\rangle$ in
Lemma~\ref{lem:YD_generation}, set $X_0=V'$. The preceding
equality gives $g(F^{\perp}\ot Y_0)=0$. \par Recall that
\[
X_1
=
\sum_{y\in Y_0}
\operatorname{span}
\{y_0\mid\delta_V(y)=y_{-1}\ot y_0\}.
\]
We claim that $g(F^{\perp}\ot X_1)=0$. Since $g$ is a morphism
in $\mathcal C$, we have
\[
(\operatorname{id}_H\ot g)\circ
\delta(F^{\perp}\ot Y_0)
=
\delta(g(F^{\perp}\ot Y_0))
=
0.
\]
By the definition of $X_1$,
$\operatorname{span}
\{z_0\mid
\delta(z)=z_{-1}\ot z_0,\ 
z\in F^{\perp}\ot Y_0\}
=
F^{\perp}\ot X_1.
$
Hence $g(F^{\perp}\ot X_1)=0$.
 By iterating this argument, we can prove that 
  $$ g(F^{\perp} \ot \langle V' \rangle)=0.$$
Note that we have such a non-degenerate pairing
$$ A/{F^{\perp}}\ot F\rightarrow \mathbbm{k}, \ \overline{a}\ot b \mapsto \omega(a,b).$$
Since $F$ is finite-dimensional, by the situation considered before, we have a bijective map
$$ \operatorname{Hom}_{\mathcal{C}}(V, F\ot W) \cong \operatorname{Hom}_{\mathcal{C},\operatorname{rat}}(A/F^{\perp} \ot V, W).$$
Let $\overline g$ be the morphism induced by $g$, and let
$f\in\operatorname{Hom}_{\mathcal C}(V,F\ot W)$ be its preimage.
Then the preimage of $g$ is
$$ i \circ f : V \rightarrow B \ot W, $$ 
where $i : F \ot W \rightarrow B \ot W $ is an inclusion. Thus $G_V: \operatorname{Hom}_{\mathcal{C}}(V, B\ot W) \cong \operatorname{Hom}_{\mathcal{C},\operatorname{rat}}(A \ot V, W)$ is bijective in this case.

\textbf{Step4:}\par
It remains to prove the bijectivity for an arbitrary object $V \in \mathcal{C}$. Let 
$$ \mathcal{V}=\lbrace \langle V' \rangle \mid V'\subseteq V \ \operatorname{finite-dimensional} \ H-\operatorname{subcomodule}  \rbrace.$$
By the finiteness theorem for comodules, we have $V= \bigcup\limits_{V' \ \operatorname{finite-dimensional}}V'$ as comodules.
Note that for all $U_1, U_2 \in \mathcal{V}$, there is an object $U \in \mathcal{V}$ with $U_1 \cup U_2 \subseteq U$, since $U $ is closed under direct sums.  Thus $$
V=\bigcup_{U \in \mathcal{V}} U$$ as an object in $\mathcal{C}$.
Given
$g\in\operatorname{Hom}_{\mathcal C,\operatorname{rat}}(A\ot V,W)$,
let $g_U$ denote its restriction to $A\ot U$. By the preceding
case, there is a morphism $f_U:U\to B\ot W$ such that
$G_U(f_U)=g_U$. If $U_1\subseteq U_2$, then
$f_{U_2}|_{U_1}=f_{U_1}$ by the injectivity already proved.
Hence the maps $f_U$ define a linear map
$f:V\to B\ot W$ by $f(v)=f_U(v)$ whenever $v\in U$.
This map satisfies $G_V(f)=g$. Thus $G_V$ is bijective for
arbitrary $V$.
\end{proof}

\subsection{Bijection between comodules and rational modules}
We begin by introducing the notion of Hopf pairings. For our purposes, we assume that $A$ and $B$ are locally finite $\mathbb{N}_0$-graded Hopf algebras in $\mathcal{C}=\HH$ in this section. Since this does not affect our study of Nichols algebras, we assume that $A$ and $B$ have bijective antipodes in $\mathcal{C}$.
\begin{definition} 
   A morphism \( \omega : A \otimes B \to \mathbbm{k} \) in \( \mathcal{C} \) is called a Hopf pairing, if the following identities hold:
\begin{align}\label{5.12}
\omega \circ (\mu_A \otimes \operatorname{id}_B)& = \omega \circ (\operatorname{id}_A \otimes \omega \otimes \operatorname{id}_B)\circ  (\operatorname{id}_A\otimes a^{-1}_{A,B,B}) \circ a_{A,A,B\otimes B}\circ (\operatorname{id}_{A \otimes A} \otimes \Delta_B), \\ \label{5.13}
\omega \circ (\eta_A \otimes \operatorname{id}_B) &= \varepsilon_B,\\
\omega \circ (\operatorname{id}_A \otimes \mu_B)& = \omega \circ (\operatorname{id}_A \otimes \omega \otimes \operatorname{id}_B)\circ  (\operatorname{id}_A\otimes a^{-1}_{A,B,B}) \circ a_{A,A,B\otimes B} \circ (\Delta_A \otimes \operatorname{id}_{B \otimes B}),\\
\omega \circ (\operatorname{id}_A \otimes \eta_B)& = \varepsilon_A.
\end{align}
   Moreover, we call  a Hopf pairing $\omega : A \otimes B \to \mathbbm{k}$ is a dual pair of locally finite $\mathbb{N}_0$-graded Hopf algebras in $\mathcal{C}$, if $\omega$ is non-degenerate and
$$ \omega(A(n),B(m))=0, \ \operatorname{if } \ n\neq m.$$
\end{definition}
There are some useful properties of Hopf pairings, we list here for further use.
\begin{rmk}\textup{[\citealp{rootsys}, Proposition 3.3.8]}\upshape\label{rmk2.3} Under above assumptions, we have the following properties.
\begin{itemize}
    \item [(1)] Suppose  $\omega: A \otimes B \rightarrow \bfo$ is a dual pair in $\mathcal{C}$, then \begin{equation}\label{3.0}
        \omega\circ (\operatorname{id} \otimes \mathcal{S}_B)=\omega\circ (\mathcal{S}_A\otimes \operatorname{id}).
    \end{equation} 
    \item [(2)] If $\omega: A\otimes B \rightarrow \textbf{1}$  is a dual pair, then  \begin{align}\label{3.15}
        &\omega^{+}:=\omega \circ c_{B,A}\circ (\mathcal{S}_B\otimes \mathcal{S}_A): B \otimes A \rightarrow \mathbbm{k}, \\
        \label{3.16}&\omega^{+\operatorname{cop}}:=\omega^{+}\circ (\operatorname{id}_B\ot \mathcal{S}^{-1}_A ): B^{\operatorname{cop}}\ot A^{\operatorname{cop}}\rightarrow \mathbbm{k}
    \end{align} 
    are dual pairs too.
\end{itemize}
\end{rmk}
Recall that $\overline{C}$ means the reverse braided monoidal category.
The following lemma holds in general braided monoidal categories; here we cite this lemma for simplicity.
 \begin{lemma}\textup{[\citealp{rootsys}, Proposition 3.3.9]}\label{lem5.8}
  Let $\mathcal{C}$  be an arbitrary braided  monoidal category. Suppose $A$, $B$ are Hopf algebras in $\mathcal{C}$. Assume that $p: A \ot B \rightarrow \bfo$ is a Hopf pairing, then we have two strict monoidal functors: 
  \begin{equation}\label{5.19}
      \begin{aligned}
          D_1: \ & {^B\mathcal{C}}\longrightarrow {_{A^{\operatorname{cop}}}\overline{\mathcal{C}}}, \ \ (V,\delta) \mapsto (V, \overline{\lambda}), \\
          \textup{with } \ &\overline{\lambda}=(A \ot V \xrightarrow{\operatorname{id}_A\ot \delta} A \ot (B \ot V)\xrightarrow{a^{-1}_{A,B,V}}(A\ot B) \ot V \xrightarrow{p \ot \operatorname{id}_V} V) 
      \end{aligned}
  \end{equation}
  and 
   \begin{equation}
      \begin{aligned}
          D_2: \ & {^{A^{\operatorname{cop}}}\overline{\mathcal{C}}} 
          \longrightarrow {_B\mathcal{C}} , \ \ (V,\overline{\delta}) \mapsto (V, \lambda), \\
          \textup{with } \ &\lambda=(B \ot V \xrightarrow{\operatorname{id}_B\ot \overline{\delta}} B \ot (A \ot V)  \xrightarrow{a^{-1}_{B,A,V} } (B \ot A) \ot V \xrightarrow{p^{+\operatorname{cop}} \ot \operatorname{id}_V} V).
      \end{aligned}
  \end{equation}
 \end{lemma}
The main result of this subsection is to give a bijection between $B$-comodules  in $\mathcal{C}$ and rational $A^{\operatorname{cop}}$-modules in $\overline{\mathcal{C}}$.
\begin{prop}\label{prop5.9}
Let $(A,B, \omega)$ be a dual pair of locally finite $\mathbb{N}_0$-graded Hopf algebras in $\mathcal{C}=\HH$. The functor
 \begin{align*}
          D_1: \ & {^B\mathcal{C}}\longrightarrow {_{A^{\operatorname{cop}}}\overline{\mathcal{C}}_{\operatorname{rat}}}, \ \ (V,\delta) \mapsto (V, \lambda), \\
          \textup{with} \ &\lambda=(A \ot V \xrightarrow{\operatorname{id}_A\ot \delta} A \ot (B \ot V) \xrightarrow{a^{-1}_{A,B,V}} (A\ot B) \ot V \xrightarrow{\omega \ot \operatorname{id}_V} V)
      \end{align*}
    is a  strict  monoidal isomorphism.
\end{prop}
\begin{proof}
Let $V$ be an object in $\HH$. By Proposition \ref{prop5.5}, we have a bijection:
    \begin{align*}
        &G_V : \operatorname{Hom}_{C}(V, B \otimes V) \to \operatorname{Hom}_{C,\text{rat}}(A \otimes V, V),\\
&f \mapsto \lambda=(A \otimes V \xrightarrow{\operatorname{id}_A \otimes f} A \otimes (B \otimes V) \xrightarrow{a^{-1}_{A,B,V}} (A\ot B) \ot  V \xrightarrow{\omega\ot \operatorname{id}}  V).  
    \end{align*} 
    Note that $\operatorname{Hom}_{C,\text{rat}}(A \otimes V, V)$ is the set of all $\lambda$ in $\operatorname{Hom}_{C}(A \otimes V, V)$ such that for all $v\in V$, there is  a finite-dimensional object $F \subseteq B$, such that $\lambda(F^{\perp}\ot v)=0.$ Since $B$ is locally finite, there is 
    a natural number $n_0$  such that $F \subseteq \bigoplus_{i=0}^{n_0} B(i)$. Therefore $\lambda(A(n)\ot v)=0$ for all $n > n_0$. Thus if $(V, \lambda) \in {_A\mathcal{C}}$, where $\lambda \in\operatorname{Hom}_{C,\text{rat}}(A \otimes V, V) $, then $(V, \lambda) \in {_A\mathcal{C}_{\operatorname{rat}}}$ automatically.\par 
   We now prove that
$(V,\lambda)\in{}_A\mathcal C_{\operatorname{rat}}$ if and only if
$(V,\delta)\in{}^B\mathcal C$. The implication from the latter to
the former follows from Lemma~\ref{lem5.8} and the preceding
observation. Conversely, $(V,\lambda)\in{}_A\mathcal C$ is
equivalent to  
    $$ 
    \lambda \circ (\operatorname{id}_A \ot \lambda) =\lambda \circ (\mu_A \ot \operatorname{id}_V) \circ a_{A,A,V}^{-1}.
    $$
According to the definition of $\lambda$,
\begin{align*}
  \lambda \circ (\operatorname{id}_A \ot \lambda)=(\omega \ot \operatorname{id}_V)\circ a^{-1}_{A,B,V} \circ (\operatorname{id}_A \ot \delta) \circ (\operatorname{id}_A \ot (\omega \ot \operatorname{id}_V))\circ (\operatorname{id}_A \ot a^{-1}_{A,B,V})\circ (\operatorname{id}_A \ot (\operatorname{id}_A \ot \delta)).  
\end{align*}
Since $\delta$, $\omega$ are morphisms in $\mathcal{C}$, and the  associative isomorphism is  natural, we have
\begin{equation}\label{5.20}
\begin{aligned}
\lambda\circ(\operatorname{id}_A\ot\lambda)
&=(\omega\ot\operatorname{id}_V)
\circ a^{-1}_{A,B,V}
\circ(\operatorname{id}_A\ot(\omega\ot\operatorname{id}_{B\ot V}))\\
&\quad{}\circ(\operatorname{id}_A\ot a^{-1}_{A,B,B\ot V})
\circ\bigl(\operatorname{id}_A\ot
(\operatorname{id}_A\ot(\operatorname{id}_B\ot\delta)\circ\delta)\bigr)\\
&=(\omega\ot\operatorname{id}_V)
\circ\bigl((\operatorname{id}_A\ot
(\omega\ot\operatorname{id}_B))\ot\operatorname{id}_V\bigr)
\circ(a^{-1}_{A,A\ot B,B}\ot\operatorname{id}_V)\\
&\quad{}\circ(\operatorname{id}_A\ot a^{-1}_{A\ot B,B,V})
\circ(\operatorname{id}_A\ot a^{-1}_{A,B,B\ot V})\\
&\quad{}\circ\bigl(\operatorname{id}_A\ot
(\operatorname{id}_A\ot(\operatorname{id}_B\ot\delta)\circ\delta)\bigr).
\end{aligned}
\end{equation}
Now we write down explicit formula of $\lambda \circ (\mu_A \ot \operatorname{id}_V) \circ a_{A,A,V}^{-1}$
\begin{align*}
    \lambda \circ (\mu_A \ot \operatorname{id}_V) \circ a_{A,A,V}^{-1}= (\omega \ot \operatorname{id}_V) \circ a^{-1}_{A,B,V}\circ  (\operatorname{id}\ot \delta)\circ (\mu_A \ot \operatorname{id})\circ a^{-1}_{A,A,V}.
\end{align*}
Since $\mu_A$ is a morphism in $\mathcal{C}$, we have 
\begin{align*}
     \lambda \circ (\mu_A \ot \operatorname{id}_V) \circ a_{A,A,V}^{-1}=(\omega \ot \operatorname{id}_V) \circ ((\mu_A \ot \operatorname{id}_B)\ot \operatorname{id}_V)\circ a^{-1}_{A\ot A,B,V} \circ (\operatorname{id}_{A\ot A} \ot \delta)\circ a^{-1}_{A,A,V}.
\end{align*}
By equation (\ref{5.12}), and by naturality of associative isomorphism,
\begin{equation}\label{5.21}
\begin{aligned}
   &\lambda \circ (\mu_A \ot \operatorname{id}_V) \circ a_{A,A,V}^{-1}\\&= (\omega \ot \operatorname{id}_V) \circ ((\operatorname{id}_A \ot (\omega \ot \operatorname{id}_B))\ot \operatorname{id}_V) \circ ((\operatorname{id}_A \ot a^{-1}_{A,B,B})\circ a_{A,A,B\ot B} \circ (\operatorname{id}_{A\ot A} \ot \Delta_B) \ot \operatorname{id}_V)\\ & \circ a^{-1}_{A \ot A, B,V} \circ (\operatorname{id}_{A\ot A}\ot \delta)\circ a^{-1}_{A,A,V}\\
   &=(\omega \ot \operatorname{id}_V) \circ ((\operatorname{id}_A \ot (\omega \ot \operatorname{id}_B))\ot \operatorname{id}_V)  \circ ((\operatorname{id}_A \ot a^{-1}_{A,B,B})\circ a_{A,A,B\ot B} \ot \operatorname{id}_V)\\ & \circ a^{-1}_{A\ot A, B\ot B, V} \circ a^{-1}_{A,A,(B\ot B)\ot V} \circ (\operatorname{id}_A \ot (\operatorname{id}_A \ot a_{B,B,V}^{-1})) \circ (\operatorname{id}_{A} \ot (\operatorname{id}_A \ot (a_{B,B,V}\circ ((\Delta_B \ot \operatorname{id}_V) \circ  \delta)))).\\
\end{aligned}
\end{equation}

Since we have assumed $(V,\lambda) \in {_A\mathcal{C}}$,
for all $a,a' \in A$, $v \in V$, we have
$ \lambda \circ (\operatorname{id}_A \ot \lambda)(a\ot (a' \ot v))=\lambda \circ (\mu_A \ot \operatorname{id}_V) \circ a_{A,A,V}^{-1}(a\ot (a' \ot v)).$
  By non-degeneracy of $\omega$ and Mac-Lane's coherence theorem, this is equivalent to
$$ (\operatorname{id}\ot \delta) \circ \delta(v)= a_{B,B,V} \circ (\Delta_B \ot \operatorname{id}_V)\circ \delta(v).$$
Furthermore, we have $v=(\varepsilon_B \ot \operatorname{id}_V)\circ \delta(v)$ by equation (\ref{5.13}).
Therefore 
 we deduce that $(V, \lambda) \in {_A\mathcal{C}_{\operatorname{rat}}}$ if and only if $(V,\delta)\in {^B\mathcal{C}}$.
\par 
Now let $(V,\delta), (V',\delta') \in {^B\mathcal{C}}$, and $(V, \lambda)$, $(V',\lambda' )$ be the corresponding modules in ${_A\mathcal{C}_{\operatorname{rat}}}$, where $\lambda=G_V(\delta)$, $\lambda'=G_{V'}(\delta')$. A morphism
$f\in\operatorname{Hom}_{\mathcal C}(V,V')$ is a
$B$-comodule morphism if and only if it is an $A$-module
morphism in $\mathcal C$.\par 
 Note that  $(V, \lambda)\in {_A\mathcal{C}_{\operatorname{rat}}}$ if and only if $(V,\lambda) \in {_{A^{\operatorname{cop}}}\overline{\mathcal{C}}_{\operatorname{rat}}} $.
 Since the category ${_{A^{\operatorname{cop}}}\overline{\mathcal{C}}_{\operatorname{rat}}}$ is a monoidal subcategory of ${_{A^{\operatorname{cop}}}\overline{\mathcal{C}}}$ by Lemma \ref{lem5.2} and  $D_1$ is a  strict monoidal functor by Lemma \ref{lem5.8}, we deduce that $D_1$ is a strict monoidal equivalence. Since $G_V$ is bijective, it is easy to see that 
  the inverse functor is given by 
  \begin{align*}
      D_1': {_{A^{\operatorname{cop}}}\overline{\mathcal{C}}_{\operatorname{rat}}}\rightarrow {^B\mathcal{C}}, \ \ (V, \lambda') \mapsto (V,\delta'),  \end{align*} 
     $\text{where} 
       \ \delta' \ \text{is determined by}\  G_V^{-1}(\lambda').$
Thus $D_1$ is a strict  monoidal isomorphism.
\end{proof}
\subsection{Monoidal equivalence between Yetter-Drinfeld module categories related by a dual pair}
Let $(A,B,\omega)$ be a dual pair of locally finite
$\mathbb N_0$-graded Hopf algebras with bijective antipodes in
$\mathcal C=\HH$. In this subsection, we construct a braided
monoidal equivalence
\[
(\Omega,\beta):\BBCR\longrightarrow\AACR.
\]
\par 
By Remark \ref{rmk2.3}, we have a dual pair of locally finite $\mathbb{N}_0$-graded Hopf algebras in $\overline{\mathcal{C}}$:
\begin{align*}
 \omega^{+\operatorname{cop}}&:B^{\operatorname{cop}}\ot A^{\operatorname{cop}}\longrightarrow \mathbbm{k}, \\
    \omega^{+\operatorname{cop}}&=\omega \circ c_{B,A}\circ (\mathcal{S}_B \ot \mathcal{S}_A)\circ (\operatorname{id}_B \ot \mathcal{S}^{-1}_A)\\
    &=\omega \circ c_{B,A}\circ (\mathcal{S}_B \ot \operatorname{id}_A)\\
    &=\omega \circ c_{B,A}\circ (\operatorname{id}_B \ot \mathcal{S}_A).
\end{align*} 
Recall that by Proposition \ref{prop5.9}, we now have two monoidal isomorphisms:
 \begin{align}
&D_1: {^B\mathcal{C}}\rightarrow {_{A^{\operatorname{cop}}}\overline{\mathcal{C}}_{\operatorname{rat}}}, \ D_2: {^{A^{\operatorname{cop}}}\overline{\mathcal{C}}}\rightarrow {_B\mathcal{C}_{\operatorname{rat}}}, \\
\label{5.23}
&D_1(V,\delta)=(V,\overline{\lambda}), \ \overline{\lambda}=(A \ot V \xrightarrow{\operatorname{id}_A\ot \delta} A \ot (B \ot V) \xrightarrow{a^{-1}_{A,B,V}} (A\ot B) \ot V \xrightarrow{\omega \ot \operatorname{id}_V} V), \\ \label{5.24}
&D_2(V,\overline{\delta})=(V,{\lambda}), \ {\lambda}=(B \ot V \xrightarrow{\operatorname{id}_B\ot \overline{\delta}} B \ot (A \ot V) \xrightarrow{a^{-1}_{B,A,V}} (B\ot A) \ot V \xrightarrow{\omega^{+\operatorname{cop}} \ot \operatorname{id}_V} V).
 \end{align}
 
Our first objective is to prove the following braided monoidal equivalence:
\begin{thm}\label{thm5.10}
    The functor 
    \begin{equation}
        \Gamma: \overline{\BBCR} \longrightarrow \AACOPCR, \ (V,\lambda,\delta)\mapsto (V,\overline{\lambda},\overline{\delta}),
    \end{equation}
where $\overline{\lambda}$ and $\overline{\delta}$ are defined by
\eqref{5.23} and \eqref{5.24}, respectively, and
$\Gamma(f)=f$ for every morphism $f$, is a braided monoidal
equivalence.
\end{thm}
The following lemma, which holds in the general setting of a braided monoidal category 
$\mathcal{C}$
, is used to establish this theorem. Without loss of generality, we assume 
$\mathcal{C}$ is strict.
Let $A$, $B$ be Hopf algebras in $\mathcal{C}$, and $p: A\ot B \rightarrow \bfo $ is a Hopf pairing, now we have two braided monoidal categories $ \overline{\BBCR}$ and $\AACOPCR$, with braiding $\overline{c}^{\mathcal{YD}(\mathcal{C})}$ and $c^{\mathcal{YD}(\overline{\mathcal{C}})}$ respectively.
\begin{lemma}\label{lem5.11}
    With above assumptions, let $(X,\overline{\delta}_X) \in {^{A^{\operatorname{cop}}}\overline{\mathcal{C}}}$, $(V,\delta) \in {^B\mathcal{C}}$, and define 
    $$ (X,\lambda_X)=D_2(X,\overline{\delta}_X), \ (V,\overline{\lambda})=D_1(V,\delta).$$
    where $D_1,D_2$ is given by Lemma \ref{lem5.8}, then
    \begin{equation}
        \overline{c}^{\mathcal{YD}(\mathcal{C})}_{(X,\lambda_X),(V,\delta)}=c^{\mathcal{YD}(\overline{\mathcal{C}})}_{(X,\overline{\delta}_X),(V,\overline{\lambda})}.
    \end{equation}
\end{lemma}
\begin{proof}
    By definition of $c^{\mathcal{YD}}$, properties of Hopf pairing and naturality of braiding, we have
   \begin{align*}
       \overline{c}^{\mathcal{YD}(\mathcal{C})}_{(X,\lambda_X),(V,\delta)}
       &=\overline{c}_{X,V}\circ ((p^{+\operatorname{cop}}\ot \operatorname{id}_X)\circ (\operatorname{id}_B \ot \overline{\delta}_X)\ot \operatorname{id}_V)\circ (\mathcal{S}_B^{-1}\ot \operatorname{id}_X \ot \operatorname{id}_V) 
       \circ (\overline{c}_{X,B}\ot \operatorname{id}_V) \circ (\operatorname{id}_X\ot \delta)\\
       &= \overline{c}_{X,V}\circ (p\circ c_{B,A} \ot \operatorname{id}_X \ot \operatorname{id}_V) \circ (\operatorname{id}_B \ot \overline{\delta}_X \ot \operatorname{id}_V) \circ (\overline{c}_{X,B}\ot \operatorname{id}_V)\circ (\operatorname{id}_X\ot \delta)\\
       &=(p \ot \overline{c}_{X,V})\circ (c_{B,A}\ot \operatorname{id}_X)\circ (\overline{c}_{A\ot X,B} \ot \operatorname{id}_V) \circ (\operatorname{id}_A \ot \operatorname{id}_X \ot \delta)\circ (\overline{\delta}_X \ot \operatorname{id}_V)\\
       &=(p \ot \overline{c}_{X,V})\circ (\operatorname{id}_A \ot \overline{c}_{X,B}\ot \operatorname{id}_V)\circ (\operatorname{id}_A \ot \operatorname{id}_X \ot \delta)\circ (\overline{\delta}_X \ot \operatorname{id}_V)\\
       &=(p\ot \overline{c}_{X,B\ot V})\circ (\operatorname{id}_A\ot \operatorname{id}_X\ot \delta)\circ (\overline{\delta}_X \ot  \operatorname{id}_V)
       \\ &=((p\ot \operatorname{id}_V )\circ (\operatorname{id}_A \ot \delta )  \ot \operatorname{id}_X)\circ (\operatorname{id}_A \ot \overline{c}_{X,V}) \circ (\overline{\delta}_X \ot \operatorname{id}_V)=c^{\mathcal{YD}(\overline{\mathcal{C}})}_{(X,\overline{\delta}_X),(V,\overline{\lambda})}.
   \end{align*} 
   Here the second equality uses (\ref{3.0}) and (\ref{3.16}), the third and the sixth equalities use the naturality of the braiding, the fourth and the fifth equalities use the braiding axiom.
\end{proof}
Now we return to the setting of coquasi-Hopf algebras. We want to restrict the Yetter-Drinfeld module criterion to rational modules. Let $B$ be a  locally finite $\mathbb{N}_0$-graded Hopf algebra in $\mathcal{C}=\HH$ with bijective antipode.
Let $(V, \lambda)\in {_B\mathcal{C}_{\operatorname{rat}}}$, and $(V,\delta)\in {^B\mathcal{C}}$. 
\begin{lemma}\label{lem3.12}
    Let $X \in {_B\mathcal{C}}$, and assume that there is an index set $I$, a family $X_i, i \in I$, of objects in $_B \mathcal{C}$, and morphisms $\pi_i:X \rightarrow X_i$ in $_B\mathcal{C}$ for all $i \in I$ with $\bigcap_{i \in I}\operatorname{ker}(\pi_i)=0$. Then if $c^{\mathcal{YD}}_{V,X_i}$ is a morphism in $_B\mathcal{C}$ for all $i\in I$, then $c^{\mathcal{YD}}_{V,X}$ is a morphism in $_B\mathcal{C}$.
\end{lemma}
\begin{proof}
    The proof of this lemma is parallel to [\citealp{rootsys}, Lemma 12.2.5]. We omit here for simplicity.
\end{proof}
\begin{lemma}\label{lem5.13}
The following are equivalent.\par     
$ {(1)}$ For all $(X, \lambda_X)\in {_B\mathcal{C}_{\operatorname{rat}}}$,

$c^{\mathcal{YD}}_{V,X}$ is a morphism in $_B\mathcal{C}_{\operatorname{rat}}$. \par 
$ {(2)}$ The object $(V,\lambda,\delta)$ belongs to $\BBCR$.
\end{lemma}
\begin{proof}
    Suppose for all $(X, \lambda_X)\in {_B\mathcal{C}_{\operatorname{rat}}}$, the map $c^{\mathcal{YD}}_{V,X}$ is a morphism in $_B \mathcal{C}_{\operatorname{rat}}$. Then by Remark \ref{rmk2.13}, we only need to prove that 
    $$ c^{\mathcal{YD}}_{V,B}: V \ot B \rightarrow B \ot V$$ is a morphism in $_B\mathcal{C}$. Now for all $n\geq 0$, let $X_n=B/{\bigoplus_{i \geq n}B(i)}$. The quotient map  $\pi_n: B \rightarrow X_n$ is a morphism in $_B\mathcal{C}$. We have 
    $$ \bigcap_{n \geq 0}\operatorname{ker}(\pi_n)=0, \ \ \mu_B(B(m)\ot X_n)=0 $$
    for all $m \geq n$. Since $X_n$ is a rational $B$-module in $\mathcal{C}$, $c^{\mathcal{YD}}_{V,X_n}$ is a morphism in $_B\mathcal{C}$. By Lemma \ref{lem3.12}, $c^{\mathcal{YD}}_{V,B}$ is a morphism in $_B\mathcal{C}$. Hence $(V,\lambda,\delta) \in \BBCR$.
    \par 
    Suppose $(V,\lambda,\delta) \in \BBCR$, we deduce that $c^{\mathcal{YD}}_{V,X}$ is a morphism in $_B\mathcal{C}$ for any $(X, \lambda_X)\in {_B\mathcal{C}_{\operatorname{rat}}}$ by Remark \ref{rmk2.13}. Furthermore,  ${_B\mathcal{C}_{\operatorname{rat}}}$ is a monoidal full subcategory of $_B\mathcal{C}$, and $V\ot X, X\ot V \in {_B\mathcal{C}_{\operatorname{rat}}}$. Thus $c^{\mathcal{YD}}_{V,X}$ is a morphism in $_B\mathcal{C}_{\operatorname{rat}}$.
\end{proof}
\begin{prop}\label{prop5.14}
    The following are equivalent.\par 
     {(1)} The object $(V,\lambda,\delta)$ belongs to $\BBCR$.\par 
        {(2)} The object $(V,\overline{\lambda},\overline{\delta})$ belongs to $\AACOPCR$.\par 
        {(3)} For all $(X, \lambda_X)\in {_B\mathcal{C}_{\operatorname{rat}}}$, the following map is a morphism in $_B\mathcal{C}_{\operatorname{rat}}$
       $$ \overline{c}^{\mathcal{YD}(\mathcal{C})}_{(X,\lambda_X),(V,\delta)}: (X, \lambda_X)\ot (V,\lambda)\rightarrow (V, \lambda)\ot (X, \lambda_X).$$\par 
        {(4)} For all $(X,\overline{\delta}_X) \in {^{A^{\operatorname{cop}}}\overline{\mathcal{C}}}$, the following map is a morphism in $^{A^{\operatorname{cop}}}\overline{\mathcal{C}}$.
       $$ c^{\mathcal{YD}(\overline{\mathcal{C}})}_{(X,\overline{\delta}_X),(V,\overline{\lambda})}: (X,\overline{\delta}_X)\ot (V,\overline{\delta})\rightarrow (V,\overline{\delta})\ot (X,\overline{\delta}_X).$$ 
\end{prop}

\begin{proof}
    (1) $\Leftrightarrow$ (3): The condition (1), by Lemma \ref{lem5.13}, is equivalent to the condition that for all $(X, \lambda_X)\in {_B\mathcal{C}_{\operatorname{rat}}}$, 
    $c^{\mathcal{YD}}_{(V,\delta),(X, \lambda_X)}$ is a morphism in $_B\mathcal{C}_{\operatorname{rat}}$. Since we have assumed $B$ has bijective antipode, 
    $c^{\mathcal{YD}}_{(V,\delta),(X, \lambda_X)}$ is an isomorphism with inverse 
    $\overline{c}^{\mathcal{YD}}_{(X,\lambda_X),(V,\delta)}.$ Hence (1) is equivalent to (3).\par 
    (2) $\Leftrightarrow$ (4): This follows from Remark \ref{rmk2.13}.\par 
    (3) $\Leftrightarrow$ (4): By Lemma \ref{lem5.11}, 
   $$ \overline{c}^{\mathcal{YD}(\mathcal{C})}_{(X,\lambda_X),(V,\delta)}= c^{\mathcal{YD}(\overline{\mathcal{C}})}_{(X,\overline{\delta}_X),(V,\overline{\lambda})}$$ as a morphism in $\mathcal{C}$. The equivalence of  {(3)} and  {(4)} follows because
\[
D_2:{}^{A^{\operatorname{cop}}}\overline{\mathcal C}
\longrightarrow{}_B\mathcal C_{\operatorname{rat}}
\]
is a monoidal isomorphism.
\end{proof}
\begin{proof}[Proof of Theorem \ref{thm5.10}]
    Let $V \in \mathcal{C}$, and 
    \begin{align*}
        &S_1=\lbrace (\lambda,\delta)\mid (V, \lambda)\in {_B\mathcal{C}_{\operatorname{rat}}}, (V,\delta)\in {^B\mathcal{C}}\rbrace \\
        &S_2=\lbrace (\overline{\lambda},\overline{\delta})\mid (V, \overline{\lambda})\in {_{A^{\operatorname{cop}}}\overline{\mathcal{C}}_{\operatorname{rat}}}, (V,\overline{\delta})\in {^{A^{\operatorname{cop}}}\overline{\mathcal{C}}}\rbrace.
    \end{align*}
    We define a map $\phi: S_1 \rightarrow S_2$, by $(\lambda,\delta) \mapsto (\overline{\lambda},\overline{\delta}) $, where $\overline{\lambda},\overline{\delta}$ are given by
    $$ D_1(V,\delta)=(V,\overline{\lambda}), \ \ \ D_2(V,\overline{\delta})=(V,\lambda).$$
Thus $\phi$ is bijective because $D_1$ and $D_2$ are monoidal
isomorphisms. If $(V,\lambda,\delta)\in\BBCR$, then
$(V,\overline\lambda,\overline\delta)\in\AACOPCR$ by
Proposition~\ref{prop5.14}. Hence $\Gamma$ is a monoidal
equivalence.
It is strict because $D_1$ and $D_2$ are strict. To prove that
$\Gamma$ is braided, let
$X=(X,\lambda_X,\delta_X)\in\BBCR$. The morphism
$\overline c^{\mathcal{YD}(\mathcal C)}_{(X,\lambda_X),(V,\delta)}$
is the braiding in $\overline{\BBCR}$, whereas
$c^{\mathcal{YD}(\overline{\mathcal C})}
_{(X,\overline\delta_X),(V,\overline\lambda)}$
is the corresponding braiding in $\AACOPCR$. By Lemma \ref{lem5.11}, the functor $\Gamma$ is braided. 
Hence $\Gamma$ is a strict braided monoidal equivalence.

\end{proof}
Our next goal is to build a monoidal equivalence between $\AACOPCR$ and $\overline{\AAC}_{\operatorname{rat}}$. It is standard that 
\begin{align*}
   & _A\mathcal{C} \rightarrow {\mathcal{C}_A}, \ (V, \lambda) \mapsto (V, \lambda_+), \ \text{where}\
\lambda_+=\lambda \circ c_{V,A}\circ(\operatorname{id}_V \ot \mathcal{S}_A), \\
&\mathcal{C}_A \rightarrow {_A\mathcal{C}}, \ (V, \lambda) \mapsto (V, \lambda_-), \ \text{where}\
\lambda_-=\lambda\circ \overline{c}_{A,V}\circ( \mathcal{S}_A^{-1}\ot \operatorname{id}_V),
\end{align*}
are inverse isomorphisms of monoidal categories.
Furthermore, suppose $(V, \lambda ) \in {_A\mathcal{C}_{\operatorname{rat}}}$, for $v \in V$, $c_{V,A}(v\ot a)$ is a finite sum of elements in $A \ot V$. By $\mathcal{S}_A$ is a $\mathbb{N}_0$-graded map, we have $(V,\lambda_+) \in {\mathcal{C}_A}_{,\operatorname{rat}}$. Similarly, suppose $(V, \lambda ) \in {\mathcal{C}_A}_{,\operatorname{rat}}$, we can deduce that $(V, \lambda_-)\in {_A\mathcal{C}_{\operatorname{rat}}}$.
\par Now we restrict the monoidal isomorphisms in [\citealp{rootsys}, Theorem 3.4.15] to the rational Yetter-Drinfeld modules, we have a braided monoidal equivalence:
\begin{align}
(F_{rl},\alpha): {\AARC}_{,\operatorname{rat}}\rightarrow \AACR, \ &(V,\lambda,\delta)\mapsto (V, \lambda_-,(\mathcal{S}_A\ot \operatorname{id}_V)\circ c_{V,A} \circ \delta),\\
& \alpha_{X,Y}=c_{Y,X}^{\AARC} \circ \overline{c}_{X,Y}, 
\end{align}
where morphisms $f$ are mapped onto $f$. From [\citealp{rootsys}, Theorem 3.4.16], we have a braided strict monoidal equivalence:
\begin{align}
    F_{lr}: \AACOPCR \rightarrow \overline{\AARC_{,\operatorname{rat}}}, \ (V,\lambda, \delta)\mapsto (V,\lambda_+,c_{A,V}\circ \delta),
\end{align}
where morphisms $f$ are mapped onto $f$.
\par 
Our main result in this section establishes the braided monoidal equivalence between the categories of rational modules over dual pairs. The construction of the functor $\Omega$ relies on the interplay between the equivalences established in the previous subsections, as illustrated in the following diagram:

\[
\begin{tikzcd}[row sep=large, column sep=huge]
    \overline{\BBCR}\arrow[r, "\Gamma", "\cong"'] \arrow[d, "\Omega"', dashed]
    & \AACOPCR \arrow[d, "F_{lr}"] \\
\overline{\AACR}
    & \overline{\AARC_{,\operatorname{rat}}}\arrow[l, "\overline{F_{rl}}"]
\end{tikzcd}
\]
\begin{thm}\label{thm5.15}
    The following  functor is a braided monoidal equivalence:
\begin{align*}
    &(\Omega, \beta): \BBCR \rightarrow \AACR, \ \text{where} \ (V,\lambda_B, \delta_B) \mapsto (V, \lambda_A, \delta_A), \text{with} \\ 
    &\lambda_A=(A\ot V \xrightarrow{\operatorname{id}\ot \delta_B} A\ot (B\ot V)\xrightarrow{a^{-1}_{A,B,V}}(A\ot B) \ot V\xrightarrow{\omega\ot \operatorname{id}}V), \\
        &\delta_A=(V\xrightarrow{\delta'}A\ot V\xrightarrow{\mathcal{S}_A^2\ot \operatorname{id}_V} A\ot V \xrightarrow{c^2_{A,V}} A\ot V), \ \text{ where } \ \delta' \ \text{is defined by}\\
 &\lambda_B=(B\ot V \xrightarrow{\operatorname{id}_B \ot \delta' }B \ot (A\ot V) \xrightarrow{a^{-1}_{B,A,V}}(B\ot A) \ot V \xrightarrow{ \omega^{+} \ot \operatorname{id}_V} V),
\end{align*}
On morphisms, $\Omega(f)=f$. Its tensor structure is given by
\[
\beta_{X,Y}
=
c_{Y,X}^{\BBC}\circ\overline{c}_{X,Y}.
\]
\end{thm}

\begin{proof}
As depicted in the diagram above, define $\Omega$ to be the
composition
\[
(\Omega,\overline\beta)
:=
F_{rl}\circ F_{lr}\circ\Gamma:
\overline{\BBCR}\longrightarrow\overline{\AACR}.
\]
We now compute the explicit structure of this composition. Let
$(V,\lambda_B,\delta_B)\in\BBCR$. Then
    \begin{align*}
F_{rl}F_{lr}\Gamma(V,\lambda_B,\delta_B)=&F_{rl}F_{lr}(V,\overline{\lambda_A},\overline{\delta_A})\\
&=F_{rl}(V,\overline{\lambda_A}_+,c_{A,V}\circ \overline{\delta_A})\\
&=(V,(\overline{\lambda_A}_{+})_-,(\mathcal{S}_A \ot \operatorname{id}_V) \circ c_{V,A}\circ c_{A,V}\circ \overline{\delta_A} ).
    \end{align*}
Note that $(\overline{\lambda_A}_{+})_-=\overline{\lambda_A}$, while 
$$\overline{\lambda_A}=(A \ot V \xrightarrow{\operatorname{id}_A\ot \delta_B} A \ot (B \ot V) \xrightarrow{a^{-1}_{A,B,V}} (A\ot B) \ot V \xrightarrow{\omega \ot \operatorname{id}_V} V).$$
This coincides with our definition of $\lambda_A$.\par 
Now we need to show $(\mathcal{S}_A \ot \operatorname{id}_V) \circ c_{V,A}\circ c_{A,V}\circ \overline{\delta_A}$ equals $\delta_A$. Since $\mathcal{S}_A$ and $c_{A,V}$ are isomorphisms, we only need to prove:
$$ \delta'=(\mathcal{S}_A^{-2}\ot \operatorname{id}_V) \circ \overline{c}^2_{A,V} \circ (\mathcal{S}_A \ot \operatorname{id}_V) \circ   c^2_{A,V} \circ \overline{\delta_A}=(\mathcal{S}_A^{-1}\ot \operatorname{id}_V)\circ\overline{\delta_A}.$$
It suffices to show 
$$ (\omega^+ \ot \operatorname{id}_V) \circ a^{-1}_{B,A,V} \circ (\operatorname{id}_B\ot (\mathcal{S}_A^{-1}\ot \operatorname{id}_V)\circ \overline{\delta_A})=\lambda_B.$$
Since $a$ is a natural isomorphism, and $\omega^{+\operatorname{cop}}=\omega^+\circ (\operatorname{id}_B\ot \mathcal{S}_A^{-1})$, we have 
\begin{align*}
&(\omega^+\ot\operatorname{id}_V)
\circ a^{-1}_{B,A,V}
\circ\bigl(\operatorname{id}_B\ot
(\mathcal{S}_A^{-1}\ot\operatorname{id}_V)
\circ\overline{\delta_A}\bigr)\\
&\quad=(\omega^+\ot\operatorname{id}_V)
\circ\bigl((\operatorname{id}_B\ot\mathcal{S}_A^{-1})
\ot\operatorname{id}_V\bigr)
\circ a^{-1}_{B,A,V}
\circ(\operatorname{id}_B\ot\overline{\delta_A})\\
&\quad=(\omega^{+\operatorname{cop}}\ot\operatorname{id}_V)
\circ a^{-1}_{B,A,V}
\circ(\operatorname{id}_B\ot\overline{\delta_A})\\
&\quad=\lambda_B.
\end{align*}
The last equality follows from the relation between
$\overline{\delta_A}$ and $\lambda_B$ under the functor $D_2$. Thus $\Omega$ induces a monoidal equivalence $\Omega: \overline{\BBCR} \rightarrow \overline{\AACR}$.\par 
For the monoidal structure of $ \Omega$, note that only the functor $(F_{rl}, \alpha)$ is non-strict. Since $F_{lr}$ and $\Gamma$ are braided, we have
$$ \overline{\beta}_{X,Y}=\overline{c}^{\overline{\AARC}}_{F_{lr}\Gamma(Y),F_{lr}\Gamma(X)}\circ \overline{c}_{F_{lr}\Gamma(X),F_{lr}\Gamma(Y)}=\overline{c}^{\overline{\BBC}}_{Y,X}\circ \overline{c}_{X,Y}.$$ Thus $(\Omega, \overline{\beta}): \overline{\BBCR} \rightarrow \overline{\AACR}$ is a braided monoidal equivalence. After passing to the reverse braidings, it follows immediately that
\[
(\Omega,\beta):\BBCR\longrightarrow\AACR,
\qquad
\beta_{X,Y}=c_{Y,X}^{\BBC}\circ\overline c_{X,Y},
\]
is a braided monoidal equivalence.

\end{proof}

\subsection{Dual pairs of Nichols algebras over coquasi-Hopf algebras}
We now specialize to the  setting of Nichols algebras.

 Let $\mathcal{C}$ be an abelian braided monoidal category. We first recall the definition of  Nichols algebras $\bB(V)$ in $\mathcal{C}$, where $V \in \mathcal{C}$ is an arbitrary object. Then we will show that there is a non-degenerate Hopf pairing between $\mathcal{B}(V^*)$ and $\bB(V)$ when $V$ has a left dual $V^*$.\par 
We denote by $T(V)$  the tensor algebra in $\mathcal{C}$ generated freely by $V$. Here we denote that
$$ V^{\otimes n }:=\left( \cdots\left( \left( V \otimes V\right) \otimes V\right) \cdots \otimes V\right). $$
Then
$T(V)$ is isomorphic to $\bigoplus_{n \geq 0} V^{\otimes n}$ as an object. The tensor algebra  $T(V)$ is naturally a graded Hopf algebra in $\mathcal{C}$.\par 
Let $i_n:V^{\otimes n} \rightarrow T(V)$ be the  canonical injection. We denote the total Woronowicz symmetriser by 
$$ \operatorname{Wor}(c)=\bigoplus_{n \geq 1} i_n [n]!_c,
$$ 
where $[n]!_c \in \operatorname{End}(V^{\otimes n})$ is the braided symmertriser. For the explicit definition of $[n]!_c$, one may refer to [\citealp{nicholscatdef}, Section 5.6].
\begin{definition}
Let $V \in \mathcal{C}$. The Nichols algebra of $V$ is defined  to be the quotient Hopf algebra
    $$ \mathcal{B}(V):=T(V)/\operatorname{ker}\operatorname{Wor}(c).$$
\end{definition}
To facilitate further analysis, we assume $\mathcal{C}$ is a $\mathbbm{k}$-linear braided monoidal abelian category. In this setting, the following equivalent definition of the Nichols algebra is more convenient.
\begin{definition}\textup{[\citealp{defofnichols}, Definition 2.4]}\label{def3.18}
    For a given object $V \in \mathcal{C}$, the Nichols
algebra $\mathcal{B}(V)$ is the unique Hopf algebra in $\mathcal{C}$ that satisfies the following conditions:\par 
\text{$(1)$} The Hopf algebra $\mathcal{B}(V)$ is graded by the non-negative integers.\par
\text{$(2)$} The zeroth component of the grading satisfies $\mathcal{B}(V)_0=\mathbbm{k}$.\par
\text{$(3)$} The first component of the grading satisfies $\mathcal{B}(V)_1=V$ , and $\mathcal{B}(V)$ is generated
by $V$ as an algebra in $\mathcal{C}$.\par
\text{$(4)$} The subobject of primitive elements of $\mathcal{B}(V)$ is $V$.\par
\end{definition}

The following lemma is standard, but plays an important role in our proof.
\begin{lemma}\label{NicholstoNichols}
    Let $H$, $H'$ be two coquasi-Hopf algebras with bijective antipode,  and $F: \HH \xrightarrow{\sim} \HHY$ as braided monoidal categories. For any $V \in \HH$, we have 
    $$  \Psi: T(F(V)) \cong F(T(V))  $$
 as Hopf algebras in $\HHY$, which is identity on $F(V)$. Moreover, $\Psi$ induces an isomorphism of quotient Hopf algebras
 $$  \Psi: \bB(F(V))\cong F(\bB(V)).$$
\end{lemma}
\begin{proof}
   Let $V\in\mathcal C$, and let
\[
J_{V,V}:F(V\ot V)\xrightarrow{\sim}F(V)\ot F(V)
\]
be the tensor structure of $F$.  Using  $J$ repeatedly, we have
$$F(V^{\otimes n}) \cong F(V)^{\otimes n}$$ for each $n\in \mathbb{N}$.
Therefore we have $F(\bigoplus\limits_{n \in \mathbb{N}}V^{\otimes n})\cong \bigoplus\limits_{n \in \mathbb{N}}F(V^{\otimes n})\cong \bigoplus\limits_{n \in \mathbb{N}}F(V)^{\otimes n}$. That is $F(T(V)) \cong T(F(V))$.\par
   Recall that a Nichols algebra in $\mathcal C$ is of the form
$T(V)/I$, where $I$ is the unique maximal graded Hopf ideal of
$T(V)$ generated by homogeneous elements of degree at least $2$.
The isomorphism $F(T(V))\cong T(F(V))$ gives  
$$F(T(V)/I) \cong T(F(V))/F(I).$$
Here $F(I)$ is a homogeneous Hopf ideal of degree greater than or equal to $2$ of $T(F(V))$. Note that the Nichols algebra generated by $F(V)$ must be of the form $T(F(V))/J$, where $J$ is the unique maximal homogeneous graded Hopf ideal of $T(F(V)) \in \mathcal{D}$ with degree greater than or equal to $2$.  Hence   $ F(I) \subset J$ and $T(F(V))/F(I) \twoheadrightarrow T(F(V))/J$.
On the other hand, $F^{-1}: \mathcal{D} \longrightarrow \mathcal{C}$ is the  inverse of $F$, which is an exact monoidal functor. Hence
$$F^{-1}(T(F(V))/J) \cong T(F^{-1}(F(V)))/F^{-1}(J)\cong T(V)/F^{-1}(J) \supseteq T(V)/I.$$
Thus $F^{-1}(J)\subseteq I$. Together with $F(I)\subseteq J$,
this gives $F(I)=J$, and hence
$F(T(V)/I)\cong T(F(V))/J$.
\end{proof}
We are going to investigate the Hopf pairing between $\mathcal{B}(V^*)$ and $\bB(V)$ when $V$ has a left dual $V^*$.
  There is a unique  Hopf pairing \( \omega : T(V^*) \otimes T(V) \to \bfo \). It is shown in [\citealp{nicholscatdef}] that \( \omega \) is given by
\[ 
\omega\mid_{V^* \ot V}= \operatorname{ev}_V. 
\]
and 
\[ 
\omega_n := \omega |_{V^{*\otimes n} \otimes V^{\otimes n}} = \operatorname{ev}_{V^{\otimes n}} \circ (\operatorname{id}_{V^{*\otimes n}} \otimes [n]!_{c}) = \operatorname{ev}_{V^{\otimes n}} \circ ([n]!_{c^{*}} \otimes \operatorname{id}_{V^{\otimes n}}).
\]
The Hopf  pairing between \( V^{\otimes n} \) and \( V^{*\otimes m} \) is 0 unless \( n = m \). 
Here $c^*$ denotes the braided symmetriser of $V^*$. \par

    The Hopf pairing between $T(V)$ and $T(V^*)$ induces a  Hopf pairing  between $\mathcal{B}(V^*)$ and $\bB(V)$. 
 Note that the Hopf pairing $\omega$ vanishes on the kernels of the symmetrisers, and for all positive integer $n$, we have
 $$ \omega_n({\operatorname{ker}[n]!_c^*, \ V^{\ot n}})=0, \
  \omega_n({V^{*\ot n},\ \operatorname{ker}[n]!_c })=0.$$
\par Now we return to the case of a coquasi-Hopf algebra $H$ with bijective antipode. Let $V$ be a finite-dimensional object in $\HH$. Consequently, the Hopf pairing on the tensor algebras descends to a well-defined Hopf pairing
$$ \omega: \mathcal{B}(V^*) \ot \bB(V) \rightarrow \mathbbm{k}.$$
In particular, 
$$ \omega(V^{*\ot m}, V^{\ot n})=0, \ \text{for all} \ m \neq n,$$
and both $\bB(V)$ and $\bB(V^*)$ are locally finite $\mathbb{N}_0$-graded Hopf algebras in $\HH$. \par 
We consider $T(V)^{\perp}$, by expression of $\omega$ it equals $\operatorname{ker}\operatorname{Wor}(c^*)=\bigoplus_{n \geq 1} i_n \operatorname{ker}([n]!_c^*)$. Similarly, $T(V^*)^{\perp}=\operatorname{ker}\operatorname{Wor}(c)$. Since $\mathcal{B}(V)=T(V)/\operatorname{ker}\operatorname{Wor}(c)$ and $\mathcal{B}(V^*)=T(V^*)/\operatorname{ker}\operatorname{Wor}(c^*)$, we deduce that $\omega: \mathcal{B}(V^*) \ot \mathcal{B}(V)\rightarrow \mathbbm{k}$ is a dual pair of locally finite $\mathbb{N}_0$-graded Hopf algebras in $\HH$. 

\begin{cor}\label{cor3.8}
   For any finite-dimensional $V\in \HH$, there exists a braided monoidal equivalence:
\begin{equation}
    (\Omega_V, \beta_V): \BVCC_{\operatorname{rat}} \longrightarrow \BVdCC_{\operatorname{rat}}.
\end{equation}
\end{cor}
\begin{proof}
  This follows immediately from the existence of a dual pair $\omega:
    \BVd \ot \BV \rightarrow \mathbbm{k}$ and Theorem \ref{thm5.15}
\end{proof}
Let $A$, $B$ be  $\mathbb{N}_0$-graded locally finite Hopf algebra in $\mathcal{C}=\HH$ with bijective antipodes and $\omega: A \ot B\rightarrow \mathbbm{k}$  a dual pair. Recall the braided monoidal equivalence in Theorem \ref{thm5.15}:
$$\Omega: \BBCR \longrightarrow \AACR.$$
\begin{cor}
The Nichols algebra $\mathcal{B}(V)$ of any $V \in \BBC$ is rational if
$V$ is.    
\end{cor}
\begin{proof}
If $V$ is rational, then $V^{\ot n}$ is rational in $\BBC$ for
every $n\geq1$. Since $\BBC_{\operatorname{rat}}$ is closed under
direct sums and quotients by Lemma~\ref{lem5.2}, both
\[
T(V)=\bigoplus_{n\geq0}V^{\ot n}
\quad\text{and}\quad
\bB(V)
\]
are rational.
\end{proof}
Now for  $Q \in \BBCR$,  $\mathcal{B}(Q)$  is an $\mathbb{N}_0$-graded Hopf algebra  in $\BBCR$ as well. The following result is straightforward.

\begin{cor}\label{cor4.5}
For every $Q\in\BBCR$, there is an isomorphism
\[
\Omega(\mathcal B(Q))
\cong
\mathcal B(\Omega(Q))
\]
of $\mathbb N_0$-graded Hopf algebras in
$\AAC_{\operatorname{rat}}$.
\end{cor} 

\begin{proof}
        It follows from Lemma \ref{NicholstoNichols} immediately.   
\end{proof}      
\subsection{The functor $\Omega$  under $\mathbb{Z}$-gradings}
Having established the duality for Nichols algebras, we now examine
the interaction of $\Omega$ with gradings. Let
$\NN\Gr\Mod_{\mathbbm{k}}$ and $\ZZ\Gr\Mod_{\mathbbm{k}}$ denote the
categories of $\mathbb N_0$-graded and $\mathbb Z$-graded vector
spaces, respectively.
The functor $\NN\Gr \Mod_{k} \rightarrow \ZZ\Gr \Mod_{k}$, which extends the $\NN$-grading of an object $V$ in $\NN\Gr \Mod_{k}$ to a $\ZZ$-grading by setting $V(n) = 0$ for all $n < 0$, is monoidal.  This functor allows us to view $\NN$-graded coalgebras and coquasi-Hopf algebras as $\ZZ$-graded coalgebras and $\ZZ$-graded coquasi-Hopf algebras, respectively.\par 
Furthermore, suppose $H$ is a $\mathbb{N}_0$-graded coquasi-Hopf algebra. A $\ZZ$-graded Yetter-Drinfeld module $V$ over $H$ is by definition an object $V$ in $\HH(\ZZ\Gr \Mod_{k})$. In other words, $V$ is a $\ZZ$-graded vector space such that the map $H\otimes V \rightarrow V$ and the comodule structure maps  $V \rightarrow H \otimes V$ are $\mathbb{Z}$-graded.
\begin{lemma}\textup{[\citealp{reflection1}, Lemma 3.9]} Suppose $H$ is a $\mathbb{Z}$-graded coquasi-Hopf algebra.\par 
\text{$(1)$} Let $K$ be a Nichols algebra in $\HH(\Zgr)$, and $K(1)=\oplus_{\gamma \in \mathbb{Z}} K(1)_\gamma$ a $\ZZ$-graded object in $\HH(\Zgr)$. Then $K$ admits a unique $\mathbb{Z}$-grading that extends the grading  on $K(1)$. Moreover, $K(n)$ is $\ZZ$-graded in $\HH(\Zgr)$ for all $n \geq 0$.\par
\text{$(2)$}  Let $K$ be a $\ZZ$-graded braided Hopf algebra in $\HH(\ZZ\Gr \Mod_{k})$. Then the bosonization $K\#H$ is a $\ZZ$-graded Hopf algebra with $\deg K(\gamma)\#H(\lambda) = \gamma + \lambda$ for all $\gamma, \lambda \in \ZZ$.\par
\text{$(3)$}  Let $H_0 \subseteq H$ be a coquasi-Hopf subalgebra of degree $0$, and $\pi : H \rightarrow H_0$ a coquasi-Hopf algebra map with $\pi|_{H_0} = \mathrm{id}$. Define $R = H^{\operatorname{co}H_0}$. Then $R$ is a $\ZZ$-graded braided Hopf algebra in ${^{ H_0}_{H_0}\mathcal{YD}}(\ZZ\Gr \Mod_{k})$ with $R(\gamma) = R \cap H(\gamma)$ for all $\gamma \in \ZZ$.
\end{lemma}

Now we return to the purpose of this subsection.
We first introduce the following notation for further use.
\begin{definition}
Let $H$ be a coquasi-Hopf algebra with bijective antipode, and $\mathcal{C}=\HH$.
Suppose $R$ is an $\mathbb{N}_0$-graded Hopf algebra in  $\HH$. Let $(X,\lambda_X)$ be a  $R$-module in $\mathcal{C}$ and $(Y,\delta_Y)$ be a $R$-comodule in $\mathcal{C}$. For $n \geq 0$, we define:
\begin{align*}
\mathcal{F}_nX &= \{x \in X \mid \lambda_X(R(i)\ot x) = 0 \text{ for all } i > n\}, \ \\
\mathcal{F}^nY &= \{y \in Y \mid \delta_Y(y) \in \bigoplus_{i=0}^n R(i) \otimes Y\}. 
\end{align*}
\end{definition}
 With  above notations, the functor $\Omega$ relates the operator $\mathcal{F}_n$ to $\mathcal{F}^n$.
\begin{lemma}\label{lemOmega} Let $N$ be a finite-dimensional object in $\HH$. Recall 
we have the following  equivalence of monoidal categories:
    $$ \Omega:{^{\bB(N)}_{\bB(N)}\mathcal{YD}(\mathcal{C})}_{\operatorname{rat}} \cong {^{\bB(N^*)}_{\bB(N^*)}\mathcal{YD}(\mathcal{C})}_{\operatorname{rat}} .$$ Let $V \in {^{\bB(N)}_{\bB(N)}\mathcal{YD}(\mathcal{C})}_{\operatorname{rat}}$, then we have\\[1em]
     {(1)} $\mathcal{F}_n \Omega(V)=\mathcal{F}^nV$ for all $n\geq 0$.\\  {(2)} $\mathcal{F}^n\Omega(V)=\mathcal{F}_n(V)$ for all $n\geq 0$.\\
    {(3)} Suppose  $V$ is  a $\mathbb{Z}$-graded object in ${^{\bB(N)}_{\bB(N)}\mathcal{YD}(\mathcal{C})}$, then $\Omega(V)$ is a $\mathbb{Z}$-graded object in ${^{\bB(N^*)}_{\bB(N^*)}\mathcal{YD}(\mathcal{C})} $, where the grading of  $\Omega(V)$ is given by $\Omega(V)(n)=V(-n)$ for all $n\in \mathbb{Z}$.
\end{lemma}
\begin{proof}
Recall that the braided monoidal equivalence
\[
\Omega:
{}^{\bB(N)}_{\bB(N)}
\mathcal{YD}(\mathcal C)_{\operatorname{rat}}
\longrightarrow
{}^{\bB(N^*)}_{\bB(N^*)}
\mathcal{YD}(\mathcal C)_{\operatorname{rat}}
\]
sends
\[
(V,\lambda_{\bB(N)},\delta_{\bB(N)})
\longmapsto
(V,\lambda_{\bB(N^*)},\delta_{\bB(N^*)}).
\]
This equivalence is induced by the dual pair
\[
\omega:\bB(N^*)\ot\bB(N)\longrightarrow\mathbbm{k}.
\]
(1) Recall
$$\lambda_{\bB(N^*)}=(\bB(N^*)\ot V \xrightarrow{\operatorname{id}\ot \delta_{\bB(N)}} \bB(N^*)\ot (\bB(N)\ot V)\xrightarrow{a^{-1}_{\bB(N^*),\bB(N),V}}(\bB(N^*)\ot \bB(N)) \ot V\xrightarrow{\omega\ot \operatorname{id}_V}V),$$
    Now for fixed $n \geq 0$, the  kernel of the induced map:
    \begin{equation}\label{2.3}
        \bB(N^*)\otimes V \longrightarrow \operatorname{Hom}(\bB(N)(n),V), \ \ a\otimes v\mapsto(b \mapsto \Phi(a_{-1},b_{-1},v_{-1})^{-1}\langle a_0,b_0\rangle v_0)
    \end{equation} 
    is $\bigoplus_{m\neq n}\bB(N^*)(m) \otimes V$. This means,  if $v \in \mathcal{F}^n V$, then for each $i >n$, $\lambda_{\bB(N^*)}(\bB(N^*)(i)\otimes v)=0$ by (\ref{2.3}). On the other hand, if $v \in \mathcal{F}_n(V,\lambda_{\bB(N^*)})$, this will imply $\delta(v)\in \bigoplus_{i=0}^n\BN(i)\otimes V$.  This completes the proof of (1): $\mathcal{F}_n \Omega(V)=\mathcal{F}_n(V,\lambda_{\bB(N^*)})=\mathcal{F}^nV$.\par
  (2)  For the second statement, recall $\delta_{\bB(N^*)}$ is given by 
 \begin{align*}
     &\delta_{\bB(N^*)}=(V\xrightarrow{\delta'}\bB(N^*)\ot V\xrightarrow{\mathcal{S}_{\bB(N^*)}^2\ot \operatorname{id}_V} \bB(N^*)\ot V \xrightarrow{c^2_{\bB(N^*),V}}\bB(N^*) \ot V), \ \text{ where } \ \delta' \ \text{is defined by}\\
 &\lambda_{\bB(N)}=(\bB(N)\ot V \xrightarrow{\operatorname{id} \ot \delta' }\bB(N) \ot (\bB(N^*)\ot V) \xrightarrow{a^{-1}_{\bB(N),\bB(N^*),V}}(\bB(N)\ot \bB(N^*)) \ot V \xrightarrow{ \omega^{+} \ot \operatorname{id}_V} V).
 \end{align*} 
 By [\citealp{rootsys}, Corollary 3.3.6, Lemma 12.2.11], $(V,\delta')$ is a $\bB(N^*)$-comodule in $\HH$, and we have:
$$\mathcal{F}^n(V,\delta')=\mathcal{F}^n\Omega(V).$$
 A similar argument to that in (1) shows that $\mathcal{F}_n(V)=\mathcal{F}^n(V,\delta')$.  Together, these equalities show that
\[
\mathcal F_n(V)
=
\mathcal F^n(V,\delta')
=
\mathcal F^n\Omega(V).
\]\par 
(3) We regard $\bB(N)$ and $\bB(N^*)$ as $\mathbb{Z}$-graded Hopf algebras in $\HH$. Let $V=\bigoplus_{n \in \mathbb{Z}}V(n)$ be its gradation, and set  $\Omega(V)(n)=V(-n)$ for all $n \in \mathbb{Z}$. Then  for $m\geq 0$, by definition of $\rho_{\bB(N^*)}$, we have
$$ \lambda_{\bB(N^*)} ( \bB(N^*)(m)\otimes \Omega(V)(n) )=\lambda_{\bB(N^*)} (\bB(N^*)(m)\otimes V(-n))\subseteq V(-m-n)=\Omega(V)(m+n).   $$
On the other hand, by definition of $\delta_{\bB(N^*)}$
$$ \delta_{\bB(N^*)}(\Omega(V)(n))=\delta_{\bB(N^*)}(V(-n))\subseteq \bigoplus_{i\in \mathbb{Z}}\bB(N^*)(i)\otimes V(i-n)=\bigoplus_{i \in \mathbb{Z}} \bB(N^*)(i)\otimes \Omega(V)(n-i).$$
The above two equations imply that
 $\Omega(V)$ is a well-defined  $\mathbb{Z}$-graded object in ${^{\bB(N^*)}_{\bB(N^*)}\mathcal{YD}(\mathcal{C})}$.
\end{proof}

\section{Reflections of Nichols algebras over coquasi-Hopf algebras}
Throughout this section, $H$ is a coquasi-Hopf algebra with
bijective antipode. We study canonical projections and coinvariant
subalgebras of Nichols algebras and use them to construct
reflections.
\subsection{The coinvariant subalgebra $K$}
We begin with the canonical projection of a Nichols algebra onto the
Nichols algebra of one of its summands.
  \par For any $M,N \in \HH$, there is a canonical surjection of braided Hopf algebra in $\HH$
$$ \pi_{\bB(N)}:\bB(M\oplus N) \rightarrow \bB(N), \  \  \ \pi_{\bB(N)}\mid_N=\operatorname{id}, \ \ \ \pi_{\bB(N)}\mid_M=0. $$  This surjection naturally induces a canonical projection of coquasi-Hopf algebras:
$$ \pi=\pi_{\bB(N)}\#\operatorname{id}_{H}:\bB(M\oplus N)\#H\rightarrow \bB(N)\# H.$$
To simplify our notation, we introduce the following conventions:
$$ \mathcal{A}(M \oplus N):=\bB(M\oplus N)\#H,\ \mathcal{A}(N):=\bB(N)\# H.$$
This projection allows us to construct the associated space of coinvariants. There is a natural injection $\iota: \mathcal{A}(N)\rightarrow \mathcal{A}(M\oplus N)$ such that $\pi\circ \iota=\operatorname{id}_{\mathcal{A}(N)}$. \par 
Let $\pi_{H}: \AMN\rightarrow H$, and $\pi'_{H}:\AN\rightarrow H$ be the natural projection. These maps satisfy a compatibility condition, namely, there exists a commutative diagram relating these projections.
\begin{center}
	\begin{tikzpicture}
		\node (A) at (-2,0) {$\AMN$};
		\node (B) at (3,0) {$\AN$};
		\node (C) at (3,-2) {$H$};
		\node (D) at (3,2) {$\AN$};
  \draw[->] (A) --node [above ] {$\pi$} (B);
		\draw[->] (B) --node [ right] {$\pi'_{H}$} (C);	
		\draw[->] (A) --node [below ] {$\pi_{H}$} (C);
  \draw[->](D)--node [above left]{$\iota$}(A);
  \draw[->](D)--node[right]{$=$}(B);
	\end{tikzpicture}
\end{center}

Let 
$$K= \mathcal{A}(M \oplus N)^{\operatorname{co}\AN}$$
be the space of right $\AN$-coinvariant elements with respect to the projection $\pi$. By virtue of Lemma \ref{Lemma 1.6}, 
$K$ inherits the structure of a braided Hopf algebra in $\BNG$ 
with $\AN$-coaction:
$$ \delta: K \rightarrow \AN \otimes K , \ \ \  x_{-1} \otimes x_0=:\delta(x):=\pi(x_1)\otimes x_2,$$
and  a linear map:
\begin{equation}\label{3.1}
    \operatorname{ad}: \AN \otimes K \rightarrow K, \ \ \ \  a\otimes x \mapsto \operatorname{ad}(a)(x)=\Phi(a_1x_{-1}, \mathbb{S}(a_3)_1,a_4)(a_2x_0)\mathbb{S}(a_3)_2.
\end{equation}
   Here, $\mathbb{S}$ denotes the preantipode of $\AN$. We have  the following two observations.
   \begin{lemma}\textup{[\citealp{reflection1}, Lemma 4.1, Lemma 4.2]}\par  \label{lem5.1}
        {(1)} The space $K$ coincides with $\mathcal{B}(M \oplus N)^{\operatorname{co} \mathcal{B}(N)}$, which is precisely the space of right $\mathcal{B}(N)$-coinvariant elements with respect to $\pi_{\mathcal{B}(N)}$.\par 
         {(2)} The space of primitive elements $P(K)=\{ x\in K \mid \Delta_K(x)=x\otimes 1+1\otimes x\}$ is a subobject of $K$ in $\BNG$.
   \end{lemma}

One may expect that $K$ is a Nichols algebra generated by $\operatorname{ad}(\bB(N))(M)$. The following lemma establish that $\operatorname{ad}(\bB(N))(M)$ is a subobject of $K$.
\begin{lemma}\label{lem3.2}
    Let $M,N \in \HH$, $K=(\AMN)^{\operatorname{co}\AN}$.
   The standard $\mathbb{N}_0$-grading of $\bB(M\oplus N)$ induces an $\mathbb{N}_0$-grading on 
    $$L:=\operatorname{ad}(\bB(N))(M)=\bigoplus_{n \in \mathbb{N}_0}\operatorname{ad}(N)^{n}(M)$$
    with degree $\operatorname{deg}(\operatorname{ad}(N)^n(M))=n+1$. Then $L$ is a $\mathbb{N}_0$-graded object in $\BNG$. Moreover, $L \subseteq K$ is a subobject of $K$ in $\BNG$.\par 
\end{lemma}
\begin{proof}
     Let $a\in N$ and $x \in \bB(M \oplus N)$ be homogeneous elements.
By equation (\ref{1.27}), we have
    $$ \operatorname{ad}(a)(x)=ax-\operatorname{ad}(a_{-1})(x)a_0.$$
The element $\operatorname{ad}(a)(x)$ is of degree $\operatorname{deg}(x)+1$ in the Nichols algebra $\mathcal{B}(M\oplus N)$. Moreover, we have $ \operatorname{ad}(N)^{n}(M) \cap \operatorname{ad}(N)^m(M)=0$ for $ n\neq m$. This  implies the direct sum decomposition of $L$. \par 
Since $M \subseteq K$ and $K \in \BNG$, we conclude that $L=\operatorname{ad}(\BN)(M)\subseteq K$.
Note that\\ $\operatorname{ad}(\BN)(M)\subseteq \operatorname{ad}(\AN)(M).$ Conversely, since $M \in \HH$, we have 
    $$ \operatorname{ad}(\AN)(M) \subseteq \operatorname{ad}(\BN)(\operatorname{ad}(H)(M)) \subseteq \operatorname{ad}(\BN)(M).$$
  Hence $L=\operatorname{ad}(\AN)(M)$. \par 
 It is clear that the map 
    $$ \operatorname{ad}:\AN \otimes L \rightarrow L$$ is well-defined since  $\operatorname{ad}$ satisfies the first  axiom of Yetter-Drinfeld modules.
Furthermore, the map
$ \operatorname{ad}:\AN \otimes L\rightarrow L$ is $\mathbb{N}_0$-graded. 
Next, we want to show $L$ is a $\AN$-comodule. The comodule structure is given by 
$ \delta_K$. 

For $a\in N$ and $x\in L$, since $a$ is primitive in $\BMN$, we
have
$$\Delta_{\AMN}(a)=a_{-1} \ot a_0 + a  \ot 1.$$ A  direct computation shows:
\begin{align*}
\delta_K(\operatorname{ad}(a)(x))&=p_R((\operatorname{ad}(a_3)(x_0))_{-1},a_4)   q_R(a_1x_{-2},\mathcal{S}(a_6))    (a_2x_{-1})\mathcal{S}(a_5)      \otimes (\operatorname{ad}(a_3)(x_0))_{0}
\\&=p_R((\operatorname{ad}(a_0)(x_0))_{-1},1)q_R(a_{-2}x_{-2},1)a_{-1}x_{-1}\ot (\operatorname{ad}(a_0)(x_0))_{0}\\
&+p_R((\operatorname{ad}(a_{-2})(x_0))_{-1},a_{-1})q_R(a_{-4}x_{-2},1)(a_{-3}x_{-1})\mathcal{S}(a_0)\ot (\operatorname{ad}(a_{-2})(x_0))_0
\\&=a_{-1}x_{-1}\ot \operatorname{ad}(a_{0})(x_0)+p_R((\operatorname{ad}(a_{-2})(x_0))_{-1},a_{-1})(a_{-3}x_{-1})\mathcal{S}(a_0)\ot (\operatorname{ad}(a_{-2})(x_0))_0.
\end{align*}

For $n\in\mathbb N_0$, set
$L_n:=\operatorname{ad}(N)^n(M),
L_{\leq n}:=\bigoplus_{r=0}^{n}L_r.
$
We prove by induction on $n$ that
\[
\delta_K(L_{\leq n})\subseteq \mathcal A(N)\otimes L_{\leq n}.
\]
For $n=0$, we have $L_0=M$. Since $M$ is an object in
$\HH$, its $H$-coaction satisfies
$\delta_K(M)\subseteq H\otimes M
   \subseteq \mathcal A(N)\otimes L_{\leq 0}.
$
 Suppose that, for some $n\in\mathbb N_0$,
$
\delta_K(L_{\leq n})
   \subseteq \mathcal A(N)\otimes L_{\leq n}.
$
Since $L_{n+1}$ is spanned by elements of the form
$\operatorname{ad}(a)(x)$, where $a\in N$ and $x\in L_n$,
it suffices to consider such an element. By the inductive
hypothesis, the second tensor components of $\delta_K(x)$
belong to $L_{\leq n}$. Hence
$
\operatorname{ad}(a_0)(x_0)\in L_{\leq n+1}.
$
 As $a_{-2}\in H$, it follows that
$\operatorname{ad}(a_{-2})(x_0)\in L_{\leq n}
$
and therefore
$
\bigl(\operatorname{ad}(a_{-2})(x_0)\bigr)_0
   \in L_{\leq n}.
$
Applying the preceding formula for
$\delta_K(\operatorname{ad}(a)(x))$, we obtain
$
\delta_K\bigl(\operatorname{ad}(a)(x)\bigr)
   \in \mathcal A(N)\otimes L_{\leq n+1}.
$
Thus
\[
\delta_K(L_{\leq n+1})
   \subseteq \mathcal A(N)\otimes L_{\leq n+1},
\]
which completes the induction. Consequently,
\[
\delta_K(L)\subseteq\mathcal A(N)\otimes L.
\] The third axiom holds automatically since $L \subseteq K$. Thus $L$ is a $\mathbb{N}_0$-graded object in $\BNG$ and  is a subobject of $K$.\par
\end{proof}
The following lemma is used to prove the 
 semisimplicity of  the object $L$ in the case that $M$ is semisimple. 
\begin{lemma}\label{lem3.3}
    For all $x \in L$, we have  $\Delta_{\bB(M\oplus N)}(x) - x \otimes 1 \in \bB( N)\otimes L$. \par 
\end{lemma}
\begin{proof} Recall $L=\bigoplus_{n \in \mathbb{N}_0}\operatorname{ad}(N)^{n}(M)$.
 We proceed by induction on $n$. Now let $a \in N$ and $x \in M$ be homogeneous elements. A direct computation yields:
\begin{align*}
\Delta_{\BMN}(\operatorname{ad}(a)(x))
&=\Delta_{\BMN}\bigl(
ax-\operatorname{ad}(a_{-1})(x)a_0\bigr)\\
&=\operatorname{ad}(a)(x)\otimes1
+1\otimes\operatorname{ad}(a)(x)+a\otimes x\\
&\quad-\operatorname{ad}\bigl(
(\operatorname{ad}(a_{-1})(x))_{-1}\bigr)(a_0)
\otimes(\operatorname{ad}(a_{-1})(x))_0.
\end{align*}
Note that $x=\operatorname{ad}(1)(x) \in L$. Since $\operatorname{ad}(a_{-1})(x) \in L$, we have $(\operatorname{ad}(a_{-1})(x))_0 \in L$ by Lemma \ref{lem3.2}.
Therefore $\Delta_{\BMN}(\operatorname{ad}(a)(x))- \operatorname{ad}(a)(x)\otimes 1\in \BN \otimes L$.\par 
Fix $m\in\mathbb N_0$, and assume that
\[
\Delta_{\bB(M\oplus N)}(y)-y\ot1\in\bB(N)\ot L
\]
for every
$y\in\operatorname{ad}(N)^{m'}(M)$ with $0\leq m'\leq m$.
Now let $a\in N$ and
$x\in\operatorname{ad}(N)^m(M)$ be homogeneous. Then
\begin{align*}
    \Delta_{\BMN}(\operatorname{ad}(a)(x))&=\Delta_{\BMN}(ax-(\operatorname{ad}(a_{-1})(x))a_{0})\\
    &=(a \otimes 1+1 \otimes a)(x^1 \otimes x^2)\\&-\frac{\Phi (a_{-5}, x^1_{-1} , x^2_{-2})\Phi((\operatorname{ad}(a_{-4})(x^1_0))_{-1}, (\operatorname{ad}(a_{-2})(x^2_0))_{-1}, a_{-1})}{\Phi ((\operatorname{ad}(a_{-4})(x^1_0))_{-2}, a_{-3}, x^2_{-1}) }\times \\
&  ((\operatorname{ad}(a_{-4})(x^1_0))_{0} \otimes ((\operatorname{ad}(a_{-2})(x^2_0))_{0})(1\otimes a_{0}+a_{0} \otimes 1).
  \end{align*}
 Since the term $x \ot 1 $ lies in $\Delta_{\BMN}(x)$, let us consider this expression separately. The term $\operatorname{ad}(a)(x) \ot 1$ will occur in $\Delta_{\BMN}(\operatorname{ad}(a)(x))$.
 For brevity, set
\begin{align*}
C_1&=
\frac{\Phi(a_{-3},x^1_{-1},x^2_{-2})}
{\Phi((\operatorname{ad}(a_{-2})(x^1_0))_{-1},
a_{-1},x^2_{-1})},\\
C_2&=
\frac{\Phi(a_{-6},x^1_{-1},x^2_{-2})}
{\Phi((\operatorname{ad}(a_{-5})(x^1_0))_{-3},
a_{-4},x^2_{-1})}
\cdot
\frac{\Phi((\operatorname{ad}(a_{-5})(x^1_0))_{-2},
(\operatorname{ad}(a_{-3})(x^2_0))_{-2},a_{-2})}
{\Phi((\operatorname{ad}(a_{-5})(x^1_0))_{-1},
(\operatorname{ad}(a_{-3})(x^2_0))_{-1},a_{-1})},\\
C_3&=
\frac{\Phi(a_{-5},x^1_{-1},x^2_{-2})}
{\Phi((\operatorname{ad}(a_{-4})(x^1_0))_{-3},
a_{-3},x^2_{-1})}\cdot
\Phi((\operatorname{ad}(a_{-4})(x^1_0))_{-2},
(\operatorname{ad}(a_{-2})(x^2_0))_{-3},a_{-1})\\
&\quad{}\cdot
\Phi\bigl((\operatorname{ad}(a_{-4})(x^1_0))_{-1},
(\operatorname{ad}((\operatorname{ad}(a_{-2})(x^2_0))_{-2})(a_0))_{-1},
(\operatorname{ad}(a_{-2})(x^2_0))_{-1}\bigr).
\end{align*}
The remaining terms are
\begin{align*}
&\sum_{x^1\otimes x^2\neq x\otimes1}
\Phi(a_{-1},x^1_{-1},x^2_{-1})a_0x^1_0\otimes x^2_0
+C_1
(\operatorname{ad}(a_{-2})(x^1_0))_0\otimes a_0x^2_0\\
&\quad-C_2
(\operatorname{ad}(a_{-5})(x^1_0))_0
\ot(\operatorname{ad}(a_{-3})(x^2_0))_0a_0\\
&\quad-C_3
(\operatorname{ad}(a_{-4})(x^1_0))_0\cdot
(\operatorname{ad}((\operatorname{ad}(a_{-2})(x^2_0))_{-2})(a_0))_0
\otimes(\operatorname{ad}(a_{-2})(x^2_0))_0.
\end{align*}
By assumptions, $$ \Delta_{\bB(M\oplus N)}(x) - x \otimes 1 \in \bB( N)\otimes L.$$ Therefore the first term lies in $\BN \otimes{ L}$. For the third term, note that $\Phi$ is convolution invertible, thus the third term becomes
$$\sum_{x^1\otimes x^2 \neq x\ot 1} \frac{\Phi(a_{-3},x^1_{-1},x^2_{-2})}{\Phi((\operatorname{ad}(a_{-2})(x^1_0))_{-1},a_{-1},x^2_{-1})}(\operatorname{ad}(a_{-2})(x^1_0))_{0}\otimes \operatorname{ad}(a_{-1})(x_0^2)a_0.$$
Combining the second and third terms together yields
$$\sum_{x^1\otimes x^2 \neq x\ot 1}\frac{\Phi(a_{-3},x^1_{-1},x^2_{-2})}{\Phi((\operatorname{ad}(a_{-2})(x^1_0))_{-1},a_{-1},x^2_{-1})}(\operatorname{ad}(a_{-2})(x^1_0))_{0}\otimes \operatorname{ad}(a_0)(x^2_0).$$
By assumption, $\sum_{x^1\otimes x^2 \neq x\ot 1} x^1  \ot x^2 \in \bB(N) \ot L$, therefore this term lies in $\bB(N) \ot L$.
The last term belongs to $\BN \otimes{ L}$ by similar reason, which completes the induction and proves the lemma.
    \end{proof}

    \begin{lemma}\label{C-lem3.5}
    If $L' \subseteq L$ is a non-zero $\AN$-subcomodule, then $L' \cap M \neq 0$.
\end{lemma}
\begin{proof}
We begin by observing that $\mathcal{B}(M \oplus N)$ carries a natural left $\mathcal{B}(N)$-comodule structure in the category $\HH$, defined via the composition: $$ \delta: \bB(M\oplus N)\xrightarrow {\Delta_{\bB(M\oplus N)}}\bB(M\oplus N) \otimes \bB(M\oplus N)\xrightarrow {\pi_{\BN} \otimes \operatorname{id}} \bB(N) \otimes \bB(M\oplus N).$$
  Let $X\in L$. By Lemma~\ref{lem3.3},
$\delta(X)\in\BN\ot L$. Hence $L$ inherits an
$\mathbb N_0$-graded $\BN$-comodule structure by restricting
$\delta$ to
\[
\delta:L\longrightarrow\BN\ot L.
\]

For $X\in L(n)$ with $n\geq1$, consider the component
\[
\delta_{n-1,1}:L(n)\longrightarrow\BN(n-1)\ot L(1)
=\BN(n-1)\ot M.
\]
By Lemma~\ref{lem3.3},
$\Delta_{\bB(M\oplus N)}(X)=X\ot1+\BN\ot L$, and hence
\[
\delta_{n-1,1}(X)
=
(\Delta_{\bB(M\oplus N)})_{n-1,1}(X).
\]
Since Nichols algebras are strictly graded connected coalgebras,
$(\Delta_{\bB(M\oplus N)})_{n-1,1}$, and therefore
$\delta_{n-1,1}$, is injective. As $L(0)=0$,
[\citealp{rootsys}, Proposition~13.2.3] gives
\[
L'\cap L(1)=L'\cap M\neq0.
\]
\end{proof}
\begin{lemma}\label{lem3.5}\par 
  \text{$(1)$} Assume $M=\bigoplus_{i \in I} M_i$ is a direct sum  in $\HH$. Let $L_i=\operatorname{ad}\bB(N)(M_i)$ for all $i \in I$. Then we have such a decomposition in $\BNG$
    $$ L=\bigoplus_{i \in I} L_i.$$\par 
  \text{$(2)$} If $M$ is simple in $\HH$, then $L=\operatorname{ad}\left(\BN\right)(M)$ is simple in  $\BNG$. 
\end{lemma}
\begin{proof}
 (1)  It suffices to show that the sum is direct. Suppose, for contradiction, that the sum is not a direct sum. Then there exists an index $k \in I$ such that
 $L_k \cap \sum_{i \neq k}L_i \neq 0$. By Lemma \ref{C-lem3.5}, this implies $$(L_k \cap \sum_{i \neq k}L_i) \cap M_k\neq 0,$$ hence $\sum_{i \neq k}L_i \cap M_k \neq 0 $. But since $M_k$ lies in degree one in $L_k$, we obtain 
$$\sum_{i \neq k} M_i \cap M_k \neq 0,$$
contradicting the directness of the sum $M = \bigoplus_{i \in I} M_i.$
    \par 
    (2) Let $0\neq L'\subseteq L$ be a subobject in
$\BNG$. By Lemma~\ref{C-lem3.5}, $L'\cap M\neq0$. Moreover,
$L'\cap M$ is a subobject of $M$ in $\HH$. Since $M$ is simple,
we obtain $L'\cap M=M$. Hence
    $$ L=\operatorname{ad}(\AN)(M)=\operatorname{ad}(\AN)(L'\cap M) \subseteq L',$$
    which implies $L^{\prime} = L$. We conclude that $L$ is simple.
\end{proof}

We now proceed to prove that $K$ possesses the structure of a Nichols algebra in an appropriate Yetter-Drinfeld module category. 
\begin{thm}\label{thm3.7}
    There is an isomorphism 
    \begin{equation}
        K \cong \bB(L)
    \end{equation}
    of Hopf algebras in the category $\BNG$. In particular, $P(K)=L$.
\end{thm}
\begin{proof}
The proof follows the strategy of
\cite[Theorem~4.9]{reflection1}; we include the details for
completeness.\par 
 We endow a non-standard grading on $\bB(M\oplus N)$ via setting $\operatorname{deg}(M)=1$ and $\operatorname{deg}(N)=0$. Under this grading, $\bB(M\oplus N)$ remains  a $\mathbb{N}_0 $-graded Hopf algebra in $\HH$. Furthermore, $\bB(M\oplus N) \# H$ becomes a $\mathbb{N}_0$-graded coquasi-Hopf algebra with $\operatorname{deg}(H)=0$. Now we are going to show $K$ is a Nichols algebra by verifying the axiom in Definition \ref{def3.18}.
 \par (1)
      It follows directly that $K$ inherits an  $\mathbb{N}_0$-graded Hopf algebra in $\BNG$ with $K(n)=K \cap (\bB(M\oplus N) \# H)(n)$.
     \par (2) We have $K(0)=K \cap (\bB(M\oplus N) \# H)(0)=K \cap (\bB(N)\#H)=\mathbbm{k}$.
     \par (3) 
 We are going to show  $K$ is generated by  $L$. Let  $K'$ be the subalgebra of  $K$ generated by $L$ in $\BNG$, which is an object in  $\BNG$ with structures induced by tensor product of $L$. Now we let $W$ be the image of $K' \# \BN$ under the isomorphism $K \# \BN \cong \bB(M\oplus N)$. It suffices to show that $W=\bB(M\oplus N)$. Note that $M \oplus N \subseteq W$ since $M,N \subseteq W$. Thus, we need only verify that $W$ forms a subalgebra of $\bB(M\oplus N)$. \par 
The     key observation is the stability of $K'$ under the map of $\mathcal{B}(N)$.
    Indeed, $L$ itself is stable under the  map of $\BN$. Now consider arbitrary homogeneous elements $y,z\in L$ and $a \in \AN$, by (\ref{1.12}),
    \begin{align*}
        &\operatorname{ad}(a)(y\otimes z)\\&=\frac{\Phi(a_1,y_{-1},z_{-2})\Phi((\operatorname{ad}(a_2)(y_0))_{-1},(\operatorname{ad}(a_4)(z_0))_{-1},a_5))}{\Phi(({\operatorname{ad}(a_2)(y_0)})_{-2},a_3,z_{-1})}(\operatorname{ad}(a_2)(y_0))_0\otimes(\operatorname{ad}(a_4)(z_0))_0.
    \end{align*} Since $L$ is an object in $\BNG$,
    it follows that $\operatorname{ad}(a)(y\otimes z)\in L \otimes L$. Thus $K'$ remains stable under  the  map induced by $\BN$ by induction.

    As both $K'$ and $\mathcal{B}(N)$ are subalgebras of $\mathcal{B}(M \oplus N)$, we only need to examine the case when $a \in N$ and $x \in K'$ . In this situation:
    \begin{align*}
        \operatorname{ad}(a)(x)=ax-\operatorname{ad}(a_{-1})(x)a_{0}.
    \end{align*}
    Since $\operatorname{ad}(a)(x)\in K'$ and $\operatorname{ad}(a_{-1})(x)a_{0}\in K' \# \bB(N)$, we conclude that  $ax \in K' \# \bB(N)$. Hence its image lies in $W$. Thus $W$ is a subalgebra of $\bB(M\oplus N)$, which establishes that $K$ is generated by $L$ in $\BNG$. 
    Note that $L \subseteq K(1)$, this forces $K(1)=L$. Hence $K$ is generated by $K(1)$.
    \par (4) To show $P(K)=K(1)$, that is, there is no primitive element in $K(n)$, $n\geq 2$. Suppose, for contradiction, that there exists a nonzero subspace $U \subseteq P(K(n))$ for some  $n\geq 2$. By Lemma \ref{lem5.1}, $U$ is an object in $\BNG$, since $\operatorname{deg}(N)=\operatorname{deg}(H)=0$. Now, consider the coalgebra filtration on $\BN\# H$ given by
    $(\BN\#H)_0=H$, $(\BN\#H)_1=H+N$. Applying [\citealp{reflection1}, Lemma 4.11], we find a nonzero element $u \in U$ such that $\delta_K(u)=u_{-1}\otimes u_0\in H\otimes U$. Note that 
    $$ \delta_K(u)=(\pi\otimes \operatorname{id})\Delta_{\AMN}(u).$$ Therefore
     $$\Delta_{\AMN}(u)=u \otimes 1+u_{-1}\otimes u_0.$$
    Applying the map $\mathrm{id} \# \varepsilon \otimes \mathrm{id}$, we obtain $$\Delta_{\bB(M\oplus N)}(u)=(\operatorname{id}\# \varepsilon \otimes \operatorname{id})\Delta(u)=1\otimes u+u\otimes 1,$$
      which implies that $u$ is a primitive element in $\mathcal{B}(M \oplus N)$. However, since $K(n)$ is generated by $L = \operatorname{ad}(\mathcal{B}(N))(M)$, the element $u$ must have degree at least $n$ in the standard grading of $\mathcal{B}(M \oplus N)$. This leads to a contradiction, as primitive elements in a Nichols algebra lie in degree one.
    
\end{proof}

We now turn to a converse of the preceding result under an additional restriction.\par 
\begin{lemma}\label{C-lem 3.10}
Let $N \in \HH$ and $W \in \BNG$. Assume that $W=\bigoplus_{i \in I}W_i$ is a decomposition into simple objects in $\BNG(\Zgr)$. Let $$M=W(H)=\{ w\in W\mid \delta(w)\in H \otimes W\},$$ and $M_i=M \cap W_i$ for all $i \in I$. \par
\text{$(1)$} $M=\bigoplus_{i\in I}M_i$ is a decomposition into simple
objects in $\HH(\Zgr)$.\par 
\text{$(2)$} For all $i \in I$, $M_i$ is the $\mathbb{Z}$-homogeneous component of $W_i$ of minimal degree, and $W_i=\BN\rhd M_i=\bigoplus_{n\geq 0}N^{\otimes n}\rhd M_i$.

\end{lemma}

\begin{proof}
    Let $W=\bigoplus_{n\in \mathbb{Z}}W(n)$ be the $\mathbb{Z}$-grading of $W$ in $\BNG$. Then $M$ is a $\mathbb{Z}$-graded object in $\HH$ with homogeneous components $M(n)=M \cap W(n)$ for all $n \in \mathbb{Z}$. It is clear that $M=\bigoplus_{n\in \mathbb{Z}}\bigoplus_{i \in I}M_i(n)=\bigoplus_{i \in I}M_i$, where $M_i=M \cap W_i=W_i(H)$.\par
    We now show that each  $M_i$ is an  simple object in $\HH(\Zgr)$ for each $i \in I$. First, by [\citealp{reflection1}, Lemma 4.11], 
  we have $M_i \neq 0$. Suppose $0 \neq M_i' \subseteq M_i$ is a homogeneous subobject in $\HH$, and let $m_i$ be the degree of $M_i'$. Define
 $$W_i':=\BN \rhd M_i'.$$
Using a method similar to that in Lemma \ref{lem3.2}, one verifies that $W_i'$ is a $\mathbb{Z}$-graded $\mathcal{A}(N)$-subcomodule. Moreover,
  $\AN \rhd (\BN \rhd M_i') \subseteq \AN \rhd M_i'=\BN \rhd M_i'$. Therefore $W_i'$ is a $\mathbb{Z}$-graded subobject of $W_i$ in $\BNG$.
The object $W_i'$ is concentrated in degrees $\geq m_i$, and its degree $m_i$'s component is exactly $M_i'$. Since $W_i$ is simple, we must have $W_i' = W_i$. Therefore,
$M_i'=W_i'\cap M=W_i \cap M=M_i$, which establishes the simplicity of $ M_i$. Moreover, $M_i$ is indeed the homogeneous component of $W_i$ of minimal degree. Finally, for each $i \in I$ and $n \in \mathbb{N}_0$, 
    $$ \operatorname{deg}(N^{\otimes n} \rhd M_i)=n+\operatorname{deg}(M_i),$$ since the  map $\BN \# H \otimes W_i \rightarrow W_i$ is $\mathbb{Z}$-graded. It follows that $\BN \rhd M_i=\bigoplus\limits_{n \geq 0}N^{\otimes n} \rhd M_i$ as required.
\end{proof}
\begin{thm}\label{prop3.11}
    Let $N \in \HH$ and  $(W, \rhd ,\delta)$ be a semisimple object in the category $\BNG(\Zgr)$ with linear map $\rhd: \AN \ot W \rightarrow W$. Here $\AN$ is equipped with the standard $\mathbb{N}_0$-grading. Let $\mathcal{B}(W)$ be the Nichols algebra of $W$ in $\BNG$ and define $M=\lbrace w\in W\mid \delta(w) \in H \ot W\rbrace$. Then there is an isomorphism 
    \begin{equation}
        \bB(W) \# \BN\cong \mathcal{B}(M\oplus N)
    \end{equation} of braided Hopf algebras in $\HH$ which is the identity on $M \oplus N$.
    \end{thm}
\begin{proof}
     Let $W=\bigoplus_{i \in I} W_i$ be the decomposition of $W$ into simple objects in the category \\$\BNG(\Zgr)$.  For each $i \in I$, define $M_i=M\cap W_i$. By Lemma \ref{C-lem 3.10}, $W_i=\bigoplus_{n \geq 0}N^{\otimes n}\rhd M_i$. Denote $\operatorname{deg}(M_i)=m_i$ for each $i \in I$. We endow $W_i$ with a new grading by $\widetilde{W_i}=W_i$ with  
    $$ \widetilde{W_i}(n)=W_i(n+m_i-1)=N^{\otimes n-1}\rhd M_i,\ \text{for all} \ n \in \mathbb{N}_{>0}.$$
 With the new grading, $\widetilde{W}=\bigoplus_{i \in I}\widetilde{W_i}$ remains an object in $\BNG$ since degree shifiting preserves graded modules and graded comodules. Moreover, $\widetilde{W}(n)=0$ for all $n \leq 0$.  
So the Nichols algebra $\bB(\widetilde{W})$  is an $\mathbb{N}_0$-graded Hopf algebra in $\BNG$.  Note that $\bB(\widetilde{W})\cong \bB(W)$, as both are determined by the same module and comodule maps.
Under this grading, $\widetilde{W}(0)=0$ and $\widetilde{W}(1)=M$, and the degree zero and degree one part of the Nichols algebra $\bB(\widetilde{W})$ are   $$\bB(\widetilde{W})(0)=\mathbbm{k}, \ \ \bB(\widetilde{W})(1)=M.$$ Moreover, $\bB(W)\#(\BN \# H)$ is a $\mathbb{N}_0$-graded coquasi-Hopf algebra, and 
$$ R:=\bB(W)\#\BN=(\bB(W) \#(\BN\# H))^{\operatorname{co}H}$$
is an $\mathbb{N}_0$-graded Hopf algebra in $\HH$. Its degree $0$ part is $\mathbbm{k}$, and its degree $1$ part is 
$M\oplus N$. \par 
We are going to show $R$ is a pre-Nichols algebra in $\HH$, it remains to show $R$ is generated by  $M \oplus N$ as an algebra in $\HH$. Let $C$ denote the subalgebra of $R$ generated by
$\bB(N)$ and $M$. It is direct to see that $R$ is generated by $W=\BN \rhd M$ and $N$. It therefore suffices to prove that $\BN \rhd M=\bigoplus_{n \geq 0}N^{\otimes n} \rhd M$ is contained in $C$. To see this, we proceed by induction on 
 $n$. For the base case, take  elements $x \in M$ and $ y\in N$. Inside the coquasi-Hopf algebra $\bB(W)\#(\BN \# H)$, we have
$$ yx=(y_1\rhd x)y_2=(y_{-1} \rhd x)y_0+(y \rhd x).$$
Since $M \in \HH$ and $y_{-1} \in H$, $yx $ and $(y_{-1} \rhd x)y_0$ are contained in $C$. Therefore, the term $y \rhd x$ is contained in $C$. 
 
 Assume that the assertion holds for all degrees not exceeding
$n$, and let $z\in\bB(N)(n)$, $y\in N$, and $x\in M$ be
homogeneous. Since $\bB(N)$ is generated by $N$, it suffices to
show that $(yz)\rhd x\in C$. By the gradedness of the
coproduct, every homogeneous component of $z_2$ has degree at
most $n$. Hence the inductive hypothesis gives
$
z_2\rhd x_0\in C.
$
Moreover, 
$
(z_2\rhd x_0)_0,
y_{-1}\rhd(z_2\rhd x_0)_0\in C.
$
Using
\begin{align*}
(yz)\rhd x
&=
\frac{\Phi\bigl(y_{-2},(z_2\rhd x_0)_{-1},z_3\bigr)}
     {\Phi\bigl(y_{-3},z_1,x_{-1}\bigr)}
\bigl(y_{-1}\rhd(z_2\rhd x_0)_0\bigr)y_0
\\
&\quad-
\frac{\Phi\bigl(y_{-1},(z_2\rhd x_0)_{-1},z_3\bigr)}
     {\Phi\bigl(y_{-2},z_1,x_{-1}\bigr)}
y_0(z_2\rhd x_0)_0,
\end{align*}
we conclude that $(yz)\rhd x\in C$. This completes the
induction. Therefore,
\[
\bB(N)\rhd M\subseteq C,
\]
and hence $R$ is generated by $M\oplus N$.

 By the universal property of $\bB(M\oplus N)$, there exists a surjective morphism of $\mathbb{N}_0$-graded Hopf algebras in $\HH$:
$$ \rho : R \rightarrow \bB(M \oplus N), \rho\mid_{M\oplus N}=\operatorname{id}.$$
This induces a surjective coquasi-Hopf algebra map:
$$ \rho \# \operatorname{id}:R \# H \rightarrow \bB(M \oplus N)\# H.$$
Let $K=(\bB(M\oplus N)\# H)^{\operatorname{co}\BN\#H}$. Then we obtain two bijective coquasi-Hopf algebra maps:
$$ R\# H \rightarrow \bB(W) \# (\BN \# H), \ \ K \# (\BN \# H)\rightarrow \bB(M\oplus N)\# H.$$
Then the map $\rho \# \operatorname{id}$ thus induces a surjective map of coquasi-Hopf algebras
$$ \rho': \bB(W)\#(\BN \#H)\rightarrow K\#(\BN \# H), \ \ \rho'\mid_{(M \oplus N)}=\operatorname{id}.$$
Moreover,  the following diagram commutes:
 \begin{center}
	\begin{tikzpicture}
		\node (A) at (0,0) {$\bB(W) \# (\BN \# H)$};
		\node (B) at (6,0) {$K\#(\BN\#H)$};
		\node (C) at (0,-3) {$\BN \# H$};
  	\node (D) at (6,-3) {$\BN \# H$};
		\draw[->] (A) --node [above] {$\rho'$} (B);
		\draw[->] (A) --node [  right] {$\varepsilon_{\bB(W)} \# \operatorname{id}_{\BN \# H}$} (C);	
		\draw[->] (B) --node [right] {$\varepsilon_{K} \# \operatorname{id}_{\BN \# H} $} (D)	;
		\draw[->] (C) --node [above] {$\operatorname{id} $} (D)	;
	\end{tikzpicture}
\end{center}
since ${\rho\mid}_{M\oplus N}=\operatorname{id}$. As $\BN \# H$ acts on $K$ via adjoint map.  Theorem $\ref{thm3.7}$ implies that  $K\cong \bB(\operatorname{ad}(\BN)(M))$.
Therefore $\rho'$ induces a surjective map 
$$\phi: \bB(W)\rightarrow \bB(\operatorname{ad}(\BN)(M))$$
in $\BNG$ between the right coinvariant subspaces of $\varepsilon_{\bB(W)} \# \operatorname{id}_{\BN \# H}$ and $\varepsilon_{K} \# \operatorname{id}_{\BN \# H} $, satisfying $\phi\mid_M=\operatorname{id}$.
Furthermore, there is a surjective map in $\BNG$:
$$ \phi_1: \BN \rhd M \rightarrow \operatorname{ad}(\BN)(M), \ \ \phi_1\mid_M=\operatorname{id}.$$
Since  $M=\bigoplus_{i\in I}M_i$ is a decomposition into simple
objects in $\HH$, each $\operatorname{ad}(\BN)(M_i)$ is simple in $\BNG$. On the other hand, each $\BN \rhd M_i$ is simple as well.  Thus $\phi_1$ is an isomorphism in $\BNG$, and it follows that  $\phi$ is an isomorphism of Hopf algebra in $\BNG$. Consequently,
$\rho \# \operatorname{id}_{H}=\phi\#\operatorname{id}_{\AN}$ is an isomorphism of coquasi-Hopf algebras. 
Therefore,
$$ \rho=(\rho \#\operatorname{id}_{H})^{\operatorname{co}H} : \bB(W)\# \BN \rightarrow \bB(M\oplus N)$$ is an isomorphism of Hopf algebras in $\HH$.

\end{proof}

\subsection{Reflected tuples and the associated semi-Cartan graph}
We now define reflections of tuples of finite-dimensional
Yetter--Drinfeld modules and construct the associated semi-Cartan
graph.
\par 
Fix a positive integer $\theta$ and  denote the index set $\mathbb{I}=\{ 1,2,...,\theta\}$. 
\begin{definition}\label{def5.1}
    Let $\mathcal{F}_{\theta}$ denote the class of all $\theta$-tuples $M = (M_1, \ldots, M_{\theta})$, where $M_1, \ldots, M_{\theta} \in \HH$ are finite-dimensional  Yetter-Drinfeld modules. If $M \in \mathcal{F}_{\theta}$, we define

\[
\mathcal{B}(M): = \mathcal{B}(M_1 \oplus \cdots \oplus M_{\theta}).
\]
Two tuples $M, M' \in \mathcal{F}_{\theta}$ are called isomorphic, denoted $M \cong M'$, if $M_j \cong M_j'$ in $\HH$ for all j.
The isomorphism class of $M \in \mathcal{F}_{\theta}$ is denoted by $[M]$.
\par 
For $1 \leq i \leq \theta$ and $M \in \mathcal{F}_{\theta}$, we say the tuple $M$ admits the $i$-th reflection $R_i(M)$  if for all $j \neq i$ there is a natural number $m_{ij}^M \geq 0$ such that $(\mathrm{ad}\, M_i)^{m_{ij}^M} (M_j)$ is a non-zero finite-dimensional subspace of $\mathcal{B}(M)$, and $(\mathrm{ad}\, M_i)^{m_{ij}^M + 1} (M_j) = 0$. Assume  $M$ admits the $i$-th reflection. Then we set $R_i(M) = (V_1, \ldots, V_{\theta})$, where

\[
V_j = 
\begin{cases} 
M_i^*, & \text{if } j = i, \\
\mathrm{ad}(M_i)^{m_{ij}^M} (M_j), & \text{if } j \neq i.
\end{cases}
\]
\end{definition}
Such a question arises naturally: how do the simplicity properties behave under reflections? The following result provides a reassuring answer.
\begin{lemma}\label{lem3.13}
    Suppose $M \in \mathcal{F}_{\theta}$ admits the $i$-th reflection for some $i \in \mathbb{I}$, and $M_j$ is simple for each $j \in \mathbb{I}$. Then each $R_i(M)_j$ is simple in $\HH$ for  $1 \leq j \leq \theta$.
\end{lemma}
\begin{proof}
    By definition,  $R_i(M)$ is defined, thus ${R_i(M)}_i \cong M_i^*$ is simple since $M_i$ is. For $j \neq i$, we observe that
    $\operatorname{ad}(\bB(M_i))(M_j)=\bigoplus_{n=0}^{m_{ij}}\operatorname{ad}(M_i)^n(M_j)$. Consider the embedding
    $H \rightarrow \mathcal{A}(M_i) $ and the projection $\mathcal{A}(M_i) \rightarrow H$. The object 
    $\operatorname{ad}(M_i)^{m_{ij}}(M_j)$ belongs to $\HH$. Moreover, it is obvious that $\operatorname{ad}(\bB(M_i))(M_j)$ is generated by $\operatorname{ad}(M_i)^{m_{ij}}(M_j)$ as a $\mathcal{A}(M_i)$-comodule. Now suppose $0 \neq P \subset \operatorname{ad}(M_i)^{m_{ij}}(M_j)$ is a Yetter-Drinfeld submodule in $\HH$. 
    Let $\langle P \rangle$ be the $\mathcal{A}(M_i)$-subcomodule of $\operatorname{ad}\bB(M_i)(M_j)$ generated by $P$,  defined explicitly as
    $$ \langle P \rangle:=\operatorname{span}\{\langle f,x_{-1}\rangle x_0\mid x \in P, f \in \operatorname{Hom}(\mathcal{A}(M_i),\mathbbm{k})\}.$$
   We want to prove
    $ \langle P\rangle$ is a Yetter-Drinfeld submodule of  $\operatorname{ad}\bB(M_i)(M_j)$ in $\AMI$. To see this, for all  $a \in \mathcal{A}(M_i)$ and  $x \in P$, by Lemma \ref{lem2.6},
\begin{align*}
    \operatorname{ad}(a)&(\langle f,x_{-1}\rangle x_0)=\langle f,x_{-1}\rangle \operatorname{ad}(a)(x_0)\\
    &=\langle f,p_L(\mathcal{S}(a_1),(\operatorname{ad}(a_4)(x_0))_{-1}a_6 )q_L(a_3,x_{-1})\mathcal{S}(a_2)((\operatorname{ad}(a_4)(x_0))_{-2}a_5)\rangle (\operatorname{ad}(a_4)(x_0))_0.
\end{align*}
One may view $s:=p_L(\mathcal{S}(a_1), \rule[0.1ex]{0.4cm}{0.6pt}\cdot  a_6 )$ and $t:=\langle f, q_L(a_3,x_{-1})\mathcal{S}(a_2)( \rule[0.1ex]{0.4cm}{0.6pt}\cdot a_5)\rangle$ as  two linear functions, thus 
 $$ \operatorname{ad}(a)(\langle f,x_{-1}\rangle x_0)=t*s\left((\operatorname{ad}(a_4)(x_0))_{-1}\right)(\operatorname{ad}(a_4)(x_0))_0.$$
Hence $\operatorname{ad}(a)(\langle f,x_{-1}\rangle x_0) \in \langle P\rangle.$
This ensures the first axiom of a Yetter-Drinfeld module, and the third axiom holds automatically. Therefore $ \langle P\rangle$ is a Yetter-Drinfeld submodule of  $\operatorname{ad}\bB(M_i)(M_j)$ in $\AMI$. \par 
    Since $\operatorname{ad}(\bB(M_i))(M_j)$  is simple in 
$\AMI$ by Lemma \ref{lem3.5} (2), and $P \neq 0$, we have $\langle P \rangle= \operatorname{ad}(\bB(M_i))(M_j)$. It is direct to see that $ \langle P \rangle= \bigoplus_{i=0}^{m_{ij}+1} \langle P \rangle\cap \bB(M_i \oplus M_j)(i)$. Then $$P=\langle P \rangle\cap \bB(M_i \oplus M_j)({m_{ij}}+1)=\operatorname{ad}(\bB(M_i))(M_j)\cap \bB(M_i \oplus M_j)({m_{ij}}+1)=\operatorname{ad}(M_i)^{m_{ij}}(M_j).$$ This establishes the simplicity.
\end{proof}
The following lemma is standard, we list here for completeness. 
\begin{lemma}\textup{[\citealp{reflection1}, Lemma 5.4, Lemma 5.5]}\label{lem5.4}\par 
 {(1)}
Suppose $M \in \mathcal{F}_{\theta}$ and $M$ admits the $i$-th reflection for each $i \in \mathbb{I}$.  We define $a_{ii}^M = 2$ for all $1\leq i \leq \theta$ and define $a_{ij}^M = -m_{ij}^M$. Then $(a_{ij}^M)_{i,j\in \mathbb{I}}$  is a generalized Cartan matrix.\par 
   {(2)} Suppose $M \cong N$ in $\mathcal{F}_{\theta}$, if  $M$  admits the $i$-th reflection for some $i \in \mathbb{I}$, so does $N$. Furthermore, $R_i(M)\cong R_i(N)$ and $a_{ij}^M=a_{ij}^N $ for each $j \in \mathbb{I}$.
\end{lemma}
Recall from Remark \ref{rmk2.13}, let $A$ be a Hopf algebra in $\HH$ with bijective antipode,
we have such a braided monoidal isomorphism 
$$ F: \AAC \cong {^{A\#H}_{A\#H}\mathcal{YD}}.$$ 
Now restricting to rational Yetter-Drinfeld modules, we denote the image of $F$ by ${^{A\#H}_{A\#H}\mathcal{YD}}_{\operatorname{rat}}$, which is a monoidal full subcategory of $^{A\#H}_{A\#H}\mathcal{YD}$.

From now on, we always fix a tuple $M = (M_1, \ldots, M_{\theta})\in \mathcal{F}_{\theta}$, with each component $M_i$ being simple for $i\in \mathbb{I}$.
 By Corollary \ref{cor3.8}, we have such a braided monoidal equivalence for each $i\in \mathbb{I}$.
\begin{equation}
 \Omega_i,: \BMICC_{\operatorname{rat}} \longrightarrow \BMIdCC_{\operatorname{rat}}.
\end{equation}
By abusing notation, we again denote  the following monoidal equivalence by $\Omega_i$:
$$ \AMI_{\operatorname{rat}}\cong \BMICC_{\operatorname{rat}} \rightarrow  \BMIdCC_{\operatorname{rat}}\cong  \AMID_{\operatorname{rat}}.$$

\begin{lemma}
    Fix $i \in \mathbb{I}$, the following are  equivalent:\par 
     {(1)} $M$  admits the $i$-th reflection.\par
     {(2)} $\bB(M)^{\operatorname{co}\bB(M_i)}$ belongs to   $\AMI_{\operatorname{rat}}$.  
\end{lemma}
\begin{proof}
Let $W=\bigoplus_{j\neq i}M_j$, and define $Q=\operatorname{ad}\bB(M_i)(W)$. By Lemma \ref{lem3.5}, we have the decomposition $Q=\bigoplus_{j \neq i}Q_j$, where each $Q_j=\operatorname{ad}\bB(M_i)(M_j)$ is simple in $\AMI$.

    (1) $\Rightarrow$ (2):     Applying  Theorem \ref{thm3.7} yields an isomorphism of Hopf algebras:
    $$ \bB(M)^{\operatorname{co}\bB(M_i)} \cong \bB(Q).$$
    Now by assumption, there is an integer $m$ such that $\operatorname{ad}(M_i)^{m}(W)=0$. This implies that $\operatorname{ad}(M_i)^{m}(Q)=0$. Thus $Q$ is rational. Lemma~\ref{lem5.2}(3) then implies that both
$\mathcal B(Q)$ and
$\bB(M)^{\operatorname{co}\bB(M_i)}$ are rational. \par 
$(2)\Rightarrow(1)$:
By Theorem \ref{thm3.7}, we have
$\bB(M)^{\operatorname{co}\bB(M_i)}\cong\bB(Q)$ in $\AMI$.
Since rational objects are closed under subobjects, the object
$Q$ is rational. Fix $j\neq i$. As $M_j$ is
finite-dimensional, by taking the maximum of the rationality bounds
of a basis of $M_j$, there exists $N_j\geq1$ such that
$\bB(M_i)(n)\rhd M_j=0$ for all $n\geq N_j$. Equivalently,
$\operatorname{ad}(M_i)^n(M_j)=0$ for all $n\geq N_j$.

Let
$m_{ij}^M=\max\{n\geq0\mid
\operatorname{ad}(M_i)^n(M_j)\neq0\}$. Then
\[
\operatorname{ad}(M_i)^{m_{ij}^M}(M_j)\neq0,
\qquad
\operatorname{ad}(M_i)^{m_{ij}^M+1}(M_j)=0.
\]
 Hence $M$ admits the $i$-th
reflection.
\end{proof}

     \begin{lemma}\label{lem5.12}
     Under the above assumptions on $M$, and suppose $M$ admits the $i$-th reflection  with $R_i(M)=(V_1,\ldots,M_i^*, \ldots,V_{\theta})$, then the following equalities hold.\par
     \begin{align*}
        \operatorname{ad}(\bB(M_i))(M_j)=\bigoplus_{n=0}^{m_{ij}}\operatorname{ad}(M_i)^n(M_j)&, \ \ \ \ V_j= \mathcal{F}_0(\operatorname{ad}\bB(M_i)(M_j)), \\
       \Omega_i(\operatorname{ad}(\bB(M_i)(M_j)))=\bigoplus_{n=0}^{m_{ij}}\operatorname{ad}(M_i^*)^n(V_j)&, \ \ \ \ M_j \cong \operatorname{ad}(M_i^*)^{m_{ij}}(V_j).
    \end{align*}
\end{lemma}
    \begin{proof}
    The first equality follows directly from the decomposition:  $$\operatorname{ad}\left(\bB(M_i)(M_j)\right)=\bigoplus_{n \geq 0} \operatorname{ad}(M_i)^n(M_j).$$
    \par For the second statement, the inclusion $V_j \subseteq \mathcal{F}_0(\operatorname{ad}\bB(M_i)(M_j))$  is immediate. Suppose, for contradiction, that the containment is proper. Then there exists a nonzero $X \in \operatorname{ad}(M_i)^{l}(M_j)$ for $0 \leq l \leq m_{ij}-1$ such that  $ X \in \mathcal{F}_0(\operatorname{ad}\bB(M_i)(M_j))$. The Yetter-Drinfeld submodule generated by $X$ is contained in $\bigoplus_{k=0}^{l}\operatorname{ad}(M_i)^{k}(M_j)$. However, since  $V_j \neq 0$, and  $\operatorname{ad}(\bB(M_i)(M_j))$ is simple in $\AMI$ by Lemma \ref{lem3.5}(ii), this leads to a contradiction. Hence, $$V_j=\mathcal{F}_0(\operatorname{ad}\bB(M_i)(M_j)).$$\par  
   For the third statement, we consider the objects in the category $\AMI(\Zgr)$ by setting $\operatorname{deg}(M_k)=1$ for $1\leq k \leq \theta$. Then $\operatorname{ad}(M_i)^l(M_j)$ has degree $l+1$ for $1\leq l \leq m_{ij}$. By Lemma \ref{lemOmega}, $\Omega_i(\operatorname{ad}(M_i)^l(M_j))=\operatorname{ad}(M_i)^l(M_j)$ is  $\mathbb{Z}$-graded with $\operatorname{deg}(\Omega_i(\operatorname{ad}(M_i)^l(M_j)))=-1-l$. Moreover,
    $$ \mathcal{F}^0\Omega_i(\operatorname{ad}\bB(M_i)(M_j))=\mathcal{F}_0(\operatorname{ad}\bB(M_i)(M_j))=\operatorname{ad}(M_i)^{m_{ij}}(M_j)=V_j.$$
  The object
$\Omega_i(\operatorname{ad}(\bB(M_i))(M_j))$ is simple in
$\AMID$, and its component of minimal degree is $V_j$.
Lemma~\ref{C-lem 3.10} therefore gives
    $$\Omega_i\left(\operatorname{ad}\left(\bB(M_i)(M_j)\right)\right)=\operatorname{ad}\left(\bB(M_i^*)(V_j)\right)=\bigoplus_{n=0}^{m_{ij}}\operatorname{ad}(M_i^*)^n(V_j).$$ In particular, we have $$M_j=\mathcal{F}^0\operatorname{ad}\bB(M_i)(M_j)=\mathcal{F}_0\operatorname{ad}\bB(M_i^*)(V_j)=\operatorname{ad}(M_i^*)^{m_{ij}}(V_j).$$
\end{proof}

The next theorem gives a natural explanation of reflections of tuples of Yetter-Drinfeld modules. The proof parallels \cite[Theorem~5.7]{reflection1}; we include the
details for completeness and the reader's convenience.
\begin{thm}\label{thm4.12}
  Let $M$ be a $\theta$-tuple with each $M_i$  finite-dimensional and simple. If $M$ admits the $i$-th reflection, then there is an isomorphism of Hopf algebras in $\HH$:
     \begin{equation}
         \Theta_i:\bB(R_i(M))\cong \Omega_i\left(\bB(M)^{\operatorname{co}\bB(M_i)}\right)\# \bB(M_i^*).
     \end{equation} 
     \end{thm}
\begin{proof}
    Let $N=\bigoplus_{j\neq i}M_j$, and define $Q=\operatorname{ad}\bB(M_i)(N)$. By Lemma \ref{lem3.5}(1), we have the decomposition $Q=\bigoplus_{j \neq i}Q_j$, where each $Q_j=\operatorname{ad}\bB(M_i)(M_j)$ is simple in $\AMI$.  Applying  Theorem \ref{thm3.7} yields an isomorphism of Hopf algebras:
    $$ \bB(M)^{\operatorname{co}\bB(M_i)} \cong \bB(Q).$$
    Now observe that the functor $\Omega_i$ sends Nichols algebras to Nichols algebras by Corollary \ref{cor4.5}. Therefore,
    $$\Omega_i\left(\bB(M)^{\operatorname{co}\bB(M_i)}\right)\cong \Omega_i(\bB(Q))\cong \bB(\Omega_i(Q)).$$
    Hence, in the category $\HH$,
    $$ \Omega_i\left(\bB(M)^{\operatorname{co}\bB(M_i)}\right) \# \bB(M_i^*)\cong \bB(\Omega_i(Q))\#\bB(M_i^*).$$
    Since $Q$ is a direct sum of simple objects, it is semisimple. The equivalence $\Omega_i$ preserves semisimplicity, so $\Omega_i(Q)$ is also semisimple in $\AMID(\Zgr)$ by Lemma \ref{lemOmega}(3).   Then we apply Theorem \ref{prop3.11} to obtain
    $$ \bB\left(\Omega_i\left(Q\right)\right)\#\bB(M_i^*)\cong \bB\left(\mathcal{F}^0(\Omega_i(Q))\oplus M_i^*\right).$$
   According to Lemma \ref{lem5.12} and Lemma \ref{lemOmega}(2),  $\mathcal{F}^0(\Omega_i(Q))=\mathcal{F}_0(Q)=\bigoplus_{j \neq i}V_j$.
    Thus 
    $$ \bB(\Omega_i(Q))\#\bB(M_i^*)\cong \bB\left(\bigoplus_{j \neq i}V_j \oplus M_i^*\right)\cong \bB(R_i(M)).$$
   Combining the above isomorphisms, we conclude that in $\HH$, there is an isomorphism of Hopf algebras:
   $$  \Theta_i:\bB(R_i(M))\cong \Omega_i\left(\bB(M)^{\operatorname{co}\bB(M_i)}\right)\# \bB(M_i^*).$$
\end{proof}

Having defined individual reflections, we now extend this notion to sequences of reflections, which will be important for our study of repeated reflections of  tuples.
\begin{definition}\label{def5.2}
Let $M\in\mathcal F_\theta$, with each $M_i$ simple for
$i\in\mathbb I$, and let
$l\in\mathbb N_0$ and $i_1,\ldots,i_l\in\mathbb I$.

 {(1)} We say that $M$ admits the reflection sequence
$(i_1,\ldots,i_l)$ if $l=0$, or if $M$ admits the $i_1$-th
reflection and $R_{i_1}(M)$ admits the reflection sequence
$(i_2,\ldots,i_l)$.

 {(2)} We say that $M$ admits all reflections if $M$ admits
the reflection sequence $(i_1,\ldots,i_l)$ for every
$l\in\mathbb N_0$ and every $i_1,\ldots,i_l\in\mathbb I$.
\end{definition}
For $M \in \mathcal{F}_{\theta}$ admitting all reflections, we denote 
$$\mathcal{F}_{\theta}(M)=\{R_{i_l}(...(R_{i_1}(M))...)\mid l\in \mathbb{N}_0,\ i_1,..,i_l \in \mathbb{I}\}.    $$
\par

As a consequence of  Theorem \ref{thm4.12}, 
we obtain the following corollary.
\begin{cor}\label{cor5.8}
Let $M=(M_1,\ldots,M_\theta)\in\mathcal F_\theta$, with each
$M_i$ simple.

 {(1)} If $M$ admits the $i$-th reflection, then $R_i(M)$
also admits the $i$-th reflection. Moreover,
\[
R_i^2(M)\cong M,
\qquad
a_{ij}^M=a_{ij}^{R_i(M)}
\]
for all $j\in\mathbb I$.

 {(2)} Suppose that $M$ admits all reflections. Let
$
\mathcal X
=
\{[P]\mid P\in\mathcal F_\theta(M)\},
$
and define
$
r_i([P])=[R_i(P)]
$
for $i\in\mathbb I$ and $[P]\in\mathcal X$. For each
$[P]\in\mathcal X$, let
$
A^{[P]}=(a_{ij}^{[P]})_{i,j\in\mathbb I}.
$
Then
\[
\mathcal G(M)
=
\bigl(\mathbb I,\mathcal X,r,
(A^{[P]})_{[P]\in\mathcal X}\bigr)
\]
is a semi-Cartan graph.
\end{cor}
\begin{proof}
  (1)  By Theorem \ref{thm4.12}, $R_i(M)=(V_1,...,M_i^*,...,V_{\theta}).$ 
    Note that 
    $$ \operatorname{ad}\bB(M_i^*)(V_j)=\Omega_i\left(\operatorname{ad}\bB(M_i)(M_j)\right).$$
It follows that 
    $$ \operatorname{ad}\bB(M_i^*)^{1-a_{ij}^M}(V_j)=0,  \ \ \operatorname{ad}\bB(M_i^*)^{-a_{ij}^M}(V_j) \cong M_j. $$  These imply that $R_i(M)$ admits the $i$-th reflection and $ a_{ij}^M=a_{ij}^{R_i(M)}$ for all $j \neq i$.
    It is clear that the $i$-th position of $R_i^2(M)$ is isomorphic to $M_i$. For $j \neq i$, 
    $$ \mathcal{F}_0\left(\operatorname{ad}\bB(M_i^*)(V_j)\right)=\mathcal{F}_0\left(\Omega_i\left(\operatorname{ad}\left(\bB\left(M_i\right)(M_j)\right)\right)\right)=\mathcal{F}^0(\operatorname{ad}(\bB(M_i))(M_j))=M_j.$$
    Therefore $R_i^2(M)_j \cong M_j$ for all $j \neq i$, and hence 
    $R_i^2(M)\cong M$.
    
    (2) If follows by definition, Lemma \ref{lem5.4} and  statement (1).
\end{proof}
We can then prove that reflection preserves the Gelfand-Kirillov dimensions.
\begin{prop}
Let $M=(M_1,\ldots,M_\theta)\in\mathcal F_\theta$, with each
$M_i$ simple, and suppose that $M$ admits the $i$-th reflection.

 {(1)} If $\bB(M)$ is finite-dimensional, then
\[
\dim\bB(M)=\dim\bB(R_i(M)).
\]

 {(2)} One has
\[
\operatorname{GKdim}\bB(M)
=
\operatorname{GKdim}\bB(R_i(M)).
\]
\end{prop}

\begin{proof}
(1)   For each $i\in \mathbb{I}$, there is a vector-space isomorphism
\[
\bB(M)\cong K_i^M\ot\bB(M_i).
\]. Therefore $K_i^M$ is finite-dimensional, and $\operatorname{ad}(\bB(M_i))(M_j)$ is finite-dimensional for each $j \in \mathbb{I}$. It follows by definition that $M$ admits the $i$-th reflection. Applying the  equivalence $\Omega_i$, we obtain  $\bB(R_i(M))\cong K_i^M \otimes \bB(M_i^*)$ as vector spaces. Since $\bB(M_i)$ is finite-dimensional, we have $\operatorname{dim}\bB(M_i)=\operatorname{dim}\bB(M_i^*)$. Therefore $\operatorname{dim}\bB(R_i(M))=\operatorname{dim}\bB(M)$ for each $i \in \mathbb{I}$.

(2) Since $M_i$ is finite-dimensional, the Nichols algebra $\mathcal{B}(M_i)$ is finitely generated in degree 1. Consequently, the Gelfand-Kirillov dimension is well-defined. 
Since $M$ admits the $i$-th reflection, the space
\[
K_i^M(1)
=
\bigoplus_{j\neq i}
\operatorname{ad}\bB(M_i)(M_j)
\]
is finite-dimensional. We have \( \mathcal{B}(M) \simeq \Omega_i(K_i^M) \# \bB(M_i) \)  by Theorem \ref{thm4.12}. Further, \( \bB(M)^{\operatorname{co}\bB(M_i)} \) is generated by \( \bB(M)^{\operatorname{co}\bB(M_i)}(1) \) by Theorem \ref{thm3.7}. Let
\(U=\mathbbm{k}1+K_i^M(1)\), \(\widetilde U=\mathbbm{k}1+\Omega_i(K_i^M(1))\),
\(X=\mathbbm{k}1+M_i\), and \(X^*=\mathbbm{k}1+M_i^*\).
Then \(V:=U+X\) and \(V':=\widetilde U+X^*\) are
finite-dimensional subspaces containing \(1\) which generate
\(K_i^M\#\bB(M_i)\) and \(\Omega_i(K_i^M)\#\bB(M_i^*)\), respectively.

For \(n\in\mathbb N_0\), define
\[
F_n=\sum_{r+s\leq n}U^r\#X^s,
\qquad
F_n'=\sum_{r+s\leq n}\widetilde U^r\#(X^*)^s.
\]
Since \(K_i^M(1)\) is stable under the \(\bB(M_i)\)-action, the
bosonization relations allow every product of elements of \(V\)
to be written with its \(K_i^M\)-factor on the left and its
\(\bB(M_i)\)-factor on the right, without increasing the total
word length. Conversely, every element of \(U^r\#X^s\) is a
linear combination of products of at most \(r+s\) elements of
\(V\). Hence \(V^n=F_n\). The same argument applied to
\(\Omega_i(K_i^M)\#\bB(M_i^*)\) gives \((V')^n=F_n'\).

Put $\Omega_i(K_i^M)(p)=\Omega_i(K(p))$. Since $K_i^M$ and
$\Omega_i(K_i^M)$ are generated in degree one, while $\bB(M_i)$ and
$\bB(M_i^*)$ are Nichols algebras, we have
\[
U^r=\bigoplus_{p=0}^rK_i^M(p),\qquad
\widetilde U^r=\bigoplus_{p=0}^r\Omega_i(K_i^M)(p),\qquad
X^s=\bigoplus_{q=0}^s\bB(M_i)(q),\qquad
(X^*)^s=\bigoplus_{q=0}^s\bB(M_i^*)(q).
\]
Therefore,
\[
F_n=\bigoplus_{p+q\leq n}K_i^M(p)\#\bB(M_i)(q),
\qquad
F_n'=\bigoplus_{p+q\leq n}
\Omega_i(K_i^M)(p)\#\bB(M_i^*)(q).
\]
The equivalence $\Omega_i$ and the non-degenerate graded pairing
between $\bB(M_i)$ and $\bB(M_i^*)$ give
\[
\dim\Omega_i(K_i^M)(p)=\dim K(p),\qquad
\dim\bB(M_i^*)(q)=\dim\bB(M_i)(q).
\]
Hence $\dim F_n=\dim F_n'$ for every $n\in\mathbb N_0$, and thus
$\dim V^n=\dim(V')^n$.
Consequently,
$
\operatorname{GKdim}\bigl(K_i^M\#\bB(M_i)\bigr)
 =\operatorname{GKdim}\bigl(\Omega_i(K_i^M)\#\bB(M_i^*)\bigr).
$
Finally, using
\(\bB(M)\simeq K_i^M\#\bB(M_i)\) and
\(\bB(R_i(M))\simeq\Omega_i(K_i^M)\#\bB(M_i^*)\), we obtain
\[
\operatorname{GKdim}\bB(M)
 =\operatorname{GKdim}\bB(R_i(M)).
\]
\end{proof}

\section{One-sided coideal subalgebras of braided Hopf algebras in $\HH$ }
 In this section, we study one-sided coideal subalgebras of braided Hopf algebras in $\HH$. We construct bijections between these subalgebras under reflections. These results provide the main algebraic tools to decompose Nichols algebras and verify the Cartan graph axioms in the next section.
\subsection{Basic properties of one-sided coideal subalgebras}
\begin{definition}
    Let $X$ be a bialgebra in $\mathcal{C}=\HH$. A left (respectively, right) coideal subalgebra $Y$ is a subobject of $X$, and an algebra in $\mathcal{C}$ such that the inclusion map $Y \subseteq X$ is an algebra morphism, and $\Delta(Y) \subseteq X \otimes Y$ (respectively, $\Delta(Y) \subseteq Y \otimes X$).
\end{definition}
Suppose $A$ and $P$ are Hopf algebras in $\mathcal{C}$. Let $\gamma: A \rightarrow P$ be an injective Hopf algebra morphism in $\mathcal{C}$, and assume there exists a surjective Hopf algebra morphism $\pi: P \rightarrow A$ such that $\pi \circ \gamma = \operatorname{id}_A$. The triple $(P, \pi, \gamma)$ is called a Hopf algebra triple.

 By Lemma \ref{lem2.12},  there exists a Hopf algebra $R = P^{\operatorname{co}A}$ in the braided monoidal category $^A_A\mathcal{YD}(\mathcal{C})$ such that $P \cong R \# A$. Denoting the comultiplication of $P$ by $\Delta_P: P \longrightarrow P \otimes P$, $p \mapsto p_1 \otimes p_2$, there is a well-defined map \begin{equation}
    \vartheta: P \rightarrow R, \ p \mapsto p_1\mathcal{S}_A(\pi(p_2)).
\end{equation}
Moreover, $R$ is a subalgebra of $P$ in $\mathcal{C}$ such that $\Delta_P(R)\subseteq P \ot R$. The $A$-action on $R$ is given by 
$ \operatorname{ad}: A \ot R \rightarrow R, $ satisfying 
\begin{align}
    \operatorname{ad}(a)(r)&=\vartheta(ar), \label{3.2}\\
\vartheta(pa)&=\vartheta(p)\varepsilon(a)\label{3.3}
\end{align}
for all $a \in A$, $r \in R$, and $p \in P$. The $A$-coaction $\delta: R \rightarrow A \otimes R$ is defined by $\delta(r) = \pi(r_1) \otimes r_2$.
The inverse of the Hopf algebra isomorphism $P \rightarrow R \# A$, $p \mapsto \vartheta(p_1) \otimes \pi(p_2)$, is given by the multiplication map. 

For our purposes, we introduce the following notation.
\begin{definition}
  (1)  Let
\[
{S}_r^+(P) = \{E \mid E \subseteq P \text{ right coideal subalgebra in } \mathcal{C}, A \subseteq E\},
\]
\[
S_r(P, X) = \{E \mid E \subseteq P \text{ right coideal subalgebra in } \mathcal{C}, E \subseteq X\},
\]
where \(X \subseteq P\) is a subobject in $\HH$, and
\[
T_r(P) = \{F \mid F \subseteq R \text{ subalgebra in } \mathcal{C},\
\Delta_R(F) \subseteq F \otimes R, F \subseteq R \ \text{is left} \  A-\text{submodule}\}.
\]

(2) Let
\[
S_l^+(P) = \{E \mid E \subseteq P \text{ left coideal subalgebra in } \mathcal{C}, A \subseteq E\},
\]
\[
S_l(P, X) = \{E \mid E \subseteq P \text{ left coideal subalgebra in } \mathcal{C}, E \subseteq X\},
\]
where \(X \subseteq P\) is a subobject in \(\HH\), and
\[
T_l(P) = \{F \mid F \subseteq R \text{ subalgebra in } \mathcal{C},
\
\Delta_R(F) \subseteq R \otimes F, F \subseteq R \ \text{is left } \  A-\text{subcomodule}\}.
\]
\end{definition}
\begin{lemma}\label{lem6.3}
    (1) For all \( E \in S_r^+(P) \), the multiplication map
\[
(E \cap R) \otimes A \to E
\]
is an isomorphism in \( \HH \).

(2) The map \( S_r^+(P) \to T_r(P) \), \( E \mapsto E \cap R \), is bijective with inverse given by $T_r(P)\rightarrow S^+_r(P),$ \( F \mapsto FA \).
\par 
(3) $S_l(P,R)=T_l(P).$
\end{lemma}
\begin{proof}
    (1) This multiplication map is well-defined since $A\subseteq E$. Now we consider the map
\[ \psi:
E \to (E \cap R) \otimes A, \quad x \mapsto \vartheta(x_{1}) \otimes \pi(x_{2}),
\]
For $x \in E \cap R$, $a \in A$, we  have
$$\psi(xa)=\vartheta(x_1a_1)\ot \pi (x_2a_2)=\vartheta(x_1)\varepsilon(a_1)\ot \varepsilon(x_2)a_2=\vartheta(x)\ot a=x \ot a.$$
Thus, $\psi$ is inverse to the multiplication map, making both maps bijective. The map
\[
(E \cap R) \otimes A \to E
\]
 is a morphism in $\HH$ because it is induced by the isomorphism $R \ot A \rightarrow P$.\par 
(2) We first show that both maps are well-defined. Let \( E \in S_r^+(P) \). Then \( E \cap R \subseteq R \) is a subalgebra in \( \mathcal{C} \). For all \( x \in E \cap R \), we have $\Delta_P(x)=x_1 \ot x_2 \in (E \ot P) \cap (P \ot R)$. 
Since $E$ is a right coideal subalgebra of $P$, we have $\Delta_{P}(e)=e_1\otimes e_2\in E\otimes P$ for any $e\in E$. By definition, $\vartheta(e) = e_1 \mathcal{S}_A(\pi(e_2))$. Moreover, $\mathcal S_A(\pi(e_2))\in A$. Since $E$ is an algebra
in $\HH$ and $A\subseteq E$, we have
\[
e_1\mathcal S_A(\pi(e_2))
\in E\cdot A
\subseteq E\cdot E
\subseteq E.
\]
Therefore $\vartheta(E)\subseteq E$. We can deduce that
\[
\Delta_R(x) = \vartheta(x_1) \ot x_2 \in (E \otimes P) \cap (R \otimes R) =(E\cap R) \ot R.
\]
Let $c$ be the braiding isomorphism in $\HH$. Since $c(A \ot E) \subseteq E \ot A$, we have
\[
\Delta_P(ax) = (a_1 \otimes a_2) (x_1 \otimes x_2) \in AE \otimes AP \subseteq E \otimes P
\]
for all \( x \in E \), \( a \in A \).
A direct computation shows that for $x \in E\cap R$,
\[
\operatorname{ad} (a)(x) = \vartheta(ax) = (ax)_1 S_A\pi((ax)_2) \in EA \cap R=E \cap R
\]
Thus, \( E \cap R \) is an \( A \)-submodule of \( R \). This shows the map \( S_r^+(P) \to T_r(P) \), \( E \mapsto E \cap R \), is well-defined.

Let \( F \in T_r(P) \) and $y \in F$. Since $F \subseteq R$ and
 $$ \Delta_P(y)=\vartheta(y_1)\pi(y_2)\ot y_3,$$
 we obtain $\Delta_P(F) \subseteq FA \otimes R$.
Thus, for all \( y \in F \), \( a \in A \), \( \Delta_P(ya) = \Delta_P(y)\Delta_P(a) \in FA \otimes P \). 

To show that \( FA \subseteq P \) is a subalgebra, it suffices to prove  \( AF \subseteq FA \). For all \( a\in A \), \( y \in F \),
\[
\Delta_P(ay) = \Delta_P(a)\Delta_P(y) \in A(FA) \otimes P.
\]
Since \( F \subseteq R \) is a \( A \)-submodule, we have
\[
ay = \vartheta((ay)_{1})\pi((ay)_{2}) \in \vartheta(A(FA))A \overset{(\ref{3.3})}{\subseteq} \vartheta((AF)A)A \subseteq \vartheta(AF)A\overset{(\ref{3.2})}{=}\operatorname{ad}(A)(F)A \subseteq FA.
\]
Therefore 
 $FA$ is a subalgebra of $P$ and the map $T_r(P)\rightarrow S^+_r(P),$ \( F \mapsto FA \) is well-defined.

 If \( E \in S_r^+(P) \), the following composition is the identity map by (1): 
 $$  S_r^+(P) \rightarrow{} T_r(P) \rightarrow{} S_r^+(P),  \  E \mapsto E\cap R\mapsto (E\cap R)A.  $$
If $F \in T_r(P)$, we have the composition:
$$ T_r(P) \rightarrow{} S_r^+(P) \rightarrow{} T_r(P), \ F \mapsto FA \mapsto (FA) \cap R.$$
Since $F \subseteq R$, the multiplication maps \( (FA \cap R) \otimes A \to (FA \cap R)A = FA \) and \( F \otimes A \to FA \) are bijective. This yields \( F = FA\cap R \). It follows that the two maps are mutually inverse bijections.\par 
 (3) Let $E \in S_l(P,R)$, then $E \subseteq R$ is a subalgebra in $\mathcal{C}$ and
$\Delta_R(x)=\vartheta(x_1)\ot x_2 \in R\ot E$ for $x \in E$. Moreover, $\delta(x)=\pi(x_1)\ot x_2 \in A \ot E$, hence $E$ is a $A$-subcomodule. This shows $E \in T_l(P)$.
 
Conversely, let $F \in T_l(P),$ then $F \subseteq P $ is a subalgebra automatically. For $x \in F$, 
 $\Delta_P(x)=\vartheta(x_1)\pi(x_2)\ot x_3\ \in P \ot F$, because $\Delta_R(x) \in R \otimes F$. Hence, $F \in S_l(P, R)$.
\end{proof}

\subsection{Comparing one-sided coideal subalgebras via the isomorphism $T$}

Let $K$ be a Hopf algebra in $\BBCR$ with a bijective antipode. We obtain a Hopf algebra triple $(P:=K\#B, \pi, \gamma)$ in $\mathcal{C}$, where $\pi: P \rightarrow B$ and $\gamma:B \rightarrow P$. Let $\Omega(K)$ denote the corresponding Hopf algebra in $\AACR$ obtained via the braided monoidal equivalence $\Omega$. This yields another Hopf algebra triple $(Q:=\Omega(K)\#A, \overline{\pi},\overline{\gamma})$ in $\mathcal{C}$, where $$\overline{\pi}: Q \rightarrow A, \overline{\gamma}:A \rightarrow Q.$$ 

We also obtain a Hopf algebra triple $(Q^{\operatorname{cop}},\overline{\pi},\overline{\gamma})$ in the reverse braided category $\overline{\mathcal{C}}$. The antipodes of  $\Omega(K)$ and $Q$ are bijective because the antipode of $K$ is bijective.\par
By definition, $\Omega(K)=Q^{\operatorname{co}A}$, which consists of the right coinvariant elements with respect to the projection $\overline{\pi}$. Meanwhile, we define 
$$ 
L:=(Q^{\operatorname{cop}})^{\operatorname{co}\overline{\pi}}.
$$
Then $L$  is an  object in $ \AACOPCR$. The following proposition establishes the relation between $L$ and $K$.
\begin{prop}\label{prop3.8}
    With the above notation, 
there is an isomorphism of objects in $\mathcal{C}$
\begin{equation}
    T:L \longrightarrow K, \ x \mapsto \mathcal{S}_K^{-1}\mathcal{S}_{Q}(x).
\end{equation}
    In fact, the morphism $T$  is an isomorphism of Hopf algebras in $\AACOPCR$  \begin{equation}
        T:L \longrightarrow \Gamma(K^{\operatorname{cop}}).
    \end{equation} 
    
\end{prop}
\begin{proof}
By Lemma \ref{thm5.10} (2) and (3), $\Omega(K)^{\operatorname{cop}} \in \overline{\AARC}$. Then we have $$F_{lr}\circ F^{-1}_{rl}(\Omega(K)^{\operatorname{cop}})\in \AACOPCR.$$
 
 By [\citealp{rootsys}, Theorem 3.10.6], there is an isomorphism of objects $T:L \rightarrow K$ in $\mathcal{C}$. This map is an isomorphism of Hopf algebras in $\AACOPCR$:
    $$ T: L \longrightarrow F_{lr}\circ F^{-1}_{rl}(\Omega(K)^{\operatorname{cop}})$$
    satisfying $\overline{\gamma}\circ T^{-1}=\mathcal{S}_{Q}^{-1} \circ \overline{\gamma} \circ \mathcal{S}_{\Omega(K)}$.
\par Since $\Omega$ is a braided monoidal functor, we have $\Omega(K)^{\operatorname{cop}}\cong\Omega(K^{\operatorname{cop}})$. Hence 
$$ F_{lr}\circ F^{-1}_{rl}(\Omega(K)^{\operatorname{cop}})\cong F_{lr}\circ F^{-1}_{rl}(\Omega(K^{\operatorname{cop}}))=\Gamma(K^{\operatorname{cop}}).$$
Therefore, $T:L \rightarrow \Gamma(K^{\operatorname{cop}})$ is a Hopf algebra isomorphism.  For any $x \in K$, we compute $T^{-1}(x)=\mathcal{S}_{Q}^{-1} \circ \overline{\gamma} \circ \mathcal{S}_{\Omega(K)}(x)=\mathcal{S}_{Q}^{-1} \circ \overline{\gamma} \circ \mathcal{S}_K(x)=\mathcal{S}_{Q}^{-1}\circ \mathcal{S}_K(x)$ using the identification $\mathcal{S}_{\Omega(K)}=\mathcal{S}_K$. Thus, $T(x)=\mathcal{S}_{K}^{-1}\mathcal{S}_{Q}(x)$ for all $x \in L$.
\end{proof}

Let $(P:=K\#B,\pi,\gamma)$, $(Q:=\Omega(K)\#A, \overline{\pi},\overline{\gamma})$ and $L$ be defined as above.
 Define 
 $$ \overline{K}:=^{\operatorname{co}\overline{\pi}}Q.$$
to be the space of left coinvariant elements with respect to the projection $\overline{\pi}$. The isomorphism $T$ induces bijection between the one-sided coideal subalgebras of $P$ and $Q$. 
\begin{prop}\label{prop3.9}
(1)    With the above notation, there is a bijection 
    $$ S_l(Q^{\operatorname{cop}},L) \longrightarrow S_r^+(P), \ E \mapsto T(E)B$$ with inverse given by $E \mapsto T^{-1}(E\cap K)$. Furthermore, the multiplication map 
 $$ T(E) \ot B \rightarrow{} T(E)B $$ is bijective for all $E \in S_l(Q^{\operatorname{cop}},L)$.\par
    (2) There is a bijection 
    $$ S_l(P,K)\longrightarrow S_l^+(Q), \  E \mapsto T^{-1}(E)A$$
 with inverse given by $E \mapsto T(E\cap L)$. Furthermore,
the multiplication map 
 $$ T^{-1}(E) \ot A \rightarrow{} T^{-1}(E)A $$ is bijective for all $E \in S_l(P,K)$.\par
\end{prop}
\begin{proof}
(1) By definition, the object $\overline{K} \cong L$ in $\mathcal{C}$, hence 
$$ S_r(Q,\overline{K})=S_l(Q^{\operatorname{cop}},L).$$
By Lemma \ref{lem6.3} (3), $S_l(Q^{\operatorname{cop}},L)=T_l(Q^{\operatorname{cop}}).$ We now establish the bijection between $T_l(Q^{\operatorname{cop}})$ and $S_r^+(P)$.\par 
By definition,
\begin{align*}
T_l(Q^{\operatorname{cop}})=\{E\mid E\subseteq L \  \text{ subalgebra in } \overline{\mathcal{C}},
\
\Delta_L(E) \subseteq L \otimes E, E  \ \text{is} \  A^{\operatorname{cop}}-\text{subcomodule}\}.
\end{align*}
We apply the Hopf algebra isomorphism $T$, which is an $A^{\operatorname{cop}}$-colinear map in $\overline{\mathcal{C}}$. For any $ E \in  T_l(Q^{\operatorname{cop}}) $,
 $T(E)$ is a $A^{\operatorname{cop}}$-subcomodule of $\Gamma(K^{\operatorname{cop}})$ in $\overline{\mathcal{C}}$ and is a subalgebra of $K$ in $\mathcal{C}$. Since $D_2: {^{A^{\operatorname{cop}}}\overline{\mathcal{C}}}\rightarrow {_B\mathcal{C}_{\operatorname{rat}}}$ is a braided monoidal equivalence by  Proposition \ref{prop5.9}, $T(E)$ is a $B$-submodule of $K$.
  Finally, 
$$ \overline{c}^{\BBC}_{K,K}\circ \Delta_K(T(E))=\Delta_{\Gamma(K^{\operatorname{cop}})}(T(E))\subseteq K\ot T(E).$$
Here $\overline{c}^{\BBC}_{K,K}$ denotes  the reverse braiding of $c^{\BBC}_{K,K}$.
Since $K$ has a bijective antipode, $\overline{c}^{\BBC}_{K,K}$ is an isomorphism by [\citealp{rootsys}, Proposition 3.4.8]. Now 
 $$ \overline{c}^{\BBC}_{K,K}\circ \Delta_K(T(E))\subseteq K\ot T(E) \Leftrightarrow \Delta_K(T(E))\subseteq T(E)\ot K.$$
Consequently, $T(E) \in T_r(P).$ The inverse mapping can be established similarly.
 Therefore, the isomorphism $T$ establishes a bijection between $ T_l(Q^{\operatorname{cop}})$ and $T_r(P)$. By Lemma \ref{lem6.3}(2), the map 
 $$ T_r(P)\longrightarrow S_r^+(P), \ F \mapsto FB$$ is bijective. Therefore, the map $S_l(Q^{\operatorname{cop}},L) \longrightarrow S_r^+(P), \ E \mapsto T(E)B$ is bijective with inverse  $E \mapsto T^{-1}(E\cap K)$.

 The bijectivity of the multiplication map $T(E) \otimes B \rightarrow T(E)B$ for all $E \in S_l(Q^{\operatorname{cop}}, L)$ follows directly from Lemma  \ref{lem6.3} (1).
 \par
(2) By a similar argument, there is a bijection $$T_r(Q^{\operatorname{cop}}) \longrightarrow T_l(P), \ E \mapsto T(E).$$ According to Lemma \ref{lem6.3} (2), (3), there is a bijective map
$T_r(Q^{\operatorname{cop}}) \rightarrow S^+_r(Q^{\operatorname{cop}})$ given by $F \mapsto FA$, and $T_l(P)=S_l(P,K)$. Since $S^+_r(Q^{\operatorname{cop}})=S_l^+(Q)$, the map $S_l(P,K)\rightarrow S_l^+(Q)$ is bijective. 

For any $E \in S_l(P,K)$,
since $T^{-1}(E) \in T_r(Q^{\operatorname{cop}})$, the multiplication map 
 $$ T^{-1}(E) \ot A \rightarrow{} T^{-1}(E)A $$ is bijective  by Lemma \ref{lem6.3}(1).
\end{proof}
\subsection{One-sided coideal subalgebras of Nichols algebras}
In this subsection, we fix a tuple $M=(M_1, \dots, M_\theta) \in \mathcal{F}_\theta$, where each $M_i$ is a finite-dimensional simple Yetter-Drinfeld module in $\HH$. We endow $\bB(M)$ with an $\mathbb{N}_0^{\theta}$-grading by setting $\operatorname{deg}(M_i)=\alpha_i$.
Consider the following maps:
$$ \pi_i: \bB(M) \rightarrow \bB(M_i),\ \gamma_i:\bB(M_i)\rightarrow \bB(M).$$
such that $\pi_i\circ \gamma_i=\operatorname{id}_{\bB(M_i)}$.
Let
$$ K_i^M=\bB(M)^{\operatorname{co}\bB(M_i)}, \ L_i^M={^{\operatorname{co}\bB(M_i)}\bB(M)}$$
be the right and left coinvariant spaces with respect to $\pi_i$. A direct computation shows that 
\begin{equation}\label{4.4}
    K_i^M= \mathcal{S}_{\bB(M)}(L_i^M).
\end{equation}
For any $x \in {\bB(M)},$
\begin{align*}
( \pi_i  \ot \operatorname{id})(\Delta(\mathcal{S}_{\bB(M)})(x))&= (\pi_i  \ot \operatorname{id})\circ (\mathcal{S}_{\bB(M)}\ot \mathcal{S}_{\bB(M)})\circ c_{{\bB(M)},{\bB(M)}}(\Delta(x))\\
&=c_{{\bB(M)},\bB(M_i)}\circ (\mathcal{S}_{\bB(M)}\ot \mathcal{S}_{\bB(M_i)})\circ (\operatorname{id}\ot \pi_i )(\Delta(x))
\end{align*}
Hence, the condition $x \in K_i^M$ is equivalent to $\mathcal{S}_{\bB(M)}(x) \in L_i^M$.

\begin{lemma}\label{lem4.10}
For each $i \in \mathbb{I}$. \par (1) Let \( E \subseteq {\bB(M)} \) be an $\mathbb{N}_0^\theta$-graded right coideal subalgebra in \( {}_H^H \mathcal{YD} \). Then
$$  E \subseteq L^{M}_i  \Longleftrightarrow  M_i \nsubseteq E .$$\par 
(2) Let \( F \subseteq {\bB(M)} \) be an $\mathbb{N}_0^\theta$-graded left coideal subalgebra in \( {}_H^H \mathcal{YD} \). Then
$$  F \subseteq K^{M}_i  \Longleftrightarrow  M_i \nsubseteq F .$$\par 
\end{lemma}
\begin{proof}
We only prove (1), as (2) is dual. Assume $E\subseteq L_{i}^{M}$ and $M_{i}\subseteq E$; it follows that $M_{i}\subseteq L_{i}^{M}$, which forces $M_{i}=0$, yielding a contradiction.
    
    Now we assume $M_i \nsubseteq E$.  Since $M_{i}$ is simple, $M_{i}\cap E$ must be zero. Since $E$ is an $\mathbb{N}_0^\theta$-graded right coideal subalgebra and $\pi_i$ is  $\mathbb{N}_0^\theta$-graded morphism, we deduce that $\pi_i(E)$ is a right coideal subalgebra of $\bB(M_i)$.
   Assume $\pi_{i}(E)\ne\mathbbm{k}1$. Because $M_{i}$ is simple, $\pi_{i}(E)\cap P(\mathcal{B}(M_{i}))\ne 0$. Furthermore, the simplicity of $M_{i}$ and the fact that $P(\mathcal{B}(M_{i}))=M_{i}$ jointly force $M_{i}\subseteq\pi_{i}(E)$. This contradicts  $M_i \cap E=0$. Therefore $\pi_i(E)=\mathbbm{k}1$. Now for any $x \in E$, 
    $$ (\pi_i \ot \operatorname{id})\Delta(x)=\pi_i(x_1)\ot x_2=1 \ot x.$$
    Hence $E \subseteq L_i^M$.
\end{proof}
\begin{definition}
   For each $i \in \mathbb{I}$. \par 
  (1) We define
\begin{align*}
    \mathcal{K}({M}) &= \{E \mid E \subseteq {\bB(M)} \text{ is an } \mathbb{N}_0^{\theta}\text{-graded right coideal subalgebra in } {}_H^H\mathcal{YD}\},\\
\mathcal{K}_i^+({M}) &= \{E \mid E \in \mathcal{K}({M}), \ M_i \subseteq E\},\\\mathcal{K}_i^-({M}) &= \{E \mid E \in \mathcal{K}({M}), \ M_i \nsubseteq E\}.
\end{align*}
(2) Dually, we define 
\begin{align*}
    \mathcal{L}({M}) &= \{F\mid F \subseteq \bB(M) \text{ is an } \mathbb{N}_0^{\theta}\text{-graded left coideal subalgebra in } {}_H^H\mathcal{YD}\},\\
\mathcal{L}_i^+({M}) &= \{F \mid F \in \mathcal{L}({M}), \ M_i \subseteq F\},\\\mathcal{L}_i^-({M}) &= \{F \mid F \in \mathcal{L}({M}), \ M_i \nsubseteq F\}.
\end{align*}

\end{definition}
From now on, we assume that $M$ admits the $i$-th reflection. Recall that
\[
R_i(M)=\bigl(R_i(M)_1,\ldots,R_i(M)_\theta\bigr),
\]
where
\[
R_i(M)_j=
\begin{cases}
M_i^*, & \text{if } j=i,\\
\operatorname{ad}(M_i)^{m_{ij}^M}(M_j), & \text{if } j\neq i.
\end{cases}
\]
For simplicity, we denote $R_i(M_j)=\widetilde{M_j}$ for each $j$.
\begin{lemma}
   (1)  \[
\operatorname{ad}(M_i^*)^{n}(\widetilde{M_j}) = 
\begin{cases} 
0, & \text{if } n > m_{ij}^M, \\
\operatorname{ad}(M_i)^{m_{ij}^M-n}(M_j) & \text{if } 0\leq n \leq m_{ij}^M.
\end{cases}
\]\par 
(2) The Nichols algebra $\bB(R_i(M))\cong \Omega(K_i^M) \# \mathcal{B}(M_i^*)$
is an \( \mathbb{N}_0^\theta \)-graded Hopf algebra in \( {}_H^H\mathcal{YD} \) with
\[
\deg(x \otimes y) = s_i^M \bigl( \deg^{\bB(M)}(x) + \deg(y) \bigr)
\]
for all homogeneous elements \( x \in K_i^M \) and \( y \in \mathcal{B}(M_i^*) \), where \( \mathcal{B}(M_i^*) \) is a \( \mathbb{Z}^\theta \)-graded algebra with \( \deg(M_i^*) = -\alpha_i \), and \( \deg^{\bB(M)} \) denotes the degree in \( \bB(M) \).
\end{lemma}

\begin{proof}
    (1) The first statement follows from   Lemma \ref{lem5.12}.\par 
    (2) According to  Lemma \ref{lemOmega}, $\bB(M_i^*)$ is a $\mathbb{Z}^{\theta}$-graded Hopf algebra with $\operatorname{deg}(M_i^*)=-\alpha_i$ and $\Omega(K_i^M)$ is an $\mathbb{Z}^{\theta}$-graded Hopf algebra with $\Omega(K_i^M)(n)=K_i^M(-n)$ for $n \in \mathbb{Z}^{\theta}$. Under this grading, $\Omega_i(K_i^{M}) \# \bB(M_i^*)$ is an $\mathbb{Z}^\theta$-graded Hopf algebra in $\HH$.  We equip $\mathcal{B}(R_{i}(M))$ with a new grading by shifting the degree via $s_{i}^{M}$. Explicitly, for $x \in K_i^M$ and $y \in \bB(M_i^*)$, we define 
$$ \operatorname{deg}(x\#y)=s_i^M(\operatorname{deg}^{\bB(M)}(x)+\operatorname{deg}(y)).$$
Hence $\operatorname{deg}(\widetilde{M_i})=s_i^M(-\alpha_i)=\alpha_i$, and for all $j \neq i$,
$$ \operatorname{deg}(\widetilde{M_j})=s_i^M(\alpha_j-a_{ij}^M\alpha_i)=\alpha_j.$$ Hence $\bB(R_i(M))$ is an $\mathbb{N}^{\theta}_0$-graded Hopf algebra with $\operatorname{deg}(\widetilde{M_j})=\alpha_j$ for all $j \in \mathbb{I}$ under this new grading.

\end{proof}
We consider the action of the isomorphism $T$ mentioned in Proposition \ref{prop3.8}. The Hopf algebra isomorphism $T$ restricts to an algebra isomorphism in $\HH$:
\begin{equation}
    T_i^{M}: L_i^{R_i(M)}={^{\operatorname{co}\bB(M_i^*)}(\Omega_i(K_i^M)\# \bB(M_i^*))}\xrightarrow{\sim} K_i^M.
\end{equation}
\begin{lemma}\label{lem4.12}
   (1) With the above notation, for all $j \neq i \in \mathbb{I}$, $0 \leq n \leq m_{ij}^M$, and $x \in \operatorname{ad}(M_i^*)^n(\widetilde{M_j})$, 
   \begin{align*}
       T_i^M(\mathcal{S}_{\bB(R_i(M))}^{-1}(x))&=-x,\\
T_i^M(\mathcal{S}_{\bB(R_i(M))}^{-1}(\operatorname{ad}(M_i^*)^n(\widetilde{M_j})))&=\operatorname{ad}(M_i)^{m_{ij}^M-n}(M_j).
   \end{align*}\par
   (2) $T_i^M$ is an isomorphism of $\mathbb{N}_0^{\theta}$-graded objects. In particular 
   $$ (T_i^M(L_i^{R_i(M)}(\alpha)))=K_i^M(s_i^{R_i(M)}(\alpha)),$$
   where $\alpha \in \mathbb{N}_0^\theta$.
\end{lemma}
\begin{proof}
    (1) 
     We have 
    \[
\operatorname{ad}(M_i^*)^{n}(\widetilde{M_j}) = 
\begin{cases} 
0, & \text{if } n > m_{ij}^M, \\
\operatorname{ad}(M_i)^{m_{ij}^M-n}(M_j) & \text{if } 0\leq n \leq m_{ij}^M.
\end{cases}
\]
    Since $x \in \operatorname{ad}(M_i^*)^n(\widetilde{M_j})=\operatorname{ad}(M_i)^{m^{M}_{ij}-n}(M_j)\subseteq \operatorname{ad}(\bB(M_i))(M_j)$,  we deduce that $x$ is primitive in $K_i^{M}$ by  Theorem \ref{thm3.7}.
    
   Set $y=\mathcal S_{\bB(R_i(M))}^{-1}(x)$. Then
$y\in L_i^{R_i(M)}$ by (\ref{4.4}). By the definition of $T$,
\[
T_i^M(y)
=
\mathcal S_{K_i^M}^{-1}
\mathcal S_{\bB(R_i(M))}(y)
=
\mathcal S_{K_i^M}^{-1}(x)
=
-x.
\]
    The second equation holds because $T_{i}^{M}$ is an isomorphism.\par 
    (2) (2) Let $x\in\operatorname{ad}(M_i^*)^n(\widetilde{M_j})$.
Under the $\mathbb{N}_0^\theta$-grading of $\bB(R_i(M))$, we have
$\operatorname{deg}(x)=n\alpha_i+\alpha_j$. Since
$\mathcal S_{\bB(R_i(M))}$ preserves this grading,
$\operatorname{deg}(y)=n\alpha_i+\alpha_j$. By (1),
\[
\operatorname{deg}(T_i^M(y))
=(m_{ij}^M-n)\alpha_i+\alpha_j
=s_i^M(\operatorname{deg}(y))
\in\mathbb N_0^\theta.
\]
Since $M$ admits the $i$-th reflection, $R_i(M)$ does as well, and
$s_i^M=s_i^{R_i(M)}$ by Corollary \ref{cor5.8}. Hence, for every
$\alpha\in\mathbb N_0^\theta$,
\[
\operatorname{deg}\bigl(T_i^M(L_i^{R_i(M)}(\alpha))\bigr)
=s_i^{R_i(M)}(\alpha).
\]
Therefore,
\[
T_i^M(L_i^{R_i(M)}(\alpha))
=K_i^M(s_i^{R_i(M)}(\alpha)).
\]
\end{proof}
\begin{prop}\label{prop4.13}
     Assume $M$ admits the $i$-th
reflection for some $i \in \mathbb{I}$.\par 
(1) The map  
    \[
    t_i^M : \mathcal{K}_i^- (R_i(M)) \longrightarrow \mathcal{K}_i^+ (M), \quad E \longmapsto T_i^{M}(E){\bB(M_i)},
    \]  
    is bijective with inverse given by \( E \longmapsto (T_i^{M})^{-1}(E \cap K_i^{M}) \). Furthermore, the  following multiplication map is bijective for all $E \in \mathcal{K}_i^- (R_i(M))$.
    $$ T_i^{M}(E) \ot {\bB(M_i)} \longrightarrow T_i^{M}(E){\bB(M_i)}. $$\par 
    (2) The map  
    \[
    \mathbbm{t}_i^M : \mathcal{L}_i^+ (R_i(M)) \longrightarrow \mathcal{L}_i^- (M), \quad F \longmapsto T_i^{M}(F\cap L_i^{R_i(M)}),
    \]  
    is bijective with inverse given by \( F \longmapsto (T_i^{M})^{-1}(F)\bB(M_i^*) \). 
    Furthermore, the  following multiplication map is bijective for all $F \in \mathcal{L}_i^- (M)$,
    $$  (T_i^{M})^{-1}(F) \ot \bB(M_i^*) \longrightarrow (T_i^{M})^{-1}(F)\bB(M_i^*).
    $$
\end{prop}
\begin{proof}
   (1) We apply Proposition \ref{prop3.9}(1) by setting $P=\bB(M)$, $Q=\bB(R_i(M))$, $\overline{K}=L_i^{R_i(M)}$, $K=K_i^M$ and $B=\bB(M_i)$. We obtain a bijection:
 $$ S_r(\bB(R_i(M)),L_i^{R_i(M)}) \longrightarrow S_r^+(\bB(M)), \ E \mapsto T_i^M(E)\bB(M_i)$$ with inverse given by $E \mapsto (T_i^M)^{-1}(E\cap K_i^M)$.
 Since $T_i^M$ preserves $\mathbb{N}_0^\theta$-grading, the bijection still holds when restricting to $\mathbb{N}_0^\theta$-graded right coideal subalgebras. By Lemma \ref{lem4.10}, $  E \subseteq L^{R_i(M)}_i  \Longleftrightarrow  M_i^* \nsubseteq E $. Thus, the bijection restricts to the desired bijection
 $$ t_i^M: \mathcal{K}_i^- (R_i(M)) \longrightarrow \mathcal{K}_i^+ (M).$$
 The bijectivity of the multiplication map     $$ T_i^{M}(E) \ot {\bB(M_i)} \longrightarrow T_i^{M}(E){\bB(M_i)}.$$ for all $E \in \mathcal{K}_i^- (R_i(M))$ follows from Proposition \ref{prop3.9}(1).
 \par 
 (2) It follows from Proposition \ref{prop3.9}(2) and similar argument to (1).
\end{proof}

\section{The associated Cartan graph}
In this section, we prove that the associated semi-Cartan graph
$\mathcal G(M)$ is a Cartan graph by verifying the equivalent axioms
 {(CG3')} and  {(CG4')}.
 \subsection{Proof of (CG3')}

 Let $l\in\mathbb N_0$ and
$i_1,\ldots,i_l\in\mathbb I$, and assume that $M$ admits the
reflection sequence $(i_1,\ldots,i_l)$. We construct the simple
Yetter--Drinfeld modules $M_{\beta_k}$ associated with the
corresponding roots and the right coideal subalgebra
$E^M(i_1,\ldots,i_l)$.
 
For simplicity, we  write 
$$ R_{(i_1,...,i_k)}(M)=R_{i_k}(\cdots R_{i_1}(M))$$ inductively for $1 \leq k \leq l$. For clarity, we denote $R_{()}=M$. Now we consider the following two maps:
\begin{align*}
     T_{i_1}^M: L_{i_1}^{R_{(i_1)}(M)}&\xrightarrow{\sim} K_{i_1}^M,\\
     T_{i_2}^{R_{(i_1)}(M)}:L_{i_2}^{R_{(i_1,i_2)}(M)} &\xrightarrow{\sim}K_{i_2}^{R_{(i_1)}(M)}.
\end{align*}
To compose these maps, the domain must be restricted to
 $$(T_{i_2}^{R_{(i_1)}(M)})^{-1}(K_{i_2}^{R_{(i_1)}(M)}\cap L_{i_1}^{R_{(i_1)}(M)}).$$
Therefore  $T_{i_1}^M \circ T_{i_2}^{R_{(i_1)}(M)}$ is well-defined. We apply similar restrictions when considering the composition of $t_i^M$ and $\mathbbm{t}_i^M$. 
\begin{rmk}\upshape\label{rmk4.15}
With above assumptions, for $1 \leq k \leq l$, we define inductively
\begin{align*}
L_{(i_1, \ldots, i_k)}^{M} &= \bigl( T_{i_k}^{R_{(i_1, \ldots, i_{k-1})}(M)} \bigr)^{-1} 
\Bigl( K_{i_k}^{R_{(i_1, \ldots, i_{k-1})}(M)} \cap L_{(i_1, \ldots, i_{k-1})}^{M} \Bigr), \\[4pt]
T_{(i_1, \ldots, i_k)}^{M} &= T_{i_1}^{M} \circ T_{i_2}^{R_{(i_1)}(M)} \circ \cdots \circ  T_{i_k}^{R_{(i_1,...,i_{k-1})}(M)} : L_{(i_1, \ldots, i_k)}^{M} \longrightarrow K_{i_1}^M, \\
\mathcal{K}_{(i_1, \ldots, i_k)}^{-} \bigl( R_{(i_1, \ldots, i_k)}(M) \bigr) &=
\bigl( t_{i_k}^{R_{(i_1, \ldots, i_{k-1})}(M)} \bigr)^{-1} \Bigl( 
\mathcal{K}_{i_k}^{+} \bigl( R_{(i_1, \ldots, i_{k-1})}(M) \bigr) \cap 
\mathcal{K}_{(i_1, \ldots, i_{k-1})}^{-} \bigl( R_{(i_1, \ldots, i_{k-1})}(M) \bigr) \Bigr), \\
t_{(i_1, \ldots, i_k)}^{M} &= t_{i_1}^{M} \circ \cdots \circ  t_{i_k}^{R_{(i_1,...,i_{k-1})}(M)} : 
\mathcal{K}_{(i_1, \ldots, i_k)}^{-} \bigl( R_{(i_1, \ldots, i_k)}(M) \bigr) \longrightarrow \mathcal{K}^+_{i_1}(M),\\
\mathcal{L}_{(i_1, \ldots, i_k)}^{+} \bigl( R_{(i_1, \ldots, i_k)}(M) \bigr) &=
\bigl( \mathbbm{t}_{i_k}^{R_{(i_1, \ldots, i_{k-1})}(M)} \bigr)^{-1} \Bigl( 
\mathcal{L}_{i_k}^{-} \bigl( R_{(i_1, \ldots, i_{k-1})}(M) \bigr) \cap 
\mathcal{L}_{(i_1, \ldots, i_{k-1})}^{+} \bigl( R_{(i_1, \ldots, i_{k-1})}(M) \bigr) \Bigr),\\
\mathbbm{t}_{(i_1, \ldots, i_k)}^{M} &= \mathbbm{t}_{i_1}^{M} \circ \cdots \circ  \mathbbm{t}_{i_k}^{R_{(i_1,...,i_{k-1})}(M)} : 
\mathcal{L}_{(i_1, \ldots, i_k)}^{+} \bigl( R_{(i_1, \ldots, i_k)}(M) \bigr) \longrightarrow \mathcal{L}^-_{i_1}(M).
\end{align*}
By convention, we set $$ L_{()}^M=\bB(M), \ \mathcal{K}_{()
}^{-}(M)=\mathcal{K}(M), T^M_{()}=\operatorname{id}, \ t^M_{()}=\operatorname{id}, \ \mathcal{L}^+_{()}=\mathcal{L}(M),\ \mathbbm{t}^M_{()}=\operatorname{id}.$$ 
\end{rmk}
From now on, we assume $M$ admits all reflections. Let the reflection sequence {$(i_1,...,i_l)$ be $[M]$-reduced} in the semi-Cartan graph $\mathcal{G}(M)$. Note that in this case, the reflection sequence $(i_k,...,i_l)$ is $[R_{(i_1,...,i_{k-1})}(M)]$-reduced in $\mathcal{G}(R_{(i_1,...,i_{k-1})}(M))$ for all $1\leq k \leq l$ by Remark \ref{rmk2.7} (2). For all \( 1 \leq k \leq l \), let
\[
\beta_k = \text{id}_{[M]} s_{i_1} \cdots s_{i_{k-1}} (\alpha_{i_k}),
\]

The following observation is immediate and serves as the base case for our induction. Let $l=1$, then $\beta_1=\alpha_{i_1}$.
\begin{itemize}
    \item $\bbk1 \in \mathcal{K}_{i_1}^{-}(R_{i_1}(M))$, it follows from the fact that $M^*_{i_1} \nsubseteq\bbk1$. Then we have $E^M(i_1)=t_{i_1}^M(\bbk1)=\bB(M_{\beta_1}).$
    \item $M_{\beta_1}=M_{i_1}$ is a finite-dimensional simple object of degree $\beta_1=\alpha_{i_1}$.
\end{itemize}

For $l>1$, the preceding argument corresponds to the first reflection step from $R_{(i_{1},...,i_{l-1})}(M)$ to $R_{(i_{1},...,i_{l})}(M)$. By induction, we assume the following:
\begin{enumerate}
    \item [(i)] For any $2 \leq k \leq l$, if we define $\gamma_k = \text{id}_{[R_{i_1}(M)]} s_{i_2} \cdots s_{i_{k-1}} (\alpha_{i_k}),$ then $\gamma_2,...,\gamma_l$ are pairwise distinct non-zero elements of $\mathbb{N}_0^{\theta}$ (note that if $l=1$ this assumption is trivial).
    \item[(ii)] For any $2 \leq k \leq l$, $R_{(i_1,...,i_{k-1})}(M)_{i_k}\subseteq L^{R_{i_1}(M)}_{(i_2,...,i_{k-1})}$ and $\bbk1 \in \mathcal{K}^-_{(i_2,...,i_l)}(R_{(i_1,...,i_l)}(M))$. Therefore the following two objects are well-defined.
    \begin{align*}
        &\widetilde{M}_{\gamma_k}:=T^{R_{i_1}(M)}_{(i_2,\ldots,i_{k-1})}(R_{(i_1,\ldots,i_{k-1})}(M)_{i_k})\\ 
        &E^{R_{i_1}(M)}(i_2,\ldots,i_l):=t^{R_{i_1}(M)}_{(i_2,\ldots,i_{l})}(\bbk1)
    \end{align*}
 for all $2\leq k \leq l$. (Note that for $k=1$, it reduces to $M_{i_1}\subseteq \bB(M)$, which is trivial.)
 \item[(iii)] For any $2\leq k \leq l$, $\widetilde{M}_{\gamma_k} \subseteq E^{R_{i_1}(M)}(i_2,\ldots,i_l)$ is a finite-dimensional simple subobject in $\HH$ of degree $\gamma_k$. 
 \item[(iv)]    The multiplication map 
 $$ (\cdots (\bB(\widetilde{M}_{\gamma_l}) \ot \bB(\widetilde{M}_{\gamma_{l-1}})) \cdots )\ot\bB(\widetilde{M}_{\gamma_2}) \longrightarrow E^{R_{i_1}(M)}(i_2,\ldots,i_l)$$ is an isomorphism of $\mathbb{N}_0^{\theta}$-graded objects in $\HH$.
\end{enumerate}
With the above notation and hypothesis, we establish the following proposition.
\begin{prop} \label{prop4.16}Let $M\in\mathcal F_\theta$ be a tuple of simple objects admitting
all reflections. Let
\[
\kappa=(i_1,\ldots,i_l)
\]
be an $[M]$-reduced sequence, and put
\[
\beta_k
=
\operatorname{id}_{[M]}
s_{i_1}\cdots s_{i_{k-1}}(\alpha_{i_k}),
\qquad 1\leq k\leq l.
\]
Then the following statements hold.

 (1) For any $1 \leq k \leq l$, $R_{(i_1,...,i_{k-1})}(M)_{i_k}\subseteq L^{M}_{(i_1,...,i_{k-1})}$ and $\bbk1 \in \mathcal{K}^-_{(i_1,...,i_l)}(R_{(i_1,...,i_l)}(M))$. Therefore the following two objects are well-defined.
    \begin{align*}
        &M_{\beta_k}:=T^M_{(i_1,\ldots,i_{k-1})}(R_{(i_1,\ldots,i_{k-1})}(M)_{i_k}),\\ 
&E^M(i_1,\ldots,i_l):=t^M_{(i_1,\ldots,i_{l})}(\bbk1)
    \end{align*}
 for all $1\leq k \leq l$.\par 
    (2) The elements  $\beta_1,...,\beta_l$ are pairwise distinct non-zero elements of $\mathbb{N}_0^{\theta}$.\par
   (3)  The multiplication map 
 $$ (\cdots (\bB({M}_{\beta_l}) \ot \bB({M}_{\beta_{l-1}})) \cdots )\ot\bB({M}_{\beta_1})\longrightarrow E^M(i_1,\ldots,i_l)$$ is an isomorphism of $\mathbb{N}_0^{\theta}$-graded objects in $\HH$.
 
 (4) For any $1\leq k \leq l$, $M_{\beta_k}\subseteq E^M(i_1,\ldots,i_l)$ is a finite-dimensional simple subobject in $\HH$ of degree $\beta_k$. \par 
 
\end{prop}
\begin{proof}
  Since the sequence $(i_1,\ldots,i_l)$ is $[M]$-reduced, it follows from the definition that $$\alpha_{i_1} \notin \Lambda^{R_{i_1}(M)}(i_2,\ldots,i_l)=\{ \gamma_2,\ldots,\gamma_l\}.$$
  For $2\leq k \leq l$,  assumption (ii) and (iii) say
$\widetilde{M}_{\gamma_k} $ has degree $\gamma_k$.
By assumption  {(iv)}, every degree occurring in
$E^{R_{i_1}(M)}(i_2,\ldots,i_l)$ is a sum of
$\gamma_2,\ldots,\gamma_l$. Since
$\alpha_{i_1}\neq\gamma_k$ for $2\leq k\leq l$, we deduce that
$$ R_{i_1}(M)_{i_1} \nsubseteq E^{R_{i_1}(M)}(i_2,\ldots,i_l).$$ Therefore by definition of $\mathcal{K}_{i_1}^-(R_{i_1}(M))$, we have 
\begin{equation}
E^{R_{i_1}(M)}(i_2,\ldots,i_l) \in \mathcal{K}_{i_1}^-(R_{i_1}(M)).\label{4.7}
\end{equation} 
That is, $t^{R_{i_1}(M)}_{(i_2,\ldots,i_{l})}(\bbk1) \in  \mathcal{K}_{i_1}^-(R_{i_1}(M))$, hence $\bbk1 \in \mathcal{K}^-_{(i_1,\ldots,i_l)}(R_{(i_1,\ldots,i_l)}(M))$ and $E^M(i_1,\ldots,i_l):=t^M_{(i_1,\ldots,i_l)}(\bbk1)$ is well-defined.
By Lemma \ref{lem4.10}, 
$$ \widetilde{M}_{i_1}=R_{i_1}({M})_{i_1} \nsubseteq E^{R_{i_1}(M)}(i_2,\ldots,i_l) \Leftrightarrow E^{R_{i_1}(M)}(i_2,\ldots,i_l) \subseteq L_{i_1}^{R_{i_1}(M)}.$$
By assumption (iii), $\widetilde{M}_{\gamma_k} \subseteq L_{i_1}^{R_{i_1}(M)}$ for any $2\leq k \leq l$, hence
$$ T_{i_1}^M(\widetilde{M}_{\gamma_k})=T_{i_1}^M(T^{R_{i_1}(M)}_{(i_2,\ldots,i_{k-1})}(R_{(i_1,\ldots,i_{k-1})}(M)_{i_k}))\subseteq K_{i_1}^{M},$$
which implies $R_{(i_1,\ldots,i_{k-1})}(M)_{i_k}\subseteq L^{M}_{(i_1,\ldots,i_{k-1})}$ and  
  $      M_{\beta_k}:=T^M_{(i_1,\ldots,i_{k-1})}(R_{(i_1,\ldots,i_{k-1})}(M)_{i_k}) $
is well-defined.   So we  have proved (1).\par 
By Lemma~\ref{lem4.12}{(2)},
\[
\operatorname{deg}
\bigl(
T_{i_1}^M
(L_{i_1}^{R_{i_1}(M)}(\alpha))
\bigr)
=
s_{i_1}^{R_{i_1}(M)}(\alpha).
\]
Hence the elements
\[
\beta_k
=
s_{i_1}^{R_{i_1}(M)}(\gamma_k)
\in\mathbb N_0^\theta,
\qquad 2\leq k\leq l,
\]
are pairwise distinct by assumption {(i)}.. If $\beta_{1}=\beta_{i}$ for some $2\le i\le l$, then $-\beta_{1}=\gamma_{i}\in\mathbb{N}_{0}^{\theta},$ which yields a contradiction. This establishes (2).\par
Recall (\ref{4.7}), $E^{R_{i_1}(M)}(i_2,\ldots,i_l) \in \mathcal{K}_{i_1}^-(R_{i_1}(M))$, hence by Proposition \ref{prop4.13}, we have a bijective map:
$$ T_{i_1}^M(E^{R_{i_1}(M)}(i_2,\ldots,i_l)) \ot \bB(M_{i_1}) \rightarrow t_{i_1}^M(E^{R_{i_1}(M)}(i_2,\ldots,i_l))=E^M(i_1,\ldots,i_l).$$
The map $T$ itself is an algebra isomorphism. 
By assumption (iv), we have a bijective map
$$ (\cdots (T_{i_1}^M(\bB(\widetilde{M}_{\gamma_l})) \ot T_{i_1}^M(\bB(\widetilde{M}_{\gamma_{l-1}}))) \cdots )\ot \bB(M_{i_1}) \longrightarrow E^M(i_1,\ldots,i_l).$$
By the above construction, for $2 \leq k \leq l$, we have 
  $ M_{\beta_k}=T_{i_1}^M(\widetilde{M}_{\gamma_k})$.
Then 
$$ (\cdots (\bB({M}_{\beta_l}) \ot \bB({M}_{\beta_{l-1}})) \cdots )\ot\bB({M}_{\beta_1})\longrightarrow E^M(i_1,\ldots,i_l)$$ is an isomorphism of $\mathbb{N}_0^{\theta}$-graded objects in $\HH$, which shows (3).

  Note that $M_{\beta_1}=M_{i_1} \subseteq E^M(i_1,\ldots,i_l)$ is simple.
  By the above construction, for $2 \leq k \leq l$, we have 
  $$ M_{\beta_k}=T_{i_1}^M(\widetilde{M}_{\gamma_k}) \in E^M(i_1,\ldots,i_l).$$
Since $T$ is an isomorphism and $\widetilde{M}_{\gamma_k}$ is finite-dimensional and simple in $\HH$ by assumption (iii), we deduce that 
for any $1\leq k \leq l$, $M_{\beta_k}\subseteq E^M(i_1,\ldots,i_l)$ is a finite-dimensional simple object  in $\HH$ of degree $\beta_k$. 
This implies (4).
\end{proof}
We now proceed to verify (CG3').
\begin{thm}\label{thm5.3}
Let $M\in\mathcal F_\theta$, with each $M_i$ simple, and assume
that $M$ admits all reflections. Then, for every
$X\in\mathcal F_\theta(M)$ and every $[X]$-reduced sequence
$\kappa$,
\[
\Lambda^{[X]}(\kappa)\subseteq\mathbb N_0^\theta.
\]
\end{thm}
\begin{proof}
 For any $X \in \mathcal{F}_{\theta}(M)$, there is a sequence $(i_1,...,i_l)$ such that $[X]=[R_{(i_1,\ldots,i_l)}(M)]$. Clearly,  the tuple $X$ admits all reflections and gives rise to a semi-Cartan graph $\mathcal{G}(X)$.  For any $[X]$-reduced sequence $\kappa$, we have $\Lambda^{[X]}(\kappa)\subseteq \mathbb{N}_0^{\theta}$ by Proposition \ref{prop4.16} (2). This proves (CG3').
\end{proof}
\subsection{Decomposition of Nichols algebras}

 Let $M \in \mathcal{F}_{\theta}$. Assume that  $M$ admits all reflections. Let $\kappa=(i_1,...,i_l)\in \mathbb{I}^l$ be an $[M]$-reduced sequence. 
 Recall that in Remark \ref{rmk4.15}, we defined 
$$\mathbbm{t}_{(i_1, \ldots, i_k)}^{M} = \mathbbm{t}_{i_1}^{M} \circ \cdots \circ  \mathbbm{t}_{i_k}^{R_{(i_1,...,i_{k-1})}(M)} : 
\mathcal{L}_{(i_1, \ldots, i_k)}^{+} \bigl( R_{(i_1, \ldots, i_k)}(M) \bigr) \longrightarrow \mathcal{L}^-_{i_1}(M).$$ 
We use $\mathbbm t^M_\kappa$ and $E^M(\kappa)$ to obtain a
decomposition of $\bB(M)$.
\begin{lemma}
    With above assumption, we have $\bB(R_{\kappa}(M))\in \mathcal{L}_{\kappa}^{+} \bigl( R_\kappa(M) \bigr)$. Furthermore, the multiplication map $$  \mathbbm{t}^M_\kappa(\bB(R_{\kappa}(M)))\ot E^M(\kappa) \rightarrow \bB(M)$$ is bijective in $\HH$. 
\end{lemma}
\begin{proof}
    We proceed by induction on $l$. The assertion is immediate for l=0. For $l=1$, $M_{i_{1}}^{*}\subseteq\mathcal{B}(R_{i_{1}}(M))$ holds automatically and 
    $$ \mathbbm{t}_{i_1}^M(\bB(R_{i_1}(M)))=T_{i_1}^M(\bB(R_{i_1}(M))\cap \mathcal{L}_{i_1}^{R_{i_1}(M)})=K_{i_1}^M.$$
Recall $E^M(i_1)=\bB(M_{i_1})$ in this case, hence the multiplication map $\mathbbm{t}_{i_1}^M(\bB(R_{i_1}(M))) \ot E^M(i_1) \rightarrow {\bB(M)}$ is bijective.\par 
 Now for $l\geq 2$, we denote  $M'=R_{i_1}(M)$ and $\kappa'=(i_2,...,i_l)$. We assume that $\bB(R_{\kappa}(M))\in \mathcal{L}_{\kappa'}^+(R_{\kappa'}(M'))$ and the multiplication map $\mathbbm{t}_{\kappa'}^{M'}(\bB(R_{\kappa}(M)))\ot E^{M'}(\kappa')\rightarrow {\bB(M')}$ is bijective. From our assumptions, $\mathbbm{t}_{\kappa'}^{M'}(\bB(R_{\kappa}(M))) \in \mathcal{L}^-_{i_2}(M')$. Note that $M_{i_{1}}^{\prime}\not \subseteq E^{M^{\prime}}(\kappa^{\prime})$. Because the multiplication map $\mathbbm{t}_{\kappa^{\prime}}^{M^{\prime}}(\mathcal{B}(R_{\kappa}(M)))\otimes E^{M^{\prime}}(\kappa^{\prime})\rightarrow\mathcal{B}(M^{\prime})$ is bijective, it follows that $M_{i_{1}}^{\prime}\subseteq\mathbbm{t}_{\kappa^{\prime}}^{M^{\prime}}(\mathcal{B}(R_{\kappa}(M)))$. Hence $\mathbbm{t}_{\kappa'}^{M'}(\bB(R_{\kappa}(M))) \in \mathcal{L}_{i_1}^+(M')$, $\mathbbm{t}_{i_1}^M (\mathbbm{t}_{\kappa'}^{M'}(\bB(R_{\kappa}(M)))) \in \mathcal{L}_{i_1}^-(M)$ and \\$\bB(R_{\kappa}(M)) \in \mathcal{L}_{\kappa}^{+}(R_\kappa(M))$. By Lemma \ref{lem4.10} (2), the condition $ \mathbbm{t}^M_\kappa(\bB(R_{\kappa}(M))) \in \mathcal{L}_{i_1}^-(M)$ is equivalent to $$\mathbbm{t}^M_\kappa(\bB(R_{\kappa}(M))) \subseteq K_{i_1}^M.$$\par 
By Proposition \ref{prop4.16} (3) and \ref{prop4.13} (1)
$$E^M (\kappa) \cong T^M_{i_1}(E^{M'}(\kappa'))\ot \bB(M_{i_1}). $$ To show
$ \mathbbm{t}^M_\kappa(\bB(R_{\kappa}(M)))\ot E^M(\kappa) \rightarrow {\bB(M)}$ is bijective, it suffices to show $\mathbbm{t}^M_\kappa(\bB(R_{\kappa}(M))) \ot T^M_{i_1}(E^{M'}(\kappa'))$ is isomorphic to $K_{i_1}^M$ in $\HH$.
Since $T^M_{i_1}$ is an algebra isomorphism, it suffices to show 
\begin{equation}\label{4.8}
    (T^{M}_{i_1})^{-1}(\mathbbm{t}^M_\kappa(\bB(R_{\kappa}(M)))) \ot E^{M'}(\kappa') \rightarrow  L_{i_1}^{M'}
\end{equation}is an isomorphism.
The multiplication map $(T^{M}_{i_1})^{-1}(\mathbbm{t}^M_\kappa(\bB(R_{\kappa}(M)))) \ot \bB(M'_{i_1}) \rightarrow \mathbbm{t}_{\kappa'}^{M'}(\bB(R_{\kappa}(M)))$ is bijective.  Hence 
$$ ((T^{M}_{i_1})^{-1}(\mathbbm{t}^M_\kappa(\bB(R_{\kappa}(M)))) \ot \bB(M'_{i_1})) \ot E^{M'}(\kappa')\rightarrow {\bB(M')}$$ is bijective.
By Proposition \ref{prop4.16}(2), and the  definition of $T_{i_1}^M$, $$(T^{M}_{i_1})^{-1}(\mathbbm{t}^M_\kappa(\bB(R_{\kappa}(M)))), E^{M'}(\kappa')\subseteq  L_{i_1}^{M'}.$$
Hence (\ref{4.8}) is an injective map in $\HH$.

Now fix any $\alpha\in \mathbb{N}_0^{\theta}$,
$$\operatorname{dim}{\bB(M')}(\alpha)=\sum_{k=0}^{\infty}\operatorname{dim}((T^{M}_{i_1})^{-1}\mathbbm{t}^M_\kappa(\bB(R_{\kappa}(M)))) \ot E^{M'}(\kappa'))(\alpha-k\alpha_{i_1})\operatorname{dim}\bB(M'_{i_1})(k\alpha_{i_1}).$$
On the other hand, for any $\alpha \in \mathbb{N}_0^{\theta}$, we have\begin{align*}
\operatorname{dim}{\bB(M')}(\alpha)&=\sum_{k=0}^{\infty}\operatorname{dim}K_{i_1}^{M'}(\alpha-k\alpha_{i_1})\operatorname{dim}\bB(M'_{i_1})(k\alpha_{i_1})\\
&=\sum_{k=0}^{\infty}\operatorname{dim}L_{i_1}^{M'}(\alpha-k\alpha_{i_1})\operatorname{dim}\bB(M'_{i_1})(k\alpha_{i_1}),
\end{align*}
where the second equation follows by (\ref{4.4}). By comparing dimensions, we deduce that for any $\alpha\in \mathbb{N}_0^{\theta}$, 
$\operatorname{dim}L_{i_1}^{M'}(\alpha)=\operatorname{dim}((T^{M}_{i_1})^{-1}(\mathbbm{t}^M_\kappa(\bB(R_{\kappa}(M)))) \ot E^{M'}(\kappa'))(\alpha).$ Therefore (\ref{4.8}) is bijective and $$  \mathbbm{t}^M_\kappa(\bB(R_{\kappa}(M)))\ot E^M(\kappa) \rightarrow \bB(M)$$ is an isomorphism of $\mathbb{N}_0^{\theta}$-graded objects  in $\HH$. 
\end{proof}
The next proposition characterizes $\mathbbm{t}^M_\kappa(\bB(R_{\kappa}(M)))$ and $ E^M(\kappa) $ for a special reduced sequence $\kappa=\kappa_{ij}=(i,j,i,j,...)$, where $i\neq j \in \mathbb{I}.$
\begin{prop}\label{prop7.5}
Let $i,j\in\mathbb I$ with $i\neq j$. Assume that $M$ admits all
reflections and that
$
m:=\overline m_{ij}^{[M]}<\infty.
$
Put
$
\kappa=\kappa_{ij}^{[M]}.
$
Then
\[
E^M(\kappa)=\bB(M_i+M_j),
\qquad
\mathbbm t_\kappa^M\bigl(\bB(R_\kappa(M))\bigr)
=
\bB(M)^{\operatorname{co}\bB(M_i+M_j)}.
\]
\end{prop}
\begin{proof}
    We know that $\mathcal{G}(M)$ is a semi-Cartan graph satisfying (CG3') by Theorem \ref{thm5.3}. 
    Since $\kappa$ only involves $i$ and $j$, $\beta_{k}^{[M],\kappa}$ must be linear combination of $\alpha_i$ and $\alpha_j$ for all $1\leq k \leq \overline{m}_{ij}^{[M]}$. According to Proposition \ref{prop4.16}(4), we have
    $$ E^M(\kappa) \subseteq \bigoplus_{k_1,k_2 \geq 0}\bB(M)(k_1\alpha_i+k_2\alpha_j)= \bB(M_i+M_j).$$

    Conversely, by Lemma \ref{lem2.10}     (3), $\beta_{1}^{[M],\kappa}=\alpha_i$ and $\beta_{\overline{m}_{ij}^{[M]}}^{[M],\kappa}=\alpha_j$. 
    According to Proposition \ref{prop4.16}(4), we have
$ M_i +M_j \subseteq E^M(\kappa),$ and then 
$$ \bB(M_i+M_j)\subseteq E^M(\kappa),$$
since $E^M(\kappa)$ is a  subalgebra.
    Hence $E^M(\kappa)=\bB(M_i+M_j)$.\par
    For the second statement, we first claim that $\mathbbm{t}_{\kappa}^M(\bB(R_{\kappa}(M))) \subseteq \bB(M)^{\operatorname{co}\bB(M_i+M_j)}$.
    Without loss of generality, assume $i<j$ and set $M'=(M_1,...,M_i+M_j,...,M_{j-1},M_{j+1},...,M_{\theta})$, 
 Since the multiplication map 
 $$ \mathbbm{t}^M_\kappa(\bB(R_{\kappa}(M)))\ot \bB(M_i+M_j)\rightarrow \bB(M)$$ is bijective, we have $\pi_i^{M'}(\mathbbm{t}^M_\kappa(\bB(R_{\kappa}(M))))\cap (M_i+M_j)=0$. Note that $\pi_i^{M'}(\mathbbm{t}^M_\kappa(\bB(R_{\kappa}(M))))$ is a nonzero left coideal subalgebra of $\bB(M_i+M_j)$ and $\bB(M_i+M_j)$ is a strictly graded coalgebra. Hence, $\pi_{i}^{M^{\prime}}(\mathbbm{t}_{\kappa}^{M}(\mathcal{B}(R_{\kappa}(M)))) =\mathbbm{k}1$. For any $x\in\mathbbm{t}_{\kappa}^{M}(\mathcal{B}(R_{\kappa}(M)))$, this implies $\pi_{i}^{M^{\prime}}(x)=\epsilon(\pi_{i}^{M^{\prime}}(x))=\epsilon(x)$. Therefore $\mathbbm{t}^M_\kappa(\bB(R_{\kappa}(M))) \subseteq \bB(M)^{\operatorname{co}\pi_i^{M'}}=\bB(M)^{\operatorname{co}\bB(M_i+M_j)}$. 
We have
\[
\mathbbm t_\kappa^M(\bB(R_\kappa(M)))
\ot\bB(M_i+M_j)
\cong\bB(M)
\cong
\bB(M)^{\operatorname{co}\bB(M_i+M_j)}
\ot\bB(M_i+M_j).
\]
Since both isomorphisms are multiplication maps, we obtain
\[
\mathbbm t_\kappa^M(\bB(R_\kappa(M)))
=
\bB(M)^{\operatorname{co}\bB(M_i+M_j)}.
\]
\end{proof}

\subsection{Proof of (CG4')}
The main result of this section is the following theorem. 
\begin{thm}\label{thm5.6}
Let $M\in\mathcal F_\theta$, with each $M_i$ simple, and assume
that $M$ admits all reflections. Then $\mathcal G(M)$ satisfies
 {(CG4')} and is therefore a Cartan graph.
\end{thm}
For $i\neq j \in \mathbb{I}$,
Let $\kappa_{ij}^{[M]}$ be the $[M]$-reduced sequence beginning of $(i,j,i,j,...)$. Assuming $m:=\overline{m}_{ij}^{[M]}< \infty$, we rewrite $\kappa_{ij}^{[M]}=(i_1,i_2,..,i_{m})$. Now we denote $\kappa_{ji}^{[M]}=(i_2,...,i_m,i_{m+1})$, where $i_{m+1}=i_{m-1}$. By Lemma \ref{lem2.10} (2), $\kappa_{ji}^{[M]}$ is $[M]$-reduced. To consider $(r_ir_j)^m$, it is natural to consider the composition of these two reflection sequence: $\kappa_{ij}^{[M]} \circ (\kappa_{ji}^{[M]})^{-1}$.
\begin{lemma}
   Under the above notation, we consider two following reflection sequences:
    \begin{align*}
        &\kappa'=(i_1,...,i_m,k),\\
        &\kappa''=(i_2,...,i_{m+1},k)
    \end{align*}
    where $k \neq i,j$. Then $\kappa'$ and $\kappa''$ are $[M]$-reduced sequences.
\end{lemma}
\begin{proof}
Since $\kappa_{ij}^{[M]}$ is $[M]$-reduced.  $\beta_{p}^{[M],\kappa_{ij}^{[M]}} \neq  -\beta_{q}^{[M],\kappa_{ij}^{[M]}} $ for all $1\leq p <q \leq m$. 
    It is straightforward to see that for any $1\leq n \leq m$, 
    $\beta_{n}^{[M],\kappa'} \in \mathbb{Z}\alpha_i+\mathbb{Z}\alpha_j$ and 
    $\beta_{m+1}^{[M],\kappa'}=\operatorname{id}_{[M]}s_{i_1}\cdots s_{i_m}(\alpha_k)\in \alpha_k+\mathbb{Z}\alpha_i+\mathbb{Z}\alpha_j$. Hence,
 $$\beta_{p}^{[M],\kappa'}\neq -\beta_{q}^{[M],\kappa'}
 $$
  for all $1\leq p <q \leq m+1$. Therefore, by Remark \ref{rmk2.7} (3), $\kappa'$ is $[M]$-reduced. The proof that $\kappa^{\prime\prime}$ is $[M]$-reduced is analogous.
\end{proof}
Now we define $M'=R_{\kappa_{ij}^{[M]}}(M)$, and $M''=R_{\kappa_{ji}^{[M]}}(M)$. By Proposition \ref{prop4.16}(1), $M'_k:=R_{(i_1,...,i_m)}(M)_k \subseteq L^{M}_{(i_1,...,i_m)}$, and 
$M''_k:=R_{(i_2,...,i_{m+1})}(M)_k \subseteq L^{M}_{(i_2,...,i_{m+1})}.$ Furthermore, recall that $\mathbbm{t}_i^\mathcal{M} : \mathcal{L}_i^+ (R_i(M)) \longrightarrow \mathcal{L}_i^- (M), \ F \longmapsto T_i^M(F\cap L_i^{R_i(M)})$. Therefore,  by Proposition \ref{prop7.5}, we have
\begin{align*}
   & T^{M}_{(i_1,...,i_m)}(M_k') \subseteq  T^{M}_{(i_1,...,i_m)}(\bB(M')\cap L_{(i_1,..,i_m)}^{M})=\mathbbm{t}^{M}_{(i_1,..,i_m)}(\bB(M'))=\bB(M)^{\operatorname{co}\bB(M_i+M_j)},\\
       & T^{M}_{(i_2,...,i_{m+1})}(M_k'') \subseteq  T^{M}_{(i_2,...,i_{m+1})}(\bB(M'') \cap L_{(i_2,...,i_{m+1})}^{M})=\mathbbm{t}^{M}_{(i_2,...,i_{m+1})}(\bB(M''))=\bB(M)^{\operatorname{co}\bB(M_i+M_j)}.
\end{align*}
With above preparations, we can now prove our main result.
\begin{proof}[Proof of Theorem \ref{thm5.6}]
Regardless of whether $i_m=i$ or $i_m=j$, the inclusions
\[
T^M_{(i_1,\ldots,i_m)}(M_k'),
\quad
T^M_{(i_2,\ldots,i_{m+1})}(M_k'')
\subseteq\bB(M)^{\operatorname{co}\bB(M_i+M_j)}
\]
show that the relevant compositions are well-defined. Since $T$
preserves the $\mathbb N_0^\theta$-grading and simplicity, we obtain
\begin{align*}
&(T^M_{(i_2,\ldots,i_{m+1})})^{-1}
\circ T^M_{(i_1,\ldots,i_m)}(M_k')
\subseteq\bB(M''),\\
&(T^M_{(i_1,\ldots,i_m)})^{-1}
\circ T^M_{(i_2,\ldots,i_{m+1})}(M_k'')
\subseteq\bB(M')
\end{align*}
as $\mathbb N_0^\theta$-graded simple subobjects. For all
$l\in\mathbb I$, we have
    $$s_{i_{m+1}}s_{i_m}\cdots s_{i_2}s_{i_1}\cdots s_{i_m}(\alpha_l)=(s_{i_{m+1}}s_{i_m})^m(\alpha_l),$$
     $$s_{i_{m}}s_{i_{m-1}}\cdots s_{i_1}s_{i_2}\cdots s_{i_{m+1}}(\alpha_l)=(s_{i_{m}}s_{i_{m+1}})^m(\alpha_l).$$
     By Lemma \ref{lem2.10} (3), we have $(s_{i_{m+1}}s_{i_m})^m(\alpha_i)=\alpha_i$ and $(s_{i_{m+1}}s_{i_m})^m(\alpha_j)=\alpha_j$.
On the other hand, since $T$ preserves $\mathbb{N}_0^\theta$-grading, we have
$$ (s_{i_{m+1}}s_{i_m})^m(\alpha_k), \ (s_{i_{m}}s_{i_{m+1}})^m(\alpha_k) \in  \alpha_k+\mathbb{N}_0\alpha_i+\mathbb{N}_0\alpha_j.$$   
              We assume $(s_{i_{m+1}}s_{i_m})^m(\alpha_k)=\alpha_k+a\alpha_i+b\alpha_j$ for $a, b\in \mathbb{N}_0$. Then 
              $$(s_{i_{m}}s_{i_{m+1}})^m \circ (s_{i_{m+1}}s_{i_m})^m(\alpha_k)=\alpha_k$$ implies
$$(s_{i_{m}}s_{i_{m+1}})^m(\alpha_k)=\alpha_k-a\alpha_i-b\alpha_j.$$
By the above observations, the degree condition forces $a = b = 0$. Hence $$(s_{i_{m+1}}s_{i_m})^m(\alpha_k)=(s_{i_{m}}s_{i_{m+1}})^m(\alpha_k)=\alpha_k,$$ which implies:
\begin{align*}
     (T^{M}_{(i_1,...,i_m)})^{-1}\circ T^{M}_{(i_2,...,i_{m+1})}(M_k'') &\subseteq M_k',\\
     (T^{M}_{(i_2,...,i_{m+1})})^{-1}\circ T^{M}_{(i_1,...,i_m)}(M_k') &\subseteq M_k''.
\end{align*}
Therefore $M_k'\cong M_k''$ for $k \neq i,j$.\par 
We now proceed to show: $M_i'\cong M_i''$ and $M_j' \cong M_j''$. Note that 
$M_{i_m}'=R_{(i_1,...,i_m)}(M)_{i_m}=R_{(i_1,..,i_{m-1})}(M)_{i_m}^*$. According to Lemma \ref{lem2.10} (3), we have $\beta_{m}^{[M],\kappa}=\alpha_j$. Hence $M_{i_m}'\cong M_j^*$.
On the other hand, $M_{i_{m+1}}'=R_{(i_1,...,i_m)}(M)_{i_{m+1}}=R_{(i_2,..,i_{m})}(R_{i_1}(M))_{i_{m+1}}$. Then we have $M_{i_{m+1}}'\cong M_i^*$. Similarly, 
\begin{align*}
    &M_{i_{m+1}}''=R_{(i_2,...,i_{m+1})}(M)_{i_{m+1}}=R_{(i_2,...,i_m)}(M)_{i_{m+1}}^*,\\
    &M_{i_{m}}''=R_{(i_2,...,i_{m+1})}(M)_{i_{m}}=R_{(i_3,..,i_{m+1})}(R_{i_2}(M))_{i_{m}}.
\end{align*}
Therefore $M_{i_{m+1}}''\cong M_i^*$ and $M_{i_{m}}''\cong M_j^*$. As a result $[M']=[M'']$. 
Finally, we have
\[(r_i r_j)^{{m}}([M]) = [M], \quad \text{id}_{[M]}(s_i s_j)^{{m}}(\alpha_k) = \alpha_k\]
for all \( k \in \mathbb{I} \setminus \{i, j\} \).\par 
Now for any $[X]\in \mathcal{X}$, since $M$ admits all reflections, there is a reflection sequence $(j_1,...,j_l)$ such that $[X]=[R_{(j_1,...,j_l)}(M)]$. It is obvious that $X$ admits all reflections for any representative $X \in [X]$. Therefore (CG4') holds for any $[X]\in \mathcal{X}$. Combined with Theorem \ref{thm5.3}, we deduce that $\mathcal{G}(M)$ is a Cartan graph.
\end{proof}
\begin{cor}
Under the assumptions of Theorem~\ref{thm5.6}, the pair
\[
\bigl(
\mathcal G(M),
(\Delta^{[X],\operatorname{re}})_{[X]\in\mathcal X}
\bigr)
\]
is a root system over $\mathcal G(M)$.
\end{cor}
\begin{proof}
    This corollary follows from the definition of the Cartan graph and $\Delta^{[X],\operatorname{re}}$.
\end{proof}
\section{Nichols algebras with finite Cartan graphs}
In this section, we study Nichols algebras with finite Cartan graphs. We introduce tensor decomposable Nichols algebras over coquasi-Hopf algebras and use this concept to give a criterion for the finite-dimensionality of $\mathcal{B}(M)$.

\subsection{Tensor decompositions}
\begin{definition}
    Let $V$ be an $\mathbb{N}_0^\theta$-graded object in $\HH$. We say $V$ is tensor decomposable if there exist an integer $n\geq 0$, finite-dimensional and simple Yetter-Drinfeld modules $Q_1,...,Q_n \in \HH$ such that $$ V \cong \bigotimes_{l=1}^n\bB(Q_l):=(\cdots(\bB(Q_1)\ot \bB(Q_2))\ot \cdots )\ot \bB(Q_n)$$ as $\mathbb{N}_0^\theta$-graded objects in $\HH$. Moreover, the degree of $Q_l$ satisfies $\operatorname{deg}(Q_l)=\beta_l$, such that $\beta_1,...,\beta_n$ are pairwise distinct elements in $\mathbb{N}_0^\theta  \setminus\{0\}$.\par 
    For $M \in \mathcal{F}_\theta$, the Nichols algebra $\bB(M)$ is called tensor decomposable if $\bB(M)$ is tensor decomposable as an $\mathbb{N}_0^\theta$-graded object in $\HH$ with the standard grading $\operatorname{deg}(M_i)=\alpha_i$.
\end{definition}

{To show that this definition is well-defined, we consider the } $\mathbb{N}_{0}\text{-grading}$ on $V=\bigoplus_{\alpha \in \mathbb{N}_0^{\theta}}V(\alpha)$. For any $\alpha \in \mathbb{N}_0^{\theta}$, $\alpha=\sum_{i=1}^{\theta}n_i\alpha_i$ with $n_i \geq 0$, let $|\alpha|=\sum_{i=1}^{\theta}n_i$. We define
 $$ V(n)=\bigoplus_{|\alpha|=n}V(\alpha).$$ Then $V=\bigoplus_{n\geq 0} V(n)$ is a $\mathbb{N}_0$-grading.
 \begin{lemma}\label{lem8.3}
     Let \( U, V \) and \( W \) be \( \mathbb{N}_0^\theta \)-graded objects in \( {}_H^H\mathcal{YD} \) with finite-dimensional homogeneous components. Assume that \( W(0) \cong \mathbbm{k} \) in \( {}_H^H\mathcal{YD} \). If
\[
U \otimes W \cong V \otimes W \quad \text{or} \quad W \otimes U \cong W \otimes V
\]
as \( \mathbb{N}_0^\theta \)-graded objects in \( {}_H^H\mathcal{YD} \), then \( U \cong V \) as \( \mathbb{N}_0^\theta \)-graded objects in \( {}_H^H\mathcal{YD} \).
\end{lemma}

\begin{proof}
The proof is the same as in [\citealp{rootsys}, Lemma 14.4.4] and is omitted here.
\end{proof}

\begin{lemma}\label{lem8.4}
Let $V, W$ be  tensor decomposable $\mathbb{N}_0$-graded objects in $\HH$ with finite-dimensional
homogeneous components. Assume 
$$ V\cong \bigotimes_{l=1}^n\bB(Q_l) \cong \bigotimes_{k=1}^m\bB(P_k).$$
Then $n=m$, and there is a permutation $\sigma \in S_n$ such that $P_l \cong Q_{\sigma(l)}$ as $\mathbb{N}_0^{\theta}$-graded object in $\HH$. 
\end{lemma}

\begin{proof}
Let $\mathbb{I}_n=\{1,2,..,n\}$ and $\mathbb{I}_m=\{1,2,...,m\}$, \( r = \min\{|\deg(Q_l)| \mid 1 \leq l \leq n\} \). Let \( L \) be the set of all \( l \) such that \( 1 \leq l \leq n \) and \( |\deg(Q_l)| = r \). Then \( \bigoplus_{l \in L} Q_l \) is the \( \mathbb{N}_0 \)-homogeneous component of \( \bigotimes_{l=1}^n \mathcal{B}(Q_l) \) of minimal positive degree. According to $\bigotimes_{l=1}^n\bB(Q_l) \cong \bigotimes_{k=1}^m\bB(P_k)$, we have
$r = \min\{|\deg(P_k)| \mid 1 \leq k \leq m\}$
and \( \bigoplus_{l \in L} Q_l \cong \bigoplus_{k \in K} P_k \), where $K$ is the set of all \( k \) such that \( 1 \leq k \leq m \) and \( |\deg(P_k)| = r \). Since $Q_l$ and $P_k$ are finite-dimensional simple Yetter-Drinfeld modules, their endomorphism rings are local. Thus, the Krull-Schmidt theorem applies, yielding $|L|=|K|$ and there  is a permutation $\sigma_r \in S_{|L|}$ such that \( Q_l \cong P_{\sigma(k)} \). \par
Now we have two such  decompositions since $\HH$ is braided
$$ \bigotimes_{l=1}^n \mathcal{B}(Q_l)\cong (\bigotimes_{l' \in \mathbb{I}_n\setminus L} \bB(Q_{l'})) \ot (\bigotimes_{l'' \in L} \bB(Q_{l''})),$$
$$ \bigotimes_{k=1}^m \mathcal{B}(P_k)\cong (\bigotimes_{k' \in \mathbb{I}_m\setminus K} \bB(P_{k'})) \ot (\bigotimes_{k'' \in K} \bB(P_{k''})).$$
 Note that $\bigotimes_{l'' \in L} \bB(Q_{l''}) \cong \bigotimes_{k'' \in K} \bB(P_{k''})$.  Since \( \bigotimes_{l'' \in L} \bB(Q_{l''})(0) \) is the trivial object $\mathbbm{k}$ and each homogeneous component of $V$ is finite-dimensional,
we have $$(\bigotimes_{l' \in \mathbb{I}_n\setminus L} \bB(Q_{l'})) \cong (\bigotimes_{k' \in \mathbb{I}_m\setminus K} \bB(P_{k'}))$$ by Lemma \ref{lem8.3}.  Proceeding similarly, we deduce that $n=m$ and there exists a permutation $\sigma\in S_{n}$ such that $P_{l}\cong Q_{\sigma(l)}$ as $\mathbb{N}_{0}^{\theta}$-graded objects in $\HH$.\par

\end{proof}
Lemma~\ref{lem8.4}
shows that the following sets are independent of the chosen tensor
decomposition.
\begin{definition}
    Let $s \in \operatorname{Aut}(\mathbb{Z}^{\theta})$, $\alpha \in \mathbb{Z}^{\theta}$ and $Q \in \HH$. Let $V$ be a tensor decomposable $\mathbb{N}_0$-graded object in $\HH$, we define:
    \begin{align*}
        s([Q],\alpha)&=([Q],s(\alpha)),\\
        \Phi_+^V&=\{([Q_l],\beta_l)\mid 1\leq l \leq n\},\\
        \Phi_-^V&=\{([Q_l^*],-\beta_l)\mid 1\leq l \leq n\},\\
 \Phi^V&=\Phi_+^V\cup \Phi_-^V.
    \end{align*}
    Here $Q_l$ and $\beta_l$, $1\leq l \leq n$, are the simple objects and their degrees, respectively, in the tensor decomposition of $V$.
\end{definition}

\subsection{Tensor decomposable Nichols algebras}
For $i\in\mathbb I$, recall that
$
 K_i^M=\mathcal B(M)^{\operatorname{co}\mathcal B(M_i)}.
$
The canonical coinvariant decomposition gives a graded isomorphism
\begin{equation}\label{eq:canonical-coinvariant-factorization}
 K_i^M\otimes\mathcal B(M_i)\xrightarrow{\ \sim\ }\mathcal B(M).
\end{equation}

\begin{lemma}
\label{lem:simple-root-factor-extraction}
If $\mathcal B(M)$ is tensor decomposable, then $K_i^M$ is tensor
decomposable for every $i\in\mathbb I$.  More precisely,
\begin{equation}\label{eq:positive-factors-split}
 \Phi_+^{\mathcal B(M)}
 =\Phi_+^{K_i^M}\,\cup \{([M_i],\alpha_i)\}.
\end{equation}
\end{lemma}

\begin{proof}
Choose a tensor decomposition
$
 \mathcal B(M)\simeq\bigotimes_{l=1}^n\mathcal B(Q_l),
\ \deg Q_l=\beta_l.
$
Since $\mathcal B(M)(\alpha_i)=M_i$, the degree-$\alpha_i$ component of the
right-hand side is nonzero.  As every $\beta_l$ is a nonzero element of
$\mathbb N_0^\theta$, the only way to obtain the degree $\alpha_i$ is from a
single factor whose generating degree is $\alpha_i$.  The generating
degrees are pairwise distinct; hence there is a unique index $h$ such that
$\beta_h=\alpha_i$, and comparison of the degree $\alpha_i$ components
gives $Q_h\simeq M_i$.

By the braiding, we may move this factor to the right and write
\[
 \mathcal B(M)\simeq U_i\otimes\mathcal B(M_i),
 \qquad
 U_i=\bigotimes_{l\neq h}\mathcal B(Q_l).
\]
Comparing this isomorphism with
\eqref{eq:canonical-coinvariant-factorization} and applying
Lemma~\ref{lem8.3}, we obtain $K_i^M\simeq U_i$.  This
proves both assertions.
\end{proof}
The next proposition is the main technical step.  In particular, it removes
the assumption that reflections are already known to exist.
\begin{prop}\label{prop:tensor-decomposition-under-reflection}
Let $M=(M_1,\ldots,M_\theta)$ be a tuple of finite-dimensional
simple objects. If $\mathcal B(M)$ is tensor decomposable, then,
for every $i\in\mathbb I$, the following statements hold.

 {(1)} The tuple $M$ admits the $i$-th reflection.

 {(2)} The Nichols algebra $\mathcal B(R_i(M))$ is tensor
decomposable.

 {(3)} One has
\[
\Phi^{\mathcal B(R_i(M))}
=
s_i^M\bigl(\Phi^{\mathcal B(M)}\bigr).
\]
\end{prop}

\begin{proof}
(1) Fix $i\in\mathbb I$.  By
Lemma~\ref{lem:simple-root-factor-extraction}, choose a tensor decomposition
\begin{equation}\label{eq:Ki-tensor-decomposition}
 K_i^M\simeq\bigotimes_{l=1}^n\mathcal B(Q_l),
 \qquad \deg Q_l=\beta_l.
\end{equation}
We first prove that the $i$-th reflection exists.  Define a  grading
on $K_i^M$ by
\[
 \operatorname{ht}_i\!\left(\sum_{r=1}^\theta n_r\alpha_r\right)
 =\sum_{r\neq i}n_r,
 \qquad
 K_i^M[t]=
 \bigoplus_{\operatorname{ht}_i(\gamma)=t}K_i^M(\gamma).
\]
Since $\mathcal B(M)(n\alpha_i)=\mathcal B(M_i)(n)$ and the right
coinvariants of $\mathcal B(M_i)$ under its coproduct are just
$\mathbbm k$, one has $K_i^M(n\alpha_i)=0$ for $n>0$.  Therefore every
$\beta_l$ in
\eqref{eq:Ki-tensor-decomposition} has $\operatorname{ht}_i(\beta_l)\geq1$.
It follows from that decomposition that
$ K_i^M[1]\simeq
 \bigoplus_{\operatorname{ht}_i(\beta_l)=1}Q_l.
$
In particular, $K_i^M[1]$ is finite-dimensional.

On the other hand, apply Theorem~\ref{thm3.7} with the grading in which
$M_i$ has degree zero and every $M_j$, $j\neq i$, has degree one.  It gives
\begin{equation}\label{eq:degree-one-Ki}
 K_i^M[1]
 =\bigoplus_{j\neq i}\operatorname{ad}(\mathcal B(M_i))(M_j)
 =\bigoplus_{j\neq i}\bigoplus_{r\geq0}
   (\operatorname{ad}M_i)^r(M_j).
\end{equation}
The last sum is direct because its summands have the distinct
$\mathbb N_0^\theta$-degrees $r\alpha_i+\alpha_j$.  Since the left-hand side
of \eqref{eq:degree-one-Ki} is finite-dimensional, for each $j\neq i$ only
finitely many of the spaces $(\operatorname{ad}M_i)^r(M_j)$ are nonzero.
Let
\[
 m_{ij}^M=\max\{r\geq0\mid(\operatorname{ad}M_i)^r(M_j)\neq0\}.
\]
Then $(\operatorname{ad}M_i)^{m_{ij}^M}(M_j)$ is nonzero and
finite-dimensional, whereas
$(\operatorname{ad}M_i)^{m_{ij}^M+1}(M_j)=0$.  Thus $M$ admits the $i$-th
reflection.

(2) We next transport the tensor decomposition through this reflection.  By
Lemma~\ref{lem4.12}, the  map
$
 T_i^M:L_i^{R_i(M)}\xrightarrow{\sim}K_i^M
$
satisfies
$
 T_i^M\bigl(L_i^{R_i(M)}(\gamma)\bigr)
 =K_i^M(s_i^M(\gamma)).
$
Here we use $s_i^{R_i(M)}=s_i^M$, which follows from
Corollary~\ref{cor5.8}.  Consequently, regrading
\eqref{eq:Ki-tensor-decomposition} by $s_i^M$ gives
\begin{equation}\label{eq:Li-reflected-decomposition}
 L_i^{R_i(M)}
 \simeq\bigotimes_{l=1}^n\mathcal B(Q_l),
 \qquad \deg_{R_i(M)}Q_l=s_i^M(\beta_l).
\end{equation}
The grading rule for $T_i^M$ gives
$s_i^M(\operatorname{supp}K_i^M)=
\operatorname{supp}L_i^{R_i(M)}\subseteq\mathbb N_0^\theta$.
Hence the degrees $s_i^M(\beta_l)$ belong to
$\mathbb N_0^\theta\setminus\{0\}$.  They remain pairwise
distinct because $s_i^M$ is an automorphism of $\mathbb Z^\theta$.
The antipode induces a degree-preserving isomorphism of objects
$
 L_i^{R_i(M)}\simeq K_i^{R_i(M)}
$
by \eqref{4.4}.  Hence $K_i^{R_i(M)}$ is tensor decomposable with the
factors in \eqref{eq:Li-reflected-decomposition}.  Finally, the canonical
coinvariant factorization gives
\begin{equation}\label{eq:reflected-full-factorization}
 \mathcal B(R_i(M))
 \simeq K_i^{R_i(M)}\otimes\mathcal B(M_i^*).
\end{equation}
Here $M_i^*$ has degree $\alpha_i$.  No degree $s_i^M(\beta_l)$ is a positive
multiple of $\alpha_i$, because
$K_i^{R_i(M)}(r\alpha_i)=0$ for $r>0$.  Thus
\eqref{eq:reflected-full-factorization} is a tensor decomposition.

(3) It remains to verify $\Phi^{\mathcal B(R_i(M))}
   =s_i^M\bigl(\Phi^{\mathcal B(M)}\bigr)$.  Put
$\Psi_i=\Phi_+^{K_i^M}$.  By
\eqref{eq:positive-factors-split},
$
 \Phi_+^{\mathcal B(M)}
 =\Psi_i\,\cup\{([M_i],\alpha_i)\}.
$
Equations \eqref{eq:Li-reflected-decomposition} and
\eqref{eq:reflected-full-factorization} give
\[
 \Phi_+^{\mathcal B(R_i(M))}
 =s_i^M(\Psi_i)\,\cup\{([M_i^*],\alpha_i)\}.
\]
Since $s_i^M(\alpha_i)=-\alpha_i$, the two simple-root pairs are exchanged:
\[
 s_i^M([M_i],\alpha_i)=([M_i],-\alpha_i),
 \qquad
 s_i^M([M_i^*],-\alpha_i)=([M_i^*],\alpha_i).
\]
The same calculation for the dual negative pairs proves
(3).
\end{proof}

\begin{cor}
\label{cor:tensor-decomposition-orbit}
If $\mathcal B(M)$ is tensor decomposable, then the following statements
hold.\par 
(1) The tuple $M$ admits all reflections, and
$\mathcal B(P)$ is tensor decomposable for every
$P\in\mathcal F_\theta(M)$.\par
(2) For every morphism $w:[P]\to[Q]$ in $\mathcal W(\mathcal G(M))$ one has
\begin{equation}\label{eq:Phi-under-Weyl-morphism}
 \Phi^{\mathcal B(Q)}=w\bigl(\Phi^{\mathcal B(P)}\bigr).
\end{equation}\par
(3) For every $P\in\mathcal F_\theta(M)$,
\begin{equation}\label{eq:real-roots-contained-in-Phi}
 \Delta^{[P],\operatorname{re}}
 \subseteq
 \{\gamma\in\mathbb Z^\theta\mid
   ([V],\gamma)\in\Phi^{\mathcal B(P)}
   \text{ for some simple }V\in\HH\}.
\end{equation}

\end{cor}

\begin{proof}
For (1), Proposition~\ref{prop:tensor-decomposition-under-reflection}, applied
inductively, shows that every reflection sequence exists and that tensor
decomposability is preserved at every step.  The components of all reflected
tuples remain simple by Lemma~\ref{lem3.13}. By Proposition~\ref{prop:tensor-decomposition-under-reflection} (3), we can prove (2), namely
\eqref{eq:Phi-under-Weyl-morphism}.

Let $\gamma\in\Delta^{[P],\operatorname{re}}$.  Then
$\gamma=w(\alpha_i)$ for some $i\in\mathbb I$ and some morphism
$w:[Q]\to[P]$.  Since
$([Q_i],\alpha_i)\in\Phi^{\mathcal B(Q)}$, equation
\eqref{eq:Phi-under-Weyl-morphism} gives
$([Q_i],\gamma)\in\Phi^{\mathcal B(P)}$.  This proves (3), that is,
\eqref{eq:real-roots-contained-in-Phi}.
\end{proof}

We can now characterize tensor decomposability purely in terms of the
associated Cartan graph.

\begin{thm}
\label{thm:tensor-decomposability-criterion}
Let $M=(M_1,\ldots,M_\theta)\in\mathcal F_\theta$, with every $M_i$
finite-dimensional and simple.  The following conditions are equivalent.
\par (1) $M$ admits all reflections and $\mathcal G(M)$ is a finite Cartan
 graph.\par
 (2) $M$ admits all reflections and $\mathcal B(P)$ is tensor
 decomposable for some $P\in\mathcal F_\theta(M)$.
 
(3) $\mathcal B(M)$ is tensor decomposable.

\end{thm}

\begin{proof}
Assume first that (1) holds.  Let
$w_0\in\operatorname{Hom}(\mathcal W(\mathcal G(M)),[M])$ be a longest morphism and let
$\kappa=(i_1,\ldots,i_l)$ be a reduced expression of $w_0$.  Put
$
 \beta_k=\operatorname{id}_{[M]}s_{i_1}\cdots s_{i_{k-1}}
 (\alpha_{i_k})
$ for $1\leq k\leq l$.
By Lemma~\ref{lem2.11}, the sequence $\kappa$ is $[M]$-reduced and
\begin{equation}\label{eq:longest-inversion-set}
 \Delta_+^{[M],\operatorname{re}}
 =\{\beta_1,\ldots,\beta_l\}.
\end{equation}
Proposition~\ref{prop4.16} provides a graded right coideal subalgebra
$E^M(\kappa)$ and a multiplication isomorphism
\begin{equation}\label{eq:E-longest-factorization}
 \mathcal B(M_{\beta_l})\otimes\cdots\otimes
 \mathcal B(M_{\beta_1})
 \xrightarrow{\ \sim\ }E^M(\kappa).
\end{equation}
Every simple root $\alpha_i$ belongs to the set in
\eqref{eq:longest-inversion-set}.  If $\beta_k=\alpha_i$, then
$M_{\beta_k}$ is a nonzero subobject of
$\mathcal B(M)(\alpha_i)=M_i$; hence $M_{\beta_k}=M_i$ because $M_i$ is
simple.  Thus $E^M(\kappa)$ contains every $M_i$.  Since it is a subalgebra
of $\mathcal B(M)$ and $\mathcal B(M)$ is generated by
$\bigoplus_iM_i$, we obtain $E^M(\kappa)=\mathcal B(M)$.  Equation
\eqref{eq:E-longest-factorization} therefore proves (3).

Clearly, (3) implies (2), with $P=M$, once we know that all reflections
exist.  This existence follows from
Corollary~\ref{cor:tensor-decomposition-orbit}.

Suppose that (2) holds.  Choose a reflection sequence carrying $M$ to $P$.
By Corollary~\ref{cor5.8}, the reverse sequence carries $P$ back to a tuple
isomorphic to $M$.  Repeated application of
Proposition~\ref{prop:tensor-decomposition-under-reflection} therefore
shows that $\mathcal B(M)$ is tensor decomposable.  Hence (2) implies (3).

Finally, assume (3).  By
Corollary~\ref{cor:tensor-decomposition-orbit}, $M$ admits all reflections,
and $\mathcal B(P)$ is tensor decomposable for every
$P\in\mathcal F_\theta(M)$. By Theorem \ref{thm5.6}, $\mathcal{G}(M)$ is a Cartan graph. The set $\Phi^{\mathcal B(P)}$ is finite.
Equation \eqref{eq:real-roots-contained-in-Phi} therefore implies that
$\Delta^{[P],\operatorname{re}}$ is finite for every object $[P]$ of
$\mathcal G(M)$.  Thus $\mathcal G(M)$ is finite, proving (1).
\end{proof}

The root factorization is now a consequence, rather than an additional
input, of the preceding theorem.

\begin{cor}\label{cor:root-factorization}
Assume the equivalent conditions of
Theorem~\ref{thm:tensor-decomposability-criterion}, and let
$P\in\mathcal F_\theta(M)$. Let
$
w_0\in\operatorname{Hom}(\mathcal W(\mathcal G(M)),[P])
$
be a longest morphism, and let $\kappa=(i_1,\ldots,i_l)$ be a reduced
expression of $w_0$. For
$
 \beta_k=\operatorname{id}_{[P]}s_{i_1}\cdots s_{i_{k-1}}
 (\alpha_{i_k}),
$
one has
\[
 \Delta_+^{[P],\operatorname{re}}
 =\{\beta_1,\ldots,\beta_l\},
\]
and multiplication induces an isomorphism of
$\mathbb N_0^\theta$-graded objects in $\HH$,
\[
 \mathcal B(P_{\beta_l})\otimes(\cdots(\bB(P_{\beta_2})\otimes
 \mathcal B(P_{\beta_1})))
 \xrightarrow{\ \sim\ }\mathcal B(P).
\]
\end{cor}

\begin{proof}
Apply Lemma~\ref{lem2.11} and Proposition~\ref{prop4.16} to $P$ and
$\kappa$.  As in the proof of
Theorem~\ref{thm:tensor-decomposability-criterion}, the resulting coideal
subalgebra $E^P(\kappa)$ contains every $P_i$, and hence equals
$\mathcal B(P)$.  The factorization in Proposition~\ref{prop4.16} gives the
claimed isomorphism.
\end{proof}
\subsection{A criterion for finite-dimensionality}
In this subsection, we first assume that \(M\) admits all reflections
and that \(\mathcal G(M)\) is finite. Fix
\(P\in\mathcal F_\theta(M)\), and let
$
 w_0\in\operatorname{Hom}(\mathcal W(\mathcal G(M)),[P])
$
be a longest morphism. Choose a reduced expression
\(\kappa=(i_1,\ldots,i_m)\) of \(w_0\), and set
\[
 \beta_k
 =
 \operatorname{id}_{[P]}s_{i_1}\cdots s_{i_{k-1}}
 (\alpha_{i_k}),
 \qquad 1\leq k\leq m.
\]
By Corollary~\ref{cor:root-factorization}, multiplication induces an
isomorphism
\[
 \mathcal B(P_{\beta_m})\otimes\cdots\otimes
 \mathcal B(P_{\beta_1})
 \xrightarrow{\ \sim\ }\mathcal B(P).
\]

\begin{lemma}\label{lem6.6}
   Let \( Q \in \mathcal{F}_\theta(M) \), \( 1 \leq k \leq m \), \( i \in \mathbb{I} \), and let \( w : [Q] \to [P] \) be a morphism in \( \mathcal{W}(\mathcal{G}(M)) \). Assume that \( \beta_k = w(\alpha_i) \). Then \( P_{\beta_k} \cong Q_i \) in \( {}_H^H\mathcal{YD} \).
\end{lemma}
\begin{proof}
By Theorem \ref{thm:tensor-decomposability-criterion},
\(\mathcal{B}(M)\) is tensor decomposable. Hence Corollary
\ref{cor:tensor-decomposition-orbit} applies to every tuple in
\(\mathcal{F}_{\theta}(M)\). In particular,
$
\Phi^{\mathcal{B}(P)}
=
w\bigl(\Phi^{\mathcal{B}(Q)}\bigr).$
Since \(\mathcal{B}(Q)\) is tensor decomposable, Lemma
\ref{lem:simple-root-factor-extraction} gives
$
([Q_i],\alpha_i)\in\Phi_+^{\mathcal{B}(Q)}.
$
Therefore,
\[
w([Q_i],\alpha_i)
=
([Q_i],w(\alpha_i))
=
([Q_i],\beta_k)
\in\Phi^{\mathcal{B}(P)}.
\]
Since \(\beta_k\in\mathbb{N}_0^\theta\setminus\{0\}\), this pair
belongs to \(\Phi_+^{\mathcal{B}(P)}\).

On the other hand, Corollary \ref{cor:root-factorization} and the
definition of \(\Phi_+^{\mathcal{B}(P)}\) give
\[
\Phi_+^{\mathcal{B}(P)}
=
\bigl\{([P_{\beta_h}],\beta_h)\mid 1\leq h\leq m\bigr\}.
\]
The roots \(\beta_1,\ldots,\beta_m\) are pairwise distinct. Hence the
unique element of \(\Phi_+^{\mathcal{B}(P)}\) whose second component is
\(\beta_k\) is \(([P_{\beta_k}],\beta_k)\). Consequently,
$
[Q_i]=[P_{\beta_k}],
$
and therefore \(Q_i\cong P_{\beta_k}\) in
\({}_H^H\mathcal{YD}\).
\end{proof}
The next proposition is useful for our characterization of finite-dimensional Nichols algebras.
\begin{prop}
    Let \(i, j \in \mathbb{I}\), \(i \neq j\), and \(0 \leq t \leq -a_{ij}^{[P]}\). Then there exists \(1 \leq k \leq m\) such that \(\alpha_j + t\alpha_i = \beta_k\) and \(\operatorname{ad} (P_i)^t(P_j) \cong P_{\beta_k}\) in \( {}_H^H\mathcal{YD}\). In particular, $\operatorname{ad} (P_i)^t(P_j)$ is simple in \( {}_H^H\mathcal{YD}\).
\end{prop}
\begin{proof}
By Corollary \ref{cor:root-factorization} and Lemma
\ref{lem:simple-root-factor-extraction}, there is a unique
\(1\leq h\leq m\) such that
\(\beta_h=\alpha_i\) and \(P_{\beta_h}\cong P_i\), and
\begin{equation}\label{5.3}
K_i^P\cong
\bigotimes_{\substack{1\leq k\leq m\\k\neq h}}
\mathcal B(P_{\beta_k})
\end{equation}
as \(\mathbb N_0^\theta\)-graded objects in \({}_H^H\mathcal{YD}\).

Set
\[
\operatorname{ht}_i\left(\sum_{r=1}^{\theta}n_r\alpha_r\right)
=\sum_{r\neq i}n_r.
\]
Since \(K_i^P(n\alpha_i)=0\) for \(n>0\), we have
\(\operatorname{ht}_i(\beta_k)\geq1\) for every \(k\neq h\).
By Theorem \ref{thm3.7}, \(K_i^P\) is generated by
$
\bigoplus_{\substack{r\in\mathbb I\\r\neq i}}
\ \bigoplus_{s=0}^{-a_{ir}^{[P]}}
\operatorname{ad}(P_i)^s(P_r).
$
Every summand has \(\operatorname{ht}_i\)-degree one. Hence, for
\(\gamma=\alpha_j+t\alpha_i\),
$
K_i^P(\gamma)=\operatorname{ad}(P_i)^t(P_j).
$

A homogeneous contribution of \eqref{5.3} to degree \(\gamma\) is
determined by integers \(n_k\geq0\) satisfying
$
\sum_{k\neq h}n_k\beta_k=\gamma.
$
Applying \(\operatorname{ht}_i\) gives
$
1=\sum_{k\neq h}n_k\operatorname{ht}_i(\beta_k).
$
Thus exactly one \(n_k\) equals \(1\), all the others vanish, and
\(\beta_k=\gamma\). Therefore
\[
\operatorname{ad}(P_i)^t(P_j)
=K_i^P(\gamma)\cong P_{\beta_k}.
\]
The last object is simple by Proposition \ref{prop4.16}(4).
\end{proof}

\begin{thm}\label{thm6.7}
    Let \( M \in \mathcal{F}_\theta \) such that \( M_j \) is simple in \( {}_H^H\mathcal{YD} \) for all \( j \in \mathbb{I} \). The following are equivalent.

(1) \( \mathcal{B}(M) \) is finite-dimensional.

 (2) \( M \) admits all reflections, \( \mathcal{G}(M) \) is finite and \( \mathcal{B}(P_i) \) is finite-dimensional for all \( P \in \mathcal{F}_\theta(M) \) and \( i \in \mathbb{I} \).

\end{thm}
\begin{proof}
    $(1) \Rightarrow (2)$ For any $i \in \mathbb{I}$, the coinvariant object
$\bB(M)^{\operatorname{co}\bB(M_i)}$ is finite-dimensional. Hence,  $\bB(M)^{\operatorname{co}\bB(M_i)}$ is a rational Yetter-Drinfeld module and then $M$ admits the $i$-th reflection. Inductively, $M$ admits all
reflections and $\bB(P_i)$ is finite-dimensional for every
$P\in\mathcal F_\theta(M)$ and $i\in\mathbb I$. Since reflections
preserve dimension, one has
\[
 \dim\mathcal B(P)=\dim\mathcal B(M)
\]
for every \(P\in\mathcal F_\theta(M)\).
Let \(\kappa=(i_1,\ldots,i_l)\) be any \([P]\)-reduced sequence.
By Proposition~\ref{prop4.16}, there is an isomorphism
\[
 \mathcal B(P_{\beta_l})\otimes(\cdots (\mathcal{B}(P_{\beta_2})\otimes
 \mathcal B(P_{\beta_1})))
 \xrightarrow{\ \sim\ }E^P(\kappa).
\]
Every \(P_{\beta_k}\) is nonzero, and hence
\(\dim\mathcal B(P_{\beta_k})\geq2\). Therefore
\[
 2^l
 \leq
 \dim E^P(\kappa)
 \leq
 \dim\mathcal B(P).
\]
Thus the lengths of all reduced sequences are uniformly bounded.
Since \(\mathbb I\) is finite, only finitely many reduced sequences
can occur. It follows that
\(\Delta^{[P],\operatorname{re}}\) is finite for every
\(P\in\mathcal F_\theta(M)\). Hence \(\mathcal G(M)\) is finite.\par
$(2)\Rightarrow(1)$. Since \(M\) admits all reflections, Theorem
\ref{thm5.6} shows that \(\mathcal{G}(M)\) is a Cartan graph. Hence it
is a finite Cartan graph by assumption. Let
\[
w_0\in
\operatorname{Hom}\bigl(\mathcal{W}(\mathcal{G}(M)),[M]\bigr)
\]
be a longest morphism, and choose a reduced expression
\(\kappa=(i_1,\ldots,i_l)\) of \(w_0\). Set
\[
\beta_k
=
\operatorname{id}_{[M]}s_{i_1}\cdots s_{i_{k-1}}
(\alpha_{i_k}),
\qquad 1\leq k\leq l.
\]
By Corollary \ref{cor:root-factorization}, multiplication induces an
isomorphism
$  \bB(M)\cong (\cdots(\bB(M_{\beta_l})\ot \bB(M_{\beta_{l-1}}))\ot \cdots) \ot \bB(M_{\beta_1}).$

For each \(1\leq k\leq l\), the root \(\beta_k\) is real. Thus there
exist \(Q^{(k)}\in\mathcal{F}_{\theta}(M)\), \(j_k\in\mathbb{I}\), and
a morphism
$
w_k:[Q^{(k)}]\longrightarrow[M]
$
in \(\mathcal{W}(\mathcal{G}(M))\) such that
\(\beta_k=w_k(\alpha_{j_k})\). By Lemma \ref{lem6.6},
$
M_{\beta_k}\cong Q^{(k)}_{j_k}.
$
Assumption  {(2)} therefore implies that
\(\mathcal{B}(M_{\beta_k})\) is finite-dimensional for every
\(1\leq k\leq l\). The above root factorization now shows that
\(\mathcal{B}(M)\) is finite-dimensional.
 \end{proof}

\section{Cartan graphs under braided monoidal equivalences}

In this section, we prove that braided monoidal equivalences preserve
the existence of reflections and induce isomorphisms of the
associated Cartan graphs.

Let $H$ and $H'$ be coquasi-Hopf algebras with bijective antipodes,
and let
\[
(F,\xi):\HH\longrightarrow\HHY
\]
be a braided monoidal equivalence, with tensor structure
\[
\xi_{V,W}:F(V)\ot F(W)\xrightarrow{\sim}F(V\ot W).
\]
Let $X=(X_1,\ldots,X_\theta)\in\mathcal F_\theta$, where each $X_i$
is finite-dimensional and simple. Then
$F(X)=(F(X_1),\ldots,F(X_\theta))$ is a tuple of finite-dimensional
simple objects in $\HHY$.
According to Theorem \ref{thm5.6}, if $X$
($F(X)$, respectively) admits all reflections, then $\mathcal{G}(X)$ ($\mathcal{G}(F(X))$, respectively) is a Cartan graph. Now we recall the definition of isomorphism of Cartan graphs
\begin{definition}
    Let \( \mathcal{G} = \mathcal{G}(I, \mathcal{X}, r, A) \) and \( \mathcal{G}' = \mathcal{G}(J, \mathcal{Y}, t, B) \) be semi-Cartan graphs. A {morphism} \( (\beta, \gamma) : \mathcal{G} \to \mathcal{G}' \) of semi-Cartan graphs is a pair \( (\beta, \gamma) \), where \( \beta : I \to J \), \( \gamma : \mathcal{X} \to \mathcal{Y} \) are maps such that for all \( i, j \in I \) and \( X \in \mathcal{X} \),
\[
\gamma(r_i(X)) = t_{\beta(i)}(\gamma(X)), \quad a_{ij}^X = b_{\beta(i)\beta(j)}^{\gamma(X)}.
\]
A morphism
$(\beta,\gamma)$ is an isomorphism if and only if both maps $\beta$ and $\gamma$ are bijective.
Furthermore, if both $\mathcal{G}$ and $\mathcal{G}'$ are Cartan graphs, we say $\mathcal{G}$ is isomorphic to $\mathcal{G}'$ as  Cartan graphs.
\end{definition}

Our main goal is as follows:
\begin{thm}\label{thm7.2}
   Under the above assumptions, the tuple $X$ admits all reflections if and only if $F(X)$ does. Under this condition, $\mathcal{G}(X)$ is isomorphic to $\mathcal{G}(F(X))$ as  Cartan graphs.
\end{thm}

Now fix two simple objects $V,W\in\HH$. By Lemma \ref{lem3.2},
\[
\operatorname{ad}_H(\bB(V))(W)
=\bigoplus_{n\geq0}\operatorname{ad}_H(V)^n(W)
\]
is an object in $\AVG$, where $A(V)=\bB(V)\#H$. Furthermore,
$\operatorname{ad}_H(V)^n(W)$ is an object in $\HH$ for every
$n\geq0$.

\begin{prop}\label{prop7.4}
For each $n\geq 0$, there is an isomorphism of objects in $\HHY$
\[
\operatorname{ad}_{H'}(F(V))^n(F(W))
\cong
F\bigl(\operatorname{ad}_H(V)^n(W)\bigr).
\]
\end{prop}

\begin{proof}
Set
$B=\bB(V\oplus W),\
B'=\bB(F(V)\oplus F(W)),$
and denote their multiplications by $\mu$ and $\mu'$, respectively.
Let $c$ and $d$ denote the braidings of $\HH$ and $\HHY$,
respectively.

Since $F$ is an additive equivalence, we identify
$F(V\oplus W)\cong F(V)\oplus F(W).$
Applying Lemma~\ref{NicholstoNichols} to $V\oplus W$, we obtain an
isomorphism of Hopf algebras
\[
\Psi_{\mathcal B}:B'\xrightarrow{\sim}F(B).
\]
Let
$\iota_V:V\longrightarrow B
$,
$\iota_W:W\longrightarrow B
$
and
$
\iota_{F(V)}:F(V)\longrightarrow B'
$, 
$\iota_{F(W)}:F(W)\longrightarrow B'
$
be the canonical inclusions. Under the above identifications,
$
\Psi_{\mathcal B}\circ\iota_{F(V)}=F(\iota_V),
\qquad
\Psi_{\mathcal B}\circ\iota_{F(W)}=F(\iota_W).
$

Let $D_0=W$ and, for $n\geq1$, define
$
D_n=V\ot D_{n-1}.
$
We define morphisms
$a_n:D_n\longrightarrow B
$
inductively by
$a_0=\iota_W
$
and
\[
a_n
=
\mu\circ(\iota_V\ot a_{n-1})
-
\mu\circ(a_{n-1}\ot\iota_V)
   \circ c_{V,D_{n-1}}
\]
for $n\geq1$. By the definition of the iterated braided adjoint
action,
$
\operatorname{Im}(a_n)
=
\operatorname{ad}_H(V)^n(W).
$
Similarly, let $E_0=F(W)$ and
$E_n=F(V)\ot E_{n-1}
$
for $n\geq1$. Define morphisms
$
b_n:E_n\longrightarrow B'
$
by
$
b_0=\iota_{F(W)}
$
and
\[
b_n
=
\mu'\circ(\iota_{F(V)}\ot b_{n-1})
-
\mu'\circ(b_{n-1}\ot\iota_{F(V)})
   \circ d_{F(V),E_{n-1}}.
\]
Then
$
\operatorname{Im}(b_n)
=
\operatorname{ad}_{H'}(F(V))^n(F(W)).
$

We next compare $a_n$ and $b_n$. Define isomorphisms
$
\Xi_n:E_n\xrightarrow{\sim}F(D_n)
$
inductively by
$
\Xi_0=\operatorname{id}_{F(W)}
$
and
\[
\Xi_n
=
\xi_{V,D_{n-1}}
\circ
\bigl(\operatorname{id}_{F(V)}\ot\Xi_{n-1}\bigr)
\]
for $n\geq1$. We claim that
\begin{equation}\label{eq:ad-morphism-under-F}
\Psi_{\mathcal B}\circ b_n
=
F(a_n)\circ\Xi_n
\end{equation}
for every $n\geq0$.

For $n=0$, this follows from
$
\Psi_{\mathcal B}\circ b_0
=
\Psi_{\mathcal B}\circ\iota_{F(W)}
=
F(\iota_W)
=
F(a_0)\circ\Xi_0.
$
Suppose that \eqref{eq:ad-morphism-under-F} holds for $n-1$.
For convenience, write
$u_n=\mu\circ(\iota_V\ot a_{n-1}),\
v_n=\mu\circ(a_{n-1}\ot\iota_V)
        \circ c_{V,D_{n-1}},$
so that $a_n=u_n-v_n$. Meanwhile, we write
$
u_n'=\mu'\circ(\iota_{F(V)}\ot b_{n-1}),\
v_n'=\mu'\circ(b_{n-1}\ot\iota_{F(V)})
        \circ d_{F(V),E_{n-1}},$
so that $b_n=u_n'-v_n'$.
Since $\Psi_{\mathcal B}$ is an algebra morphism relative to the
tensor structure of $F$, we have
\[
\Psi_{\mathcal B}\circ\mu'
=
F(\mu)\circ\xi_{B,B}
\circ(\Psi_{\mathcal B}\ot\Psi_{\mathcal B}).
\]
Using this equality, the induction hypothesis, and the naturality of
$\xi$, we obtain
\begin{align*}
\Psi_{\mathcal B}\circ u_n'
&=
F(\mu)\circ\xi_{B,B}
\circ
\bigl(
F(\iota_V)\ot F(a_{n-1})\circ\Xi_{n-1}
\bigr)\\
&=
F(\mu)\circ F(\iota_V\ot a_{n-1})
\circ\xi_{V,D_{n-1}}
\circ
\bigl(
\operatorname{id}_{F(V)}\ot\Xi_{n-1}
\bigr)\\
&=
F(u_n)\circ\Xi_n.
\end{align*}
For the second term, the naturality of the braiding gives
\[
(\Xi_{n-1}\ot\operatorname{id}_{F(V)})
\circ d_{F(V),E_{n-1}}
=
d_{F(V),F(D_{n-1})}
\circ
(\operatorname{id}_{F(V)}\ot\Xi_{n-1}).
\]
Moreover, since $F$ is braided monoidal,
$
\xi_{D_{n-1},V}
\circ d_{F(V),F(D_{n-1})}
=
F(c_{V,D_{n-1}})
\circ\xi_{V,D_{n-1}}.
$
Therefore,
\begin{align*}
\Psi_{\mathcal B}\circ v_n'
&=
F(\mu)\circ\xi_{B,B}
\circ
\bigl(
F(a_{n-1})\circ\Xi_{n-1}\ot F(\iota_V)
\bigr)
\circ d_{F(V),E_{n-1}}\\
&=
F(\mu)\circ F(a_{n-1}\ot\iota_V)
\circ\xi_{D_{n-1},V}
\circ d_{F(V),F(D_{n-1})}
\circ
\bigl(
\operatorname{id}_{F(V)}\ot\Xi_{n-1}
\bigr)\\
&=
F(\mu)\circ F(a_{n-1}\ot\iota_V)
\circ F(c_{V,D_{n-1}})
\circ\xi_{V,D_{n-1}}\circ
\bigl(
\operatorname{id}_{F(V)}\ot\Xi_{n-1}
\bigr)\\
&=
F(v_n)\circ\Xi_n.
\end{align*}
Consequently,
\[
\Psi_{\mathcal B}\circ b_n
=
\Psi_{\mathcal B}\circ(u_n'-v_n')
=
F(u_n-v_n)\circ\Xi_n
=
F(a_n)\circ\Xi_n.
\]
This completes the induction.

Finally, $F$ is an exact equivalence of abelian categories and hence
preserves images. Since $\Psi_{\mathcal B}$ and $\Xi_n$ are
isomorphisms, equation \eqref{eq:ad-morphism-under-F} gives
\begin{align*}
\operatorname{ad}_{H'}(F(V))^n(F(W))
=
\operatorname{Im}(b_n)
\cong
F\bigl(\operatorname{Im}(a_n)\bigr)
=
F\bigl(\operatorname{ad}_H(V)^n(W)\bigr).
\end{align*}
This proves the proposition.
\end{proof}

Now we can prove Theorem \ref{thm7.2}.
\begin{proof}[Proof of Theorem \ref{thm7.2}]
If $X$ admits the $i$-th reflection for some $i \in \mathbb{I}$, then for each $j\neq i$, we have $ \operatorname{ad}_H(X_i)^{-a_{ij}^X}(X_j)\neq 0$ but $\operatorname{ad}_H(X_i)^{-a_{ij}^X+1}(X_j)=0$ for some $a_{ij}^X\leq 0$. By Proposition \ref{prop7.4}, we have 
$$  \operatorname{ad}_{H'}(F(X_i))^{-a_{ij}^X}(F(X_j))\neq 0, \  \operatorname{ad}_{H'}(F(X_i))^{-a_{ij}^X+1}(F(X_j))=0. $$ 
Obviously, the converse is true, hence $X$ admits the $i$-th reflection if and only if $F(X)$ does. 
Since $\operatorname{ad}_{H}(X_{i})^{-a_{ij}^{X}}(X_{j})$ is a simple object in $\HH$ by  Lemma \ref{lem3.13}, it follows that 
$$\operatorname{ad}_{H'}(F(X_i))^{-a_{ij}^X}(F(X_j)) \cong F(\operatorname{ad}_H(X_i)^{-a_{ij}^X}(X_j)) $$ is a simple object in $\HHY$ as well. Meanwhile we have $F(X_i^*)\cong F(X_i)^*$. Hence 
$$ F(R_i(X))\cong R_i(F(X)). $$ By applying Proposition \ref{prop7.4} again, if $R_i(X)$ admits the $l$-th reflection, then  for all $j \in\mathbb{I}$,
$$  a_{lj}^{R_i(X)}=a_{lj}^{F(R_i(X))}.   $$  
By induction, $X$ admits all reflections if and only if $F(X)$
does, and for every $Y\in\mathcal F_\theta(X)$ and
$i,j\in\mathbb I$,
$
a_{ij}^{Y}=a_{ij}^{F(Y)}.
$

Suppose now that $X$ admits all reflections. Then
$\mathcal G(X)$ is given by
$$ \mathcal{G}(X)=(\mathbb{I},\mathcal{X},r,(A^{[Y]})_{[Y]\in \mathcal{X}}),$$
where $\mathcal{X}=\{ [Y]\mid Y \in \mathcal{F}_{\theta}(X) \},$  the map  $r: \mathbb{I} \times \mathcal{X} \rightarrow \mathcal{X},\ i \times [Y] \mapsto [R_i(Y)]$ and $A^{[Y]}=(a_{ij}^Y)_{i,j \in \mathbb{I}}$ for all $[Y] \in \mathcal{X}$.
Then  $\mathcal{G}(F(X))$ is well-defined and can be expressed as follows,
$$ \mathcal{G}(F(X))=(\mathbb{I},\mathcal{Y},t,(A^{[Y]})_{[Y]\in\mathcal{Y}}),$$
where $\mathcal{Y}=\{ [Y]\mid Y \in \mathcal{F}_{\theta}(F(X)) \},$  the map  $t: \mathbb{I} \times \mathcal{Y} \rightarrow \mathcal{Y},\ i \times [Y] \mapsto [R_i(Y)]$ and $A^{[Y]}=(a_{ij}^Y)_{i,j \in \mathbb{I}}$ for all $[Y] \in \mathcal{Y}$.
Let us define $\beta=\operatorname{id}_{\mathbb{I}}$ and 
$$ \gamma: \mathcal{X}\rightarrow{} \mathcal{Y}, [Y]\mapsto [F(Y)].$$ 
It is straightforward to see that $\gamma$ is well-defined. For any $[Z]\in\mathcal X$, choose a representative $Z$. Then
\[
\gamma(r_i([Z]))
=[F(R_i(Z))]
=[R_i(F(Z))]
=t_i(\gamma([Z])),
\ \
a_{ij}^{[Z]}=a_{ij}^{\gamma([Z])}.
\]
The map $\gamma$ is injective because $F$ is an equivalence of categories, $$F(Y_1) \cong F(Y_2) \implies Y_1 \cong Y_2.$$ To prove surjectivity, let $[W]\in\mathcal Y$. Then
$
[W]=[R_{i_k}\cdots R_{i_1}(F(X))]
$
for some $(i_1,\ldots,i_k)\in\mathbb I^k$. Set
\[
Y=R_{i_k}\cdots R_{i_1}(X)\in\mathcal F_\theta(X).
\]
The isomorphism $F(R_i(Z))\cong R_i(F(Z))$, applied inductively,
gives $F(Y)\cong W$. Hence $\gamma([Y])=[W]$, so $\gamma$ is
surjective.
Hence $\mathcal{G}(X)$ is isomorphic to $\mathcal{G}(F(X))$ as Cartan graphs.
\end{proof}
\begin{cor}
Assume that $X$ admits all reflections, and let
\[
\gamma:\mathcal X\longrightarrow\mathcal Y,
\qquad
\gamma([Z])=[F(Z)],
\]
be the object map induced by $F$. Then
$\mathcal W(\mathcal G(X))$ and
$\mathcal W(\mathcal G(F(X)))$ are isomorphic as Weyl groupoids,
and
\[
\Delta^{[Z],\operatorname{re}}
=
\Delta^{\gamma([Z]),\operatorname{re}}
\]
for every $[Z]\in\mathcal X$.
\end{cor}
\begin{proof}
    It follows from the definition immediately.
\end{proof}
A special and important case of interest to us is when the braided monoidal equivalence originates from a twist. Let $H$ be a coquasi-Hopf algebra with bijective antipode. A convolution-invertible linear map

\[
J : H \otimes H \to \bbk
\]
is called a twist on \(H\) if
$
J(h, 1) = \varepsilon(h) = J(1, h)$
for all \(h \in H\). Given a twist, we can construct a new coquasi-Hopf algebra
as follows: \(H^J = H\) as a coalgebra and the multiplication \(\circ\) on \(H^J\) is given by

\[
a \circ b := J(a_1, b_1)a_2b_2J^{-1}(a_3, b_3)
\]
for all \(a, b \in H\). The associator \(\Phi^J\) and the quasi-antipode \((\mathcal{S}^J, \alpha^J, \beta^J)\) are given as
\[
\Phi^J(a, b, c) = J(b_1, c_1)J(a_1, b_2c_2)\Phi(a_2, b_3, c_3)J^{-1}(a_3b_4, c_4)J^{-1}(a_4, b_5),
\]
\[
\mathcal{S}^J = \mathcal{S}, \quad \alpha^J(a) = J^{-1}(\mathcal{S}(a_1), a_3)\alpha(a_2), \quad \beta^J(a) = J(a_1, \mathcal{S}(a_3))\beta(a_2)
\]
for all \(a, b, c \in H\).

It is standard that $$\operatorname{Comod}(H) \cong \operatorname{Comod}(H^J) $$ as tensor categories and 
$$ (\operatorname{id}, \mathcal{J}): \HH \xrightarrow{\sim} \HHJ, \ \ V \mapsto V^J $$
is an equivalence  of braided monoidal categories, where  $V^J=V$ as vector spaces and $\mathcal{J}_{V,W}: V^J\ot W^J \rightarrow{} V\ot W, v\ot w\mapsto J^{-1}(v_{-1},w_{-1})v_0\ot w_0.$
\begin{cor}\label{cor7.6}
Let $H$ be a coquasi-Hopf algebra with bijective antipode, let $J$
be a twist on $H$, and let $X\in\mathcal F_\theta$. Then $X$
admits all reflections if and only if $X^J$ does. If these
equivalent conditions hold, then
\[
\mathcal G(X)\cong\mathcal G(X^J)
\]
as Cartan graphs.
\end{cor}
\begin{proof}
    It follows from Theorem \ref{thm7.2} immediately.
\end{proof}

 \section{Applications}
\subsection{Reflections of Nichols algebras of diagonal type}

Let \(G\) be a finite abelian group and let \(\Phi\) be a normalized
\(3\)-cocycle on \(G\). Nichols algebras of diagonal type in
\({}_G^G\mathcal{YD}^{\Phi}\) were studied in \cite{QQG}. In this
subsection, we apply the reflection theory developed above to such
Nichols algebras. We first recall the relevant Yetter--Drinfeld module
structure and then compute the iterated braided adjoint actions,
reflections, and Cartan matrix for tuples of diagonal type.
\begin{exm}\upshape
    The left $\mathbbm{k}G$-comodule $(V, \delta)$ is a left-left Yetter--Drinfeld module over the coquasi-Hopf algebra $H = (\mathbbm{k}G, \Phi)$ if there is a linear map $\rhd : G \otimes V \rightarrow V$ such that for all $e, f \in G$ and $v \in {}^gV$:
\begin{align}
&e \rhd (f \rhd v) = \frac{\Phi(e, f, g)\Phi(g, e, f)}{\Phi(e, g, f)} (ef) \rhd v, \\
&1_{G} \rhd v = v, \\
&e \rhd v \in {}^{g}V.
\end{align}

The category of all left-left Yetter--Drinfeld modules over $(\mathbbm{k}G, \Phi)$ is denoted by $\GG$. 
We always denote $$\widetilde{\Phi}_g(e,f):= \frac{\Phi(e, f, g)\Phi(g, e, f)}{\Phi(e, g, f)}.$$
A Direct computation shows $\widetilde{\Phi}_g$ is a $2$-cocycle on $G$. Furthermore, a special but important class of $3$-cocycles is called abelian $3$-cocycle.  If $\Phi$ is abelian, then, up to cohomology, it can be chosen in
the standard form given in \cite[Equation~(2.13)]{huang2024classification}.
For this representative, one has
\[
\Phi(x,y,z)=\Phi(x,z,y)
\]
for all $x,y,z\in G$.
\end{exm}

Now let $(H,\Phi)$ be a coquasi-Hopf algebra with bijective antipode and $V$ be a one dimensional vector space. Let $\rhd: H \ot V \rightarrow V$ and $\delta: V \rightarrow H \ot V$ be maps. Now we choose $v \in V$ and $g\in H$ and $\chi \in \operatorname{Hom}(H,\bbk)$ such that 
$$ \delta(v)=g\ot v,\ \ h\rhd v=\chi(h)v \ \text{for all} \ h \in H.$$

It is straightforward to see that $V$ is a $H$-comodule if and only if $\Delta(g)=g\ot g$ and $\varepsilon(g)=1$. On the other hand, one may note that (\ref{1.11}) holds if and only if 
$ gh_2 \ot \chi(h_1)v=h_1g\ot \chi(h_2)v$. This is equivalent to
\begin{equation}\label{7.1}
    gh_2 \chi(h_1)=h_1g \chi(h_2).
\end{equation} 
 Now for any $h,l \in H$, $V$ satisfies (\ref{1.9}) if and only if 
 \begin{align}\label{7.2}
     \chi(hl)= \frac{\Phi(h_2, g, l_3)}{\Phi(h_1, l_1, g)\Phi(g, h_4, l_4)} 
    \chi(h_3) \chi(l_2).
 \end{align}
This implies $\chi$ is far from being an algebra homomorphism. Nevertheless, it exhibits several useful properties. Since $H$ has bijective antipode, if $h\in G(H)$, then $h^{-1}\in G(H)$. Moreover, 
\begin{equation}\label{7.3}
   1 = \chi(1)=\frac{\Phi(h, g, h^{-1})}{\Phi(h, h^{-1}, g)\Phi(g, h, h^{-1})} 
    \chi(h) \chi(h^{-1}).
\end{equation} Since elements in $G(H)$ are invertible, one deduces that $\chi(h)\ne 0$ for any $h\in G(H)$. According to (\ref{7.1}), we have $gh=hg$ for all $h \in G(H)$.
Now for any one-dimensional Yetter-Drinfeld module $(W=\operatorname{span}\{w\},\rhd,\delta)$, given by $\delta(w)=l \ot w$, $h \rhd w=\zeta(h)w$, where $\zeta \in \operatorname{Hom}(H,\bbk)$ satisfying (\ref{7.1}). If $h$ is group-like, the equation reduces to $$\zeta(gh)\chi(h)=\zeta(hg)\chi(h).$$
Since $hg=gh$ and
 $\zeta(g) \neq 0$,  we can deduce that $\zeta\chi(h)=\chi\zeta(h)$ for all $h \in G(H)$.
\begin{lemma}
    Let $V \in \HH$ be a Yetter-Drinfeld module of diagonal type. Denote 
    $$ \operatorname{Supp}V:= \{g\in G(H)\mid {^gV \neq 0}\}.$$ Here $^gV=\{ v \in V\mid \delta(v)=g \otimes v\}.$ Then $\operatorname{Supp}V$ generates an abelian group $G$, and $\Phi|_{G}$ is an abelian $3$-cocycle on $G$.
\end{lemma}
\begin{proof}
The preceding analysis shows that $\operatorname{Supp}V$
generates an abelian group $G$ and that $\Phi|_G$ is a
$3$-cocycle on $G$. By [\citealp{huang2024classification},
Lemma~2.10], $V$ is of diagonal type if and only if $\Phi|_G$ is
an abelian $3$-cocycle.
\end{proof}

Let $M=(M_1,...,M_{\theta})\in \mathcal{F}_{\theta}$ be a tuple of  Yetter-Drinfeld modules of diagonal type. For all \( j \in \mathbb{I} \), let \( x_j \) be a basis of \( M_j \), and let \( g_j \in H \), \( \chi_j \in \operatorname{Hom}(H, \mathbbm{k}) \) such that
\[
\delta_{M_j}(x_j) = g_j \otimes x_j, \qquad h \cdot x_j = \chi_j(h)x_j
\]
for all \( h \in H \). Then \( g_j \) is an invertible group-like element for all \( j \in \mathbb{I} \). For all \( j, k \in \mathbb{I} \) let \( q_{jk} \in \mathbbm{k}^\times \) such that
\[
c_{M_j, M_k}(x_j \otimes x_k) = q_{jk} \, x_k \otimes x_j.
\]
Therefore $G:=G_M=<g_1,\cdots, g_{\theta}>$ and $\Phi$ is an abelian $3$-cocycle on $G$ and $M \in \GG$. \par
For $n \geq 3$, we will use the notation 
$$ x_i^n:=((x_i \cdot x_i)\cdots )\cdot x_i.$$
to represent iterated multiplication.
\begin{lemma}
    For each $i \in \mathbb{I}$, $x_i^n=0$ in $\bB(M)$ if and only if $(n)!_{q_{ii}}=0.$
\end{lemma}
\begin{proof} We first consider $n=3$ as a base case.
    By direct computation:
    \begin{align*}
        \Delta(x_i^3)&=(x_i^2\ot 1+1\ot x_i^2+(1+q_{ii})x_i\ot x_i)(x_i\ot 1+1\ot x_i)\\
        &=x_i^3\ot 1+1 \ot x_i^3+(1+q_{ii}+q_{ii}^2)x_i^2 \ot x_i+ \frac{1}{\Phi(g_i,g_i,g_i)}(1+q_{ii}+q_{ii}^2)x_i\ot x_i^2.
    \end{align*}
    It is easy to see that $x_i^3=0$ if and only if $(3)!_{q_{ii}}=0$, since $\Phi(g_i,g_i,g_i)\neq 0$.
    
We claim that 
\begin{equation}
    \Delta(x_i^n)=\sum_{j=0}^n \binom{n}{n-j}_{q_{ii}}\prod_{k=1}^{j-1}\frac{1}{\Phi(g_i^{n-j},g_i^k,g_i)}x_i^{n-j} \ot x_i^{{j}}.
\end{equation}
Here, the associators are trivial when $j=0$ and $j=1$. We prove this claim by induction. 
\begin{align*}
\Delta(x_i^{n+1})&=\Delta(x_i^n)\Delta(x_i)\\&=(\sum_{j=0}^n \binom{n}{n-j}_{q_{ii}}\prod_{k=1}^{j-1}\frac{1}{\Phi(g_i^{n-j},g_i^k,g_i)}x_i^{n-j} \ot x_i^{{j}})(1\ot x_i +x_i\ot 1).
\end{align*}
Let us focus on coefficient of the term $x_i^{n-j+1}\ot x_i^{j}$, which is given by
\begin{align*}
   &\binom{n}{n-j+1}_{q_{ii}}\prod_{k=1}^{j-2}\frac{1}{\Phi(g_i^{n-j+1},g_i^k,g_i)}\frac{1}{\Phi(g_i^{n-j+1},g_i^{j-1},g_i)}\\
   &+\binom{n}{n-j}_{q_{ii}}\prod_{k=1}^{j-1}\frac{1}{\Phi(g_i^{n-j},g_i^k,g_i)}\frac{\Phi(g_i^{n-j},g_i,g_i^{j})}{\Phi(g_i^{n-j},g_i^{j},g_i)}\prod_{k=1}^{j-1}\frac{1}{\widetilde{\Phi}_{g_i}(g_i^k,g_i)}q_{ii}^{j}.
\end{align*}
A Direct computation shows
$$ \Phi(g_i^{n-j},g_i^k,g_i)=\Phi(g_i^{n-j},g_i,g_i^k)=\frac{\Phi(g_i^{n-j+1},g_i^k,g_i)\Phi(g_i^{n-j},g_i,g_i^{k+1})}{\Phi(g_i^{n-j},g_i^{k+1},g_i)\Phi(g_i,g_i^k,g_i)} $$
Therefore 
$$\prod_{k=1}^{j-1}\frac{1}{\Phi(g_i^{n-j},g_i^k,g_i)}\frac{\Phi(g_i^{n-j},g_i,g_i^{j})}{\Phi(g_i^{n-j},g_i^{j},g_i)}\prod_{k=1}^{j-1}\frac{1}{\widetilde{\Phi}_{g_i}(g_i^k,g_i)}=\prod_{k=1}^{j-1}\frac{1}{\Phi(g_i^{n-j+1},g_i^k,g_i)}.$$
On the other hand,
$$ \binom{n}{n-j+1}_{q_{ii}}+\binom{n}{n-j}_{q_{ii}}q_{ii}^{j}=\binom{n+1}{n-j+1}_{q_{ii}}.$$
The above calculation implies
$$ \Delta(x_i^{n+1})=\sum_{j=0}^{n+1} \binom{n+1}{n-j+1}_{q_{ii}}\prod_{k=1}^{j-1}\frac{1}{\Phi(g_i^{n-j+1},g_i^k,g_i)}x_i^{n-j+1} \ot x_i^{{j}}.$$
Since $\Phi$ is invertible, $\Delta(x^n)=0$ if and only if $\binom{n}{n-j}_{q_{ii}}=0$ for all $0 \leq j \leq n$. This is equivalent to $(n)!_{q_{ii}}=0$.
\end{proof}

\begin{lemma}\label{lem7.2}
    For any $i,j \in \mathbb{I}$, let $n \in \mathbb{N}_0$, then
    \begin{align*}
        &\delta(\operatorname{ad}(x_i)^n(x_j))=g_i^ng_j\ot \operatorname{ad}(x_i)^n(x_j),\\
        & g \rhd \operatorname{ad}(x_i)^n(x_j)=\prod_{k=0}^{n-1}\widetilde{\Phi}_{g}(g_i,g_i^kg_j)\chi_i^n(g)\chi_j(g)\operatorname{ad}(x_i)^n(x_j).
    \end{align*}
\end{lemma}
\begin{proof}
Both equations follow from induction on $n$ and the fact that $G$ is abelian.
\end{proof}

\begin{lemma}\label{lem7.3}
    For any $i,j \in \mathbb{I}$. The condition $\operatorname{ad}(x_i)^{n}(x_j)\neq 0$, but 
 $\operatorname{ad}(x_i)^{n+1}(x_j)=0$ is equivalent to 
 
 \begin{equation}
     (n+1)_{q_{ii}}(q_{ii}^nq_{ij}q_{ji}-1)=0, \ \ (k)_{q_{ii}}(q_{ii}^{k-1}q_{ij}q_{ji}-1)\neq 0, \ \text{for all} \ 1\leq k\leq n.
 \end{equation}
 \end{lemma}
\begin{proof}
   \textbf{ We claim that }
\begin{equation}\label{7.5}
    \begin{aligned}
        \Delta(\operatorname{ad}(x_i)^{n}(x_j))&=\operatorname{ad}(x_i)^{n}(x_j)\ot 1+1\ot \operatorname{ad}(x_i)^{n}(x_j)\\&+\sum_{k=1}^n\binom{n}{k}_{q_{ii}}\prod_{l=n-k}^{n-1}
(1-q_{ii}^lq_{ij}q_{ji})\prod_{q=1}^k\Phi(g_i^{q-1},g_i,g_i^{n-q}g_j)x_i^k\ot \operatorname{ad}(x_i)^{n-k}(x_j).       
    \end{aligned}
\end{equation}
The case $n=1$ is trivial, yielding
$$
 \Delta(\operatorname{ad}(x_i)(x_j))=\operatorname{ad}(x_i)(x_j)\ot 1+1\ot \operatorname{ad}(x_i)(x_j)+(1-q_{ij}q_{ji})x_i\ot x_j.
$$
For $n=2$, direct computation shows
\begin{align*}
\Delta(\operatorname{ad}(x_i)^2(x_j))&=\Delta(x_i)\Delta(\operatorname{ad}(x_i)(x_j))-\Delta(g_i \rhd (\operatorname{ad}(x_i)(x_j)))\Delta(x_i)\\
&=\operatorname{ad}(x_i)^2(x_j)\ot 1+1 \ot \operatorname{ad}(x_i)^2(x_j)+(1-q_{ii}^2q_{ij}q_{ji}+q_{ii}-q_{ii}q_{ij}q_{ji})x_i\ot \operatorname{ad}(x_i)(x_j)\\&+\Phi(g_i,g_i,g_j)(1-q_{ij}q_{ji})(1-q_{ii}q_{ij}q_{ji})x_i^2\ot x_j\\
&=\operatorname{ad}(x_i)^2(x_j)\ot 1+1 \ot \operatorname{ad}(x_i)^2(x_j)+(1+   q_{ii})(1-q_{ii}q_{ij}q_{ji})x_i\ot \operatorname{ad}(x_i)(x_j)\\&+\Phi(g_i,g_i,g_j)(1-q_{ij}q_{ji})(1-q_{ii}q_{ij}q_{ji})x_i^2\ot x_j.
\end{align*}
Now we assume (\ref{7.5}) holds for $\operatorname{ad}(x_i)^n(x_j).$ Then 
\begin{align*}
&\Delta(\operatorname{ad}(x_i)^{n+1}(x_j))=\Delta(x_i)\Delta(\operatorname{ad}(x_i)^n(x_j))-\Delta(g_i\rhd (\operatorname{ad}(x_i)^n(x_j)))\Delta(x_i)\\
&=\Delta(x_i)\Delta(\operatorname{ad}(x_i)^n(x_j))-\prod_{p=0}^{n-1}\widetilde{\Phi}_{g_i}(g_i,g_i^{n-1-p}g_j)q_{ii}^{n}q_{ij}\Delta(\operatorname{ad}(x_i)^n(x_j))\Delta(x_i)\\
&=\operatorname{ad}(x_i)^{n+1}(x_j)\ot 1+1\ot \operatorname{ad}(x_i)^{n+1}(x_j)\\&+x_i \ot \operatorname{ad}(x_i)^{n}(x_j)-\prod_{p=0}^{n-1}\widetilde{\Phi}_{g_i}(g_i,g_i^{n-1-p}g_j)q_{ii}^{n}q_{ij}(g_i^ng_j\rhd x_i)\ot \operatorname{ad}(x_i)^{n}(x_j)\\&+\prod_{p=0}^{n-1}\widetilde{\Phi}_{g_i}(g_i,g_i^{n-1-p}g_j)q_{ii}^{n}q_{ij}\operatorname{ad}(x_i)^n(x_j)\ot x_i-\prod_{p=0}^{n-1}\widetilde{\Phi}_{g_i}(g_i,g_i^{n-1-p}g_j)q_{ii}^{n}q_{ij}\operatorname{ad}(x_i)^n(x_j)\ot x_i\\
&+(1\ot x_i+x_i \ot 1)(\sum_{k=1}^n\binom{n}{k}_{q_{ii}}\prod_{l=n-k}^{n-1}
(1-q_{ii}^lq_{ij}q_{ji})\prod_{q=1}^k\Phi(g_i^{q-1},g_i,g_i^{n-q}g_j)x_i^k\ot \operatorname{ad}(x_i)^{n-k}(x_j))\\
&-\prod_{p=0}^{n-1}\widetilde{\Phi}_{g_i}(g_i,g_i^{n-1-p}g_j)q_{ii}^{n}q_{ij}(\sum_{k=1}^n\binom{n}{k}_{q_{ii}}\prod_{l=n-k}^{n-1}
(1-q_{ii}^lq_{ij}q_{ji})\prod_{q=1}^k\Phi(g_i^{q-1},g_i,g_i^{n-q}g_j)\\& \hspace{11cm} x_i^k\ot \operatorname{ad}(x_i)^{n-k}(x_j))\Delta(x_i).
\end{align*}
We focus on the coefficient of $x_i^{k+1} \ot \operatorname{ad}(x_i)^{n-k}(x_j)$, where $0 \leq k \leq n$. For $k=0$, this term equals:
\begin{align*}
&(1-q_{ii}^{n}q_{ij}q_{ii}^nq_{ji})x_i\ot \operatorname{ad}(x_i)^n(x_j)+\frac{\Phi(g_i,g_i,g_i^{n-1}g_j)
}{\Phi(g_i,g_i,g_i^{n-1}g_j)}\binom{n}{1}_{q_{ii}}(1-q_{ii}^{n-1}q_{ij}q_{ji})q_{ii}x_i\ot x_i\operatorname{ad}(x_i)^{n-1}(x_j)\\
&-\prod_{p=0}^{n-1}\widetilde{\Phi}_{g_i}(g_i,g_i^{n-1-p}g_j)q_{ii}^{n}q_{ij}\binom{n}{1}_{q_{ii}}(1-q_{ii}^{n-1}q_{ij}q_{ji})\frac{1}{\Phi(g_i,g_i^{n-1}g_j,g_i)}x_i\ot \operatorname{ad}(x_i)^{n-1}(x_j)x_i\\
&=(1-q_{ii}^{2n}q_{ij}q_{ji}+\binom{n}{1}_{q_{ii}}(1-q_{ii}^{n-1}q_{ij}q_{ji})q_{ii})x_i\ot \operatorname{ad}(x_i)^n(x_j)\\
&=\binom{n+1}{1}_{q_{ii}}(1-q_{ii}^{n}q_{ij}q_{ji})x_i \ot \operatorname{ad}(x_i)^n(x_j).
\end{align*}
Now for $1\leq k \leq n-1$, if we denote the coefficient $\prod_{p=0}^{n-1}\widetilde{\Phi}_{g_i}(g_i,g_i^{n-1-p}g_j)q_{ii}^{n}q_{ij}$ by $\gamma$, then the term $x_i^{k+1} \ot \operatorname{ad}(x_i)^{n-k}(x_j)$ is given by
\begin{align*}
   &(x_i\ot 1)(\binom{n}{k}_{q_{ii}}\prod_{l=n-k}^{n-1}
(1-q_{ii}^lq_{ij}q_{ji})\prod_{q=1}^k\Phi(g_i^{q-1},g_i,g_i^{n-q}g_j)x_i^k\ot \operatorname{ad}(x_i)^{n-k}(x_j))\\&+
(1\ot x_i)(\binom{n}{k+1}_{q_{ii}}\prod_{l=n-k-1}^{n-1}
(1-q_{ii}^lq_{ij}q_{ji})\prod_{q=1}^{k+1}\Phi(g_i^{q-1},g_i,g_i^{n-q}g_j)x_i^{k+1}\ot \operatorname{ad}(x_i)^{n-k-1}(x_j))\\
&-\gamma(\binom{n}{k}_{q_{ii}}\prod_{l=n-k}^{n-1}
(1-q_{ii}^lq_{ij}q_{ji})\prod_{q=1}^k\Phi(g_i^{q-1},g_i,g_i^{n-q}g_j)x_i^k\ot \operatorname{ad}(x_i)^{n-k}(x_j))(x_i \ot 1)\\
&-\gamma(\binom{n}{k+1}_{q_{ii}}\prod_{l=n-k-1}^{n-1}
(1-q_{ii}^lq_{ij}q_{ji})\prod_{q=1}^{k+1}\Phi(g_i^{q-1},g_i,g_i^{n-q}g_j)x_i^{k+1}\ot \operatorname{ad}(x_i)^{n-k-1}(x_j))(1 \ot x_i).
\end{align*}
Let $A_{1}, A_{2}, A_{3}, A_{4}$ denote these four terms, respectively.
The associators $\operatorname{Asso}(A_i)$ are as follows:
\begin{align*}
  &\operatorname{Asso}(A_1)=  \prod_{q=1}^k\Phi(g_i^{q-1},g_i,g_i^{n-q}g_j)\Phi(g_i,g_i^k,g_i^{n-k}g_j)
  \prod_{r=1}^{k-1}\Phi(g_i,g_i^r,g_i)=\prod_{q=1}^{k+1}\Phi(g_i^{q-1},g_i,g_i^{n+1-q}g_j),\\
  &\operatorname{Asso}(A_2)=\prod_{q=1}^{k+1}\Phi(g_i^{q-1},g_i,g_i^{n-q}g_j)\frac{\Phi(g_i,g_i^{k+1},g_i^{n-k-1}g_j)}{\Phi(g_i^{k+1},g_i,g_i^{n-k-1}g_j)}\prod_{r=1}^{k}\widetilde{\Phi}_{g_i}(g_i^r,g_i)=\prod_{q=1}^{k+1}\Phi(g_i^{q-1},g_i,g_i^{n+1-q}g_j),\\
  &\operatorname{Asso}(A_3)=\prod_{p=0}^{n-1}\widetilde{\Phi}_{g_i}(g_i,g_i^{n-1-p}g_j)\prod_{q=1}^k\Phi(g_i^{q-1},g_i,g_i^{n-q}g_j)\prod_{r=1}^{n-k} \widetilde{\Phi}_{g_i}^{-1}(g_i,g_i^{n-k-r}g_j)\\&=\prod_{q=1}^{k+1}\Phi(g_i^{q-1},g_i,g_i^{n+1-q}g_j),\\
  &\operatorname{Asso}(A_4)=\prod_{p=0}^{n-1}\widetilde{\Phi}_{g_i}(g_i,g_i^{n-1-p}g_j)\prod_{q=1}^{k+1}\Phi(g_i^{q-1},g_i,g_i^{n-q}g_j)\frac{1}{\Phi(g_i^{k+1},g_i^{n-k-1}g_j,g_i)}\\&=\prod_{r=1}^{n-k}\widetilde{\Phi}_{g_i}^{-1}(g_i,g_i^{n-k-r}g_j)\prod_{q=1}^{k+1}\Phi(g_i^{q-1},g_i,g_i^{n+1-q}g_j).
\end{align*}
Combining these terms yields
\begin{align*}
    &\prod_{q=1}^{k+1}\Phi(g_i^{q-1},g_i,g_i^{n+1-q}g_j)
      \Biggl[ \binom{n}{k}_{q_{ii}} \prod_{l=n-k}^{n-1} (1-q_{ii}^l q_{ij} q_{ji}) \\
    &\qquad + \binom{n}{k+1}_{q_{ii}} \prod_{l=n-k-1}^{n-1} (1-q_{ii}^l q_{ij} q_{ji}) \Biggr]
      x_i^{k+1} \otimes \operatorname{ad}(x_i)^{n-k}(x_j) \\
    &= \prod_{q=1}^{k+1}\Phi(g_i^{q-1},g_i,g_i^{n+1-q}g_j)
      \binom{n+1}{k+1}_{q_{ii}} \prod_{l=n-k}^{n} (1-q_{ii}^l q_{ij} q_{ji})
      x_i^{k+1} \otimes \operatorname{ad}(x_i)^{n-k}(x_j).
\end{align*}
For $k=n$, a similar computation shows the term is 
\begin{align*}
    \prod_{l=0}^{n}(1-q_{ii}^lq_{ij}q_{ji})\Phi(g_i^l,g_i,g_i^{n-l}g_j)x_i^{n+1}\ot x_j.
\end{align*}
This completes the proof of the claim (\ref{7.5}).

Now according to (\ref{7.5}), it is straightforward to see that  the condition $\operatorname{ad}(x_i)^{n}(x_j)\neq 0$, 
 $\operatorname{ad}(x_i)^{n+1}(x_j)=0$ is equivalent to 
 \begin{equation*}
     (n+1)_{q_{ii}}(q_{ii}^nq_{ij}q_{ji}-1)=0, \ \ (k)_{q_{ii}}(q_{ii}^{k-1}q_{ij}q_{ji}-1)\neq 0, \ \text{for all} \ 1 \leq k\leq n.
 \end{equation*}
\end{proof}
   \begin{cor}
       Let $M=(M_1,...,M_{\theta})$ be the $\theta$-tuple as above. Then $M$ admits the $i$-th reflection if and only if for each $j\neq i$, there is $n_{ij} \in \mathbb{N}_0$ such that
    $$ (n_{ij}+1)_{q_{ii}}(q_{ii}^{n_{ij}}q_{ij}q_{ji}-1)=0. $$ 
   \end{cor}
\begin{proof}
    This follows from Lemma \ref{lem7.3} immediately.
\end{proof}

\begin{lemma}
    Let \(i, j \in \mathbb{I}\). Assume  that \(M\) admits the $i$-th reflection. 

(1) If $j\neq i$, \(a_{ij}^M = -\min\{m \in \mathbb{N}_0 \mid (m+1)_{q_{ii}}(q_{ii}^m q_{ij} q_{ji} - 1) = 0\}\).

(2) $$R_i(M)_j = 
\begin{cases} 
\mathbbm{k}{x_i^*}, & \text{if } j = i, \\
\mathbbm{k} (\operatorname{ad} x_i)^{-a_{ij}^M}(x_j), & \text{if } j \neq i.
\end{cases}
$$ 

\end{lemma}

\begin{proof}
    (1) and (2) follows from Lemma \ref{lem7.3} and definition of reflection immediately.
    \end{proof}
Now we choose basis for $R_i(M)_j$. We define 
$$ y_j:=\begin{cases} 
{x_i^*}, & \text{if } j = i, \\
 (\operatorname{ad} x_i)^{-a_{ij}^M}(x_j), & \text{if } j \neq i.
\end{cases}$$ 
Here we choose $x_i^*$ such that $x_i^*(x_i)=1$.
\begin{lemma}
    With the above notation and assume that $M$ admits the $i$-th reflection. Then 
    \begin{equation}
        \delta(y_j) = g_j g_i^{-a_{ij}^M} \otimes y_j
    \end{equation} for all $j \in \mathbb{I}$. Meanwhile, for any $g \in G$, 
    \begin{equation}\label{7.10}
         g \rhd y_j=\begin{cases}
           \chi_i(g)^{-1}\Phi(g,g_i^{-1},g_i)^{-1}y_i,& \text{if} \ j=i,\\
           \prod_{k=0}^{-a^M_{ij}-1}\widetilde{\Phi}_{g}(g_i,g_i^kg_j)\chi_i^{-a_{ij}}(g)\chi_j(g)y_j, & \text{if} \ j\neq i.
        \end{cases}
    \end{equation}
\end{lemma}
\begin{proof}
    For $j \neq i$, since $G$ is abelian,   the  statement follows from Lemma \ref{lem7.2}.\par
    For $j=i$, note that the evaluation map $\operatorname{ev}:\mathbbm{k}y_i \ot \bbk x_i \rightarrow \bbk$ is given by 
    $$ \operatorname{ev}(y_i \ot x_i)=y_i(\alpha(g_i)x_i). $$ In $(\bbk G,\Phi)$, $\alpha(g)=1$ for all $g \in G$. Hence
    $\operatorname{ev}(y_i\ot x_i)=1$.
    Since $\operatorname{ev}$ is a morphism in $\GG$,  it follows that $$ 
   g \rhd \operatorname{ev}(y_i \ot x_i)=1.
    $$
    This implies  $\delta(y_i)=g_i^{-1}\ot y_i$.  Meanwhile, we have
    $$ \widetilde{\Phi}_g(g_i^{-1},g_i)<g \rhd y_i,g \rhd x_i>=1.$$
We denote $g \rhd y_i=\zeta(g)y_i$, then $\zeta(g)\chi(g)\widetilde{\Phi}_g(g_i^{-1},g_i)=1$. Hence 
$$  
\zeta(g)= \chi_i(g)^{-1}\widetilde{\Phi}_g^{-1}(g_i^{-1},g_i)=\chi_i(g)^{-1}\Phi(g,g_i^{-1},g_i)^{-1}.
$$
For further statement, Equation (\ref{7.3}) implies
    $$ \chi_i(g)\chi_i(g^{-1})\widetilde{\Phi}^{-1}_{g_i}(g,g^{-1})=1,$$
    we obtain that
    $$ \zeta(g)=\widetilde{\Phi}^{-1}_{g_i}(g,g^{-1})\widetilde{\Phi}^{-1}_{g}(g_i^{-1},g_i)\chi_i(g^{-1})=\Phi(g_i,g,g^{-1})^{-1}\Phi(g,g_i^{-1},g_i)^{-1}\chi_i(g^{-1}).   $$
\end{proof}
\begin{lemma}\label{lem 7.14}
    Let $i \in \mathbb{I}$ and $M$ admits the $i$-th reflection. \par 
    (1) Assume the braiding matrix of $R_i(M)$ is $(q_{jk}')_{j,k \in \mathbb{I}}$, then for $j,k \neq i$, \begin{align*}
    &q_{ii}'=q_{ii},\\
       & q_{ji}'=\prod_{l=0}^{-a^M_{ij}-1}\widetilde{\Phi}^{-1}_{g_i^{-1}}(g_jg_i^{l},g_i)\zeta(g_j)\zeta^{-a^M_{ij}}(g_i),\\
        &q_{ij}'=\prod_{l=0}^{-a^M_{ij}-1}\widetilde{\Phi}_{g_i^{-1}}(g_i,g_i^lg_j)\chi_i^{-a^M_{ij}}(g_i^{-1})\chi_j(g_i^{-1}),\\
        &q_{jk}'=\prod_{l=0}^{-a^M_{ik}-1}
        \widetilde{\Phi}_{g_i^{-a_{ij}^M}g_j}(g_i,g_i^lg_k)\\&\prod_{m=0}^{-a_{ij}^M-1}(\widetilde{\Phi}^{-1}_{g_i}(g_jg_i^m,g_i)\chi_i(g_i)^{-a_{ij}^M}\chi_i(g_j))^{-a_{ik}^M}\widetilde{\Phi}_{g_k}^{-1}(g_jg_i^m,g_i)\chi_k(g_i)^{-a^M_{ij}}\chi_k(g_j).
    \end{align*}\par 
    (2) The above elements $(q_{jk}')_{j,k \in \mathbb{I}}$ satisfy the following relations. 
    \begin{align*}
        &q_{jj}'=q_{jj}q_{ij}^{-a_{ij}^M}q_{ji}^{-a_{ij}^M}q_{ii}^{(a^M_{ij})^2},\\
        &q_{ij}'q_{ji}'=q_{ij}^{-1}q_{ji}^{-1}q_{ii}^{2a_{ij}^M},\\
        &q_{kj}'q_{jk}'=(q_{jk}q_{kj})(q_{ik}q_{ki})^{-a^M_{ij}}(q_{ij}q_{ji})^{-a^M_{ik}}q_{ii}^{2a^M_{ij}a^M_{ik}}.
    \end{align*}
\end{lemma}
\begin{proof}
    The proof proceeds by direct computation:\par 
    (1) We know that $q_{ii}'=\zeta(g_i^{-1})$, while
    \begin{align*}
\zeta(g_i^{-1})=\Phi(g_i,g_i^{-1},g_i)^{-1}\Phi(g_i^{-1},g_i^{-1},g_i)^{-1}\chi_i(g_i)=q_{ii}.
    \end{align*}
For $j,k \neq i$, 
\begin{align*}
    q_{ji}'&=\zeta(g_i^{-a_{ij}^M}g_j)\overset{(\ref{7.2})}{=}\prod_{l=0}^{-a^M_{ij}-1}\widetilde{\Phi}^{-1}_{g_i^{-1}}(g_jg_i^{l},g_i)\zeta(g_j)\zeta^{-a^M_{ij}}(g_i)\\
    q_{ij}'&\overset{(\ref{7.10})}{=} \prod_{l=0}^{-a^M_{ij}-1}\widetilde{\Phi}_{g_i^{-1}}(g_i,g_i^lg_j)\chi_i^{-a_{ij}^M}(g_i^{-1})\chi_j(g_i^{-1})\\
  q_{jk}'&\overset{(\ref{7.10})}{=}
\prod_{l=0}^{-a^M_{ik}-1}
\widetilde{\Phi}_{g_i^{-a_{ij}^M}g_j}(g_i,g_i^lg_k)
\chi_i^{-a_{ik}^M}(g_i^{-a_{ij}^M}g_j)
\chi_k(g_i^{-a_{ij}^M}g_j)\\
&\overset{(\ref{7.2})}{=}
\prod_{l=0}^{-a^M_{ik}-1}
\widetilde{\Phi}_{g_i^{-a_{ij}^M}g_j}(g_i,g_i^lg_k)
\prod_{m=0}^{-a_{ij}^M-1}
\bigl(
\widetilde{\Phi}^{-1}_{g_i}(g_jg_i^m,g_i)
\chi_i(g_i)^{-a_{ij}^M}\chi_i(g_j)
\bigr)^{-a_{ik}^M}\\
&\quad{}\cdot
\widetilde{\Phi}_{g_k}^{-1}(g_jg_i^m,g_i)
\chi_k(g_i)^{-a^M_{ij}}\chi_k(g_j)
\end{align*}
\par 
(2) Since $a_{ii}^M=2$, it is straightforward to see $q_{ii}'=q_{ii}^{1-2a_{ii}^M+(a_{ii}^M)^2}$. For $j \neq i$,
\begin{align*}
    q_{jj}'&=\prod_{l=0}^{-a^M_{ij}-1}\widetilde{\Phi}_{g_i^{-a_{ij}^M}g_j}(g_i,g_i^lg_j)\prod_{m=0}^{-a_{ij}^M-1}(\widetilde{\Phi}^{-1}_{g_i}(g_jg_i^m,g_i)\chi_i(g_i)^{-a_{ij}^M}\chi_i(g_j))^{-a_{ij}^M}\widetilde{\Phi}_{g_j}^{-1}(g_jg_i^m,g_i)\chi_j(g_i)^{-a^M_{ij}}\chi_j(g_j)\\
    &=q_{jj}q_{ij}^{-a_{ij}^M}q_{ji}^{-a_{ij}^M}q_{ii}^{(a^M_{ij})^2},
\end{align*}
and \begin{align*}
q_{ij}'q_{ji}'&=\prod_{l=0}^{-a^M_{ij}-1}\widetilde{\Phi}_{g_i^{-1}}(g_i,g_i^lg_j)\chi_i^{-a_{ij}^M}(g_i^{-1})\chi_j(g_i^{-1})\prod_{l=0}^{-a^M_{ij}-1}\widetilde{\Phi}^{-1}_{g_i^{-1}}(g_jg_i^{l},g_i)\zeta(g_j)\zeta^{-a^M_{ij}}(g_i)\\
&=\chi_i^{-a_{ij}^M}(g_i^{-1})\chi_j(g_i^{-1})\zeta(g_j)\zeta^{-a^M_{ij}}(g_i)\\
&=\chi_i^{-a_{ij}^M}(g_i^{-1})\chi_j(g_i^{-1})\chi_i(g_j)^{-1}\Phi(g_j,g_i^{-1},g_i)^{-1}(\chi_i(g_i)^{-1}\Phi(g_i,g_i^{-1},g_i)^{-1})^{-a^M_{ij}}\\
&=\chi_i^{-a_{ij}^M}(g_i^{-1})\chi_j(g_i^{-1})\chi_i(g_j)^{-1}\Phi(g_j,g_i^{-1},g_i)^{-1}(\chi_i(g_i)^{-1}\chi_i(g_i)^{-1}\chi_i(g_i^{-1})^{-1})^{-a^M_{ij}}\\
&=\chi_j(g_i^{-1})\chi_i(g_j)^{-1}\Phi(g_j,g_i^{-1},g_i)^{-1}\chi_i(g_i)^{2a_{ij}^M}\\
&=\chi_j(g_i)^{-1}\chi_i(g_j)^{-1}\chi_i(g_i)^{2a_{ij}^M}=q_{ij}^{-1}q_{ji}^{-1}q_{ii}^{2a_{ij}^M}.
\end{align*}
Now for $j,k\neq i$,
\begin{align*}
    q_{jk}'q_{kj}'&=\prod_{l=0}^{-a^M_{ik}-1}\widetilde{\Phi}_{g_i^{-a_{ij}^M}g_j}(g_i,g_i^lg_k)\prod_{m=0}^{-a_{ij}^M-1}(\widetilde{\Phi}^{-1}_{g_i}(g_jg_i^m,g_i)q_{ii}^{-a_{ij}^M}q_{ji})^{-a_{ik}^M}\widetilde{\Phi}_{g_k}^{-1}(g_jg_i^m,g_i)q_{ik}^{-a^M_{ij}}q_{jk}\\
    &\prod_{l=0}^{-a^M_{ij}-1}\widetilde{\Phi}_{g_i^{-a_{ik}^M}g_k}(g_i,g_i^lg_j)\prod_{m=0}^{-a_{ik}^M-1}(\widetilde{\Phi}^{-1}_{g_i}(g_kg_i^m,g_i)q_{ii}^{-a_{ik}^M}q_{ki})^{-a_{ij}^M}\widetilde{\Phi}_{g_j}^{-1}(g_kg_i^m,g_i)q_{ij}^{-a^M_{ik}}q_{kj}\\
    &=(q_{ii}^{-a_{ij}^M}q_{ji})^{-a_{ik}^M}q_{ik}^{-a^M_{ij}}q_{jk}(q_{ii}^{-a_{ik}^M}q_{ki})^{-a_{ij}^M}q_{ij}^{-a^M_{ik}}q_{kj}\\
    &=(q_{jk}q_{kj})(q_{ik}q_{ki})^{-a^M_{ij}}(q_{ij}q_{ji})^{-a^M_{ik}}q_{ii}^{2a^M_{ij}a^M_{ik}}.
\end{align*}

\end{proof}

\begin{definition}
    Let \( (V, c) \) be a finite-dimensional braided vector space of diagonal type and let \( q = (q_{ij})_{i,j \in \mathbb{I}} \) be a braiding matrix of \( V \). We say that \( (V, c) \) is  of {Cartan type}, if there exists a Cartan matrix \( A = (a_{ij})_{i,j \in \mathbb{I}} \) such that for all \( i,j \in \mathbb{I} \),

\begin{equation}
    \quad q_{ij} q_{ji} = q_{ii}^{a_{ij}}, \quad \text{where } 0 \leq -a_{ij} < \operatorname{ord}(q_{ii}) \text{ if } i \neq j.
\end{equation}
 
\end{definition}

\begin{lemma}\label{lem7.17}
Assume that \( M=(M_1,...M_{\theta}) \in \GG \) is of Cartan type with Cartan matrix \( A \). Then \( M \) admits the $i$-th reflection for all \( i \in \mathbb{I} \) and \( a_{ij}^M = a_{ij} \) for all \( i,j \in \mathbb{I} \).
\end{lemma}

\begin{proof}
Let \( i \in \mathbb{I} \). Recall that if $j\neq i$, \(a_{ij}^M = -\min\{m \in \mathbb{N}_0 \mid (m+1)_{q_{ii}}(q_{ii}^m q_{ij} q_{ji} - 1) = 0\}\). If $a_{ij}=0$, then it follows from the definition that $a_{ij}^M=0$. Now  if  $-a_{ij} \geq 1$, it is obvious that $-a_{ij}^M \leq -a_{ij}$. Suppose $-a_{ij}^M < -a_{ij}$. Since $q_{ij}q_{ji} = q_{ii}^{a_{ij}}$ and $0 \le -a_{ij} < \operatorname{ord}(q_{ii})$, the condition $-a_{ij}^{M} < -a_{ij}$ implies $q_{ii}^{-a_{ij}^{M}}q_{ij}q_{ji} - 1 = q_{ii}^{-a_{ij}^{M}}q_{ii}^{a_{ij}} - 1 \neq 0$. Therefore, to satisfy the minimality condition, we must have:$$(-a_{ij}^{M}+1)_{q_{ii}} = 0 \implies q_{ii}^{-a_{ij}^{M}+1} = 1.$$Consequently, we obtain $(q_{ij}q_{ji})^{-a_{ij}^{M}+1} = (q_{ii}^{a_{ij}})^{-a_{ij}^{M}+1} = 1$ for any $j \in \mathbb{I} \setminus \{i\}$. This forces $-a_{ij}^{M}+1$ to be a multiple of $\operatorname{ord}(q_{ii})$. Since $\operatorname{ord}(q_{ii}) > -a_{ij}$, it follows that:$$-a_{ij}^{M}+1 \ge \operatorname{ord}(q_{ii}) > -a_{ij}.$$This contradicts the assumption $-a_{ij}^{M} < -a_{ij}$. Thus, $a_{ij}^{M} = a_{ij}$ for all $i, j \in \mathbb{I}$.
\end{proof}

\begin{lemma}\label{lem7.18}
Assume that \( M=(M_1,...M_{\theta}) \in \GG \) is of Cartan type with Cartan matrix \( A \). Let \( i \in \mathbb{I} \), then the following hold.

(1) The labels of the generalized Dynkin diagram of  \( R_i(M) \) are
\[
q_{jj}' = q_{jj}, \quad q_{jk}' q_{kj}' = q_{jk} q_{kj}
\]
for all \( j,k \in \mathbb{I} \).

 (2) The tuple \( R_i(M) \) is of Cartan type with Cartan matrix \( A \).

\end{lemma}

\begin{proof}
By Lemma~\ref{lem7.17}, $M$ admits the $i$-th reflection, so
$R_i(M)$ is well-defined. Since
$q_{ii}^{a_{ij}^M}=q_{ij}q_{ji}$, the first statement follows from
Lemma~\ref{lem 7.14} {(2)}, and the second follows from the
definition of Cartan type.
\end{proof}

\begin{thm}
Let $M=(M_1,\ldots,M_\theta)\in\mathcal F_\theta$ be a tuple of
finite-dimensional simple objects in $\GG$, and assume that $M$ is
of Cartan type. Then the following statements hold.

 {(1)} The tuple $M$ admits all reflections.

 {(2)} Every $M'\in\mathcal F_\theta(M)$ is of Cartan type.

 {(3)} The Cartan graph of $M$ is standard; that is,
$
A^{[X]}=A^{[Y]}
$
for all $[X],[Y]\in\mathcal X$.

 {(4)} The Cartan graph of $M$ is finite if and only if
$A^M$ is of finite type.

 {(5)} The Nichols algebra $\mathcal B(M)$ is
finite-dimensional if and only if $A^M$ is of finite type and, for
every $i\in\mathbb I$, there exists $m\in\mathbb N_0$ such that
$
(m+1)_{q_{ii}}=0.
$
\end{thm}
\begin{proof}
  Statements ($1$), ($2$), and ($3$) follow from repeated applications of Lemma \ref{lem7.17} and Lemma \ref{lem7.18}.  Since $\mathcal{G}(M)$ is standard, it follows from the definition of Weyl groupoid that 
  $$ \operatorname{Hom}(\mathcal{W}(\mathcal{G}(M)),X)=W(A^M)$$
  for each $X \in \mathcal{X}$.
  Here $W(A)$ denotes the Weyl group of the Cartan matrix $A^M$. Hence, $\mathcal{G}(M)$ is finite if and only if $W(A)$ is finite, which is equivalent to $A^{M}$ being of finite type. This proves ($4$).

By Theorem~\ref{thm6.7}, $\mathcal B(M)$ is finite-dimensional if
and only if $M$ admits all reflections, $\mathcal G(M)$ is finite,
and $\bB(P_i)$ is finite-dimensional for every
$P\in\mathcal F_\theta(M)$ and $i\in\mathbb I$. Moreover, the
last condition is equivalent to the existence, for every
$i\in\mathbb I$, of an $m\in\mathbb N_0$ such that
$(m+1)_{q_{ii}}=0$. This proves  {(5)}.
\end{proof}
\subsection{Cartan graphs of diagonal tuples}
In this subsection, we compare the Cartan graphs associated with
diagonal tuples in \({}_G^G\mathcal{YD}^{\Phi}\), where \(G\) is
finite abelian and \(\Phi\) is an abelian \(3\)-cocycle, with those
associated with diagonal tuples in \({}_K^K\mathcal{YD}\), where
\(K\) is finite abelian. We first show that the first class properly
contains the second. We then prove that every Cartan graph in the
first class is a quotient of one in the second and compare their
real-root sets.
\begin{prop}\label{prop:nonordinary-diagonal-Cartan-graph}
There exists a rank-two diagonal tuple \(M\) over a finite abelian
group equipped with an abelian \(3\)-cocycle such that, for no finite
abelian group \(K\) and no diagonal tuple
\(P\in{}_K^K\mathcal{YD}\), the Cartan graphs
\(\mathcal G(M)\) and \(\mathcal G(P)\) are isomorphic. 
\end{prop}

\begin{proof}
Let
$
G=\mathbb Z_2=\langle g\mid g^2=e\rangle
$
and define
$
\Phi(g^a,g^b,g^c)=(-1)^{abc},\
 a,b,c\in\{0,1\}.
$
This is a normalized abelian \(3\)-cocycle on \(G\). Let \(\zeta\) be
a primitive fourth root of unity, and let \(S=\mathbbm{k}x\) be the
one-dimensional object in \({}_G^G\mathcal{YD}^{\Phi}\) defined by
\[
\delta(x)=g\otimes x,
\qquad
g^a\rhd x=\zeta^a x.
\]

The tensor square of \(S\) is isomorphic to the unit object. In fact,
its coaction is trivial, and Equation~\eqref{1.12} gives
\[
g^a\rhd(x\otimes x)
=
\frac{\Phi(g^a,g,g)\Phi(g,g,g^a)}
     {\Phi(g,g^a,g)}
\zeta^{2a}(x\otimes x)
=
x\otimes x.
\]
Thus
$
S\otimes S\cong\mathbbm{k},
S^*\cong S.
$

Consider the rank-two tuple
$
M=(S,S).
$
Its braiding matrix satisfies
$q_{rs}=\zeta,
r,s\in\{1,2\}.
$
In particular, \(q_{rs}q_{sr}=\zeta^2=-1\) for \(r\neq s\).
By Lemma~\ref{lem7.3}, the integer \(-a_{rs}^M\) is the least
\(m\in\mathbb N_0\) such that
\[
(m+1)_\zeta\bigl(-\zeta^m-1\bigr)=0.
\]
The expressions corresponding to \(m=0\) and \(m=1\) are nonzero,
whereas for \(m=2\) one has
$
-\zeta^2-1=0.
$
Hence
\[
a_{12}^M=a_{21}^M=-2.
\]

Moreover, \(\operatorname{ad}(S)^2(S)\neq0\). The morphism defining
this object has the one-dimensional source \(S^{\otimes3}\), and
therefore its nonzero image is isomorphic to \(S^{\otimes3}\). Since
\(S^{\otimes2}\cong\mathbbm{k}\), it follows that
$
\operatorname{ad}(S)^2(S)\cong S.
$
Together with \(S^*\cong S\), this gives
\[
R_1(M)\cong M,
\qquad
R_2(M)\cong M.
\]
Consequently, \(\mathcal G(M)\) has one object and its Cartan matrix
is
\[
A=
\begin{pmatrix}
2&-2\\
-2&2
\end{pmatrix}.
\]

Suppose now that \(\mathcal G(M)\) were isomorphic to
\(\mathcal G(P)\) for some diagonal tuple
$
P=(P_1,P_2)\in{}_K^K\mathcal{YD},
$
where \(K\) is an finite abelian group. Since
\(\mathcal G(M)\) has one object, so does \(\mathcal G(P)\). Hence
$
[R_r(P)]=[P]
$
for $r=1,2$.
By the definition of isomorphism classes of tuples, comparison of
the \(r\)-th components gives
$
P_r^*\cong P_r.
$

Write \(P_r=\mathbbm{k}y_r\), where
$
\delta(y_r)=h_r\otimes y_r, \
h\rhd y_r=\chi_r(h)y_r
$
for some \(h_r\in K\) and \(\chi_r\in\widehat K\). In the ordinary
Yetter--Drinfeld category, \(P_r^*\) has degree \(h_r^{-1}\) and
character \(\chi_r^{-1}\). Therefore \(P_r^*\cong P_r\) implies
$h_r=h_r^{-1},\
\chi_r=\chi_r^{-1}.
$
In particular, if \(q_{rr}=\chi_r(h_r)\), then
$
q_{rr}^2=1.
$

Let \(r\neq s\). If \(q_{rr}=1\), then
\[
(m+1)_{q_{rr}}=m+1\neq0
\]
for every \(m\in\mathbb N_0\). Since \(P\) admits the \(r\)-th
reflection, the defining formula for \(a_{rs}^P\) forces
\(q_{rs}q_{sr}=1\), and hence \(a_{rs}^P=0\). If \(q_{rr}=-1\), then
$
(2)_{-1}=0,
$
so the same formula gives
$
-a_{rs}^P\leq1.
$
Thus, in either case,
\[
a_{rs}^P\neq-2.
\]
This contradicts the fact that the unique Cartan matrix of
\(\mathcal G(P)\) must be
\[
\begin{pmatrix}
2&-2\\
-2&2
\end{pmatrix}.
\]
Therefore \(\mathcal G(M)\) cannot arise from any diagonal tuple in an
ordinary Yetter--Drinfeld category over a finite abelian group.
\end{proof}

\begin{cor}
The isomorphism classes of Cartan graphs arising from diagonal tuples
in ordinary Yetter--Drinfeld categories over finite abelian groups form
a proper subset of those arising from diagonal tuples over finite
abelian groups equipped with abelian \(3\)-cocycles.
\end{cor}

\begin{proof}
The inclusion follows by taking the \(3\)-cocycle to be trivial. It is
proper by Proposition~\ref{prop:nonordinary-diagonal-Cartan-graph}.
\end{proof}

\begin{rmk}\upshape
The Cartan graph in the preceding proposition can also be obtained
from a tuple with non-isomorphic components. Let
\[
G=\mathbb Z_2\times\mathbb Z_2=\langle g,h\rangle,
\qquad
\Phi(g^ah^b,g^ch^d,g^eh^f)=(-1)^{ace},
\]
and let \(\zeta\) be a primitive fourth root of unity. Put
\(u_1=g\) and \(u_2=gh\), and define \(M_r=\mathbbm{k}x_r\),
\(r=1,2\), by
$
\delta(x_r)=u_r\otimes x_r
$,
$(g^ah^b)\rhd x_r=\zeta^a x_r.
$
 Their coaction degrees are different,
and hence \(M_1\not\cong M_2\). A direct calculation gives
\[
M_1^{\otimes2}\cong M_2^{\otimes2}\cong\mathbbm{k},
\qquad
q_{11}=q_{12}=q_{21}=q_{22}=\zeta.
\]
Thus both \(M_1\) and \(M_2\) are self-dual. As in the proof of
Proposition~\ref{prop:nonordinary-diagonal-Cartan-graph}, the Cartan
entries are \(a_{12}=a_{21}=-2\). Moreover,
\[
\operatorname{ad}(M_1)^2(M_2)\cong M_2,
\qquad
\operatorname{ad}(M_2)^2(M_1)\cong M_1.
\]
Consequently, \(\mathcal G(M_1,M_2)\) is the same one-object Cartan
graph as in the preceding proposition, although
\(M_1\not\cong M_2\).
\end{rmk}

We now establish the general relationship between the two settings
by constructing a covering of Cartan graphs.

Write
\[
G=\mathbb Z_{m_1}\times\cdots\times\mathbb Z_{m_n}
 =\langle g_1\rangle\times\cdots\times\langle g_n\rangle,
\qquad m_i\mid m_{i+1}.
\]
By \cite[Proposition~3.8]{QQG}, every normalized \(3\)-cocycle on
\(G\) is cohomologous to one of the standard cocycles
\(\Phi_{\underline a}\). Since cohomologous cocycles give braided
monoidally equivalent Yetter--Drinfeld categories, Corollary
\ref{cor7.6} allows us to assume that the cocycle is of this form.

Suppose that \(\Phi_{\underline a}\) is abelian. Set
\[
\mathbb G
 =\mathbb Z_{m_1^2}\times\cdots\times\mathbb Z_{m_n^2}
 =\langle\sg_1\rangle\times\cdots\times\langle\sg_n\rangle
\]
and define
\[
\pi:\mathbb G\longrightarrow G,\qquad \pi(\sg_i)=g_i.
\]
Choose the set-theoretic section
\[
\iota(g_1^{r_1}\cdots g_n^{r_n})
 =\sg_1^{r_1}\cdots\sg_n^{r_n},
\qquad 0\leq r_i<m_i.
\]
By \cite[Proposition~3.15]{QQG}, there is a normalized \(2\)-cochain
\(J_{\underline a}\) on \(\mathbb G\) such that
\[
\pi^*(\Phi_{\underline a})=\partial J_{\underline a}.
\]

Let \(M=(M_1,\ldots,M_\theta)\in\GGAA\). Following
\cite[Lemma~4.6 and Proposition~4.7]{QQG}, define
\(\widetilde M\in\BGGA\) by
\begin{equation}\label{7.12}
\rho(x_i)=\widetilde d_i\otimes x_i,
\qquad
a\blacktriangleright x_i=\pi(a)\rhd x_i,
\end{equation}
The Nichols algebra \(\mathcal{B}(\widetilde{M}) \in \BGGA \) is isomorphic to \(\mathcal{B}(M)\) in the sense of [\citealp{QQG}, Definition 4.3], while
\[
M'=\widetilde M^{J_{\underline a}^{-1}}
\]
is a diagonal tuple in
\({}_{\mathbb G}^{\mathbb G}\mathcal{YD}\).

\begin{definition}\textup{[\citealp{rootsys}, Definition 10.1.1]}
Let
$\mathcal G=\mathcal G(\mathbb I,\mathcal X,r,A),
\mathcal G'=\mathcal G(\mathbb I,\mathcal Y,t,B)$
be semi-Cartan graphs, and let \(p:\mathcal Y\to\mathcal X\) be a map.
We say that \(\mathcal G'\) is a covering of \(\mathcal G\), and that
\(p\) is a {covering map}, if \(p\) is surjective and
\[
p(t_i(Y))=r_i(p(Y)),
\qquad
b_{jk}^{Y}=a_{jk}^{p(Y)}
\]
for all \(i,j,k\in\mathbb I\) and \(Y\in\mathcal Y\).  In this case,
\(\mathcal G\) is called a quotient semi-Cartan graph of \(\mathcal G'\).
\end{definition}

\begin{prop}\label{prop7.19}
The tuple \(M\in\GGAA\) admits all reflections if and only if
\(\widetilde M\in\BGGA\) admits all reflections. If these equivalent conditions hold, then
\(\mathcal G(\widetilde M)\) is a covering of \(\mathcal G(M)\).
\end{prop}

\begin{proof}
Write \(\iota(d_i)=\widetilde d_i\). For all \(i,j\in\mathbb I\), we have
$$\rho(x_i)=\widetilde d_i\otimes x_i,
\qquad
\widetilde d_j\blacktriangleright x_i
=d_j\rhd x_i
=q_{ji}x_i.$$
Thus \(M\) and \(\widetilde M\) have the same braiding matrix.

We claim that, along every reflection sequence, the corresponding
tuples are simultaneously defined and their braiding matrices
\((q_{jk})\) and \((\widetilde q_{jk})\) satisfy
\begin{equation}\label{eq:braiding-data-under-lifting}
q_{jj}=\widetilde q_{jj},
\qquad
q_{jk}q_{kj}=\widetilde q_{jk}\widetilde q_{kj}
\end{equation}
for all \(j,k\in\mathbb I\). We prove this by induction on the length
of the reflection sequence. The assertion holds for \(M\) and
\(\widetilde M\). Suppose that it holds for two corresponding tuples
\(P\) and \(\widetilde P\). By Lemma~\ref{lem7.3}, \(P\) admits the
\(i\)-th reflection if and only if \(\widetilde P\) does, and in this
case
\[
a_{jk}^{P}=a_{jk}^{\widetilde P}
\qquad j,k\in\mathbb I.
\]
If these reflections exist, the formulas in
Lemma~\ref{lem 7.14}(2) show that the braiding matrices of \(R_i(P)\)
and \(R_i(\widetilde P)\) again satisfy
\eqref{eq:braiding-data-under-lifting}. This proves the claim.

Consequently, \(M\) admits all reflections if and only if
\(\widetilde M\) does. Moreover, corresponding tuples obtained by the
same reflection sequence have the same Cartan matrix.

Assume that these equivalent conditions hold. Let
$
Y=(Y_1,\ldots,Y_\theta)\in\mathcal F_\theta(\widetilde M).
$
Each \(Y_j\) is one-dimensional. Choose \(0\neq y_j\in Y_j\) and write
\[
\rho(y_j)=h_j\ot y_j,
\qquad
a\blacktriangleright y_j=\chi_j^Y(a)y_j
\]
for \(h_j,a\in\mathbb G\). We claim that \(\chi_j^Y(a)\) depends only
on \(\pi(a)\). For \(Y=\widetilde M\), this follows from
\eqref{7.12}. The claim is preserved under reflections, since the
components of a reflected tuple are obtained by taking duals and
iterated braided adjoints, whose action formulas involve only the
action scalars of the original components and the values of
\(\pi^*(\Phi_{\underline a})\). Since
\[
\pi^*(\Phi_{\underline a})(a,b,c)
 =\Phi_{\underline a}(\pi(a),\pi(b),\pi(c)),
\]
the claim follows by induction on the length of a reflection sequence.

Define a one-dimensional object
\(\overline{Y_j}\in\GGAA\), with the same underlying vector space as
\(Y_j\), by
\[
\delta_{\overline{Y_j}}(y_j)=\pi(h_j)\ot y_j,
\qquad
g\rhd y_j=\chi_j^Y(a)y_j,
\]
where \(a\in\mathbb G\) satisfies \(\pi(a)=g\). The preceding claim
shows that the action is well-defined. The Yetter--Drinfeld
compatibility follows from that of \(Y_j\), because the associator of
\(\BGGA\) is the pullback of \(\Phi_{\underline a}\) along \(\pi\).

Set
$
\overline Y
 =(\overline{Y_1},\ldots,\overline{Y_\theta}).
$
The definitions of duals and braided adjoint actions, together with
\(\pi^*(\Phi_{\underline a})=\Phi_{\underline a}\circ\pi^{\otimes3}\),
give
$
\overline{Y_i^*}\cong\overline{Y_i}^{\,*}
$,
$\overline{\operatorname{ad}(Y_i)^n(Y_j)}
 \cong
\operatorname{ad}(\overline{Y_i})^n(\overline{Y_j}).
$
Hence
\[
\overline{R_i(Y)}\cong R_i(\overline Y),
\qquad
a_{jk}^{Y}=a_{jk}^{\overline Y}
\]
for all \(i,j,k\in\mathbb I\). In particular, if
$
Y=R_{i_l}\cdots R_{i_1}(\widetilde M),
$
then
$
\overline Y\cong R_{i_l}\cdots R_{i_1}(M).
$

Let
\[
\mathcal X
 =\{[X]\mid X\in\mathcal F_\theta(M)\},
\qquad
\mathcal Y
 =\{[Y]\mid Y\in\mathcal F_\theta(\widetilde M)\},
\]
and define
\[
p:\mathcal Y\longrightarrow\mathcal X,
\qquad
p([Y])=[\overline Y].
\]
If \(Y\cong Y'\), then \(\overline Y\cong\overline{Y'}\), so \(p\) is
well-defined. It is surjective because every reflection sequence
starting from \(M\) is the image under \(p\) of the same reflection
sequence starting from \(\widetilde M\).

Let \(r_i\) and \(t_i\) denote the reflection maps on \(\mathcal X\)
and \(\mathcal Y\), respectively. Then
\[
p(t_i([Y]))
 =[\overline{R_i(Y)}]
 =[R_i(\overline Y)]
 =r_i(p([Y]))
\]
and
\[
a_{jk}^{[Y]}=a_{jk}^{[\overline Y]}=a_{jk}^{p([Y])}.
\]
Therefore
$
(\operatorname{id}_{\mathbb I},p):
\mathcal G(\widetilde M)\longrightarrow\mathcal G(M)
$
is a covering of Cartan graphs.
\end{proof}
The covering in Proposition~\ref{prop7.19} also gives a precise
relationship between the corresponding Weyl groupoids and real roots.

\begin{cor}
Let
$
p:\mathcal G(\widetilde M)\longrightarrow\mathcal G(M)
$
be the covering map constructed in Proposition~\ref{prop7.19}. Then
\(p\) induces a covering functor
\[
F_p:
\mathcal W(\mathcal G(\widetilde M))
\longrightarrow
\mathcal W(\mathcal G(M))
\]
which is the identity on the linear parts of the morphisms. Moreover,
for every \([Y]\in\mathcal Y\),
\[
\Delta^{[Y],\operatorname{re}}
 =
\Delta^{p([Y]),\operatorname{re}}
\]
as subsets of \(\mathbb Z^{\mathbb I}\).
\end{cor}

\begin{proof}
The covering functor is determined on the generating reflections by
\[
F_p(s_i^{[Y]})=s_i^{p([Y])}.
\]
This is well-defined because
\[
A^{[Y]}=A^{p([Y])},
\qquad
p(t_i([Y]))=r_i(p([Y])).
\]
In particular, a morphism and its image under \(F_p\) have the same
linear part in \(\operatorname{Aut}(\mathbb Z^{\mathbb I})\).

Let
\(\beta\in\Delta^{[Y],\operatorname{re}}\). Then
\(\beta=w(\alpha_i)\) for some morphism \(w\) with target \([Y]\).
Applying \(F_p\) gives a morphism with target \(p([Y])\) and the same
linear part. Hence
\[
\Delta^{[Y],\operatorname{re}}
 \subseteq
\Delta^{p([Y]),\operatorname{re}}.
\]

Conversely, let
\(\beta=w(\alpha_i)\in\Delta^{p([Y]),\operatorname{re}}\), and choose
an expression of \(w\) as a product of fundamental reflections.
Starting from \([Y]\) and lifting this reflection sequence backwards
gives a morphism in
\(\mathcal W(\mathcal G(\widetilde M))\) with target \([Y]\).
Since the Cartan matrices at corresponding objects agree, the lifted
morphism has the same linear part \(w\). Thus
\(\beta\in\Delta^{[Y],\operatorname{re}}\), proving the reverse
inclusion.
\end{proof}

\begin{thm}\label{thm7.20}
Let \(M=(M_1,\ldots,M_\theta)\in\GGAA\) be a tuple of Yetter--Drinfeld
modules of diagonal type, and put
\[
 M'=\widetilde M^{J_{\underline a}^{-1}}\in\BGGo.
\]
Then \(M\) admits all reflections if and only if \(M'\) admits all
reflections. If these equivalent conditions hold, then
\(\mathcal G(M')\) is a covering of \(\mathcal G(M)\).
\end{thm}

\begin{proof}
By Proposition~\ref{prop7.19}, \(M\) admits all reflections if and
only if \(\widetilde M\) does. On the other hand, twisting by
\(J_{\underline a}^{-1}\) induces a braided monoidal equivalence.
Hence Corollary~\ref{cor7.6} implies that \(\widetilde M\) admits all
reflections if and only if
\[
 M'=\widetilde M^{J_{\underline a}^{-1}}
\]
does. Combining the two equivalences proves the first assertion.

If these conditions hold, Corollary~\ref{cor7.6} gives an isomorphism
$
 \mathcal G(M')\simeq\mathcal G(\widetilde M).
$
Composing this isomorphism with the covering
$
 \mathcal G(\widetilde M)\longrightarrow\mathcal G(M)
$
from Proposition~\ref{prop7.19}, we obtain the required covering
\[
 \mathcal G(M')\longrightarrow\mathcal G(M).
\]
\end{proof}

\begin{cor}
Every Cartan graph arising from a diagonal-type Nichols algebra over
a finite abelian group with an abelian \(3\)-cocycle is a quotient
Cartan graph of one arising from a diagonal-type Nichols algebra over
an finite abelian group. At corresponding objects, their
real-root sets coincide.
\end{cor}

\begin{proof}
It follows from Theorem~\ref{thm7.20}. The converse
follows because a group algebra with trivial associator is the special case in
which the associator is trivial.
\end{proof}

\subsection{Root systems and GK-dimension for Nichols algebras of
diagonal type}
It is well known that a braided vector space of diagonal type has an
associated generalized Dynkin diagram. We recall its definition in
the present setting.

\begin{definition}
Let \(V\) be a \(\theta\)-dimensional braided vector space of diagonal
type. Let \((x_i)_{i\in\mathbb I}\) be a basis of \(V\), and let
\((q_{ij})_{i,j\in\mathbb I}\) be its braiding matrix, so that
\[
c_{V,V}(x_i\ot x_j)=q_{ij}x_j\ot x_i
\]
for all \(i,j\in\mathbb I\). The {generalized Dynkin diagram}
of \(V\) is the labeled graph with vertices indexed by \(\mathbb I\),
where the \(i\)-th vertex is labeled by \(q_{ii}\). For
\(1\leq i<j\leq\theta\), the vertices \(i\) and \(j\) are joined by an
edge if and only if \(q_{ij}q_{ji}\neq1\), in which case the edge is
labeled by \(q_{ij}q_{ji}\).

If
$
M=(M_1,\ldots,M_\theta)\in\mathcal F_\theta
$
and \(\dim M_i=1\) for all \(i\in\mathbb I\), then the generalized
Dynkin diagram of \(M\) is the generalized Dynkin diagram of
\(M_1\oplus\cdots\oplus M_\theta\) with respect to a basis
\((x_i)_{i\in\mathbb I}\), where \(0\neq x_i\in M_i\).
\end{definition}

\begin{cor}
Let
$
M=(M_1,\ldots,M_\theta)
$
be of diagonal type in $\GGAA$ and admit all reflections. Then
\(\mathcal B(M)\) has a finite real root system if and only if the
generalized Dynkin diagram of \(M\) occurs in the classification of
\cite{Hec09}.
\end{cor}

\begin{proof}
Let
$
M'=\widetilde M^{J_{\underline a}^{-1}}
$
be the tuple in \(\BGGo\) from Theorem~\ref{thm7.20}. If $(q'_{ij})$ is the braiding matrix of $M'$, then the
the construction of \(M'\) gives gives
\[
q'_{ii}=q_{ii},
\qquad
q'_{ij}q'_{ji}=q_{ij}q_{ji}
\]
for all \(i,j\in\mathbb I\). Thus \(M\) and \(M'\) have the same
generalized Dynkin diagram. Theorem~\ref{thm7.20} also gives
\[
\Delta^{[M],\operatorname{re}}
=
\Delta^{[M'],\operatorname{re}}.
\]
The conclusion now follows from the classification of finite
arithmetic root systems in \cite{Hec09}.
\end{proof}

For Nichols algebras of diagonal type in
\({}_K^K\mathcal{YD}\), where \(K\) is a finite abelian group,
Kharchenko's PBW theorem associates a set of positive roots with the
multidegrees of the PBW generators. By
\cite[Corollary~6.12 and Theorem~6.11]{HS10}, together with
\cite[Theorem~10.4.7]{rootsys}, this set is finite if and only if the
tuple admits all reflections and its associated real-root system is
finite. Hence \cite{AG25} gives
\[
\operatorname{GKdim}\mathcal B(V)<\infty
\quad\Longleftrightarrow\quad
V\text{ admits all reflections and }\mathcal G(V)\text{ is finite}.
\]
\begin{cor}
Let \(M=(M_1,\ldots,M_\theta)\) be a tuple of Yetter--Drinfeld
modules of diagonal type in $\GGAA$ . Then the following statements are equivalent.

(1)
    \(\operatorname{GKdim}\mathcal B(M)<\infty\).
    
    (2) \(M\) admits all reflections and \(\mathcal G(M)\) is a finite
    Cartan graph.

\end{cor}

\begin{proof}
Let
$
 M'=\widetilde M^{J_{\underline a}^{-1}}
$
be the tuple of diagonal type over an ordinary Hopf algebra constructed
above. Notice that the construction of \(M'\) does not require \(M\)
to admit all reflections.

The construction of \(M'\) preserves the
\(\mathbb N_0^{\mathbb I}\)-graded dimensions of the corresponding
Nichols algebras. Hence \(\mathcal B(M)\) and \(\mathcal B(M')\) have
the same multigraded Hilbert series, and therefore
\[
 \operatorname{GKdim}\mathcal B(M)
 =
 \operatorname{GKdim}\mathcal B(M').
\]
By \cite{AG25} and the relation between PBW roots and real roots
recalled above,
\[
 \operatorname{GKdim}\mathcal B(M')<\infty
 \quad\Longleftrightarrow\quad
 M'\text{ admits all reflections and }\mathcal G(M')
 \text{ is finite}.
\]

By Theorem~\ref{thm7.20}, \(M\) admits all reflections if and only if
\(M'\) does. If these equivalent conditions hold, the covering map
$
 \pi:\mathcal G(M')\longrightarrow\mathcal G(M)
$
preserves the Cartan matrices and the reflection maps. Consequently, for every object $[Y]$ of $\mathcal G(M')$,
\[
\Delta^{[Y],\operatorname{re}}
=
\Delta^{\pi([Y]),\operatorname{re}}.
\]
Since \(\pi\) is surjective, it follows that
\[
 \mathcal G(M')\text{ is finite}
 \quad\Longleftrightarrow\quad
 \mathcal G(M)\text{ is finite}.
\]
Combining these equivalences proves the assertion.
\end{proof}
\subsection{A non-diagonal affine Nichols algebra}
In this section, we consider a rank three Nichols algebra $\mathcal{B}(M)$ of non-diagonal type, which plays a central role in the classification of finite-dimensional Nichols algebras in $\GG$, where $G$ is a finite abelian group and $\Phi$ is a $3$-cocycle on $G$. It was proved to be infinite-dimensional in different ways in [\citealp{huang2024classification}, \citealp{LiLiu1}, \citealp{reflection1}]. Moreover, it gives rise to a semi-Cartan graph $\mathcal{G}(M)$. In this section, we prove $\mathcal{G}(M)$ is an affine Cartan graph. Moreover, we show that $\mathcal{B}(M)$ is affine in the sense of \cite{affinenichols2}.

We begin by recalling the specific Nichols algebra $\mathcal{B}(M)$ under consideration.
 Let
\[
G=\mathbb Z_2\times\mathbb Z_2\times\mathbb Z_2
=\langle h_1\rangle\times\langle h_2\rangle\times\langle h_3\rangle,
\]
and define a $3$-cocycle $\Phi$ on $G$ by
\[
\Phi\left(
h_1^{i_1}h_2^{i_2}h_3^{i_3},
h_1^{j_1}h_2^{j_2}h_3^{j_3},
h_1^{k_1}h_2^{k_2}h_3^{k_3}
\right)
=(-1)^{k_1j_2i_3},
\]
where $0\leq i_l,j_l,k_l\leq1$ for $1\leq l\leq3$.

By [\citealp{huang2024classification}, Lemma 3.5], we may give a complete list of simple Yetter-Drinfeld modules over $\GG$.  In particular, when restricting to those simple Yetter-Drinfeld modules that generate finite-dimensional Nichols algebras, we see that the isomorphism classes are given by the set
\begin{equation}
    S=\{ [^hM]\mid \ ^hM \ \text{is simple},\  h \in \mathbb{Z}_2 \times \mathbb{Z}_2 \times \mathbb{Z}_2-\{1,h_1h_2h_3\}\}.
\end{equation}
There are exactly six isomorphism classes. Here $^hM $ represents that the simple Yetter–Drinfeld module has comodule structure
$$^hM:=\lbrace w\in M\mid \delta_M(w)=h\otimes w\rbrace.$$ 
We now choose explicit representatives of $S$. For $1\leq i \leq 6$, let $M_i\in S $ be pairwise non-isomorphic simple modules such that $\operatorname{dim}(M_i)=2$, and  $\operatorname{deg}(M_1)=h_1$, $\operatorname{deg}(M_2)=h_2$, $\operatorname{deg}(M_3)=h_3$, $\operatorname{deg}(M_4)=h_1h_2$, $\operatorname{deg}(M_5)=h_1h_3$, $\operatorname{deg}(M_6)=h_2h_3$. 
One may refer to \cite{reflection1} for the explicit expression of the Yetter-Drinfeld structure, we omit here for simplicity.\par 
The Nichols algebra $\mathcal{B}(M)$ which we are concerned with is generated by $M_1 \oplus M_2 \oplus M_3$.  The following result was proved in [\citealp{reflection1}].
\begin{thm}\label{thm 4.10}
\textup{[\citealp{reflection1}, Theorem~6.11]}
Let $M=(M_1,M_2,M_3)$. Then $M$ admits all reflections and
$\mathcal G(M)$ is a standard semi-Cartan graph with Cartan matrix
\[
\begin{pmatrix}
2 & -1 & -1\\
-1 & 2 & -1\\
-1 & -1 & 2
\end{pmatrix}.
\]
\end{thm}
Since $\mathcal G(M)$ is standard,
$
\Delta^{[X],\operatorname{re}}
=
\Delta^{[Y],\operatorname{re}}
$
for all $X,Y\in\mathcal F_3(M)$, and the map
\[
\operatorname{Hom}
\bigl(\mathcal W(\mathcal G(M)),[X]\bigr)
\longrightarrow W(A_2^{(1)}),
\qquad
([X],s,[Y])\longmapsto s,
\]
is bijective, where $ W(A_2^{(1)})$ is the Weyl group of the affine Lie algebra  of type $A_2^{(1)}$.\par 
For an affine Lie algebra of type $A_2^{(1)}$ (and more generally for any symmetrizable Kac–Moody algebra), the set of real roots is closed under the action of the Weyl group, and every real root can be obtained from a simple root by applying a sequence of fundamental reflections, see [\citealp{Kacliealg}, Proposition 5.1]. If we denote $\delta=\alpha_1+\alpha_2+\alpha_3$. Then by [\citealp{Kacliealg}, Proposition 6.3(a)]
\[
\Delta^{[M],\operatorname{re}}
=
\{\alpha+k\delta\mid
\alpha\in\Phi_{A_2},\ k\in\mathbb Z\}.
\]
Here   $\Phi_{A_2}=\pm \{  \alpha_2, \alpha_3, \alpha_2+\alpha_3\}$, which corresponds to the root of Lie algebra of type $A_2$.
 This argument is expected because the semi-Cartan graph is standard, and thus $\Delta^{[M],\operatorname{re}}$ coincides with the set of real roots for the affine Lie algebra of type
$A_2^{(1)}$. Unfortunately, $\Delta^{[M],\operatorname{re}}$ is an infinite set, by definition $\mathcal{G}(M)$ is not a finite semi-Cartan graph.\par 
\begin{definition}
 {(1)} A semi-Cartan graph $\mathcal{G}$ is called connected if  the groupoid $\mathcal{W}(\mathcal{G})$ is
connected, that is, if for any two objects $X, Y$ of $\mathcal{G}$ there is a morphism from $X$ to
$Y$ in $\mathcal{W}(\mathcal{G})$.\par
 {(2)} A semi-Cartan graph $\mathcal{G}$ is called simply connected if for
any two points $X, Y$ of $\mathcal{G}$ there is at most one morphism from $X$ to
$Y$ in $\mathcal{W}(\mathcal{G})$.\par
\end{definition}
\begin{cor}
The Cartan graph $\mathcal G(M)$ is connected.
\end{cor}
\begin{proof}
    It follows from Theorem \ref{thm5.6} immediately.
\end{proof}
If $\mathcal{G}$ is a  Cartan graph,    the pair $(\mathcal{G}, (\Delta^{X,\operatorname{re}})_{X \in \mathcal{X}} )$ is a reduced root system over $\mathcal{G}$ by \textup{[\citealp{rootsys}, Example 10.4.4]}. The following corollary follows immediately.
\begin{cor}
For the tuple $M$ above,
$
\bigl(
\mathcal G(M),
(\Delta^{X,\operatorname{re}})_{X\in\mathcal X}
\bigr)
$
is a reduced root system over $\mathcal G(M)$. Moreover,
$\mathcal G(M)$ admits no finite root system.
\end{cor}
\begin{proof}
    The second claim follows from the fact that $\mathcal{G}$ is not a finite Cartan graph and [\citealp{rootsys}, Theorem 10.4.7].
\end{proof}
For a connected Cartan graph permitting a reduced root system, we
call the crystallographic Tits arrangement corresponding to its
connected simply connected covering its associated Tits arrangement.
We call a Nichols algebra affine if the associated Tits arrangement
of its Cartan graph is affine in the sense of
\cite[Definition~2.7]{affinenichols2}.

According to [\citealp{rootsys}, Section 10.1], we can take a covering $\mathcal{G}'$ of $\mathcal{G}(M)$. The covering $\mathcal{G}'$ is still a standard Cartan graph with common Cartan matrix $A_2^{(1)}$. Furthermore, $\mathcal{G}'$ is connected and simply connected.

Let $(\mathcal G,(\Delta^{X,\operatorname{re}})_{X\in\mathcal X})$ be a
connected and simply connected Cartan graph with a reduced root system.  By
[\citealp{affinenichols2}, Theorem~1.1], it is equivalent to the existence of  a crystallographic
Tits arrangement $(\mathcal A,T)$.  Recall that $T$ is an open convex cone.
Following [\citealp{affinenichols2}, Definition~2.7], the arrangement is
called  {affine} if
\[
 T=\gamma^+:=\{x\in V\mid\gamma(x)>0\}
\]
for some nonzero $\gamma\in V^*$.  Thus, in arbitrary rank, the appropriate
term is  {open half-space}.

For later use, we recall the geometric realization in the form needed here.
Fix a base object $X_0\in\mathcal X$, put $r=|\mathbb I|$, and let
$V=\mathbb R^r$.  Choose a linear isomorphism
\[
 \psi:\mathbb R^{\mathbb I}\longrightarrow V^*,
 \qquad \psi(\alpha_i)=\phi_i.
\]
For every $X\in\mathcal X$, let
$w_X\in\operatorname{Hom}(X_0,X)$ be the unique morphism and set
\[
 \beta_i^X=\psi\bigl(w_X^{-1}(\alpha_i)\bigr),
 \qquad
 B^X=\{\beta_i^X\mid i\in\mathbb I\}.
\]
The chamber corresponding to $X$ is
\[
 K^X
 =\bigcap_{i\in\mathbb I}(\beta_i^X)^+
 =\{x\in V\mid\beta_i^X(x)>0\text{ for every }i\in\mathbb I\},
\]
and the Tits cone is
\begin{equation}\label{eq:tits-cone-convex-hull}
 T=\operatorname{conv}\left(\bigcup_{X\in\mathcal X}K^X\right).
\end{equation}

We now determine this cone when the Cartan matrix is of affine type.

\begin{prop}\label{prop:standard-affine-tits-cone}
Let
$\mathcal G=\mathcal G(\mathbb I,\mathcal X,
(\rho_i)_{i\in\mathbb I},(A^X)_{X\in\mathcal X})$
be a connected, simply connected, standard Cartan graph with a reduced root
system.  Let $A=(a_{ij})_{i,j\in\mathbb I}$ be its common Cartan matrix, and
assume that $A$ is an indecomposable Cartan matrix of affine type.  Then the
Tits arrangement associated with $\mathcal G$ is affine. 
\end{prop}

\begin{proof}
Since $A$ is indecomposable and affine, there exists a vector
$d\in\mathbb N^{\mathbb I}$ with all $d_i>0$ and $Ad=0$; see, for example,
[\citealp{Kacliealg}, Chapter~4].
  Let \[
\delta_0:=\sum_{i\in I}d_i\alpha_i,\qquad
\delta:=\psi(\delta_0)\in V^*.
\]
  With the convention
$
 s_i(\alpha_j)=\alpha_j-a_{ij}\alpha_i,
$
we have
\begin{align*}
 s_i(\delta_0)
 =\sum_{j\in\mathbb I}d_j
   (\alpha_j-a_{ij}\alpha_i)
 =\delta_0-\left(\sum_{j\in\mathbb I}a_{ij}d_j\right)\alpha_i
 =\delta_0
\end{align*}
for every $i\in\mathbb I$.  Hence every morphism in the Weyl groupoid fixes
$\delta_0$.

Let $X\in\mathcal X$.  Since $w_X^{-1}(\delta_0)=\delta_0$, the root basis
of the chamber $K^X$ satisfies
\begin{equation}\label{eq:delta-in-every-root-basis}
 \sum_{i\in\mathbb I}d_i\beta_i^X
 =\sum_{i\in\mathbb I}d_i\psi(w_X^{-1}(\alpha_i))
 =\psi(w_X^{-1}(\delta_0))
 =\delta.
\end{equation}
If $x\in K^X$, then $\beta_i^X(x)>0$ for every $i$.  Since all $d_i$ are
positive, \eqref{eq:delta-in-every-root-basis} implies
$$
 \delta(x)=\sum_{i\in\mathbb I}d_i\beta_i^X(x)>0.
$$
Thus $K^X\subseteq\delta^+$ for every $X$.  The set $\delta^+$ is convex,
so \eqref{eq:tits-cone-convex-hull} yields
\begin{equation}\label{eq:first-tits-inclusion}
 T\subseteq\delta^+.
\end{equation}

It remains to prove the reverse inclusion.  Consider the affine hyperplane
\[
 H_1=\{x\in V\mid\delta(x)=1\}.
\]
For the base object $X_0$, the section
\[
 C_0=K^{X_0}\cap H_1
 =\left\{x\in H_1\mid\phi_i(x)>0
   \text{ for every }i\in\mathbb I\right\}
\]
is an open $(r-1)$-simplex.  Indeed, on $H_1$ the coordinates
$t_i=\phi_i(x)$ satisfy $t_i>0$ and
$\sum_i d_it_i=1$.
Because the Cartan graph is standard, the linear parts of the morphisms in
its Weyl groupoid form the affine Weyl group $W(A)$.  Conversely, every
element of $W(A)$ is represented by a word in the simple reflections
starting at $X_0$.  Consequently, as $X$ runs through $\mathcal X$, the
sections
\begin{equation}\label{eq:alcove-sections}
 C_X=K^X\cap H_1
\end{equation}
are precisely the open alcoves of the classical affine reflection
arrangement on $H_1$, possibly with repetitions.  In particular, the
arrangement on $H_1$ is locally finite, and
\begin{equation}\label{eq:regular-points-covered}
 H_1\setminus\bigcup_{H\in\mathcal A}(H\cap H_1)
 =\bigcup_{X\in\mathcal X}C_X.
\end{equation}
These are the standard alcove properties of an affine Weyl group; see, for
example, [\citealp{H90}, Chapter~4].

We claim that every point of $H_1$ belongs to the convex hull of the open
alcoves.  Let $x\in H_1$.  If $x$ lies on no wall, then
\eqref{eq:regular-points-covered} gives $x\in C_X$ for some $X$.  Suppose
that $x$ lies on one or more walls.  By local finiteness, only finitely many
walls meet a sufficiently small neighbourhood of $x$.  Choose
$u\in\ker\delta$ transverse to every wall containing $x$.  After replacing
$u$ by a sufficiently small positive multiple, the two points
\[
 x_+=x+u,
 \qquad
 x_-=x-u
\]
lie on no wall.  Hence $x_+\in C_{X_+}$ and $x_-\in C_{X_-}$ for suitable
objects $X_+,X_-\in\mathcal X$.  Since
\[
 x=\frac12x_++\frac12x_-,
\]
we obtain
\begin{equation}\label{eq:H1-in-tits-cone}
 H_1\subseteq
 \operatorname{conv}\left(\bigcup_{X\in\mathcal X}C_X\right)
 \subseteq T.
\end{equation}

Finally, let $y\in\delta^+$ and put $\lambda=\delta(y)>0$.  Then
$x=\lambda^{-1}y\in H_1$.  By \eqref{eq:H1-in-tits-cone}, write
\[
 x=\sum_{k=1}^m t_kx_k,
 \qquad
 t_k\geq0,
 \quad \sum_{k=1}^m t_k=1,
 \quad x_k\in K^{X_k}.
\]
Every $K^{X_k}$ is a cone, so $\lambda x_k\in K^{X_k}$.  Therefore
\[
 y=\lambda x=\sum_{k=1}^m t_k(\lambda x_k)\in T.
\]
This proves $\delta^+\subseteq T$.  Together with
\eqref{eq:first-tits-inclusion}, it gives
$T=\delta^+$.
\end{proof}

We apply the proposition to the non-diagonal rank-three Nichols algebra
considered above.

\begin{cor}\label{cor:non-diagonal-affine-tits-cone}
Let $\mathcal{G}'$ be the Cartan graph introduced above.  Under the
correspondence of [\citealp{affinenichols2}, Theorem~1.1], the Tits cone of
the Cartan graph  $\mathcal{G}'$ is
\begin{equation}\label{eq:A2-affine-tits-cone}
 T
 =\left\{x\in\mathbb R^3\ \middle|
   (\phi_1+\phi_2+\phi_3)(x)>0\right\}.
\end{equation}
In particular, $T$ is an open half-space, and $\mathcal B(M)$ is affine in
the sense of [\citealp{affinenichols2}, Definition~2.7].
\end{cor}

\begin{proof}
The graph  $\mathcal{G}'$ is connected, simply connected, and standard, and
its common Cartan matrix is
\[
 A=
 \begin{pmatrix}
 2&-1&-1\\
 -1&2&-1\\
 -1&-1&2
 \end{pmatrix},
\]
the affine Cartan matrix of type $A_2^{(1)}$.  Its positive null vector is
$d=(1,1,1)^{\mathsf T}$.  Thus
\[
 \delta_0=\alpha_1+\alpha_2+\alpha_3,
 \qquad
 \delta=\phi_1+\phi_2+\phi_3.
\]
The conclusion follows from
Proposition~\ref{prop:standard-affine-tits-cone}.
\end{proof}

\begin{rmk}\upshape
The significance of
Corollary~\ref{cor:non-diagonal-affine-tits-cone} is not only that
\(\mathcal B(M)\) is infinite-dimensional. Its associated Cartan
graph is the standard affine Cartan graph of type \(A_2^{(1)}\), and
its Tits cone is an open half-space. Hence \(\mathcal B(M)\) is an
affine Nichols algebra of non-diagonal type in the sense of
\cite{affinenichols2}.
\end{rmk}

\noindent\textbf{Acknowledgments.}
This work was supported by the National Key R\&D Program of China
(No.~2024YFA1013802) and the National Natural Science Foundation of
China (No.~12271243).

\bibliographystyle{plain}\small
	\bibliography{ref}

\end{document}